\documentclass[12pt,english,british]{article}
\usepackage[T1]{fontenc}
\usepackage[latin9]{inputenc}
\usepackage{geometry}
\usepackage{amsmath}
\usepackage{amsthm}
\usepackage{amssymb}
\usepackage{esint}

\makeatletter
\theoremstyle{plain}
\newtheorem{thm}{\protect\theoremname}[section]
\theoremstyle{remark}
\newtheorem{rem}[thm]{\protect\remarkname}
\theoremstyle{definition}
\newtheorem{defn}[thm]{\protect\definitionname}
\theoremstyle{definition}
\newtheorem{example}[thm]{\protect\examplename}
\theoremstyle{plain}
\newtheorem{prop}[thm]{\protect\propositionname}
\theoremstyle{remark}
\newtheorem{notation}[thm]{\protect\notationname}
\theoremstyle{plain}
\newtheorem{lem}[thm]{\protect\lemmaname}
\theoremstyle{plain}
\newtheorem{cor}[thm]{\protect\corollaryname}

\makeatother

\usepackage{babel}
\addto\captionsbritish{\renewcommand{\corollaryname}{Corollary}}
\addto\captionsbritish{\renewcommand{\definitionname}{Definition}}
\addto\captionsbritish{\renewcommand{\examplename}{Example}}
\addto\captionsbritish{\renewcommand{\lemmaname}{Lemma}}
\addto\captionsbritish{\renewcommand{\notationname}{Notation}}
\addto\captionsbritish{\renewcommand{\propositionname}{Proposition}}
\addto\captionsbritish{\renewcommand{\remarkname}{Remark}}
\addto\captionsbritish{\renewcommand{\theoremname}{Theorem}}
\addto\captionsenglish{\renewcommand{\corollaryname}{Corollary}}
\addto\captionsenglish{\renewcommand{\definitionname}{Definition}}
\addto\captionsenglish{\renewcommand{\examplename}{Example}}
\addto\captionsenglish{\renewcommand{\lemmaname}{Lemma}}
\addto\captionsenglish{\renewcommand{\notationname}{Notation}}
\addto\captionsenglish{\renewcommand{\propositionname}{Proposition}}
\addto\captionsenglish{\renewcommand{\remarkname}{Remark}}
\addto\captionsenglish{\renewcommand{\theoremname}{Theorem}}
\providecommand{\corollaryname}{Corollary}
\providecommand{\definitionname}{Definition}
\providecommand{\examplename}{Example}
\providecommand{\lemmaname}{Lemma}
\providecommand{\notationname}{Notation}
\providecommand{\propositionname}{Proposition}
\providecommand{\remarkname}{Remark}
\providecommand{\theoremname}{Theorem}

\begin{document}
\selectlanguage{english}%
\date{}

\numberwithin{equation}{section}
\title{\selectlanguage{british}%
Conditional measure on the Brownian path and other random sets}
\author{\selectlanguage{english}%
\'Abel Farkas}
\maketitle
\selectlanguage{british}%
\begin{abstract}
Let $B$ denote the range of the Brownian motion in $\mathbb{R}^{d}$
($d\geq3$). For a deterministic Borel measure $\nu$ on $\mathbb{R}^{d}$
we wish to find a random measure $\mu$ such that the support of $\mu$
is contained in $B$ and it is a solution to the equation $E(\mu(A))=\nu(A)$
for every Borel set $A$. We discuss when it is possible to find a
solution $\mu$ and in that case we construct the solution. We study
several properties of $\mu$ such as the probability of $\mu\neq0$
and we establish a formula for the expectation of the double integral
with respect to $\mu\times\mu$. We calculate $\mu$ in terms of the
occupation measure when $\nu$ is the Lebesgue measure, i.e. we provide
an explicit deterministic density function of $\mu$ with respect
to the occupation measure. As a conclusion we calculate an explicit
formula for the expectation of the double integral with respect to
the occupation measure. We generalise the theory for more general
random sets in separable, metric, Radon spaces. As an additional example,
we also apply our results to percolation limit sets on boundaries
of trees.
\end{abstract}

\section{Introduction}

We begin by the analogy of slicing measures with straight lines to
introduce the concept. Let $\nu$ be a finite Borel probability measure
on the unit square of the plane, let $\pi(x,y)=x$ be the projection
onto the $x$-axis and $\pi_{*}\nu=\nu\circ\pi^{-1}$ be the projection
measure. Then by the existence of the regular conditional measure
\cite[Theorem 5.1.9]{Durrett} for $\pi_{*}\nu$ almost every $x\in[0,1]$
there exists a Borel probability measure $\nu_{x}$ on the slice $\{x\}\times[0,1]$
such that we get the disintegration formula $\mathrm{d}\nu(x,y)=\mathrm{d}\nu_{x}(x,y)\mathrm{d}\pi_{*}\nu(x)$.
When $\pi_{*}\nu\ll\lambda$, where $\lambda$ denotes the $1$-dimensional
Lebesgue measure, the measure $\nu_{x}$ can be obtained as a weak
limit of certain rescaled restrictions of $\nu$. Assume that $\pi_{*}\nu\ll\lambda$,
then for Lebesgue almost every $x\in[0,1]$ the weak limit of the
measures
\begin{equation}
\frac{\nu\vert_{\pi^{-1}B(x,r)}}{2r}\label{eq:sequence of slice measure}
\end{equation}
exists as $r$ approaches $0$, see \cite[Chapter 10]{Mattila book},
where $\nu\vert_{A}$ denotes the restriction of $\nu$ to $A$, that
is $\nu\vert_{A}(B)=\nu(A\cap B)$ for every Borel sets $A,B\subseteq\mathbb{R}^{2}$.
Let this weak limit be $\mu_{x}$ for Lebesgue almost every $x\in[0,1]$,
then
\begin{equation}
\mathrm{d}\nu(x,y)=\mathrm{d}\mu_{x}(x,y)\mathrm{d}\lambda(x),\label{eq:dL}
\end{equation}
see Mattila \cite[Lemma 3.4]{Mattila-Integral geometric properties of capacities}.
Thus by the uniqueness of the conditional measure

\[
\nu_{x}=\left(\frac{\mathrm{d}\pi_{*}\nu}{\mathrm{d}\lambda}(x)\right)^{-1}\cdot\mu_{x}
\]
for $\pi_{*}\nu$ almost every $x\in[0,1]$, where $\frac{\mathrm{d}\nu}{\mathrm{d}\mu}(x)$
denotes the Radon-Nykodym derivative.

It was shown by Mattila \cite[Lemma 3.4]{Mattila-Integral geometric properties of capacities}
that for a Borel function $f:[0,1]^{2}\longrightarrow\mathbb{R}$
with $\int\left|f\right|\mathrm{d}\nu(z)<\infty$ we have that
\begin{equation}
\intop f\mathrm{d}\mu_{x}(z)=\lim_{r\rightarrow0}(2r)^{-1}\intop_{\pi^{-1}(B(x,r))}f\mathrm{d}\nu(z)\label{eq:integral limit eloallitas}
\end{equation}
for Lebesgue almost every $x\in[0,1]$. Mattila \cite[Theorem 10.7]{Mattila book}
also discusses the double integral of certain kernels with respect
to $\mu_{x}\times\mu_{x}$ which gives a strong tool to analyse the
geometric measure theory of slices of measures and sets.

One can look at it as: We randomly choose a slice $B=\{x\}\times[0,1]$
where we choose $x$ uniformly in $[0,1]$. Then $\mu=\mu_{x}$ is
a random measure supported on the random slice, and
\begin{equation}
\mathrm{d}\nu(z)=\mathrm{d}\mu(z)\mathrm{d}P(x)\label{eq:regreg}
\end{equation}
holds by (\ref{eq:dL}), i.e. $\nu$ is the expectation of $\mu$.
Our main goal in this paper is to construct this kind of slice measures
on random slices but instead of taking the random slices to be straight
line segments we take the slices to be the Brownian path or other
random sets.

Let $Q_{k}(z)$ be the dyadic cube $[\frac{i_{1}}{2^{k}},\frac{i_{1}+1}{2^{k}})\times[\frac{i_{2}}{2^{k}},\frac{i_{2}+1}{2^{k}})$
for $i_{1},i_{2}\in\mathbb{Z}$ such that $z\in Q_{k}(z)$ for some
$z\in[0,1]^{2}$ and let $\mathcal{Q}_{k}=\left\{ Q_{k}(z):z\in[0,1]^{2}\right\} $.
It can be shown that for Lebesgue almost every $x\in[0,1]$ we get
the same weak limit $\mu_{x}$ if we consider the sequence
\begin{equation}
\mu_{k}=\frac{\nu\vert_{\pi^{-1}\left(\pi\left(Q_{k}(x,0)\right)\right)}}{2^{-k}}=\sum_{Q\in\mathcal{Q}_{k}}P(Q\cap B\neq\emptyset)^{-1}\cdot I_{Q\cap B\neq\emptyset}\cdot\nu\vert_{Q}\label{eq:reg sum lim}
\end{equation}
instead of (\ref{eq:sequence of slice measure}), where $I_{Q\cap B\neq\emptyset}$
is the indicator function of the event $Q\cap B\neq\emptyset$. We
can obtain from the analogue of (\ref{eq:integral limit eloallitas})
that for every Borel set $A\subseteq\mathbb{R}^{2}$
\begin{equation}
\lim_{k\rightarrow\infty}\mu_{k}(A)=\mu(A)\label{eq:reg lim}
\end{equation}
almost surely, i.e. for Lebesgue almost every slice.

If $\pi_{*}\nu$ is singular to the Lebesgue measure then by Lebesgue`s
density theorem
\begin{equation}
\lim_{k\rightarrow\infty}\mu_{k}([0,1]^{2})=0\label{eq:sing lim}
\end{equation}
for Lebesgue almost every $x\in[0,1]$, i.e. $\mu=0$ almost surely.
Hence in general we can decompose $\nu$ into two parts
\begin{equation}
\nu=\nu_{R}+\nu_{\perp}\label{eq:decomposition example}
\end{equation}
such that $\pi_{*}\nu_{R}\ll\lambda$ and $\pi_{*}\nu_{\perp}\perp\lambda$,
one part corresponds to a vanishing limit (\ref{eq:sing lim}), the
other part corresponds to an $\mathcal{L}^{1}$ limit (\ref{eq:regreg}).
Thus for the almost sure weak limit $\mu$ of the sequence of random
measures $\mu_{k}$ we obtain the disintegration formula
\begin{equation}
\mathrm{d}\nu_{R}(z)=\mathrm{d}\mu(z)\mathrm{d}P.\label{eq:integ decomp example}
\end{equation}

Our main goal is to show the existence of the limit of (\ref{eq:reg sum lim})
in the case when $B$ is the Brownian path and to obtain a disintegration
formula as in (\ref{eq:integ decomp example}). For the purpose of
application in the geometric measure theory of random intersections
it is beneficial to have good control over the double intergral with
respect $\mu\times\mu$.
\begin{thm}
\label{prop:Brownian conditional meaasure}Let $B$ be a Brownian
path in $\mathbb{R}^{d}$ for $d\geq3$ and let $\nu$ be a locally
finite Borel measure on $\mathbb{R}^{d}$ such that $\nu(\{0\})=0$.
Then $\nu=\nu_{R}+\nu_{\perp}$ such that there exists a Borel set
$A$ such that $\nu_{\perp}(\mathbb{R}^{d}\setminus A)=0$, $P(B\cap A\neq\emptyset)=0$
and there exists a random, locally finite, Borel measure $\mu$ supported
on $B$ such that
\[
\mathrm{d}\mu(z)\mathrm{d}P=\mathrm{d}\nu_{R}(z)
\]
and
\[
\mathrm{d}\mu(x)\mathrm{d}\mu(y)\mathrm{d}P=\frac{\left\Vert x\right\Vert ^{d-2}+\left\Vert y\right\Vert ^{d-2}}{\left\Vert x-y\right\Vert ^{d-2}}\mathrm{d}\nu_{R}(x)\mathrm{d}\nu_{R}(y).
\]
\end{thm}

\begin{rem}
\label{rem:nohope for singular}Note that if $\mu$ is a random measure
supported on the Brownian path $B$ and $A$ is a deterministic Borel
set such that $P(B\cap A\neq\emptyset)=0$ then $\mu(A)\leq\mu(\mathbb{R}^{d}\setminus B)=0$
almost surely and so
\[
\int\intop_{A}\mathrm{d}\mu(z)\mathrm{d}P=E(\mu(A))=0.
\]
This means that there is no hope to satisfy (\ref{eq:regreg}) for
$\nu_{\perp}$ in Theorem \ref{prop:Brownian conditional meaasure}.
We restate Theorem \ref{prop:Brownian conditional meaasure} in a
stronger form of Theorem \ref{thm:Brownian cond measure: Main}.
\end{rem}

Our construction works for more general random closed sets than the
Brownian path and in more general metric spaces than $\mathbb{R}^{d}$.
We use the sum in (\ref{eq:reg sum lim}) to define the sequence $\mu_{k}$
in the general case (see Section \ref{sec:Notations}).

\subsection{Notations\label{sec:Notations}}

\subsubsection{General assumptions throughout the paper}
\begin{defn}
A locally finite Borel measure $\nu$ on a topological space $X$
is called \textit{inner regular} if for every Borel set $A\subseteq X$
\[
\nu(A)=\sup\left\{ \nu(K):K\subseteq A\mathrm{\,compact}\right\} 
\]
and $\nu$ is called \textit{outer regular} if for every Borel set
$A\subseteq X$
\[
\nu(A)=\inf\left\{ \nu(U):A\subseteq U\mathrm{\,open}\right\} .
\]
A Hausdorff topological space $X$ is called a \textit{Radon space}
if every locally finite Borel measure on $X$ is outer and inner regular.
\end{defn}

Every Polish space is a Radon space \cite[Theorem 17.11]{Kechris}
and so every locally compact, separable, metric space is a Radon space.

Throughout the paper let $(X,d)$ be a separable metric space that
is also a Radon space. We notoriously use the inner regularity of
Borel measures on $X$ throughout the paper. Let $\varphi:[0,\infty)\longrightarrow[0,\infty]$
be a nonnegative, continuous, monotone decreasing function with finite
values on $(0,\infty)$. We consider the composition kernel $\varphi(d(x,y))$
on $X\times X$ which we denote by $\varphi(x,y)$. We note that we
use $\varphi$ to denote both $\varphi(r)$ and $\varphi(x,y)$ but
in the context it should be clear depending on what is the domain
of $\varphi$.
\begin{example}
\label{exa:kernel example}Commonly used examples are the Riesz kernel
$\varphi(x,y)=\left\Vert x-y\right\Vert ^{-\alpha}$ and the logarithmic
kernel $\varphi(r)=\max\{0,\log(1/r)\}$.
\end{example}

Let $\mathcal{Q}_{k}$ be a sequence of countable families of Borel
subsets of $X$ such that $Q\cap S=\emptyset$ for $Q,S\in\mathcal{Q}_{k}$,
$Q\neq S$ for all $k\in\mathbb{N}$ and
\begin{equation}
\lim_{k\rightarrow\infty}\sup\left\{ \mathrm{diam}(Q):Q\in\mathcal{Q}_{k}\right\} =0,\label{eq:diameter_goes_to0}
\end{equation}
where $\mathrm{diam}$ denotes the \textit{diameter} in $X$. We further
assume that for every $Q\in\mathcal{Q}_{k}$, $k>1$ there exists
a unique $D\in\mathcal{Q}_{k-1}$ such that 
\begin{equation}
Q\subseteq D.\label{eq:unique-subset}
\end{equation}
Define $X_{0}:=\bigcap_{k=1}^{\infty}(\bigcup_{Q\in\mathcal{Q}_{k}}Q)$.
\begin{example}
\label{exa:example Q_k}Let $\mathcal{Q}_{k}=\left\{ [\frac{i_{1}}{2^{k}},\frac{i_{1}+1}{2^{k}})\times\dots\times[\frac{i_{d}}{2^{k}},\frac{i_{d}+1}{2^{k}})\setminus\left\{ 0\right\} :i_{1},\dots,i_{d}\in\mathbb{Z}\right\} $
for $k\in\mathbb{N}$, i.e. the dyadic cubes of side length $2^{-k}$
which are left closed, right open and we exclude the origin $0$ from
the one that contains it. Then $X_{0}=\mathbb{R}^{d}\setminus\left\{ 0\right\} $.
\end{example}

Let $\nu$ be a finite Borel measure on $X$. Throughout most of the
paper we make the following assumption in our statements  that
\begin{equation}
\nu(X\setminus X_{0})=0.\label{eq:X_0}
\end{equation}

Let $(\Omega,\mathcal{A},P)$ be a probability space and $B=B_{\omega}\subseteq X$
be a random closed set ($\omega\in\Omega$) such that $\{B\cap K\neq\emptyset\}\in\mathcal{A}$
for every compact set $K\subseteq X$ and $\{B\cap Q\neq\emptyset\}\in\mathcal{A}$
for every $Q\in\mathcal{Q}_{k}$, $k\in\mathbb{N}$. We assume that
\begin{equation}
P(Q\cap B\neq\emptyset)>0\label{eq:psoitivity of probability}
\end{equation}
for every $Q\in\mathcal{Q}_{k}$. Note that (\ref{eq:psoitivity of probability})
can be obtained by discarding those $Q$ from $\mathcal{Q}_{k}$ for
which $P(Q\cap B\neq\emptyset)=0$.

We write
\begin{equation}
\mathcal{C}_{k}(\nu)=\sum_{Q\in\mathcal{Q}_{k}}P(Q\cap B\neq\emptyset)^{-1}\cdot I_{Q\cap B\neq\emptyset}\cdot\nu\vert_{Q}.\label{eq:C_k def}
\end{equation}
It follows from (\ref{eq:X_0}), (\ref{eq:psoitivity of probability})
and that the elements of $\mathcal{Q}_{k}$ are disjoint that
\begin{equation}
E(\mathcal{C}_{k}(\nu)(A))=\nu(A)\label{eq:martingal expectation inequality}
\end{equation}
for every Borel set $A\subseteq X$.

\subsubsection{Special assumptions for our main results\label{subsec:Special-assumptions}}

For our main results we make the assumptions of Section \ref{subsec:Special-assumptions}
on $\varphi$, $\mathcal{Q}_{k}$ and $B$. We list these assumptions
below, however, we will note in the statement of the results and in
the text when we make these assumptions.

For some $\delta>0$ there exist $c_{2},c_{3}<\infty$ such that
\begin{equation}
\varphi(r)\leq c_{2}\varphi(r\cdot(1+2\delta))+c_{3}\label{eq:Kernel restriction}
\end{equation}
for all $r>0$. We assume that
\begin{equation}
\varphi(0)=\infty,\label{eq:phi =00003Dinfty}
\end{equation}
which ensures that whenever $\nu(\{x\})>0$ for some $x\in X$ and
Borel measure $\nu$ then $\iint\varphi(x,y)\mathrm{d}\nu(x)\mathrm{d}\nu(y)=\infty$.
See Example \ref{exa:kernel example}.

There exists $\delta>0$ (that is the same $\delta$ as in (\ref{eq:Kernel restriction}))
and $M_{\delta}<\infty$ independent of $k$ such that for every $Q\in\mathcal{Q}_{k}$
\begin{equation}
\#\left\{ S\in\mathcal{Q}_{k}:\max\left\{ \mathrm{diam}(Q),\mathrm{diam}(S)\right\} \geq\delta\cdot\mathrm{dist}(Q,S)\right\} \leq M_{\delta},\label{eq:bounded sundivision}
\end{equation}
where $\mathrm{dist}(Q,S)=\inf_{x\in Q,y\in S}\left\Vert x-y\right\Vert $
and $\#A$ denotes the cardinality of $A$. There exists $0<M<\infty$,
independent of $k$, such that
\begin{equation}
0<\mathrm{diam}(Q)/M\leq\mathrm{diam}(S)\leq\mathrm{diam}(Q)\cdot M<\infty\label{eq:same size}
\end{equation}
for every $Q,S\in\mathcal{Q}_{k}$, $k\in\mathbb{N}$.

The \textit{$\varphi$-energy} of a Borel measure $\nu$ on $X$ is
\begin{equation}
I_{\varphi}(\nu)=\int\int\varphi(x,y)\mathrm{d}\nu(x)\mathrm{d}\nu(y).\label{eq:energy}
\end{equation}
The \textit{$\varphi$-capacity} of a Borel set $K\subseteq X$ is
\begin{equation}
C_{\varphi}(K)=\sup\left\{ I_{\varphi}(\nu)^{-1}:\nu\,\mathrm{is\,a\,Borel\,probability\,measure\,on\,K}\right\} .\label{eq:capacity}
\end{equation}
When $\varphi(r)=r^{-\alpha}$ for some $\alpha\geq0$ then we write
$I_{\alpha}(\nu)=I_{\varphi}(\nu)$ and $C_{\alpha}(K)=C_{\varphi}(K)$.

There exists $a>0$ such that
\begin{equation}
aC_{\varphi}(Q)\leq P(Q\cap B\neq\emptyset)\label{eq:lower hitting prob}
\end{equation}
for every $Q\in\mathcal{Q}_{k}$. Hence if $C_{\varphi}(Q)>0$ for
every $Q\in\mathcal{Q}_{k}$ then (\ref{eq:psoitivity of probability})
holds.

There exists $0<\delta<1$ and $0<c<\infty$ such that whenever $Q\in\mathcal{Q}_{k}$,
$S\in\mathcal{Q}_{n}$ and $\max\left\{ \mathrm{diam}(Q),\mathrm{diam}(S)\right\} <\delta\cdot\mathrm{dist}(Q,S)$
then
\begin{equation}
P(Q\cap B\ne\emptyset\,\mathrm{and}\,S\cap B\ne\emptyset)\leq c\cdot P(Q\cap B\ne\emptyset)\cdot P(S\cap B\ne\emptyset)\cdot\varphi(\mathrm{dist}(Q,S)).\label{eq:capacity-independence}
\end{equation}
We further assume that $\delta>0$ is the same value for (\ref{eq:Kernel restriction}),
(\ref{eq:bounded sundivision}) and (\ref{eq:capacity-independence}).
We would like to note that the assumption that $\delta<1$ is not
necessary, we just use it as a convenient assumption.
\begin{example}
\label{exa:example Q_k^i}Let $\mathcal{Q}_{k}^{i}=\left\{ [\frac{i_{1}}{2^{k}},\frac{i_{1}+1}{2^{k}})\times\dots\times[\frac{i_{d}}{2^{k}},\frac{i_{d}+1}{2^{k}})\subseteq[-2^{i},2^{i})^{d}\setminus[-\frac{1}{2^{i}},\frac{1}{2^{i}})^{d}:i_{1},\dots,i_{d}\in\mathbb{Z}\right\} $
for some $k,i\in\mathbb{N}$, $i\leq k$, i.e. the dyadic cubes of
side length $2^{-k}$ which are contained in $[-2^{i},2^{i})^{d}\setminus[-\frac{1}{2^{i}},\frac{1}{2^{i}})^{d}$.
\end{example}

We show, in Section \ref{subsec:Conditions-on}, that for fixed $i\in\mathbb{N}$
for $\mathcal{Q}_{k}^{i}$, $(k\geq i)$ in Example \ref{exa:example Q_k^i}
we have that (\ref{eq:diameter_goes_to0}), (\ref{eq:unique-subset}),
(\ref{eq:psoitivity of probability}), (\ref{eq:bounded sundivision}),
(\ref{eq:same size}), (\ref{eq:lower hitting prob}) and (\ref{eq:capacity-independence})
hold for sufficient constants $0<\delta^{i}<1$, $M_{\delta}^{i}<\infty$,
$0<M^{i}<\infty$, $a^{i}<\infty$ and $0<c^{i}<\infty$ when $B$
is the Brownian path in $\mathbb{R}^{d}$ ($d\geq3$).

\subsection{Decomposition of measures}

In the spirit of (\ref{eq:decomposition example}) we would like to
decompose $\nu$ into a vanishing part and a part for which we obtain
convergence in $\mathcal{L}^{1}$. Similar results to the following
proposition was published by Kahane \cite[Section 3]{Kahane-measdecomp}.
However, we are not aware that this kind of results appeared in the
literature in English.
\begin{prop}
\label{decomposition}Let $\nu$ be a locally finite Borel measure
on $X$. There exist two locally finite, Borel measures $\nu_{\varphi R}=\nu_{R}$
and $\nu_{\varphi\perp}=\nu_{\perp}$ with the following properties:

i) $\nu=\nu_{R}+\nu_{\perp}$

ii) $\nu_{R}\perp\nu_{\perp}$

iii) $\nu_{\perp}$ is singular to every locally finite Borel measure
with finite $\varphi$-energy

iv) there exists a sequence of disjoint Borel sets $(A_{n})_{n\in\mathbb{N}}$
such that $\nu_{R}=\nu\vert_{\cup_{n\in\mathbb{N}}A_{n}}=\sum_{n\in\mathbb{N}}\nu\vert_{A_{n}}$
and $I_{\varphi}(\nu\vert_{A_{n}})<\infty$.
\end{prop}

\begin{notation}
\label{nota:decomp meas}We call $\nu_{R}$ the $\varphi$-regular
part of $\nu$ and we call $\nu_{\bot}$ the $\varphi$-singular part
of $\nu$. These are uniquely determined by $\nu$ and $\varphi$.
When $\varphi(r)=r^{-\alpha}$ for some $\alpha\geq0$ then we say
that $\nu_{R}$ is the $\alpha$-regular part of $\nu$ and $\nu_{\bot}$
is the $\alpha$-singular part of $\nu$.
\end{notation}

If $f:X\longrightarrow\mathbb{R}$ is a nonnegative Borel function
then we denote by
\[
f(x)\mathrm{d}\nu(x)
\]
the measure $\tau$ defined by $\tau(A)=\intop_{A}f(x)\mathrm{d}\nu(x)$.
\begin{prop}
\label{prop:abs cont decomposition}Let $\nu$ be a locally finite
Borel measure on $X$ and let $\tau\ll\nu$ be a locally finite Borel
measure on $X$. Then $\frac{\mathrm{d}\tau}{\mathrm{d}\nu}(x)\mathrm{d}\nu_{\varphi R}(x)=\mathrm{d}\tau_{\varphi R}(x)$
and $\frac{\mathrm{d}\tau}{\mathrm{d}\nu}(x)\mathrm{d}\nu_{\varphi\perp}(x)=\mathrm{d}\tau_{\varphi\perp}(x)$.
\end{prop}

The $\varphi$-singular part $\nu_{\bot}$ of a measure is carried
by a set of zero $\varphi$-capacity.
\begin{prop}
\label{prop:singular felbontas loc fin intro}If $\nu$ is a locally
finite Borel measure that is singular to every finite Borel measure
with finite $\varphi$-energy then there exits a Borel set $Z\subseteq X$
such that $\nu(X\setminus Z)=0$ and $C_{\varphi}(Z)=0$.
\end{prop}

We prove Proposition \ref{decomposition}, Proposition \ref{prop:abs cont decomposition}
and Proposition \ref{prop:singular felbontas loc fin intro} in Section
\ref{sec:Decomposition-of-measure}.

\subsection{\label{subsec:Conditional-measure}Summary of the main results}

Let $\mu$ and $\mu_{k}$ be a sequence of random, finite, Borel measures
on $X$ (for the definition of random, finite, Borel measures see
Definition \ref{def:random finite measure}). We say that \textit{$\mu_{k}$
weakly converges to $\mu$ subsequentially in probability} if for
every subsequence $\left\{ \alpha_{k}\right\} _{k=1}^{\infty}$ of
$\mathbb{N}$ there exists a subsequence $\left\{ \beta_{k}\right\} _{k=1}^{\infty}$
of $\left\{ \alpha_{k}\right\} _{k=1}^{\infty}$ and an event $H\in\mathcal{A}$
with $P(H)=1$ such that $\mu_{\beta_{k}}$ converges weakly to $\mu$
on the event $H$.

Let $\mu$ and $\mu_{k}$ be a sequence of random, locally finite,
Borel measures on $X$ (for the definition of random, locally finite,
Borel measures see Definition \ref{def:random loc fin measure}).
We say that \textit{$\mu_{k}$ vaguely converges to $\mu$ in probability}
if $\intop_{X}f(x)\mathrm{d}\mu_{k}(x)$ converges to $\intop_{X}f(x)\mathrm{d}\mu(x)$
in probability for every compactly supported continuous function $f$.

For a Borel measure $\nu$ on $X$ let $\mathrm{supp}(\nu)$ denote
the \textit{support of $\nu$,} that is the smallest closed set with
full measure, i.e.
\[
\mathrm{supp}(\nu)=\bigcap\left\{ K:\nu(X\setminus K)=0,K\subseteq X\mathrm{\,is\,a\,closed\,set}\right\} .
\]

Below we define the main object of the paper.
\begin{defn}
\label{def:def of cond meas}Let $\nu$ be a finite, Borel measure
on $X$. If for every finite, Borel measure $\tau\ll\nu$ there exists
a random, finite, Borel measure $\mathcal{C}(\tau)$ that satisfies
the following:

\textit{i.)} $\mathcal{C}_{k}(\tau)$ weakly converges to $\mathcal{C}(\tau)$
subsequentially in probability,

\textit{ii.)} the sequence $\intop_{X}f(x)\mathrm{d}\mathcal{C}_{k}(\tau)(x)$
converges to a random variable $S_{\tau}(f)$ in probability with
$E(\intop_{X}f(x)\mathrm{d}\mathcal{C}(\tau)(x))=E(S_{\tau}(f))=\intop_{X}f(x)\mathrm{d}\tau_{R}(x)$
for every $f:X\longrightarrow\mathbb{R}$ Borel measurable function
such that $\intop_{X}\left|f(x)\right|\mathrm{d}\tau(x)<\infty$,

\textit{iii.)} for every countable collection of deterministic Borel
measurable functions $f_{n}:X\longrightarrow\mathbb{R}$ with $\intop_{X}\left|f_{n}(x)\right|\mathrm{d}\tau(x)<\infty$
we have that $\intop_{X}f_{n}(x)\mathrm{d}\mathcal{C}(\tau)(x)=S_{\tau}(f_{n})$
for every $n\in\mathbb{N}$ almost surely,

\textit{iv.)} for every countable collection of deterministic, Borel
sets $A_{n}\subseteq X$ with $\tau(A_{n})<\infty$ we have that $S_{\tau}(\chi_{A_{n}})=\mathcal{C}(\tau)(A_{n})$
for every $n\in\mathbb{N}$ almost surely,

\textit{v.)} $E(\mathcal{C}(\tau)(A))=E(S_{\tau}(\chi_{A}))=\tau_{R}(A)\leq\tau(A)$
for every Borel set $A\subseteq X$ with $\tau(A)<\infty$,

\textit{vi.)} $\mathcal{C}_{k}(\tau_{R})(A)$ converges to $\mathcal{C}(\tau)(A)$
in $\mathcal{L}^{1}$ for every Borel set $A\subseteq X$ with $\tau(A)<\infty$

\textit{vii.)} $\mathcal{C}(\tau_{\bot})=0$ almost surely,

\textit{viii.)} $\mathcal{C}(\tau)=\mathcal{C}(\tau_{R})$ almost
surely,

\textit{ix.)} if $\tau=\sum_{i=1}^{\infty}\tau^{i}$ for a sequence
of finite, Borel measures $\tau^{i}$ then $\mathcal{C}(\tau)=\sum_{i=1}^{\infty}\mathcal{C}(\tau^{i})$
almost surely,

\textit{x.)} if $f:X\longrightarrow\mathbb{R}$ is a nonnegative Borel
function such that $\intop_{X}f(x)\mathrm{d}\tau(x)<\infty$ then
$\mathcal{C}(f(x)\mathrm{d}\tau(x))=f(x)\mathrm{d}\mathcal{C}(\tau)(x)$,
in particular, if $\gamma\in[0,\infty)$ then $\mathcal{C}(\gamma\cdot\tau)=\gamma\cdot\mathcal{C}(\tau)$
almost surely,

\textit{xi.)} $\mathrm{supp}\mathcal{C}(\tau)\subseteq\mathrm{supp}\tau\cap B$
almost surely,

\noindent then we say that the \textit{conditional measure of $\nu$
on $B$ exists with respect to $\mathcal{Q}_{k}$ ($k\geq1$) with
regularity kernel $\varphi$} and it is\textit{ $\mathcal{C}(\nu)$}.

If $X$ is locally compact and $\nu$ is a locally finite, Borel measure
on $X$ and for every locally finite, Borel measure $\tau\ll\nu$
there exists a random, locally finite, Borel measure $\mathcal{C}(\tau)$
that satisfies \textit{ii.)-viii.), xi.)} and additionally also satisfies
the following:

\textit{i{*}.)} $\mathcal{C}_{k}(\tau)$ vaguely converges to $\mathcal{C}(\tau)$
in probability,

\textit{ix{*}.)} if $\tau=\sum_{i=1}^{\infty}\tau^{i}$ for a sequence
of locally finite, Borel measures $\tau^{i}$ then $\mathcal{C}(\tau)=\sum_{i=1}^{\infty}\mathcal{C}(\tau^{i})$
almost surely,

\textit{x{*}.)} if $f:X\longrightarrow\mathbb{R}$ is a nonnegative
Borel function such that for every $y\in X$ there exists a neigbourhood
$U$ of $y$ such that $\intop_{U}f(x)\mathrm{d}\tau(x)<\infty$ then
$\mathcal{C}(f(x)\mathrm{d}\tau(x))=f(x)\mathrm{d}\mathcal{C}(\tau)(x)$,
in particular, if $\gamma\in[0,\infty)$ then $\mathcal{C}(\gamma\cdot\tau)=\gamma\cdot\mathcal{C}(\tau)$
almost surely,

\noindent then we say that the \textit{conditional measure of $\nu$
on $B$ exists with respect to $\mathcal{Q}_{k}$ ($k\geq1$) with
regularity kernel $\varphi$} and it is\textit{ $\mathcal{C}(\nu)$}.
\end{defn}

\begin{rem}
We would like to note that when we say that the conditional measure
$\mathcal{C}(\nu)$ of $\nu$ on $B$ exists with respect to $\mathcal{Q}_{k}$
($k\geq1$) with regularity kernel $\varphi$ some may think it is
a bit incorrect because we actually mean that $\mathcal{C}(\tau)$
exists for every $\tau\ll\nu$. It follows from the definition that
if $\mathcal{C}(\nu)$ exists then $\mathcal{C}(\tau)$ also exists
for every $\tau\ll\nu$. However, there is a characterisation whether
$\mathcal{C}(\nu)$ exists which need no information about any $\tau\ll\nu$,
see Theorem \ref{thm:Gen exist both}. So we simply say that $\mathcal{C}(\nu)$
exists when we do mean that for every $\tau\ll\nu$ the Properties
in Definition \ref{def:def of cond meas} are all satisfied. Throughout
the paper when we write $S(f)$ for a function $f$ and we do not
indicate otherwise in lower case then we mean $S_{\nu}(f)$ by $S(f)$.
\end{rem}

\begin{rem}
By Property \textit{ii.) }and \textit{iv.)} of Definition \ref{def:def of cond meas}
it follows that for every countable collection of Borel sets $A_{n}$
we have that $\mathcal{C}_{k}(\nu)(A_{n})$ converges to $\mathcal{C}(\nu)(A_{n})$
in probability as $k$ goes to $\infty$. We note that in general
it is not true that $\mathcal{C}(\nu)(A)$ is the limit of $\mathcal{C}_{k}(\nu)(A)$
in probability for every Borel set $A$ simultaneously. Let $\nu=\lambda$
and $A$ be a set of $0$ Lebesgue measure. Then $\mathcal{C}_{k}(\nu)(A)=0$
for every $k$. However, $\mathcal{C}(\nu)(A)=0$ cannot always hold
for every such $A$ simultaneously because usually $\mathcal{C}(\nu)$
is supported on a set of $0$ Lebesgue measure by Property \textit{xi.)}
of Definition \ref{def:def of cond meas}. For instance, in the case
of $B$ being the Brownian path.
\end{rem}

We defined the conditional measure of finite Borel measures and in
locally compact spaces we defined the conditional measure of locally
finite Borel measures. To avoid confusion we need to show that the
two definitions of the conditional measure of finite measures in locally
compact spaces are the same. Proposition \ref{prop:two def are the same}
is proven at the end of Section \ref{subsec:SUB-General-existance-of}.
\begin{prop}
\label{prop:two def are the same}If $X$ is locally compact and $\nu$
is a finite Borel measure then the two definitions of the conditional
measure of $\nu$ on $B$ with respect to $\mathcal{Q}_{k}$ ($k\geq1$)
with regularity kernel $\varphi$ in Definition \ref{def:def of cond meas}
are equivalent and the limits $\mathcal{C}(\nu)$ in Property i.)
and in Property i{*}.) are the same almost surely.
\end{prop}

We develop the theory of weak convergence subsequentially in probability
and vague convergence in probability in Section \ref{sec:Convergence-of-random},
that is essential in showing the existence of the conditional measure.
The proofs are based on classical functional analysis and classical
probability. The following proposition provides a tool to show the
existence of the conditional measure.
\begin{prop}
\label{prop:simp conv prop}Let $\nu$ be a deterministic, finite,
Borel measure on $X$. Let $\mu_{k}$ be a sequence of random, finite,
Borel measures on $X$ such that $\mu_{k}\ll\nu$ almost surely for
every $k$, there exists $c:X\longrightarrow\mathbb{R}$ such that
$\int c(x)\mathrm{d}\nu(x)<\infty$ and $E\left(\frac{\mathrm{d}\mu_{k}}{\mathrm{d}\nu}(x)\right)\leq c(x)$
for $\nu$ almost every $x\in X$ for every $k$. Assume that $\mu_{k}(A)$
converges to a random variable $\mu_{\infty}(A)$ in probability for
every compact set $A\subseteq X$. Then $\mu_{k}$ weakly converges
to a random, finite, Borel measure $\mu$ subsequentially in probability.
\end{prop}

Proposition \ref{prop:simp conv prop} also extends to locally finite
measures when $X$ is locally compact and we consider the vague convergence
in probability. We prove these results in Section \ref{sec:Convergence-of-random}.
The notion of weak convergence subsequentially in probability was
considered independently by Berestycki in \cite[Section 6]{Berestycki}
in the context of Gaussian multiplicative chaos, the existence of
a random measure in that paper depends on a similar result to Proposition
\ref{prop:simp conv prop}, though it is not explicitly stated. The
following theorem gives a characterisation when $\mathcal{C}(\nu)$
exists and is proved in Section \ref{subsec:SUB-General-existance-of}.
\begin{thm}
\label{thm:Gen exist both}Let $\nu$ be a finite, Borel measure on
$X$ or let $X$ be locally compact and $\nu$ be a locally finite,
Borel measure on $X$. Assume that $\nu(X\setminus X_{0})=0$. Then
the conditional measure $\mathcal{C}(\nu)$ of $\nu$ on $B$ exists
with respect to $\mathcal{Q}_{k}$ ($k\geq1$) with regularity kernel
$\varphi$ if and only if $\mathcal{C}_{k}(\nu\vert_{D})(X)=\mathcal{C}_{k}(\nu)(D)$
converges in $\mathcal{L}^{1}$ for every compact set $D\subseteq X_{0}$
with $I_{\varphi}(\nu\vert_{D})<\infty$ and $\mathcal{C}_{k}(\nu_{\perp})(D)$
converges to $0$ in probability for every compact set $D\subseteq X_{0}$.
\end{thm}

The sequence $\mathcal{C}_{k}(\nu)$ might give the impression of
a $T$-martingale that was introduced by Kahane \cite{Kahane-positive martingales}.
However, when $B$ is the Brownian path $\mathcal{C}_{k}(\nu)$ is
not a $T$-martingale with respect to the natural filtration $\mathcal{F}_{k}=\sigma\{B\cap Q\neq\emptyset\}_{Q\in\mathcal{Q}_{k}}$
and we cannot prove almost sure convergence of $\mathcal{C}_{k}(\nu)$,
that is why we needed to develop the convergence of random measures
in probability. Despite that, $\mathcal{C}_{k}(\nu)$ and its limit
$\mathcal{C}(\nu)$ if the conditional measure exists, exhibits many
similar properties to Kahane`s $T$-martingales. To get around the
trouble that $\mathcal{C}_{k}(\nu)$ is not necessarily a $T$-martingale
we define the following kernels that are key objects in showing that
the limit of $\mathcal{C}_{k}(\nu)$ exists.
\begin{notation}
\label{notaQ_k(X) def}For $x\in X$ and $k\in\mathbb{N}$ let $Q_{k}(x)=Q$
if $x\in Q$ for some $Q\in\mathcal{Q}_{k}$ and $Q_{k}(x)=\emptyset$
otherwise. There is at most one such $Q$ since elements of $\mathcal{Q}_{k}$
are disjoint hence $Q_{k}(x)$ is well-defined.
\end{notation}

\begin{defn}
\label{def:F(x,y) def}For $k,n\in\mathbb{N}$ let $F_{k,n}:X\times X\longrightarrow\mathbb{R}$
be the nonnegative function
\[
F_{k,n}(x,y)=\begin{cases}
\frac{P(Q_{k}(x)\cap B\ne\emptyset\,and\,Q_{n}(y)\cap B\ne\emptyset)}{P(Q_{k}(x)\cap B\ne\emptyset)\cdot P(Q_{n}(y)\cap B\ne\emptyset)} & \mathrm{if\,}Q_{k}(x)\neq\emptyset,Q_{n}(y)\neq\emptyset\\
0 & \mathrm{otherwise}
\end{cases}.
\]
\end{defn}

\begin{defn}
\label{def:upper and loweer F(x,y)}We define the following functions
\[
\overline{F}_{N}(x,y)=\sup_{n,k\geq N}F_{k,n}(x,y)
\]
and
\[
\underline{F}_{N}(x,y)=\inf_{n,k\geq N}F_{k,n}(x,y)
\]
and their limits
\[
\overline{F}(x,y)=\limsup_{N\rightarrow\infty}\overline{F}_{N}(x,y)=\lim_{N\rightarrow\infty}\overline{F}_{N}(x,y)
\]
and
\[
\underline{F}(x,y)=\liminf_{N\rightarrow\infty}\underline{F}_{N}(x,y)=\lim_{N\rightarrow\infty}\underline{F}_{N}(x,y).
\]
\end{defn}

\begin{rem}
\label{rem:controlled F}If (\ref{eq:capacity-independence}) holds
for some $\delta>0$ then
\[
\underline{F}(x,y)\leq\overline{F}(x,y)\leq c\cdot\varphi(x,y)
\]
for $x\neq y$.
\end{rem}

In Section \ref{sec:Degenerate-case} we prove that, if $C_{\varphi}(D)=0$
for some compact set $D\subseteq X_{0}$ that implies that $B\cap D=\emptyset$
almost surely, then $\mathcal{C}_{k}(\nu_{\bot})$ converges to $0$
weakly subsequentially in probability. If $X$ is locally compact
then even for locally finite $\nu$ it follows that $\mathcal{C}_{k}(\nu_{\bot})$
converges to $0$ vaguely in probability. Throughout the paper we
assume that $\nu(X\setminus X_{0})=0$ because when $\nu(X_{0})=0$
then $\mathcal{C}_{k}(\nu)$ converges to $0$ (see Section \ref{sec:Degenerate-case}).
The existence of the kernel $F(x,y)=\underline{F}(x,y)=\overline{F}(x,y)$
is the key assumption in order to prove the existence of the conditional
measure in the nondegenerate case. Using the observation that $E\left(\mathcal{C}_{k}(\nu)(A)^{2}\right)=\intop_{A}\intop_{A}F_{k,k}(x,y)\mathrm{d}\nu(x)\mathrm{d}\nu(y)$,
we show, throughout Section \ref{sec:-boundedness} and \ref{sec:Non-degenerate-percolation-of},
that if $I_{\varphi}(\nu)<\infty$ then $\mathcal{C}_{k}(\nu)$ converges
in $\mathcal{L}^{2}$ and so in $\mathcal{L}^{1}$. Finally, using
the measure decomposition result of Proposition \ref{decomposition}
we conclude one of our deepest result in Section \ref{sec:Existence-of-the}
and Section \ref{sec:Double-integration}. Theorem \ref{cor:Double integral when W exists-summed}
is proven at the end of Section \ref{sec:Double-integration}.
\begin{thm}
\label{cor:Double integral when W exists-summed}Assume that (\ref{eq:phi =00003Dinfty}),
(\ref{eq:same size}) and (\ref{eq:lower hitting prob}) hold and
there exists $0<\delta<1$ such that (\ref{eq:Kernel restriction}),
(\ref{eq:capacity-independence}) and (\ref{eq:bounded sundivision})
hold. Assume that $F(x,y)=\underline{F}(x,y)=\overline{F}(x,y)$ for
every $(x,y)\in X\times X$. Assume that  if $C_{\varphi}(D)=0$ for
some compact set $D\subseteq X_{0}$ then $B\cap D=\emptyset$ almost
surely. Let either $\nu$ and $\tau$ be finite Borel measures on
$X$ or $X$ be locally compact and $\nu$ and $\tau$ be locally
finite Borel measures on $X$. Assume that $\nu(X\setminus X_{0})=0$
and $\tau(X\setminus X_{0})=0$. Then the conditional measure $\mathcal{C}(\nu)$
of $\nu$ and $\mathcal{C}(\tau)$ of $\tau$ on $B$ exist with respect
to $\mathcal{Q}_{k}$ ($k\geq1$) with regularity kernel $\varphi$
and 
\begin{equation}
E\left(\int\int f(x,y)\mathrm{d}\mathcal{C}(\nu)(x)\mathrm{d}\mathcal{C}(\tau)(y)\right)=\intop\intop F(x,y)f(x,y)\mathrm{d}\nu_{R}(x)\mathrm{d}\tau_{R}(y)\label{eq:double int form intro}
\end{equation}
for every $f:X\times X\longrightarrow\mathbb{R}$ Borel function with
$\intop\intop F(x,y)\left|f(x,y)\right|\mathrm{d}\nu_{R}(x)\mathrm{d}\tau_{R}(y)<\infty$,
in particular if $\intop_{X}\intop_{X}\varphi(x,y)\left|f(x,y)\right|\mathrm{d}\nu(x)\mathrm{d}\tau(y)<\infty$.
\end{thm}

We study the conditional measure on the Brownian path in Section \ref{sec:Brownian-path}.
We show that the conditions of Theorem \ref{cor:Double integral when W exists-summed}
hold for the Brownian path $B\subseteq\mathbb{R}^{d}$ ($d\geq3$)
and the sequence $\mathcal{Q}_{k}^{i}$, $(k\geq i)$ in Example \ref{exa:example Q_k^i}
for sufficient constants that depends on $i$ and $d$. However, our
main goal is to show the existence of the conditional measure with
respect to $\mathcal{Q}_{k}$, $(k\geq1)$ in Example \ref{exa:example Q_k}.
In Section \ref{sec:increaasing cond meas} we discuss how we can
extend Theorem \ref{cor:Double integral when W exists-summed} to
a $\mathcal{Q}_{k}$, $(k\geq1)$ if we know that that conditions
of Theorem \ref{cor:Double integral when W exists-summed} hold for
$\mathcal{Q}_{k}^{i}$, $(k\geq i)$ such that $\mathcal{Q}_{k}^{i}\subseteq\mathcal{Q}_{k}$
and $\mathcal{Q}_{k}^{i}$ is approaching $\mathcal{Q}_{k}$ in some
sense as $i$ goes to $\infty$. For the exact statement see Theorem
\ref{thm:increasing double int}. Applying this result to the Brownian
path we conclude in Section \ref{subsec:Existance-of-the-cond meas on Brownian path}
the following theorem.
\begin{thm}
\label{thm:Brownian cond measure: Main}Let $B$ be a Brownian path
in $\mathbb{R}^{d}$ for $d\geq3$. Let $\nu$ be a locally finite
Borel measure on $\mathbb{R}^{d}$ such that $\nu(\{0\})=0$. Then
the conditional measure $\mathcal{C}(\nu)$ of $\nu$ on $B$ exists
with respect to $\mathcal{Q}_{k}$ ($k\geq1$) (where $\mathcal{Q}_{k}$
is as in Example \ref{exa:example Q_k}) with regularity kernel $\varphi(x,y)=\left\Vert x-y\right\Vert ^{2-d}$.
Let $\tau$ also be a locally finite Borel measure on $\mathbb{R}^{d}$
such that $\tau(\{0\})=0$. Then
\begin{equation}
E\left(\int\int f(x,y)\mathrm{d}\mathcal{C}(\nu)(x)\mathrm{d}\mathcal{C}(\tau)(y)\right)=\intop\intop\frac{\left\Vert x\right\Vert ^{d-2}+\left\Vert y\right\Vert ^{d-2}}{\left\Vert x-y\right\Vert ^{d-2}}f(x,y)\mathrm{d}\nu_{R}(x)\mathrm{d}\tau_{R}(y)\label{eq:doub int brown}
\end{equation}
for every Borel function $f(x,y):X\times X\longrightarrow\mathbb{R}$
with $\intop\intop\frac{\left\Vert x\right\Vert ^{d-2}+\left\Vert y\right\Vert ^{d-2}}{\left\Vert x-y\right\Vert ^{d-2}}\left|f(x,y)\right|\mathrm{d}\nu_{R}(x)\mathrm{d}\tau_{R}(y)<\infty$.
\end{thm}

Recall, that in (\ref{eq:doub int brown}) $\nu_{R}$ and $\tau_{R}$
refers to the $(d-2)$-regular part of the measures (see Notation
\ref{nota:decomp meas}). Note that if $\nu$ is a point mass on $\{0\}$
and we consider the $\mathcal{Q}_{k}$ of Example \ref{exa:example Q_k}
with the difference that we do not remove the origin from the box
containing it then $\mathcal{C}_{k}(\nu)=\nu$ for every $n$ because
the Brownian path always intersects the cube containing $0$. In that
case $\nu$ is almost surely supported on $B$ hence the deterministic
$\mu=\nu$ itself solves the equation $\mathrm{d}\mu\mathrm{d}P=\mathrm{d}\nu$
with $\mu$ is supported on $B$ almost surely. However, (\ref{eq:doub int brown})
is not satisfied. Note that in the natural decomposition $\nu=\nu_{R}+\nu_{\bot}$
the point mass on $\{0\}$ belongs to the singular part, although
that is the only singular measure for which the problem can be solved.
It is because any singular measure with $\nu_{\bot}(\{0\})=0$ is
carried by a set that is almost surely not intersected by $B$ (see
Remark \ref{rem:nohope for singular}).
\begin{rem}
\label{rem:finite energy almost surely}Let $\nu$ be a finite Borel
measure such that $\mathrm{supp}\nu$ is bounded away from $0$ and
infinity and $I_{\beta}(\nu)<\infty$ for some $\beta>d-2$. One consequence
of the double integration formula (\ref{eq:doub int brown}) that
$I_{\beta+2-d}(\mathcal{C}(\nu))<\infty$ almost surely. This is a
useful tool in the geometric measure theory of the random intersection
$B\cap K$ for some fixed deterministic Borel set $K$.
\end{rem}

The question naturally rises what happens when $\nu$ is the Lebesgue
measure. The answer is given in terms of the occupation measure of
the Brownian motion, that counts the amount of time the Brownian motion
spends inside a set. Section \ref{sec:Conditional-ofleb and ocup}
is dedicated to deal with this question.
\begin{thm}
\label{thm:intro ocup}Let $B_{0}(t)$ be a Brownian motion in $\mathbb{R}^{d}$
for $d\geq3$, let $B$ be the range of the Brownian motion and let
$\lambda$ be the Lebesgue measure in $\mathbb{R}^{d}$. Let
\[
\tau(A)=\intop_{0}^{\infty}I_{B_{0}(t)\in A}\mathrm{d}t
\]
be the occupation measure of $B_{0}$. Then
\[
\mathrm{d}\mathcal{C}(\lambda)(x)=\frac{1}{c(d)}\left\Vert x\right\Vert ^{d-2}\mathrm{d}\tau(x)
\]
and
\[
\mathcal{C}(c(d)\left\Vert x\right\Vert ^{2-d}\mathrm{d}\lambda(x))=\mathrm{d}\tau(x)
\]
almost surely where
\begin{equation}
c(d)=\Gamma(d/2-1)2^{-1}\pi^{-d/2}\label{eq:cd def}
\end{equation}
and $\Gamma(x)=\intop_{0}^{\infty}s^{x-1}e^{-s}\mathrm{d}s$ is the
Euler's Gamma function.
\end{thm}

This result gives us a tool to calculate the occupation measure, $\tau(A)=c(d)\intop_{A}\left\Vert x\right\Vert ^{2-d}\mathrm{d}\mathcal{C}(\lambda)(x)$
can be approximated by $c(d)\intop_{A}\left\Vert x\right\Vert ^{2-d}\mathrm{d}\mathcal{C}_{k}(\lambda)(x)$
which converges to $\tau(A)$ in probability. To calculate the value
of $c(d)\intop_{A}\left\Vert x\right\Vert ^{2-d}\mathrm{d}\mathcal{C}_{k}(\lambda)(x)$
we only need to know which boxes of $\mathcal{Q}_{k}$ does $B$ intersect.
In Section \ref{sec:Conditional-ofleb and ocup}, while we prove Theorem
\ref{thm:intro ocup}, we basically show that
\begin{equation}
C_{\varphi}([0,1]^{d})\cdot2^{-2k}\#\left\{ Q\in\mathcal{Q}_{k}:Q\cap K\cap B\neq\emptyset\right\} \label{eq:ocup qube measure}
\end{equation}
converges to $\tau(K)$ in probability for every deterministic compact
set $K$, where $\varphi(x,y)=\left\Vert x-y\right\Vert ^{2-d}$ (note
that $\alpha=1/C_{\varphi}([0,1]^{d})$ in Section \ref{sec:Conditional-ofleb and ocup}
by Remark \ref{rem:what is alpha}). These results also show that
the occupation measure only depends on the trajectory of the Brownian
path, we need not to know the parametrisation, however, this fact
also follows from a result of Ciesielski and Taylor \cite[Theorem 5]{Cieselski Taylor}
which states that $\tau(K)$ equals to a deterministic constant (depending
on $d$) times the Hausdorff measure of $K\cap B$ with the gauge
function $x^{2}\log\log x^{-1}$. Note that the expression in (\ref{eq:ocup qube measure})
differs from the definition of the dyadic net measure with gauge function
$x^{2}$ in that the cubes are equal size rather than arbitrary. Since
the net measure is equivalent to the Hausdorff measure, the expression
in (\ref{eq:ocup qube measure}) cannot converge to the net measure
because that would contradict the Ciesielski-Taylor result.
\begin{rem}
\label{rem:green fn}The Green`s function of the Brownian motion in
$\mathbb{R}^{d}$ is $G(x,y)=c(d)\left\Vert x-y\right\Vert ^{2-d}$
for the constant $c(d)$ in (\ref{eq:cd def}), see \cite[Theorem 3.33]{Peres-Morters-Broanian motion}.
\end{rem}

The following proposition is known on the first moment of the occupation
measure, see \cite[Theorem 3.32]{Peres-Morters-Broanian motion}.
\begin{prop}
\label{thm:first moment}For every Borel set $A\subseteq\mathbb{R}^{d}$
($d\geq3$)
\[
E\left(\tau(A)\right)=c(d)\intop_{A}\frac{1}{\left\Vert x\right\Vert ^{d-2}}\mathrm{d}x.
\]
\end{prop}

The following theorem calculates the second moment of the occupation
measure.
\begin{thm}
\label{thm:second moment}We have that
\[
E\left(\intop\intop f(x,y)\mathrm{d}\tau(x)\mathrm{d}\tau(y)\right)=c(d)^{2}\intop\intop f(x,y)\frac{\left\Vert x\right\Vert ^{d-2}+\left\Vert y\right\Vert ^{d-2}}{\left\Vert x-y\right\Vert ^{d-2}\cdot\left\Vert x\right\Vert ^{d-2}\cdot\left\Vert y\right\Vert ^{d-2}}\mathrm{d}x\mathrm{d}y
\]
for every Borel function $f:\mathbb{R}^{d}\times\mathbb{R}^{d}\longrightarrow\mathbb{R}$
such that $\intop\intop\left|f(x,y)\right|\frac{\left\Vert x\right\Vert ^{d-2}+\left\Vert y\right\Vert ^{d-2}}{\left\Vert x-y\right\Vert ^{d-2}\cdot\left\Vert x\right\Vert ^{d-2}\cdot\left\Vert y\right\Vert ^{d-2}}\mathrm{d}x\mathrm{d}y<\infty$.
\end{thm}

Theorem \ref{thm:second moment} follows from Theorem \ref{thm:intro ocup}
and Theorem \ref{thm:Brownian cond measure: Main}. As an application
of Theorem \ref{thm:second moment}, we obtain a formula for the second
moment of the occupation measure
\begin{equation}
E\left(\tau(A)^{2}\right)=c(d)^{2}\intop_{A}\intop_{A}\frac{\left\Vert x\right\Vert ^{d-2}+\left\Vert y\right\Vert ^{d-2}}{\left\Vert x-y\right\Vert ^{d-2}\cdot\left\Vert x\right\Vert ^{d-2}\cdot\left\Vert y\right\Vert ^{d-2}}\mathrm{d}x\mathrm{d}y.\label{eq:sec mom intro}
\end{equation}
We note however, that (\ref{eq:sec mom intro}) could be deduced,
via a direct calculation, from the transition probability kernels
$p^{*}(t,x,y)$. For the definition of $p^{*}(t,x,y)$ see \cite[Theorem 3.30]{Peres-Morters-Broanian motion}.
Alternatively, (\ref{eq:sec mom intro}) can also be deduced from
Kac's moment formula (see \cite{Kac's moment}).

For a wide class of random closed sets we have that the hitting probability
of a compact set $K$ is comparable to the $\varphi$-capacity of
$K$ for a sufficient kernel $\varphi$. See Proposition  \ref{lem:tree capacity}
in case of the `percolation limit set' or Proposition \ref{thm:intersection capacity equivalence}
in case of the Brownian path. In Section \ref{sec:Probability-of-non-extinction}
we discuss the probability of the nonextinction of the conditional
measure and we establish analogous results to the hitting probabilities,
namely the probability of the nonextinction of $\mathcal{C}(\nu)$
is comparable to the capacity of the measure $\nu$. For the definition
of capacities of measures $C_{\varphi}(\nu)$ and $\overline{C_{\varphi}}(\nu)$
see Definition \ref{def:lower cap} and Definition \ref{def:upper cap}.
\begin{thm}
\label{thm:intro non extinction}Assume that the conditions of Theorem
\ref{cor:Double integral when W exists-summed} hold and $P(D\cap B\neq\emptyset)\leq b\cdot C_{\varphi}(D)$
for every compact set $D\subseteq X_{0}$. Let $\nu$ be a finite
Borel measure such that $\nu(X\setminus X_{0})=0$. Then
\[
c^{-1}\cdot C_{\varphi}(\nu)\leq P(\mathcal{C}(\nu)(X)>0)\leq b\cdot\overline{C_{\varphi}}(\nu).
\]
\end{thm}

We prove Theorem \ref{thm:intro non extinction} in Section \ref{sec:Probability-of-non-extinction}
and the conclusion can be extended for locally finite measures when
$X$ is locally compact (see Remark \ref{rem:loc fin non ext prob}).
The following result gives lower and upper bound for the probability
of nonextinction of $\mathcal{C}(\nu)$ when $B$ is the Brownian
path, that we show in Section \ref{subsec:Probability-of-non Brown}.
\begin{thm}
\label{thm:nonext prob for cond meas Brownian}Let $B$ be a Brownian
path in $\mathbb{R}^{d}$ for $d\geq3$. Let $\nu$ be a locally finite
Borel measure on $\mathbb{R}^{d}$ such that $\nu(\{0\})=0$. Let
$\mathcal{C}(\nu)$ be the conditional measure of $\nu$ on $B$ with
respect to $\mathcal{Q}_{k}$ ($k\geq1$) ($\mathcal{Q}_{k}$ is as
in Example \ref{exa:example Q_k}) with regularity kernel $\varphi(x,y)=\left\Vert x-y\right\Vert ^{2-d}$.
Then
\[
C_{F}(\nu)\leq P(\mathcal{C}(\nu)(X)>0)\leq2\overline{C_{F}}(\nu)
\]
for
\begin{equation}
F(x,y)=\frac{\left\Vert x\right\Vert ^{d-2}+\left\Vert y\right\Vert ^{d-2}}{\left\Vert x-y\right\Vert ^{d-2}}.\label{eq:F for brown}
\end{equation}
\end{thm}

The results on the probability of nonextinction suggest an analogy
between the random intersection $B\cap K$ for a fixed compact set
$K$ and the conditional measure $\mathcal{C}(\nu)$ of a fixed measure
$\nu$ on $B$. Compare Theorem \ref{thm:nonext prob for cond meas Brownian}
to the following known result on hitting probabilities.
\begin{prop}
\label{thm:intersection capacity equivalence}Let $B$ be a Brownian
path in $\mathbb{R}^{d}$ for $d\geq3$. Let $A\subseteq\mathbb{R}^{d}\setminus\{0\}$
be a compact set. Then
\begin{equation}
C_{F}(A)\leq P(B\cap A\neq\emptyset)\leq2C_{F}(A)\label{eq:cap eq segg}
\end{equation}
where $F$ is as in (\ref{eq:F for brown}).
\end{prop}

Let $G(x,y)=c(d)\left\Vert x-y\right\Vert ^{2-d}$ be the Green`s
function of the Brownian motion in $\mathbb{R}^{d}$ (see \cite[Theorem 3.33]{Peres-Morters-Broanian motion})
and let $M(x,y)=G(x,y)/G(0,y)$ be the Martin`s kernel. Then $F(x,y)=M(x,y)+M(y,x)$
and so $2I_{M}(\nu)=I_{F}(\nu)$ for every finite Borel measure $\nu$.
Thus Proposition \ref{thm:intersection capacity equivalence} is a
reformulation of \cite[Theorem 8.24]{Peres-Morters-Broanian motion}.

We stated many of our main results in general separable metric Radon
space $X$ for a `reasonable' random closed set $B$. In Section \ref{sec:Conditional-measure-on}
we let $X=\partial T$ to be the boundary of an infinite rooted tree
$T$. We discuss the theory of the conditional measure when $B$ is
a `percolation limit set'. We establish many properties of the conditional
measure such as the double integration formula and the probability
of nonextinction. In the end, we prove that $\mathcal{C}(\nu)$ is
a certain `random multiplicative cascade measure'. For the exact statement
of these results and the discussion see Section \ref{sec:Conditional-measure-on}.
Section \ref{sec:Conditional-measure-on} is complete without the
knowledge of the sections on the Brownian motion, i.e. does not rely
on Section \ref{sec:increaasing cond meas}, Section \ref{sec:Brownian-path}
and Section \ref{sec:Conditional-ofleb and ocup}.

We hope that our methods can be extended to a rich family of natural
random closed sets, however, this study is already extensive but the
question can be a subject of further investigation. The key to make
the machinery work is the existence of $F(x,y)$, the rest of the
assumptions in Section \ref{subsec:Special-assumptions} are more
straightaway if we have $F(x,y)$. In the case of the two dimensional
Brownian motion to obtain a closed set, we need to stop the Brownian
motion at some point. A natural candidate is the two dimensional Brownian
motion that is stopped at a random exponential time independent of
the Brownian motion, and let $B$ to be the range of the motion till
that stopped time. We suspect that the assumptions of Theorem \ref{cor:Double integral when W exists-summed}
also hold locally in the two dimensional case and due to Section \ref{sec:increaasing cond meas}
can be extended globally to the plane. The proof shall differ a bit
due to the fact that the Geen's function is merely different in two
dimension than in higher dimensions. Also other random processes might
satisfy our assumptions, the list could include processes like $\alpha$-stable
Levy processes, the fractional Brownian motion or the Mandelbrot percolation
set.

We note that all our results can be extended to finite signed and
complex measures or if $X$ is locally compact then to locally finite
signed and complex measures. It can be seen by applying our results
to the measure parts in the decomposition of signed and complex measures.

\subsection{Application}

We introduced the concept of the conditional measure on a random set
by the analogy of slicing measures (\ref{eq:dL}) by straight lines.
For the conditional measure we wish to solve the problem that we require
a random measure $\mu$ to be supported on $B$ and $\mathrm{d}\mu\mathrm{d}P=\mathrm{d}\nu$
for a given measure $\nu$. As we discussed in Remark \ref{rem:nohope for singular}
it is not possible to solve the problem for $\nu_{\bot}$ in the case
of the Brownian path and we can solve it for $\nu_{R}$. However,
the solution is not unique, one can take a random variable $Y$ independent
of $B$ that takes value $p^{-1}$ with probability $p$ and value
$0$ with probability $1-p$. Then if $\mu$ is a solution then $Y\mu$
is also a solution. Even if we require $\mu$ to be measurable with
respect to the $\sigma$-algebra $\mathcal{F}$ generated by the events
$\{Q\cap B\neq\emptyset\}$ where $Q\in\mathcal{Q}_{k}$, $k\in\mathbb{N}$
the solution is not unique. One can construct a random point mass
$\mu$, such that $\mu$ is $\mathcal{F}$-measurable and $\mathrm{d}\mu\mathrm{d}P=\mathrm{d}\lambda$.
However, a disadvantage of the point mass that it does not behave
well for the double integration formula (\ref{eq:doub int brown}).
The construction of $\mathcal{C}(\nu)$ defined as the limit of $\mathcal{C}_{k}(\nu)$
is in some sense a random measure that is nicely spread on $B$ unlike
a point mass. The construction is analogous to the slice measures
(\ref{eq:sequence of slice measure}). In the application of the slice
measures a weaker double integration formula \cite[Theorem 10.7]{Mattila book}
provides a strong tool for the geometric measure theory of slices.
Similarly the double integration formula for $\mathcal{C}(\nu)$ gives
us a powerful weapon and can be applied to analogous problems that
the slice measures are used for in the case of intersections with
straight lines. We outline below one further application that we believe
is merely new.

It can be deduced from Remark \ref{rem:finite energy almost surely}
that
\begin{equation}
\dim_{H}\nu-(d-2)\leq\dim_{H}\mathcal{C}(\nu)\label{eq:dim lower}
\end{equation}
conditional on $\mathcal{C}(\nu)(X)>0$, where $\dim_{H}$ denotes
the Hausdorff dimension. Let $K\subseteq\mathbb{R}^{d}\setminus\{0\}$
be a fixed compact set. Assume that $\nu$ is a finite deterministic
measure on $K$ such that
\begin{equation}
P(\mathcal{C}(\nu)(\mathbb{R}^{d})>0)=P(K\cap B\neq\emptyset).\label{eq:nonext equality}
\end{equation}
Then $\dim_{H}\nu-(d-2)\leq\dim_{H}K\cap B$ almost surely conditional
on $K\cap B\neq\emptyset$ by (\ref{eq:dim lower}), (\ref{eq:nonext equality})
and Property \textit{xi.)} of Definition \ref{def:def of cond meas}.
We suspect that (\ref{eq:nonext equality}) holds for the harmonic
measure of $K$ for the Brownian motion and for the equilibrium measure
of $K$ for the capacity kernel $\left\Vert x-y\right\Vert ^{2-d}$.
An essential almost sure upper bound for $\dim_{H}K\cap B$ has been
known for a long time \cite{Hawkes 71a,Hawkes 71b}, namely $\left\Vert \dim_{H}K\cap B\right\Vert _{\infty}=\max\{0,\dim_{H}K-(d-2)$,
however we are not aware of a general method to give almost sure lower
bounds on $\dim_{H}K\cap B$ conditional on intersection.

We note that the harmonic measure does not necessarily give the essential
lower bound. If $K$ is a ball in $\mathbb{R}^{d}$ not containing
the origin then the Harmonic measure is supported on $\partial K$
which gives lower estimate of dimension $1$ while the correct dimension
of the intersection is almost surely $2$ conditional on intersection.
For $K$ being the ball the Lebesgue measure is the ideal choice for
$\nu$.

In general, if $\nu_{n}$ is a sequence of measures supported on $K$
such that $\lim_{n\rightarrow\infty}P(\mathcal{C}(\nu_{n})(X)>0)=P(K\cap B\neq\emptyset)$
then $\limsup_{n\rightarrow\infty}\dim_{H}\nu_{n}-(d-2)\leq\dim_{H}K\cap B$
almost surely conditional on $K\cap B\neq\emptyset$. This means that
we can provide almost sure lower bound even if we cannot find the
ideal measure that satisfies (\ref{eq:nonext equality}) but we can
find a sequence of measures that in limes satisfies (\ref{eq:nonext equality}).

\section{Preliminary remarks\label{sec:Peliminary-remarks}}

In this section we summarise the background and preliminary lemmas.

For $x\in X$, $r>0$ let $B(x,r)=\left\{ y\in X:d(x,y)<r\right\} $
and for a set $A\subseteq X$ let $B(A,r)=\left\{ y\in X:\exists x\in A,d(x,y)<r\right\} $.

\begin{notation}
For $A\subseteq X$ let 
\[
\chi_{A}(x)=\begin{cases}
1 & x\in A\\
0 & x\notin A
\end{cases}
\]
 be the \textit{characteristic function of $A$}.
\end{notation}

\begin{notation}
For a probability event $A\in\mathcal{A}$ let
\[
I_{A}(\omega)=\begin{cases}
1 & \omega\in A\\
0 & \omega\notin A
\end{cases}
\]
be the\textit{ indicator function of $A$}.
\end{notation}

\subsection{Convergence in probability}

We list below some folklore properties of the convergence in probability
that we need throughout the paper.
\begin{defn}
\label{def:conv in prob}Let $Y_{n}$ ($n\in\mathbb{N}$) and $Y$
be random variables. We say that \textit{$Y_{n}$ converges to $Y$
in probability} (as $n$ goes to $\infty$) if
\[
\lim_{n\rightarrow\infty}P\left(\left|Y_{n}-Y\right|>0\right)=0
\]
for every $\varepsilon>0$. If the random variables take values in
a metric space than we can replace the difference $\left|.\right|$
by the metric to obtain a definition of convergence in probability
for random variables that take values in a metric space.
\end{defn}

\begin{rem}
\label{rem: SUBSEQ IFF conv prob}Let $Y$ and $Y_{n}$ be a sequence
of random variables that take values in a metric space. Then $Y_{n}$
converges to $Y$ in probability if and only if for every subsequence
$\left\{ n_{k}\right\} _{k=1}^{\infty}$ of $\mathbb{N}$ there exists
a subsequence $\left\{ j_{k}\right\} _{k=1}^{\infty}$ of $\left\{ n_{k}\right\} _{k=1}^{\infty}$
such that $Y_{j_{k}}$ converges to $Y$ almost surely (see \cite[Theorem 2.3.2]{Durrett}).
\end{rem}

\begin{rem}
\label{rem:rho def}The convergence in probability is a metric convergence
and is completely metrizable by the following metric $\rho$ (see
\cite[Exercise 2.3.8 and 2.3.9]{Durrett}). For random variables $Y,Z$
let 
\[
\rho(Y,Z)=E\left(\frac{\left|Y-Z\right|}{1+\left|Y-Z\right|}\right).
\]
\end{rem}

\begin{rem}
\label{rem:exp quality}Note that $\rho(Y,Z)\leq E\left(\left|Y-Z\right|\right)$.
We use this fact without reference throughout the paper.
\end{rem}

\begin{lem}
\label{lem:random lim is random}Let $Y_{n}:\Omega\longrightarrow M$
be a sequence of random variables that take values in a metric space
$M$. If $Y:\Omega\longrightarrow M$ is a function and there exists
an event $H$ with $P(H)=1$ such that $Y_{n}(\omega)$ converges
to $Y(\omega)$ for every $\omega\in H$ then $Y$ is a random variable.
\end{lem}

For the proof of Lemma \ref{lem:random lim is random} see \cite[Theorem 4.2.2]{Dudley}
\begin{lem}
\label{lem:mutual convergence subsequance}Assume that $\left\{ f_{k}^{i}\right\} _{i,k\in\mathbb{N}}$
is a family of random variables such that $f_{k}^{i}$ converges in
probability as $k$ goes to infinity for every $i$. Then for every
subsequence $\left\{ n_{k}\right\} _{k=1}^{\infty}$ of $\mathbb{N}$
there exists a subsequence $\left\{ j_{k}\right\} _{k=1}^{\infty}$
of $\left\{ n_{k}\right\} _{k=1}^{\infty}$and there exists a probability
event $H\in\mathcal{A}$ with $P(H)=1$ such that $f_{j_{k}}^{i}$
converges on the event $H$ as $k$ goes to infinity for every $i$.
\end{lem}

\begin{proof}
Let $\left\{ \alpha_{k}^{1}\right\} _{k=1}^{\infty}$ be a subsequence
of $\left\{ n_{k}\right\} _{k=1}^{\infty}$ such that $f_{\alpha_{k}^{1}}^{1}$
converges almost surely and let $j_{1}=\alpha_{1}^{1}$. If $\left\{ \alpha_{k}^{i}\right\} _{k=i}^{\infty}$
and $j_{1},\dots j_{i}$ are defined let $\left\{ \alpha_{k}^{i+1}\right\} _{k=i+1}^{\infty}$
be a subsequence of $\left\{ \alpha_{k}^{i}\right\} _{k=i+1}^{\infty}$
such that $f_{\alpha_{k}^{i+1}}^{i+1}$ converges almost surely as
$k$ goes to infinity and let $j_{i+1}=\alpha_{i+1}^{i+1}$. Then
$\left\{ j_{k}\right\} _{k=i}^{\infty}$ is a subsequence of $\left\{ \alpha_{k}^{i}\right\} _{k=i}^{\infty}$
and hence $f_{j_{k}}^{i}$ converges almost surely as $k$ goes to
infinity for every $i$.
\end{proof}
\begin{lem}
\label{lem:subseq prob 0}Let $Y$ and $Y_{n}$ ($n\in\mathbb{N}$)
be a sequence of random variables. If for every $\varepsilon>0$ and
every subsequence $\left\{ \alpha_{k}\right\} _{k=1}^{\infty}$ of
$\mathbb{N}$ we can find a subsequence $\left\{ \beta_{k}\right\} _{k=1}^{\infty}$
of $\left\{ \alpha_{k}\right\} _{k=1}^{\infty}$ such that
\begin{equation}
\lim_{k\rightarrow\infty}P\left(\left|Y_{\beta_{k}}-Y\right|>\varepsilon\right)=0\label{eq:sq}
\end{equation}
then $Y_{n}$ converges to $Y$ in probability as $n$ goes to $\infty$.
\end{lem}

\begin{proof}
In a topological space $x_{k}$ converges to $x$ if and only if for
every subsequence $\left\{ \alpha_{k}\right\} _{k=1}^{\infty}$ of
$\mathbb{N}$ we can find a subsequence $\left\{ \beta_{k}\right\} _{k=1}^{\infty}$
of $\left\{ \alpha_{k}\right\} _{k=1}^{\infty}$ such that $x_{\beta_{k}}$
converges to $x$. Applying this to $x_{k}=P\left(\left|Y_{k}-Y\right|>\varepsilon\right)$
it follows that $\lim_{k\rightarrow\infty}P\left(\left|Y_{k}-Y\right|>\varepsilon\right)=0$.
\end{proof}

\begin{lem}
\label{lem:Scheffe}Let $Y_{n}$ be a sequence of real valued random
variables such that $Y_{n}$ converges to $Y$ in probability and
$\lim_{n\rightarrow\infty}E\left|Y_{n}\right|=E\left|Y\right|<\infty$.
Then $Y_{n}$ converges to $Y$ in $\mathcal{L}^{1}$.
\end{lem}

For details of the proof see \cite[Theorem 5.5.2]{Durrett}

\begin{lem}
\label{lem:finite expectation stochastic limit}If $Y_{n}$ ($n\in\mathbb{N}$)
is a sequence of nonnegative, real valued random variables, $Y_{n}$
converges to $Y$ in probability and there exists $c<\infty$ such
that $E(Y_{n})\leq c$ for every $n\in\mathbb{N}$ then $E(Y)\leq c$.
\end{lem}

\begin{proof}
Let $n_{k}$ be a sequence such that $Y_{n_{k}}$ converges to $Y$
almost surely. Then by Fatou`s lemma
\[
E(Y)=E(\liminf_{k\rightarrow\infty}Y_{n_{k}})\leq\liminf_{k\rightarrow\infty}E(Y_{n_{k}})\leq c.
\]
\end{proof}

\subsection{Weak$^{*}$ and vague convergence of measures}

We recall some properties of the weak$^{*}$ and vague convergences
of measures and some related lemmas.

Let $C_{b}(X)$ denote the \textit{space of bounded continuous functions
on $X$ }equipped with the supremum norm and let $C_{c}(X)$ denote
the \textit{space of all compactly supported continuous functions
on $X$ }equipped with the supremum norm. We denote by $\mathrm{supp}(f)$
the support of a function $f:X\longrightarrow\mathbb{R}$.
\begin{lem}
\label{lem:countab exhaustion}Let $\nu$ be a finite, Borel measure
on $X$ and $A\subseteq X$ be a Borel set. Then there exists a sequence
of disjoint compact subsets $K_{1},K_{2},\dots$ of $A$ such that
$\nu(A\setminus\cup_{i=1}^{\infty}K_{i})=0$.
\end{lem}

\begin{rem}
\label{rem:countab exhaust vague}Lemma \ref{lem:countab exhaustion}
easily follows from the inner regularity of $\nu$. The statement
also holds for a locally finite Borel measure $\nu$. It can be deduced
from that $\nu$ is locally finite and $X$ satisfies the Lindel\"of
property as it is a separable metric space.
\end{rem}

\begin{defn}
\label{def:weak conv}Let $\mu$ and $\mu_{k}$ ($k\in\mathbb{N}$)
be finite Borel measures on $X$. We say that \textit{$\mu_{k}$ weakly
converges to $\mu$} (as $k$ goes to $\infty$) if
\[
\lim_{k\rightarrow\infty}\int f(x)\mathrm{d}\mu_{k}(x)=\int f(x)\mathrm{d}\mu(x)
\]
for every $f\in C_{b}(X)$.
\end{defn}

\begin{rem}
\label{rem:liminf open weak conv}It is well-known, that $\mu_{k}$
converges to $\mu$ weakly if and only if $\mu(G)\leq\liminf_{k\rightarrow\infty}\mu_{k}(G)$
for every open set $G$. See \cite[Theorem 2.1]{Billingsly Convergence of Prob measu}.
\end{rem}

\begin{lem}
\label{lem:integral separation}Let $\mu$ and $\nu$ be finite Borel
measures on $X$ such that $\mu\neq\nu$. Then there exists $f\in C_{b}(X)$
such that $\intop_{X}f(x)\mathrm{d}\mu(x)\neq\intop_{X}f(x)\mathrm{d}\nu(x)$.
\end{lem}

Lemma \ref{lem:integral separation} is shown in \cite[Theorem 1.2]{Billingsly Convergence of Prob measu}.
In locally compact space the same proof results the following lemma
because by inner regularity the measure of compact sets determines
the measure.
\begin{lem}
\label{lem:comp integral separation}Let $X$ be locally compact.
Let $\mu$ and $\nu$ be locally finite Borel measures on $X$ such
that $\mu\neq\nu$. Then there exists $f\in C_{c}(X)$ such that $\intop_{X}f(x)\mathrm{d}\mu(x)\neq\intop_{X}f(x)\mathrm{d}\nu(x)$.
\end{lem}

\begin{lem}
\label{lem:weak conv unique}Assume that $\mu_{k}$ converges to both
$\mu$ and $\nu$ weakly then $\mu=\nu$.
\end{lem}

Lemma \ref{lem:weak conv unique} follows from Lemma \ref{lem:integral separation}.
\begin{lem}
\label{lem:determ int}Let $f,f_{k}\in C_{b}(X)$ ($k\in\mathbb{N}$)
such that $\lim_{k\rightarrow\infty}\left\Vert f-f_{k}\right\Vert _{\infty}=0$.
Assume that $\mu$ and $\nu$ are finite Borel measures on $X$ such
that $\intop_{X}f_{k}(x)\mathrm{d}\mu(x)=\intop_{X}f_{k}(x)\mathrm{d}\nu(x)$
for every $k\in\mathbb{N}$. Then $\intop_{X}f(x)\mathrm{d}\mu(x)=\intop_{X}f(x)\mathrm{d}\nu(x)$.
\end{lem}

\begin{proof}
Let $\varepsilon>0$ and $k\in\mathbb{N}$ be such that $\left\Vert f-f_{k}\right\Vert _{\infty}<\varepsilon$.
Then
\[
\left|\intop_{X}f(x)\mathrm{d}\mu(x)-\intop_{X}f(x)\mathrm{d}\nu(x)\right|\leq\left|\intop_{X}f(x)\mathrm{d}\mu(x)-\intop_{X}f_{k}(x)\mathrm{d}\mu(x)\right|
\]
\[
+\left|\intop_{X}f_{k}(x)\mathrm{d}\nu(x)-\intop_{X}f(x)\mathrm{d}\nu(x)\right|\leq\varepsilon(\mu(X)+\nu(X)).
\]
Hence the statement follows.
\end{proof}
\begin{lem}
\label{lem:countabl determines mu}Let $\Psi\subseteq C_{b}(X)$ be
a dense subset with respect to the supremum norm. Assume that $\mu$
and $\nu$ are finite Borel measures on $X$ such that $\intop_{X}f(x)\mathrm{d}\mu(x)=\intop_{X}f(x)\mathrm{d}\nu(x)$
for every $f\in\Psi$. Then $\mu=\nu$.
\end{lem}

\begin{proof}
It follows from Lemma \ref{lem:determ int} that $\intop_{X}f(x)\mathrm{d}\mu(x)=\intop_{X}f(x)\mathrm{d}\nu(x)$
for every $f\in C_{b}(X)$. Hence by Lemma \ref{lem:integral separation}
the statement follows.
\end{proof}
\begin{defn}
\label{def:The-Prohorov-distance}The \textit{Prohorov distance} between
two finite Borel measures $\mu$ and $\nu$ on $X$ is
\[
\pi(\mu,\nu)=\inf\left\{ \varepsilon:\mu(A)\leq\nu(B(A,\varepsilon))+\varepsilon\,\mathrm{and}\,\nu(A)\leq\mu(B(A,\varepsilon))+\varepsilon\,\mathrm{for}\,\forall A\in\mathcal{B}(X)\right\} 
\]
where $\mathcal{B}(X)$ denotes the \textit{set of Borel subsets of
$X$}.
\end{defn}

\begin{lem}
\label{lem:triv bound}We have that
\[
\pi(\mu,\nu)\leq\mu(X)+\nu(X).
\]
\end{lem}

\begin{lem}
\label{lem:sum triv bound}We have that
\[
\pi(\mu+\nu,\nu)\leq\mu(X).
\]
\end{lem}

The statement of Lemma \ref{lem:triv bound} and Lemma \ref{lem:sum triv bound}
follow from the definition of the Prohorov distance.
\begin{prop}
\label{prop:The-Prohorov-distance}The Prohorov distance is a separable
metric on the set of all finite Borel measures $\mathcal{M}_{+}(X)$.
We have that $\mu_{k}$ converges to $\mu$ weakly if and only if
$\lim_{k\rightarrow\infty}\pi(\mu_{k},\mu)=0$. (Note that throughout
the paper we assume that $X$ is a separable metric space.) If we
further assume that $(X,d)$ is a complete metric space then the Prohorov
distance is a complete metric.
\end{prop}

For the proof see the section of \cite[Theorem 6.8]{Billingsly Convergence of Prob measu}
\begin{lem}
\label{lem:Separable C_b(S)}If $K\subseteq X$ is compact then $C_{b}(K)$
is separable.
\end{lem}

See \cite[page 437]{Dunford-Schwartz}.
\begin{lem}
\label{lem:deterministic con}Let $K\subseteq X$ be a compact subset,
let $\Psi\subseteq C_{b}(K)$ be a dense subset with respect to the
supremum norm and let $\mu_{k}$ be a sequence of finite, Borel measures
on $K$ such that $\intop_{X}f(x)\mathrm{d}\mu_{k}(x)$ converges
to a limit $S(f)<\infty$ for every $f\in\Psi$. Then $\mu_{k}(K)$
is bounded and $\mu_{k}$ converges weakly to a finite, Borel measure.
\end{lem}

\begin{proof}
Let $g\in\Psi$ such that $\left\Vert \chi_{K}-g\right\Vert <1/2$.
Then $\chi_{K}\leq2g$ on $K$ thus
\[
\limsup_{k\rightarrow\infty}\mu_{k}(K)\leq2\limsup_{k\rightarrow\infty}\intop_{X}g(x)\mathrm{d}\mu_{k}(x)=2S(g)<\infty
\]
and so $\mu_{k}(K)$ is bounded.

Since $\mu_{k}(K)$ is bounded and $K$ is a compact metric space
it follows, by \cite[(17.22) Theorem]{Kechris}, that there exists
a subsequence $n_{k}$ of $\mathbb{N}$ such that $\mu_{n_{k}}$ weakly
converges to a Borel measure $\tau$ of finite total mass. Then
\[
\intop_{X}g(x)\mathrm{d}\tau(x)=\lim_{k\rightarrow\infty}\intop_{X}g(x)\mathrm{d}\mu_{n_{k}}(x)=S(g)
\]
for every $g\in\Psi$. Let $f$ be a bounded continuous function and
$g\in\Psi$ be such that $\left\Vert f-g\right\Vert _{\infty}<\varepsilon$.
Then
\[
\limsup_{k\rightarrow\infty}\left|\intop_{X}f(x)\mathrm{d}\tau(x)-\intop_{X}f(x)\mathrm{d}\mu_{k}(x)\right|\leq\left|\intop_{X}f(x)\mathrm{d}\tau(x)-\intop_{X}g(x)\mathrm{d}\tau(x)\right|
\]
\[
+\limsup_{k\rightarrow\infty}\left(\left|\intop_{X}g(x)\mathrm{d}\tau(x)-\intop_{X}g(x)\mathrm{d}\mu_{k}(x)\right|+\left|\intop_{X}g(x)\mathrm{d}\mu_{k}(x)-\intop_{X}f(x)\mathrm{d}\mu_{k}(x)\right|\right)
\]
\[
\leq\intop_{X}\left|f(x)-g(x)\right|\mathrm{d}\tau(x)+0+\limsup_{k\rightarrow\infty}\intop_{X}\left|f(x)-g(x)\right|\mathrm{d}\mu_{k}(x)\leq\varepsilon\left(\tau(K)+\limsup_{k\rightarrow\infty}\mu_{k}(X)\right).
\]
By taking $\varepsilon$ goes to $0$ it follows that $\intop_{X}f(x)\mathrm{d}\tau(x)=\lim_{k\rightarrow\infty}\intop_{X}f(x)\mathrm{d}\mu_{k}(x)$. 
\end{proof}
\begin{defn}
\label{def:vague conv}Let $\mu$ and $\mu_{k}$ ($k\in\mathbb{N}$)
be locally finite Borel measures on $X$. We say that \textit{$\mu_{k}$
vaguely converges to $\mu$} (as $k$ goes to $\infty$) if
\[
\lim_{k\rightarrow\infty}\int f(x)\mathrm{d}\mu_{k}(x)=\int f(x)\mathrm{d}\mu(x)
\]
for every $f\in C_{c}(X)$.
\end{defn}

\begin{lem}
\label{lem:deterministic con-Vague}Let $X$ be locally compact. There
exists $\Psi\subseteq C_{c}(X)$ countable and dense subset with respect
to the supremum norm such that if $\mu_{k}$ is a sequence of locally
finite Borel measures on $X$ such that $\intop_{X}f(x)\mathrm{d}\mu_{k}(x)$
converges to a finite limit $S(f)$ for every $f\in\Psi$ then $\mu_{k}$
vaguely converges to a locally finite Borel measure.
\end{lem}

\begin{proof}
Since $X$ is a separable metric space it satisfies the Lindel\"of
property. Thus, because $X$ is locally compact, we can find a sequence
of open sets $G_{1}\subseteq G_{2}\subseteq\dots$ such that $\cup_{i=1}^{\infty}G_{i}=X$
and $\overline{G_{i}}$ is compact for every $i\in\mathbb{N}$. Hence
for every compact set $K$ there exists $i\in\mathbb{N}$ such that
$K\subseteq G_{i}$. We can further assume that $\overline{G_{i}}\subseteq G_{i+1}$.
By Lemma \ref{lem:Separable C_b(S)} we can find a countable and dense
subset $\Psi_{i}$ of $\left\{ f\in C_{b}(X):\mathrm{supp}(f)\subseteq\overline{G_{i}}\right\} $.
Let $\Psi=\cup_{i=1}^{\infty}\Psi_{i}$.

Let $\mu_{k}$ be a sequence of deterministic, locally finite, Borel
measures on $X$ such that $\intop_{X}f(x)\mathrm{d}\mu_{k}(x)$ converges
to a finite limit $S(f)$ for every $f\in\Psi$. If $\intop_{X}f(x)\mathrm{d}\mu_{k}(x)$
converges to a finite limit $S(f)$ for every $f\in C_{c}(X)$ then
clearly $S$ is a positive linear functional, hence by the Riesz-Markov
theorem \cite[Theorem 2.14]{Rudin} there exists a locally finite,
Borel measures $\mu$ on $X$ such that $S(f)=\intop_{X}f(x)\mathrm{d}\mu(x)$
for every $f\in C_{c}(X)$. Hence the statement would follow. Thus
to finish the proof we need to show that $\intop_{X}f(x)\mathrm{d}\mu_{k}(x)$
converges to a finite limit $S(f)$ for every $f\in C_{c}(X)$.

Let $f\in C_{c}(X)$, let $K=\mathrm{supp}(f)$. There exists $i\in\mathbb{N}$
such that $K\subseteq G_{i}\subseteq\overline{G_{i}}\subseteq G_{i+1}$.
By Tietze`s extension theorem we can find $h_{0}\in C_{c}(X)$ such
that $h_{0}(x)=1$ for $x\in\overline{G_{i}}$ and $h_{0}(x)=0$ for
$x\notin G_{i+1}$. There exists $h\in\Psi$ such that $\left\Vert h_{0}-h\right\Vert _{\infty}<1/2$.
Then $\limsup_{k\rightarrow\infty}\mu_{k}(\overline{G_{i}})\leq2\lim_{k\rightarrow\infty}\intop_{X}h(x)\mathrm{d}\mu_{k}(x)=2S(h)$.

Let $g_{n}\in\Psi_{i}$ be such that $\left\Vert f-g_{n}\right\Vert _{\infty}<1/n$.
Then
\[
\limsup_{k\rightarrow\infty}\left|\intop_{X}f(x)\mathrm{d}\mu_{k}(x)-S(g_{n})\right|\leq\limsup_{k\rightarrow\infty}\left|\intop_{X}f(x)\mathrm{d}\mu_{k}(x)-\intop_{X}g_{n}(x)\mathrm{d}\mu_{k}(x)\right|
\]
\[
+\limsup_{k\rightarrow\infty}\left|\intop_{X}g_{n}(x)\mathrm{d}\mu_{k}(x)-S(g_{n})\right|\leq\limsup_{k\rightarrow\infty}\mu_{k}(\overline{G_{i}})/n+0\leq2S(h)/n.
\]
Thus for every $n\in\mathbb{N}$ there exists $N\in\mathbb{N}$ such
that $\left|\intop_{X}f(x)\mathrm{d}\mu_{k}(x)-S(g_{n})\right|\leq(2S(h)+1)/n$
for every $k\geq N$ and so $\left|\intop_{X}f(x)\mathrm{d}\mu_{k}(x)-\intop_{X}f(x)\mathrm{d}\mu_{j}(x)\right|\leq(4S(h)+2)/n$
for every $j,k\geq N$. Hence $\intop_{X}f(x)\mathrm{d}\mu_{k}(x)$
is a Cauchy sequence so it has a finite limit $S(f)$ in $\mathbb{R}$.
\end{proof}
\begin{lem}
\label{lem:countab agree}Let $X$ be locally compact, let $\nu$
and $\mu$ be locally finite Borel measures on $X$ and let $\Psi\subseteq C_{c}(X)$
be as in Lemma \ref{lem:deterministic con-Vague}. If $\intop_{X}f(x)\mathrm{d}\nu(x)=\intop_{X}f(x)\mathrm{d}\mu(x)$
for every $f\in\Psi$ then $\mu=\nu$.
\end{lem}

\begin{proof}
Let $f\in C_{c}(X)$ be fixed. Similarly to the proof of Lemma \ref{lem:deterministic con-Vague}
it can be shown that there exist $h\in\Psi$ and $g_{n}\in\Psi$ for
every $n\in\mathbb{N}$ such that
\[
\left|\intop_{X}f(x)\mathrm{d}\mu(x)-\intop_{X}f(x)\mathrm{d}\nu(x)\right|\leq\left|\intop_{X}f(x)\mathrm{d}\mu(x)-\intop_{X}g_{n}(x)\mathrm{d}\mu(x)\right|
\]
\[
+\left|\intop_{X}g_{n}(x)\mathrm{d}\nu(x)-\intop_{X}f(x)\mathrm{d}\nu(x)\right|\leq4\intop_{X}h(x)\mathrm{d}\mu(x)/n.
\]
Thus $\intop_{X}f(x)\mathrm{d}\mu(x)=\intop_{X}f(x)\mathrm{d}\nu(x)$
for every $f\in C_{c}(X)$ and the statement of the lemma follows
by Lemma \ref{lem:comp integral separation}.
\end{proof}
\begin{prop}
\label{prop:Polish vague}Assume that $X$ is locally compact. The
set of all locally finite Borel measures $\mathcal{M}_{l}(X)$ can
be equipped with a complete separable metric such that $\mu_{k}$
converges to $\mu$ in the induced topology if and only if $\mu_{k}$
converges to $\mu$ vaguely.
\end{prop}

Proposition \ref{prop:Polish vague} is shown for $X=\mathbb{R}^{d}$
in \cite[Chapter 14]{Mattila book} in the section called `A metric
on measures' starting on page 194. The proof goes similarly for locally
compact separable metric spaces, by replacing the role of balls $B(0,n)$
by an increasing sequence of compact sets $K_{n}$ such that $\cup_{n\in\mathbb{N}}K_{n}=X$
(which we can find since $X$ satisfies the Lindel\"of property and
is locally compact). We omit the details of the proof.
\begin{lem}
\label{lem:leq meas lem}Let $\nu$ and $\tau$ be two locally finite
Borel measures on $X$ such that $\tau(G)\leq\nu(G)$ for every open
set $G$. Then $\tau(A)\leq\nu(A)$ for every Borel set $A$.
\end{lem}

Lemma \ref{lem:leq meas lem} follows from the outer regularity of
the measures.

\begin{lem}
\label{lem:open eq dense}Let $D\subseteq X$ be a countable and dense
subset and let
\[
\mathcal{G}=\left\{ \cup_{i=1}^{m}B(x_{i},r_{i}):m\in\mathbb{N},x_{i}\in D,r_{i}\in\mathbb{Q},r_{i}>0,\mathrm{\,for\,}i=1,\dots,m\right\} .
\]
Let $\nu$ and $\tau$ be two locally finite Borel measures on $X$
such that $\tau(G)=\nu(G)$ for every $G\in\mathcal{G}$. Then $\tau(A)=\nu(A)$
for every Borel set $A$.
\end{lem}

\begin{proof}
Let $A\subseteq X$ be an open set. Then
\[
\left\{ B(x,r):x\in D,r\in\mathbb{Q},r>0,B(x,r)\subseteq A\right\} 
\]
is a countable open cover of $A$, which we can enumerate $\left\{ B(x_{i},r_{i})\right\} _{i=1}^{\infty}$,
thus $A=\cup_{i=1}^{\infty}B(x_{i},r_{i})$. Then
\[
\tau(A)=\lim_{m\rightarrow\infty}\tau(\cup_{i=1}^{m}B(x_{i},r_{i}))=\lim_{m\rightarrow\infty}\nu(\cup_{i=1}^{m}B(x_{i},r_{i}))=\nu(A).
\]
We can conclude that $\tau(A)=\nu(A)$ for every open set $A\subseteq X$
and so for every Borel set $A\subseteq X$ by Lemma \ref{lem:leq meas lem}.
\end{proof}
\begin{prop}
\label{lem:charateodory ineq}Let $(\Omega,\mathcal{B})$ be a measurable
space, let $\nu$ and $\tau$ be $\mathcal{B}$-measurable finite
measures on $\Omega$. Let $\mathcal{S}$ be a semiring of sets of
$\mathcal{B}$ that generates the $\sigma$-algebra $\mathcal{B}$,
assume that $\Omega\in\mathcal{S}$ and $\nu(S)\leq\tau(S)$ for every
$S\in\mathcal{S}$. Then $\nu(A)\leq\tau(A)$ for every $A\in\mathcal{B}$.
\end{prop}

\begin{proof}
By \cite[Section 1.5.1]{Makarov-Podkorytov} we have that $\nu$ and
$\tau$ are uniquely determined by their values on $\mathcal{S}$,
and $\nu$ and $\tau$ equal to their Charat\'{e}odry extension from
$\mathcal{S}$. Hence $\nu(A)\leq\tau(A)$ for every $A\in\mathcal{B}$
by the definition of Charat\'{e}odry extension \cite[Section 1.4.4]{Makarov-Podkorytov}.
\end{proof}

\section{Convergence of random measures\label{sec:Convergence-of-random}}

We combine the convergence of random variables in probability and
the convergence of measures to obtain the convergence of measures
in probability. This section includes four subsections. Section \ref{subsec:Weak-convergence-subsequentiuall}
develops the theory of weak convergence of random measures subsequentially
in probability. In Section \ref{subsec:Weak-convergence-in prob}
we briefly introduce the concept of weak convergence of random measures
in probability. In Section \ref{subsec:Vague-convergence-in prob}
we discuss the vague convergence of random measures in probability
in the situation when $X$ is a locally compact space. Finally, in
Section \ref{subsec:SUB-General-existance-of} we use the results
on the convergence of random measures to obtain some general results
about the conditional measure of deterministic measures on random
sets, including Theorem \ref{thm:Gen exist both} and Proposition
\ref{prop:two def are the same}.
\begin{defn}
\label{def:random finite measure}The set of all finite Borel measures
$\mathcal{M}_{+}(X)$ on $X$ equipped with the weak$^{*}$-topology
on the dual space of $C_{b}(X)$ is a separable metrizable topological
space (see Proposition \ref{prop:The-Prohorov-distance}). A random,
finite, Borel measure is an element of $\mathcal{L}^{0}\left(\mathcal{M}_{+}(X)\right)$,
i.e. a finite Borel measure valued random variable.
\end{defn}

\begin{lem}
\label{lem:limit is measure}Let $\mu_{k}$ be a sequence of random,
finite Borel measures. If there exists $H\in\mathcal{A}$ with $P(H)=1$
such that for every outcome $\omega\in H$ we have that $\mu_{k}$
weakly converges to a finite, Borel measure $\mu$ (note that $\mu$
depends on $\omega\in H$) then $\mu$ is a random, finite Borel measure.
\end{lem}

The lemma follows from Lemma \ref{lem:random lim is random} and basically
states that the pointwise limit of random measures is a random measure.
We will use this fact throughout the paper without referring to it.

\subsection{Weak convergence subsequentially in probability\label{subsec:Weak-convergence-subsequentiuall}}

\begin{defn}
\label{def:weak conv in prob SEQ}Let $\mu$ and $\mu_{k}$ be a sequence
of random, finite, Borel measures on $X$. We say that \textit{$\mu_{k}$
weakly converges to $\mu$ subsequentially in probability} if for
every subsequence $\left\{ \alpha_{k}\right\} _{k=1}^{\infty}$ of
$\mathbb{N}$ there exists a subsequence $\left\{ \beta_{k}\right\} _{k=1}^{\infty}$
of $\left\{ \alpha_{k}\right\} _{k=1}^{\infty}$ and an event $H\in\mathcal{A}$
with $P(H)=1$ such that $\mu_{\beta_{k}}$ converges weakly to $\mu$
on the event $H$.
\end{defn}

\begin{rem}
It follows from the Definition \ref{def: weak conv in prob} that
if $\mu_{k}$ is a sequence of random, finite, Borel measures on $X$
such that $\mu_{k}$ weakly converges to a random, finite, Borel measure
$\mu$ almost surely then $\mu_{k}$ weakly converges to $\mu$ subsequentially
in probability.
\end{rem}

\begin{prop}
\label{prop:unique random lim}The limit in Definition \ref{def:weak conv in prob SEQ}
is unique in $\mathcal{L}^{0}\left(\mathcal{M}_{+}(X)\right)$ if
exists. 
\end{prop}

\begin{proof}
Assume that a sequence of random, finite Borel measures converges
weakly to both of the random, finite Borel measures $\mu$ and $\nu$
subsequentially in probability. Then there exists $\left\{ \alpha_{k}\right\} _{k=1}^{\infty}$
and an event $H\in\mathcal{A}$ with $P(H)=1$ such that $\mu_{\alpha_{k}}$
weakly converges to both $\mu$ and $\nu$ on the event $H$. Hence
$\mu=\nu$ on the event $H$ by Lemma \ref{lem:weak conv unique}.
\end{proof}
\begin{defn}
Let $\mu$ and $\nu$ be two random, finite, Borel measures. We define
\[
\rho_{\pi}(\mu,\nu)=E\left(\frac{\pi(\mu,\nu)}{1+\pi(\mu,\nu)}\right)
\]
where $\pi$ is the Prohorov distance defined in Definition \ref{def:The-Prohorov-distance}.
\end{defn}

The following proposition shows that the weak convergence subsequentially
in probability is the same as the convergence in probability of Borel
random variables that take values in the metric space $(\mathcal{M}_{+}(X),\pi)$.
We use this equivalence, of the weak convergence subsequentially in
probability and the convergence of $\rho_{\pi}$, throughout the paper
without referencing to it.
\begin{prop}
\label{prop:rho pi metric}We have that $\rho_{\pi}$ is a metric
on $\mathcal{L}^{0}\left(\mathcal{M}_{+}(X)\right)$ and $\mu_{k}$
weakly converges to $\mu$ subsequentially in probability if and only
if $\lim_{k\rightarrow\infty}\rho_{\pi}(\mu_{k},\mu)=0$.
\end{prop}

\begin{proof}
The fact that $\rho_{\pi}$ is a metric can be shown similarly to
that $\rho$ is a metric, depending on the fact that $x/(1+x)$ is
monotone increasing concave function on the positive reals, for reference
see \cite[Theorem 3.5]{Dudley-prob}. The statement follows from Proposition
\ref{prop:The-Prohorov-distance} and Remark \ref{rem: SUBSEQ IFF conv prob}.
\end{proof}
\begin{prop}
\label{prop:triv equivalence of conv}Let $\mu_{k}$ be a sequence
of random, finite, Borel measures on $X$. If $\mu_{k}$ weakly converges
to a random, finite, Borel measure $\mu$ subsequentially in probability
then $\intop_{X}f(x)\mathrm{d}\mu_{k}(x)$ converges to $\intop_{X}f(x)\mathrm{d}\mu(x)$
in probability for every $f\in C_{b}(X)$.
\end{prop}

\begin{proof}
Let $f\in C_{b}(X)$. For every subsequence $\left\{ \alpha_{k}\right\} _{k=1}^{\infty}$
of $\mathbb{N}$ there exists a subsequence $\left\{ \beta_{k}\right\} _{k=1}^{\infty}$
of $\left\{ \alpha_{k}\right\} _{k=1}^{\infty}$ and an event $H\in\mathcal{A}$
with $P(H)=1$ such that $\mu_{\beta_{k}}$ converges weakly to $\mu$
on the event $H$. Thus $\intop_{X}f(x)\mathrm{d}\mu_{\beta_{k}}(x)$
converges to $\intop_{X}f(x)\mathrm{d}\mu(x)$ almost surely. Hence
$\intop_{X}f(x)\mathrm{d}\mu_{k}(x)$ converges to $\intop_{X}f(x)\mathrm{d}\mu(x)$
in probability by Remark \ref{rem: SUBSEQ IFF conv prob}.
\end{proof}
\begin{prop}
\label{prop:non trivuquivalence of conv}Let $K\subseteq X$ be a
compact subset, let $\Psi\subseteq C_{b}(K)$ be a countable dense
subset with respect to the supremum norm and let $\mu$ and $\mu_{k}$
be a sequence of random, finite Borel measures on $K$. If $\intop_{X}f(x)\mathrm{d}\mu_{k}(x)$
converges to $\intop_{X}f(x)\mathrm{d}\mu(x)$ in probability for
every $f\in\Psi$ then $\mu_{k}$ weakly converges to $\mu$ subsequentially
in probability.
\end{prop}

\begin{proof}
By Lemma \ref{lem:mutual convergence subsequance} for every subsequence
$\left\{ \alpha_{k}\right\} _{k=1}^{\infty}$ of $\mathbb{N}$ there
exists a subsequence $\left\{ \beta_{k}\right\} _{k=1}^{\infty}$
of $\left\{ \alpha_{k}\right\} _{k=1}^{\infty}$ and there exists
an event $H\in\mathcal{A}$ with $P(H)=1$ such that $\intop_{X}f(x)\mathrm{d}\mu_{\beta_{k}}(x)$
converges to $\intop_{X}f(x)\mathrm{d}\mu(x)$ for every $f\in\Psi$
on the event $H$. We have that $\mu_{\beta_{k}}$ weakly converges
to a measure $\tau$ on the event $H$ by Lemma \ref{lem:deterministic con}.
We have that $\intop_{X}f(x)\mathrm{d}\mu(x)=\intop_{X}f(x)\mathrm{d}\tau(x)$
for every $f\in\Psi$ on the event $H$ and hence $\tau=\mu$ on the
event $H$ by Lemma \ref{lem:countabl determines mu}. So $\mu_{\beta_{k}}$
weakly converges to the measure $\mu$ on the event $H$.
\end{proof}
\begin{prop}
\label{thm:representation thm}Let $K\subseteq X$ be a compact subset,
let $\Psi\subseteq C_{b}(K)$ be a countable dense subset with respect
to the supremum norm and let $\mu_{k}$ be a sequence of random Borel
measures on $K$ such that $\intop_{X}f(x)\mathrm{d}\mu_{k}(x)$ converges
in probability to a random limit $S(f)$ for every $f\in\Psi$ and
$\left|S(f)\right|<\infty$ almost surely. Then $\mu_{k}$ weakly
converges to a random, finite, Borel measure $\mu$ subsequentially
in probability.
\end{prop}

\begin{proof}
By Lemma \ref{lem:mutual convergence subsequance} for every subsequence
$\left\{ \alpha_{k}\right\} _{k=1}^{\infty}$ of $\mathbb{N}$ there
exists a subsequence $\left\{ \beta_{k}\right\} _{k=1}^{\infty}$
of $\left\{ \alpha_{k}\right\} _{k=1}^{\infty}$ and there exists
an event $H\in\mathcal{A}$ with $P(H)=1$ such that $\intop_{X}f(x)\mathrm{d}\mu_{\beta_{k}}(x)$
converges to $S(f)$ for every $f\in\Psi$ on the event $H$. Then
$\mu_{\beta_{k}}$ weakly converges to a random, finite, Borel measure
$\mu$ on the event $H$ by Lemma \ref{lem:deterministic con} and
Lemma \ref{lem:limit is measure}. Thus $\intop_{X}f(x)\mathrm{d}\mu_{k}(x)$
converges to $S(f)=\intop_{X}f(x)\mathrm{d}\mu(x)$ in probability
for every $f\in\Psi$ and so $\mu_{k}$ weakly converges to $\mu$
subsequentially in probability by Proposition \ref{prop:non trivuquivalence of conv}.
\end{proof}
\begin{prop}
\label{Prop:single integral prob conv}Let $\nu$ be a deterministic
Borel measure on $X$ and $\mu_{k}$ be a sequence of random, finite,
Borel measures on $X$ such that $\mu_{k}\ll\nu$ almost surely for
every $k$, there exists $c>0$ such that $E\left(\frac{\mathrm{d}\mu_{k}}{\mathrm{d}\nu}(x)\right)\leq c$
for every $k\in\mathbb{N}$, $x\in X$ and $\mu_{k}(A)$ converges
in probability for every compact set $A\subseteq X$. Let $f:X\longrightarrow\mathbb{R}$
be a Borel measurable function such that $\intop_{X}\left|f(x)\right|\mathrm{d}\nu(x)<\infty$.
Then $\intop_{X}f(x)\mathrm{d}\mu_{k}(x)$ converges to a random variable
$Y$ in probability and $E(\left|Y\right|)\leq c\intop_{X}\left|f(x)\right|\mathrm{d}\nu(x)$.
\end{prop}

\begin{proof}
It is enough to prove the statement for a nonnegative $f$. We have
that
\begin{equation}
E\left(\intop_{X}g(x)\mathrm{d}\mu_{k}(x)\right)=E\left(\intop_{X}g(x)\frac{\mathrm{d}\mu_{k}}{\mathrm{d}\nu}(x)\mathrm{d}\nu(x)\right)\leq c\intop_{X}g(x)\mathrm{d}\nu(x)\label{eq:fub hel 1}
\end{equation}
for every nonnegative Borel function $g$ by Fubini`s theorem. Let
$g_{n}(x)=\sum_{i=1}^{N_{n}}b_{i,n}\cdot\chi_{A_{i,n}}(x)$, where
$0\leq b_{i,n}<\infty$, $N_{n}\in\mathbb{N}$ and $A_{i,n}\subseteq X$
are compact subsets, such that $g_{n}\leq f$ on $X$ and $0\leq\intop_{X}f(x)-g_{n}(x)\mathrm{d}\nu(x)<1/n$,
note that by the definition of Lebesgue integration we can find such
$g_{n}$ with $A_{i,n}$ being Borel sets and by inner regularity
 we can further assume the $A_{i,n}$ to be compact. By assumption
$\intop_{X}g_{n}(x)\mathrm{d}\mu_{k}(x)$ converges in probability
as $k$ goes to infinity to a random variable $Y_{n}$, thus $\lim_{k\rightarrow\infty}\rho\left(\intop_{X}g_{n}(x)\mathrm{d}\mu_{k}(x),Y_{n}\right)=0$.
Hence
\[
\limsup_{k\rightarrow\infty}\rho\left(\intop_{X}f(x)\mathrm{d}\mu_{k}(x),Y_{n}\right)\leq\limsup_{k\rightarrow\infty}\rho\left(\intop_{X}f(x)\mathrm{d}\mu_{k}(x),\intop_{X}g_{n}(x)\mathrm{d}\mu_{k}(x)\right)
\]
\[
+\rho\left(\intop_{X}g_{n}(x)\mathrm{d}\mu_{k}(x),Y_{n}\right)\leq\limsup_{k\rightarrow\infty}E\left|\intop_{X}f(x)\mathrm{d}\mu_{k}(x)-\intop_{X}g_{n}(x)\mathrm{d}\mu_{k}(x)\right|+0
\]
\[
\leq\limsup_{k\rightarrow\infty}E\left(\intop_{X}\left|f(x)-g_{n}(x)\right|\mathrm{d}\mu_{k}(x)\right)\leq c\intop_{X}f(x)-g_{n}(x)\mathrm{d}\nu(x)<c/n,
\]
and so there exists $m_{n}\in\mathbb{N}$ such that $\rho\left(\intop_{X}f(x)\mathrm{d}\mu_{k}(x),Y_{n}\right)<c/n$
for every $k\geq m_{n}$. If $k\geq\max\left\{ m_{n},m_{l}\right\} $
then
\[
\rho(Y_{n},Y_{l})\leq\rho\left(\intop_{X}f(x)\mathrm{d}\mu_{k}(x),Y_{n}\right)+\rho\left(\intop_{X}f(x)\mathrm{d}\mu_{k}(x),Y_{l}\right)\leq c/n+c/l.
\]
It follows that $Y_{n}$ is a Cauchy sequence for the metric $\rho$,
which is a complete metric. Let $Y$ be the limit of $Y_{n}$ in probability.
Then
\[
\limsup_{k\rightarrow\infty}\rho\left(\intop_{X}f(x)\mathrm{d}\mu_{k}(x),Y\right)\leq\limsup_{k\rightarrow\infty}\rho\left(\intop_{X}f(x)\mathrm{d}\mu_{k}(x),Y_{n}\right)+\rho\left(Y_{n},Y\right)\leq c/n+\rho\left(Y_{n},Y\right).
\]
By taking limit $n$ goes to infinity it follows that $\intop_{X}f(x)\mathrm{d}\mu_{k}(x)$
converges to $Y$ in probability. By applying Lemma \ref{lem:finite expectation stochastic limit}
twice and by (\ref{eq:fub hel 1})
\[
E(Y)\leq\liminf_{n\rightarrow\infty}E(Y_{n})\leq\liminf_{n\rightarrow\infty}\liminf_{k\rightarrow\infty}E(\intop_{X}g_{n}(x)\mathrm{d}\mu_{k}(x))\leq c\liminf_{n\rightarrow\infty}\intop_{X}g_{n}(x)\mathrm{d}\nu(x)\leq c\intop_{X}f(x)\mathrm{d}\nu(x).
\]
\end{proof}
\begin{cor}
\label{cor:compact conv cor}Let $\nu$ be a deterministic Borel measure
on $X$ such that $\mathrm{supp}\nu$ is compact. Let $\mu_{k}$ be
a sequence of random, finite, Borel measures on $X$ such that $\mu_{k}\ll\nu$
almost surely for every $k$, there exists $c>0$ such that $E\left(\frac{\mathrm{d}\mu_{k}}{\mathrm{d}\nu}(x)\right)\leq c$
for every $k\in\mathbb{N}$, $x\in X$ and $\mu_{k}(A)$ converges
to a random variable $\mu_{\infty}(A)$ in probability for every compact
set $A\subseteq X$. Then $\mu_{k}$ weakly converges to a random,
finite, Borel measure $\mu$ subsequentially in probability.
\end{cor}

\begin{proof}
There exists $\Psi\subseteq C_{b}(\mathrm{supp}\nu)$ countable and
dense subset by Lemma \ref{lem:Separable C_b(S)}. The conditions
of Proposition \ref{thm:representation thm} are satisfied for $K=\mathrm{supp}\nu$
by Proposition \ref{Prop:single integral prob conv}. Thus there exists
a random Borel measure $\mu$ such that $\mu_{k}$ weakly converges
to $\mu$ subsequentially in probability.
\end{proof}
\begin{lem}
\label{lem:abs cont weak prob conv fatou for open sets}Let $g:X\longrightarrow\mathbb{R}$
be a nonnegative Borel function and let $\nu$ be a deterministic
finite Borel measure on $X$. Let $\mu_{k}$ be a sequence of random,
finite, Borel measures on $X$ such that $\mu_{k}\ll\nu$ almost surely
for every $k$ and there exists $c>0$ such that $E\left(\frac{\mathrm{d}\mu_{k}}{\mathrm{d}\nu}(x)\right)\leq c$
for every $k\in\mathbb{N}$, $x\in X$. If $\mu_{k}$ weakly converges
to a random, finite, Borel measure $\mu$ subsequentially in probability
then $E\left(\intop_{X}g(x)\mathrm{d}\mu(x)\right)\leq c\cdot\intop_{X}g(x)\mathrm{d}\nu(x)$
and $E\left(\intop_{X}g(x)\mathrm{d}\mu_{k}(x)\right)\leq c\cdot\intop_{X}g(x)\mathrm{d}\nu(x)$.
\end{lem}

\begin{proof}
Let $G\subseteq X$ be an open set. Then
\[
E\left(\intop_{X}\chi_{G}(x)\mathrm{d}\mu_{k}(x)\right)=E\left(\intop_{X}\chi_{G}(x)\frac{\mathrm{d}\mu_{k}}{\mathrm{d}\nu}(x)\mathrm{d}\nu(x)\right)\leq c\intop_{X}\chi_{G}(x)\mathrm{d}\nu(x)
\]
by Fubini`s theorem. Since it holds for every open set $G$ it follows
by Lemma \ref{lem:leq meas lem} that
\begin{equation}
E\left(\mu_{k}(A)\right)\leq c\cdot\nu(A)\label{eq:sagg}
\end{equation}
for every Borel set $A$. Hence
\[
E\left(\intop_{X}g(x)\mathrm{d}\mu_{k}(x)\right)\leq c\cdot\intop_{X}g(x)\mathrm{d}\nu(x).
\]

Let again $G\subseteq X$ be an open set. There exists a subsequence
$\left\{ \beta_{k}\right\} _{k=1}^{\infty}$ of $\mathbb{N}$ and
an event $H\in\mathcal{A}$ with $P(H)=1$ such that $\mu_{\beta_{k}}$
weakly converges to $\mu$ on the event $H$. Then by Remark \ref{rem:liminf open weak conv},
Fatou`s lemma and (\ref{eq:sagg})
\[
E\left(\mu(G)\right)\leq E\left(\liminf_{k\rightarrow\infty}\mu_{\beta_{k}}(G)\right)
\]
\[
\leq\liminf_{k\rightarrow\infty}E\left(\mu_{\beta_{k}}(G)\right)=c\cdot\nu(G).
\]
 Since it holds for every open set $G$ it follows by Lemma \ref{lem:leq meas lem}
that $E\left(\mu(A)\right)\leq c\cdot\nu(A)$ for every Borel set
$A$ and hence
\[
E\left(\intop_{X}g(x)\mathrm{d}\mu(x)\right)\leq c\cdot\intop_{X}g(x)\mathrm{d}\nu(x).
\]
\end{proof}
\begin{prop}
\label{prop:SUM lem subseq}Let $\mu^{i}$ and $\mu_{k}^{i}$ be a
sequence of random, finite, Borel measures on $X$ for every $i\in\mathbb{N}$.
Assume that $\mu_{k}^{i}$ weakly converges to $\mu^{i}$ subsequentially
in probability for every $i\in\mathbb{N}$ as $k$ goes to $\infty$.
Assume that for every $\varepsilon>0$ there exist $N,n_{0}\in\mathbb{N}$
such that $\sum_{i=N}^{\infty}E(\mu_{k}^{i}(X))<\varepsilon$ for
every $k\geq n_{0}$. Then there exists $n_{1}\in\mathbb{N}$ such
that $\sum_{i\in\mathbb{N}}\mu_{k}^{i}$ is a sequence of random,
finite Borel measures, for $k\geq n_{1}$, that weakly converges to
the random, finite Borel measure $\sum_{i\in\mathbb{N}}\mu^{i}$ subsequentially
in probability.
\end{prop}

\begin{proof}
Since there exists $N_{1},n_{1}\in\mathbb{N}$ such that $\sum_{i=N_{1}}^{\infty}E(\mu_{k}^{i}(X))\leq1$
for every $k\geq n_{1}$ it follows that $\sum_{i=N_{1}}^{\infty}\mu_{k}^{i}(X)<\infty$
almost surely for $k\geq n_{1}$ and so $\sum_{i=1}^{\infty}\mu_{k}^{i}(X)<\infty$
almost surely for $k\geq n_{1}$ . Thus $\sum_{i=1}^{\infty}\mu_{k}^{i}$
is a random, finite, Borel measure.

Since $\mu_{k}^{i}$ weakly converges to $\mu^{i}$ subsequentially
in probability it follows by Proposition \ref{prop:triv equivalence of conv}
that $\mu_{k}^{i}(X)$ converges to $\mu^{i}(X)$ in probability as
$k$ goes to $\infty$. Thus we can find, by Lemma \ref{lem:mutual convergence subsequance},
a subsequence $\left\{ \alpha_{k}\right\} _{k=1}^{\infty}$ of $\mathbb{N}$
and an event $H\in\mathcal{A}$ with $P(H)=1$ such that $\mu_{\alpha_{k}}^{i}(X)$
converges to $\mu^{i}(X)$ on $H$ as $k$ goes to $\infty$ for every
$i\in\mathbb{N}$. Then by Fatou`s lemma and Fubini`s theorem
\begin{equation}
E(\sum_{i=N_{1}}^{\infty}\mu^{i}(X))=E(\sum_{i=N_{1}}^{\infty}\lim_{k\rightarrow\infty}\mu_{\alpha_{k}}^{i}(X))\leq\liminf_{k\rightarrow\infty}\sum_{i=N_{1}}^{\infty}E(\mu_{\alpha_{k}}^{i}(X))\leq1.\label{eq:sum lem -1}
\end{equation}
Thus $\sum_{i=N_{1}}^{\infty}\mu^{i}(X)<\infty$ almost surely and
so $\sum_{i=1}^{\infty}\mu^{i}(X)<\infty$ almost surely. Hence $\sum_{i=1}^{\infty}\mu^{i}$
is a random, finite, Borel measure.

Let $\varepsilon>0$ be fixed and let $N,n_{0}\in\mathbb{N}$ such
that $\sum_{i=N}^{\infty}E(\mu_{k}^{i}(X))<\varepsilon$ for every
$k\geq n_{0}$. Similarly to (\ref{eq:sum lem -1}) we have that $E(\sum_{i=N}^{\infty}\mu^{i}(X))\leq\varepsilon$.
Thus
\[
\rho_{\pi}(\sum_{i=1}^{\infty}\mu_{k}^{i},\sum_{i=1}^{\infty}\mu^{i})\leq
\]
\[
\rho_{\pi}(\sum_{i=1}^{\infty}\mu_{k}^{i},\sum_{i=1}^{N-1}\mu_{k}^{i})+\rho_{\pi}(\sum_{i=1}^{N-1}\mu_{k}^{i},\sum_{i=1}^{N-1}\mu^{i})+\rho_{\pi}(\sum_{i=1}^{N-1}\mu^{i},\sum_{i=1}^{\infty}\mu^{i})
\]
\[
\leq\varepsilon+\rho_{\pi}(\sum_{i=1}^{N-1}\mu_{k}^{i},\sum_{i=1}^{N-1}\mu^{i})+\varepsilon
\]
where we used the fact that $\rho_{\pi}(\mu,\nu)\leq E(\pi(\mu,\nu))$
and Lemma \ref{lem:sum triv bound}. Thus
\[
\limsup_{k\rightarrow\infty}\rho_{\pi}(\sum_{i=1}^{\infty}\mu_{k}^{i},\sum_{i=1}^{\infty}\mu^{i})\leq2\varepsilon
\]
since $\sum_{i=1}^{N-1}\mu_{k}^{i}$ weakly converges to $\sum_{i=1}^{N-1}\mu^{i}$
subsequentially in probability. By taking limit $\varepsilon$ goes
to $0$ it follows by Proposition \ref{prop:rho pi metric} that $\sum_{i=1}^{\infty}\mu_{k}^{i}$
weakly converges to $\sum_{i=1}^{\infty}\mu^{i}$ subsequentially
in probability.
\end{proof}
\begin{thm}
\label{thm:ABS cont representation}Let $\nu$ be a deterministic,
finite, Borel measure on $X$. Let $\mu_{k}$ be a sequence of random,
finite, Borel measures on $X$ such that $\mu_{k}\ll\nu$ almost surely
for every $k$, there exists $c>0$ such that $E\left(\frac{\mathrm{d}\mu_{k}}{\mathrm{d}\nu}(x)\right)\leq c$
for every $k\in\mathbb{N}$, $x\in X$ and $\mu_{k}(A)$ converges
to a random variable $\mu_{\infty}(A)$ in probability for every compact
set $A\subseteq X$. Then $\mu_{k}$ weakly converges to a random,
finite, Borel measure $\mu$ subsequentially in probability. Furthermore,
$\intop_{X}f(x)\mathrm{d}\mu_{k}(x)$ converges to a random variable
$S(f)$ in probability with $E(\left|S(f)\right|)\leq c\intop_{X}\left|f(x)\right|\mathrm{d}\nu(x)$
for every $f:X\longrightarrow\mathbb{R}$ Borel measurable function
such that $\intop_{X}\left|f(x)\right|\mathrm{d}\nu(x)<\infty$. For
every countable collection of Borel measurable functions $f_{n}:X\longrightarrow\mathbb{R}$
such that $\intop_{X}\left|f_{n}(x)\right|\mathrm{d}\nu(x)<\infty$
we have that $\intop_{X}f_{n}(x)\mathrm{d}\mu(x)=S(f_{n})$ for every
$n$ almost surely.
\end{thm}

\begin{proof}
Let $K_{1},K_{2},\dots$ be a sequence of disjoint compact subsets
of $X$ as in Lemma \ref{lem:countab exhaustion}. Let $\nu^{i}=\nu\vert_{K_{i}}$
and $\mu_{k}^{i}=\mu_{k}\vert_{K_{i}}$. Then $\mu_{k}=\sum_{i=1}^{\infty}\mu_{k}^{i}$
and $\mu_{k}^{i}$ weakly converges to a random, finite Borel measure
$\mu^{i}$ subsequentially in probability by Corollary \ref{cor:compact conv cor}.
Since $\sum_{i=N}^{\infty}E\left(\mu_{k}^{i}(X)\right)\leq c\sum_{i=N}^{\infty}\nu(K_{i})<\infty$
it follows that the conditions of Proposition \ref{prop:SUM lem subseq}
are satisfied and so $\mu=\sum_{i=1}^{\infty}\mu^{i}$ is a random,
finite, Borel measure and $\mu_{k}$ weakly converges to $\mu$ subsequentially
in probability.

By Proposition \ref{Prop:single integral prob conv} it follows that
$\intop_{X}f(x)\mathrm{d}\mu_{k}(x)$ converges to a random variable
$S(f)$ in probability with $E(\left|S(f)\right|)\leq c\intop_{X}\left|f(x)\right|\mathrm{d}\nu(x)$
for every $f:X\longrightarrow\mathbb{R}$ Borel measurable function
such that $\intop_{X}\left|f(x)\right|\mathrm{d}\nu(x)<\infty$.

Let $n\in\mathbb{N}$ be fixed. For every $\varepsilon>0$ we can
find $g\in C_{b}(X)$ such that $\intop_{X}\left|f_{n}(x)-g(x)\right|\mathrm{d}\nu(x)<\varepsilon$
since $C_{b}(X)$ is dense in $\mathcal{L}^{1}(X)$ (see \cite[Proposition 1.3.22]{Cb(X) dense in L1}).
Then
\[
\rho\left(\intop_{X}f_{n}(x)\mathrm{d}\mu(x),S(f_{n})\right)
\]

\[
\leq\rho\left(\intop_{X}f_{n}(x)\mathrm{d}\mu(x),\intop_{X}g(x)\mathrm{d}\mu(x)\right)+\rho\left(\intop_{X}g(x)\mathrm{d}\mu(x),\intop_{X}g(x)\mathrm{d}\mu_{k}(x)\right)
\]
\[
+\rho\left(\intop_{X}g(x)\mathrm{d}\mu_{k}(x),\intop_{X}f_{n}(x)\mathrm{d}\mu_{k}(x)\right)+\rho\left(\intop_{X}f_{n}(x)\mathrm{d}\mu_{k}(x),S(f_{n})\right)
\]
\[
\leq E\left(\intop_{X}\left|f_{n}(x)-g(x)\right|\mathrm{d}\mu(x)\right)+\rho\left(\intop_{X}g(x)\mathrm{d}\mu(x),\intop_{X}g(x)\mathrm{d}\mu_{k}(x)\right)
\]
\[
+E\left(\intop_{X}\left|g(x)-f_{n}(x)\right|\mathrm{d}\mu_{k}(x)\right)+\rho\left(\intop_{X}f_{n}(x)\mathrm{d}\mu_{k}(x),S(f_{n})\right)
\]
\[
\leq c\cdot\intop_{X}\left|f_{n}(x)-g(x)\right|\mathrm{d}\nu(x)+\rho\left(\intop_{X}g(x)\mathrm{d}\mu(x),\intop_{X}g(x)\mathrm{d}\mu_{k}(x)\right)+
\]
\[
c\cdot\intop_{X}\left|f_{n}(x)-g(x)\right|\mathrm{d}\nu(x)+\rho\left(\intop_{X}f_{n}(x)\mathrm{d}\mu_{k}(x),S(f_{n})\right)
\]
where we used Lemma \ref{lem:abs cont weak prob conv fatou for open sets}
and Remark \ref{rem:exp quality}. Hence taking limit $k$ goes to
infinity it follows that
\[
\rho\left(\intop_{X}f_{n}(x)\mathrm{d}\mu(x),S(f_{n})\right)\leq2c\cdot\varepsilon
\]
since $\intop_{X}g(x)\mathrm{d}\mu_{k}(x)$ converges to $\intop_{X}g(x)\mathrm{d}\mu(x)$
in probability by Proposition \ref{prop:triv equivalence of conv}
and $\intop_{X}f_{n}(x)\mathrm{d}\mu_{k}(x)$ converges to $S(f_{n})$
in probability by the definition of $S(f_{n})$. Taking limit $\varepsilon$
goes to $0$ it follows that $\intop_{X}f_{n}(x)\mathrm{d}\mu(x)=S(f_{n})$
almost surely.
\end{proof}
\begin{rem}
\label{rem:convergenc for sets to the measure}In particular, if we
take $f_{n}=\chi_{A_{n}}$ in Theorem \ref{thm:ABS cont representation}
for a countable collection of Borel sets $A_{n}$ then it follows
that $\mu(A_{n})=S(\chi_{A_{n}})$ almost surely for every $n$ where
$\mu_{k}(A_{n})$ converges to $S(\chi_{A_{n}})$ in probability as
$k$ goes to infinity.
\end{rem}

\begin{rem}
In Theorem \ref{thm:ABS cont representation} we can relax the condition
$E\left(\frac{\mathrm{d}\mu_{k}}{\mathrm{d}\nu}(x)\right)\leq c$
for every $k\in\mathbb{N}$, $x\in X$ . It is enough to assume that
there exists a nonnegative Borel function $c:X\longrightarrow\mathbb{R}$
such that $\int c(x)\mathrm{d}\nu(x)<\infty$ and $E\left(\frac{\mathrm{d}\mu_{k}}{\mathrm{d}\nu}(x)\right)\leq c(x)$
for $\nu$ almost every $x\in X$ for every $k$. It can be seen by
replacing $\mathrm{d}\nu(x)$ by $c(x)\mathrm{d}\nu(x)$.
\end{rem}

\begin{prop}
\label{prop:randomsupport}Let $\mu_{k}$ be a sequence of random,
finite, Borel measures on $X$ such that there exists a sequence of
random closed sets $F_{1}\supseteq F_{2}\supseteq\dots$ such that
$\mathrm{supp}\mu_{k}\subseteq F_{k}$ almost surely. If $\mu_{k}$
weakly converges to a random finite Borel measure $\mu$ subsequentially
in probability then $\mathrm{supp}\mu\subseteq\bigcap_{n=1}^{\infty}F_{n}$
almost surely.
\end{prop}

\begin{proof}
Let $G_{n}=X\setminus F_{n}$ be a random open set. Let $\left\{ \alpha_{k}\right\} _{k=1}^{\infty}$
be a subsequence of $\mathbb{N}$ such that $\mu_{\alpha_{k}}$ weakly
converges to $\mu$ on the event $H\in\mathcal{A}$ and $P(H)=1$.
We can further assume that $\mathrm{supp}\mu_{k}\subseteq F_{k}$
on the event $H$. Then $\mu(G_{n})\leq\liminf_{k\rightarrow\infty}\mu_{\alpha_{k}}(G_{n})=0$
on the event $H$ by Remark \ref{rem:liminf open weak conv}. Hence
$\mu\left(\bigcup_{n=1}^{\infty}G_{n}\right)=0$ on the event $H$
and so $\mathrm{supp}\mu\subseteq\bigcap_{n=1}^{\infty}F_{n}$ almost
surely.
\end{proof}
\begin{lem}
\label{lem:chang meas lem}Let $\nu$ be a deterministic, finite,
Borel measure on $X$ with compact support. Let $\mu_{k}$ be a sequence
of random, finite, Borel measures on $X$ such that $\mu_{k}\ll\nu$
almost surely for every $k$, there exists $c>0$ such that $E\left(\frac{\mathrm{d}\mu_{k}}{\mathrm{d}\nu}(x)\right)\leq c$
for every $k\in\mathbb{N}$, $x\in X$ and $\mu_{k}(A)$ converges
to a random variable $\mu_{\infty}(A)$ in probability for every compact
set $A\subseteq X$. Let $f:X\longrightarrow\mathbb{R}$ be a nonnegative
Borel measurable function such that $\intop_{X}f(x)\mathrm{d}\nu(x)<\infty$.
Then the sequence of random, finite, Borel measures $f(x)\mathrm{d}\mu_{k}(x)$
weakly converges to the random, finite, Borel measure $f(x)\mathrm{d}\mu(x)$
subsequentially in probability where $\mu$ is the random, finite,
Borel measure such that $\mu_{k}$ weakly converges to $\mu$ subsequentially
in probability.
\end{lem}

\begin{proof}
We have that $E(\int_{X}f(x)\mathrm{d}\mu_{k}(x))<c\intop_{X}f(x)\mathrm{d}\nu(x)<\infty$
thus $f(x)\mathrm{d}\mu_{k}(x)$ is a random finite Borel measure
almost surely. We can find a countable and dense $\Psi\subseteq C_{b}(\mathrm{supp}(\nu))$
by Lemma \ref{lem:Separable C_b(S)}. By Theorem \ref{thm:ABS cont representation}
it follows that $\intop_{X}g(x)f(x)\mathrm{d}\mu_{k}(x)$ converges
to $\intop_{X}g(x)f(x)\mathrm{d}\mu(x)$ in probability for every
$g\in\Psi$ and $E(\int_{X}f(x)\mathrm{d}\mu(x))<c\intop_{X}f(x)\mathrm{d}\nu(x)<\infty$
thus $f(x)\mathrm{d}\mu(x)$ is a random finite Borel measure almost
surely. It follows that $f(x)\mathrm{d}\mu_{k}(x)$ weakly converges
to $f(x)\mathrm{d}\mu(x)$ subsequentially in probability by Proposition
\ref{prop:non trivuquivalence of conv}.
\end{proof}
\begin{prop}
\label{prop:change of meas prop}Let $\nu$ be a deterministic, finite,
Borel measure on $X$. Let $\mu_{k}$ be a sequence of random, finite,
Borel measures on $X$ such that $\mu_{k}\ll\nu$ almost surely for
every $k$, there exists $c>0$ such that $E\left(\frac{\mathrm{d}\mu_{k}}{\mathrm{d}\nu}(x)\right)\leq c$
for every $k\in\mathbb{N}$, $x\in X$ and $\mu_{k}(A)$ converges
to a random variable $\mu_{\infty}(A)$ in probability for every compact
set $A\subseteq X$. Let $f:X\longrightarrow\mathbb{R}$ be a nonnegative
Borel measurable function such that $\intop_{X}f(x)\mathrm{d}\nu(x)<\infty$.
Then the sequence of random, finite, Borel measures $f(x)\mathrm{d}\mu_{k}(x)$
weakly converges to the random, finite, Borel measure $f(x)\mathrm{d}\mu(x)$
subsequentially in probability where $\mu$ is the random, finite,
Borel measure such that $\mu_{k}$ weakly converges to $\mu$ subsequentially
in probability. Additionally, $\int_{A}f(x)\mathrm{d}\mu_{k}(x)$
converges to $\int_{A}f(x)\mathrm{d}\mu(x)$ in probability for every
Borel set $A\subseteq X$.
\end{prop}

\begin{proof}
Let $K_{1},K_{2},\dots$ be compact sets, $\nu^{i}=\nu\vert_{K_{i}}$,
$\mu_{k}^{i}=\mu_{k}\vert_{K_{i}}$, $\mu^{i}$ and $\mu$ as in the
proof of Theorem \ref{thm:ABS cont representation}. Then by Lemma
\ref{lem:chang meas lem} it follows that the sequence of measures
$f(x)\mathrm{d}\mu_{k}^{i}(x)$ weakly converges to $f(x)\mathrm{d}\mu^{i}(x)$
subsequentially in probability. Since
\[
\sum_{i=N}^{\infty}E\left(\intop f(x)\mathrm{d}\mu_{k}^{i}(x)\right)\leq c\sum_{i=N}^{\infty}\intop_{K_{i}}f(x)\mathrm{d}\nu(x)=c\intop_{X}f(x)\mathrm{d}\nu(x)<\infty,
\]
it follows by the application of Proposition \ref{prop:SUM lem subseq}
that $f(x)\mathrm{d}\mu(x)=\sum_{i=1}^{\infty}f(x)\mathrm{d}\mu^{i}(x)$
is a random, finite, Borel measure and $f(x)\mathrm{d}\mu_{k}(x)$
weakly converges to $f(x)\mathrm{d}\mu(x)$ subsequentially in probability.

The same argument for $\chi_{A}\cdot f$ in place of $f$ shows that
$\chi_{A}(x)f(x)\mathrm{d}\mu_{k}(x)$ weakly converges to $\chi_{A}(x)f(x)\mathrm{d}\mu(x)$
subsequentially in probability, in particular, $\int_{A}f(x)\mathrm{d}\mu_{k}(x)$
converges to $\int_{A}f(x)\mathrm{d}\mu(x)$ in probability by Proposition
\ref{prop:triv equivalence of conv}.
\end{proof}

\subsection{Weak convergence in probability\label{subsec:Weak-convergence-in prob}}

\begin{defn}
\label{def: weak conv in prob}Let $\mu_{k}$ be a sequence of random,
finite, Borel measures on $X$. We say that \textit{$\mu_{k}$ weakly
converges to a random, finite, Borel measure $\mu$ in probability}
if $\intop_{X}f(x)\mathrm{d}\mu_{k}(x)$ converges to $\intop_{X}f(x)\mathrm{d}\mu(x)$
in probability for every deterministic, bounded, continuous function
$f:X\longrightarrow\mathbb{R}$.
\end{defn}

\begin{prop}
\label{prop:The-convergence-weakly topology}The convergence weakly
in probability is induced by the topology which has base elements
formed by finite intersection of sets in the form:
\[
\left\{ \mu\in\mathcal{L}^{0}\left(\mathcal{M}_{+}(X)\right):\rho\left(\intop_{X}f(x)\mathrm{d}\mu(x),Y\right)<r\right\} 
\]
where $f\in C_{b}(X)$, $r>0$ and $Y$ is a real-valued random variable
with almost surely finite values.
\end{prop}

Proposition \ref{prop:The-convergence-weakly topology} can be verified
easily, we leave the details for the reader.
\begin{prop}
\label{prop:trivi direct}Let $\mu$ and $\mu_{k}$ be a sequence
of random, finite, Borel measures on $X$. If $\mu_{k}$ weakly converges
to $\mu$ subsequentially in probability then $\mu_{k}$ weakly converges
to $\mu$ in probability.
\end{prop}

Proposition \ref{prop:trivi direct} is a reformulation of Proposition
\ref{prop:triv equivalence of conv}.
\begin{prop}
\label{prop:non tyriv direct}Let $\mu$ and $\mu_{k}$ be a sequence
of random, finite, Borel measures supported on a deterministic compact
subset $K\subseteq X$. If $\mu_{k}$ weakly converges to $\mu$ in
probability then $\mu_{k}$ weakly converges to $\mu$ subsequentially
in probability.
\end{prop}

\begin{proof}
We can find a countable and dense $\Psi\subseteq C_{b}(K)$ by Lemma
\ref{lem:Separable C_b(S)}. Hence the statement follows from Proposition
\ref{prop:non trivuquivalence of conv}.
\end{proof}
\begin{prop}
\label{prop:weak unique weak conv}Let $\mu_{k}$ be a sequence of
random, finite, Borel measures. Assume that $\tau$ and $\nu$ are
random, finite, Borel measures on $X$ such that $\mu_{k}(G)$ converges
to $\tau(G)$ in probability and $\mu_{k}(G)$ converges to $\nu(G)$
in probability for every open set $G\subseteq X$. Then $\mu=\tau$
almost surely.
\end{prop}

\begin{proof}
Let $D\subseteq X$ be a countable and dense subset and let $\mathcal{G}$
be as in Lemma \ref{lem:open eq dense}. Let $G\in\mathcal{G}$. Then
$\mu_{k}(G)$ converges to both $\mu(G)$ and $\tau(G)$ in probability
hence $\mu(G)=\tau(G)$ almost surely. Since $\mathcal{G}$ is countable
it follows that $\mu(G)=\tau(G)$ for every $G\in\mathcal{G}$ almost
surely, i.e. there exists an event $H\in\mathcal{A}$ with $P(H)=1$
such that $\mu(G)=\tau(G)$ for every $G\in\mathcal{G}$. Then on
the event $H$ we have that $\mu(A)=\tau(A)$ for every Borel set
$A\subseteq X$ by Lemma \ref{lem:open eq dense}. Hence $\mu=\tau$
almost surely.
\end{proof}
\begin{prop}
\label{prop:sum lem}Let $\mu^{i}$ and $\mu_{k}^{i}$ be a sequence
of random, finite, Borel measures on $X$ for every $i\in\mathbb{N}$.
Assume that $\mu_{k}^{i}$ weakly converges to $\mu^{i}$ in probability
for every $i\in\mathbb{N}$ as $k$ goes to $\infty$. Assume that
for every $\varepsilon>0$ there exist $N,n_{0}\in\mathbb{N}$ such
that $\sum_{i=N}^{\infty}E(\mu_{k}^{i}(X))<\varepsilon$ for every
$k\geq n_{0}$. Then there exists $n_{1}\in\mathbb{N}$ such that
$\sum_{i\in\mathbb{N}}\mu_{k}^{i}$ is a sequence of random, finite
Borel measures, for $k\geq n_{1}$, that weakly converges to the random,
finite Borel measure $\sum_{i\in\mathbb{N}}\mu^{i}$ in probability.
\end{prop}

The proof of Proposition \ref{prop:sum lem} goes similarly to the
proof of Proposition \ref{prop:SUM lem subseq} with the difference
that instead of
\[
\rho_{\pi}(\sum_{i=1}^{\infty}\mu_{k}^{i}(X),\sum_{i=1}^{\infty}\mu^{i}(X))
\]
we need to estimate
\[
\rho(\sum_{i=1}^{\infty}\intop_{X}f(x)\mathrm{d}\mu_{k}^{i}(x),\sum_{i=1}^{\infty}\intop_{X}f(x)\mathrm{d}\mu^{i}(x))
\]
for a given $f\in C_{b}(X)$. We leave for the reader to check the
details. We provide a similar proof in the proof of Proposition \ref{prop:SUM cov Vague}.

\subsection{Vague convergence in probability\label{subsec:Vague-convergence-in prob}}

Throughout this section we assume that $X$ is locally compact.
\begin{defn}
\label{def:random loc fin measure}The set of all locally finite Borel
measures $\mathcal{M}_{l}(X)$ on $X$ equipped with the weak$^{*}$-topology
on the dual space of $C_{c}(X)$ is a Polish space (see Lemma \ref{prop:Polish vague}).
A random, finite, Borel measure is an element of $\mathcal{L}^{0}\left(\mathcal{M}_{l}(X)\right)$,
i.e. a locally finite Borel measure valued random variable.
\end{defn}

\begin{lem}
\label{lem:limit is measure-Vague}Let $\mu_{k}$ be a sequence of
random, locally finite, Borel measures. If there exists $H\in\mathcal{A}$
with $P(H)=1$ such that for every outcome $\omega\in H$ we have
that $\mu_{k}$ vaguely converges to a locally finite, Borel measure
$\mu$ (note that $\mu$ depends on $\omega\in H$) then $\mu$ is
a random, locally finite Borel measure.
\end{lem}

The lemma states that the pointwise limit of random locally finite
Borel measures is a random locally finite Borel measure and it follows
from Lemma \ref{lem:random lim is random}. We use this fact throughout
the paper with no reference.
\begin{defn}
\label{def:vague conv in prob}Let $\mu$ and $\mu_{k}$ be a sequence
of random, locally finite, Borel measures on $X$. We say that \textit{$\mu_{k}$
vaguely converges to $\mu$ in probability} if $\intop_{X}f(x)\mathrm{d}\mu_{k}(x)$
converges to $\intop_{X}f(x)\mathrm{d}\mu(x)$ in probability for
every $f\in C_{c}(X)$.
\end{defn}

\begin{rem}
It follows from Definition \ref{def: weak conv in prob} that if $\mu_{k}$
is a sequence of random, locally finite, Borel measures on $X$ such
that $\mu_{k}$ vaguely converges to a random Borel measure $\mu$
almost surely then $\mu_{k}$ vaguely converges to $\mu$ in probability.
\end{rem}

\begin{rem}
\label{rem:weak imply vague}It follows from Definition \ref{def: weak conv in prob}
that if $\mu_{k}$ is a sequence of random, finite, Borel measures
on $X$ such that $\mu_{k}$ weakly converges to a random, finite,
Borel measure $\mu$ in probability then $\mu_{k}$ vaguely converges
to $\mu$ in probability.
\end{rem}

\begin{prop}
\label{prop:The-convergence-vaguely topology}The convergence vaguely
in probability is induced by the topology which has base elements
formed by finite intersection of sets in the form:
\[
\left\{ \mu\in\mathcal{L}^{0}\left(\mathcal{M}_{l}(X)\right):\rho\left(\intop_{X}f(x)\mathrm{d}\mu(x),Y\right)<r\right\} 
\]
where $f\in C_{c}(X)$, $r>0$ and $Y$ is a real-valued random variable
with almost surely finite values.
\end{prop}

Proposition \ref{prop:The-convergence-vaguely topology} can be verified
easily, we leave the details for the reader.
\begin{prop}
\label{prop:unique random lim-Vague}Let $X$ be locally compact.
The limit in Definition \ref{def:vague conv in prob} is unique in
$\mathcal{L}^{0}\left(\mathcal{M}_{l}(X)\right)$ if exists.
\end{prop}

\begin{proof}
Let $\mu_{k}$ be a sequence of random, finite, Borel measures on
$X$ such that $\mu_{k}$ vaguely converges to a random, finite, Borel
measure $\mu$ and also to a random, finite, Borel measure $\tau$
in probability. Let $\Psi\subseteq C_{c}(X)$ as in Lemma \ref{lem:deterministic con-Vague}.
Then $\intop_{X}f(x)\mathrm{d}\mu_{k}(x)$ converges to $\intop_{X}f(x)\mathrm{d}\mu(x)$
in probability for every $f\in\Psi$ and also to $\intop_{X}f(x)\mathrm{d}\tau(x)$
in probability. Since $\Psi$ is countable it follows that $\intop_{X}f(x)\mathrm{d}\mu(x)=\intop_{X}f(x)\mathrm{d}\tau(x)$
for every $f\in\Psi$ almost surely. Thus $\mu=\nu$ almost surely
by Lemma \ref{lem:countab agree}.
\end{proof}
\begin{lem}
\label{prop:triv equivalence of conv-Vague}Let $X$ be locally compact.
Let $\mu$ and $\mu_{k}$ be a sequence of random, locally finite,
Borel measures on $X$. Assume that for every subsequence $\left\{ \alpha_{k}\right\} _{k=1}^{\infty}$
of $\mathbb{N}$ there exists a subsequence $\left\{ \beta_{k}\right\} _{k=1}^{\infty}$
of $\left\{ \alpha_{k}\right\} _{k=1}^{\infty}$ and there exists
an event $H\in\mathcal{A}$ with $P(H)=1$ such that $\mu_{k}$ vaguely
converges to $\mu$ on the event $H$ then $\intop_{X}f(x)\mathrm{d}\mu_{k}(x)$
converges to $\intop_{X}f(x)\mathrm{d}\mu(x)$ in probability for
every $f\in C_{c}(X)$.
\end{lem}

Lemma \ref{prop:triv equivalence of conv-Vague} can be proven similarly
to the proof of Proposition \ref{prop:triv equivalence of conv}.
\begin{lem}
\label{lem:masik non triv eq}Let $X$ be locally compact, let $\Psi\subseteq C_{c}(X)$
as in Lemma \ref{lem:deterministic con-Vague} and $\mu$ and $\mu_{k}$
be a sequence of random, locally finite Borel measures on $X$. If
$\intop_{X}f(x)\mathrm{d}\mu_{k}(x)$ converges to $\intop_{X}f(x)\mathrm{d}\mu(x)$
in probability for every $f\in\Psi$ then for every subsequence $\left\{ \alpha_{k}\right\} _{k=1}^{\infty}$
of $\mathbb{N}$ there exists a subsequence $\left\{ \beta_{k}\right\} _{k=1}^{\infty}$
of $\left\{ \alpha_{k}\right\} _{k=1}^{\infty}$ and there exists
an event $H\in\mathcal{A}$ with $P(H)=1$ such that $\mu_{k}$ vaguely
converges to $\mu$ on the event $H$.
\end{lem}

Lemma \ref{lem:masik non triv eq} can be shown similarly to the proof
of Proposition \ref{prop:non trivuquivalence of conv} by replacing
the use of Lemma \ref{lem:deterministic con} by the use of Lemma
\ref{lem:deterministic con-Vague} and the use of Lemma \ref{lem:countabl determines mu}
by Lemma \ref{lem:countab agree}.
\begin{lem}
\label{prop:non trivuquivalence of conv-Vague}Let $X$ be locally
compact, let $\Psi\subseteq C_{c}(X)$ as in Lemma \ref{lem:deterministic con-Vague}
and $\mu_{k}$ be a sequence of random, locally finite Borel measures
on $X$. Assume that $\intop_{X}f(x)\mathrm{d}\mu_{k}(x)$ converges
in probability to a random finite limit $S(f)$ for every $f\in\Psi$.
Then there exists a random, locally finite Borel measure $\mu$ on
$X$ such that for every subsequence $\left\{ \alpha_{k}\right\} _{k=1}^{\infty}$
of $\mathbb{N}$ there exists a subsequence $\left\{ \beta_{k}\right\} _{k=1}^{\infty}$
of $\left\{ \alpha_{k}\right\} _{k=1}^{\infty}$ and there exists
an event $H\in\mathcal{A}$ with $P(H)=1$ such that $\mu_{k}$ vaguely
converges to $\mu$ on the event $H$.
\end{lem}

Lemma \ref{prop:non trivuquivalence of conv-Vague} can be shown similarly
to the proof Proposition \ref{thm:representation thm} by replacing
the use of Lemma \ref{lem:deterministic con} by the use of Lemma
\ref{lem:deterministic con-Vague}, the use of Lemma \ref{lem:countabl determines mu}
by Lemma \ref{lem:countab agree} and the use of Proposition \ref{prop:non trivuquivalence of conv}
by the use of Lemma \ref{lem:masik non triv eq}.
\begin{thm}
\label{thm:equiv of vague conv}Let $X$ be locally compact and let
$\mu$ and $\mu_{k}$ be a sequence of random, locally finite Borel
measures on $X$. Then $\mu_{k}$ vaguely converges to $\mu$ in probability
if and only if for every subsequence $\left\{ \alpha_{k}\right\} _{k=1}^{\infty}$
of $\mathbb{N}$ there exists a subsequence $\left\{ \beta_{k}\right\} _{k=1}^{\infty}$
of $\left\{ \alpha_{k}\right\} _{k=1}^{\infty}$ and there exists
an event $H\in\mathcal{A}$ with $P(H)=1$ such that $\mu_{k}$ vaguely
converges to $\mu$ on the event $H$.
\end{thm}

Theorem \ref{thm:equiv of vague conv} follows from Lemma \ref{prop:triv equivalence of conv-Vague}
and Lemma \ref{lem:masik non triv eq}.
\begin{rem}
\label{rem:metrizable random vague}Let $X$ be locally compact. Proposition
\ref{prop:The-convergence-vaguely topology} states that the vague
convergence in probability is a topological convergence. Due to Theorem
\ref{thm:equiv of vague conv} the `convergence vaguely subsequentially
in probability' and the convergence vaguely in probability are the
same convergence. Hence it can be shown similarly to Proposition \ref{prop:rho pi metric}
that convergence vaguely in probability is also a metrizable convergence
using the fact that the vague convergence of deterministic locally
finite measures is a metric convergence, see Proposition \ref{prop:Polish vague}.
Similarly to Proposition \ref{prop:rho pi metric} the vague convergence
in probability is the same as the convergence in probability of Borel
random variables that take values in the metric space of $\mathcal{M}_{l}(X)$
equipped with the metric of Proposition \ref{prop:Polish vague}.
\end{rem}

\begin{prop}
\label{Prop:single integral prob conv-Vague}Let $X$ be locally compact.
Let $\nu$ be a deterministic, locally finite, Borel measure on $X$
and $\mu_{k}$ be a sequence of random, locally finite, Borel measures
on $X$ such that $\mu_{k}\ll\nu$ almost surely for every $k$, there
exists $c>0$ such that $E\left(\frac{\mathrm{d}\mu_{k}}{\mathrm{d}\nu}(x)\right)\leq c$
for every $k\in\mathbb{N}$, $x\in X$ and $\mu_{k}(A)$ converges
in probability for every compact set $A\subseteq X$. Let $f:X\longrightarrow\mathbb{R}$
be a Borel measurable function such that $\intop_{X}\left|f(x)\right|\mathrm{d}\nu(x)<\infty$.
Then $\intop_{X}f(x)\mathrm{d}\mu_{k}(x)$ converges to a random variable
$Y$ in probability and $E(\left|Y\right|)\leq c\intop_{X}\left|f(x)\right|\mathrm{d}\nu(x)$.
\end{prop}

The proof of Proposition \ref{Prop:single integral prob conv-Vague}
is identical to the proof of Proposition \ref{Prop:single integral prob conv}.
\begin{prop}
\label{prop:SUM cov Vague}Let $X$ be locally compact. Let $\mu^{i}$
and $\mu_{k}^{i}$ be a sequence of random, locally finite, Borel
measures on $X$ for every $i\in\mathbb{N}$. Assume that $\mu_{k}^{i}$
vaguely converges to $\mu^{i}$ in probability for every $i\in\mathbb{N}$
as $k$ goes to $\infty$. Assume that for every compact set $K\subseteq X$
and $\varepsilon>0$ there exist $N\in\mathbb{N}$ such that $\sum_{i=N}^{\infty}E(\mu_{k}^{i}(K))<\varepsilon$
for every $k\in\mathbb{N}$. Then $\sum_{i\in\mathbb{N}}\mu_{k}^{i}$
is a sequence of random, locally finite Borel measures that vaguely
converges to the random, locally finite Borel measure $\sum_{i\in\mathbb{N}}\mu^{i}$
in probability.
\end{prop}

\begin{proof}
Since for every compact set $K\subseteq X$ there exists $N_{1}\in\mathbb{N}$
such that $\sum_{i=N_{1}}^{\infty}E(\mu_{k}^{i}(K))\leq1$ for every
$k\in\mathbb{N}$ it follows that $\sum_{i=N_{1}}^{\infty}\mu_{k}^{i}(K)<\infty$
almost surely for $k\in\mathbb{N}$ and so $\sum_{i=1}^{\infty}\mu_{k}^{i}(K)<\infty$
almost surely. Thus $\sum_{i=1}^{\infty}\mu_{k}^{i}$ is a random,
locally finite, Borel measure since $X$ is locally compact. Let $h\in C_{c}(X)$
such that $h(x)=1$ for every $x\in K$ and $\left\Vert h\right\Vert _{\infty}\leq1$,
we can find such $h$ by Tietze`s extension theorem and the fact that
$X$ is locally compact. We can find $N_{2}\in\mathbb{N}$ such that
$\sum_{i=N_{2}}^{\infty}E(\mu_{k}^{i}(\mathrm{supp}(h))\leq1$ for
every $k\in\mathbb{N}$. We can find, by Lemma \ref{lem:mutual convergence subsequance},
a subsequence $\left\{ \alpha_{k}\right\} _{k=1}^{\infty}$ of $\mathbb{N}$
and an event $H\in\mathcal{A}$ with $P(H)=1$ such that $\intop_{X}h(x)\mathrm{d}\mu_{\alpha_{k}}^{i}(x)$
converges to $\intop_{X}h(x)\mathrm{d}\mu^{i}(x)$ on $H$ as $k$
goes to $\infty$ for every $i\in\mathbb{N}$. Then by Fatou`s lemma
and Fubini`s theorem
\[
E(\sum_{i=N_{2}}^{\infty}\mu^{i}(K))\leq E(\sum_{i=N_{2}}^{\infty}\intop_{X}h(x)\mathrm{d}\mu^{i}(x))=E(\sum_{i=N_{2}}^{\infty}\lim_{k\rightarrow\infty}\intop_{X}h(x)\mathrm{d}\mu_{\alpha_{k}}^{i}(x))
\]
\begin{equation}
\leq\liminf_{k\rightarrow\infty}\sum_{i=N_{2}}^{\infty}E(\intop_{X}h(x)\mathrm{d}\mu_{\alpha_{k}}^{i}(x))\leq\liminf_{k\rightarrow\infty}\sum_{i=N_{2}}^{\infty}E(\mu_{\alpha_{k}}^{i}(\mathrm{supp}(h)))\leq1.\label{eq:sum vague}
\end{equation}
Thus $\sum_{i=N_{1}}^{\infty}\mu^{i}(K)<\infty$ almost surely and
so $\sum_{i=1}^{\infty}\mu^{i}(K)<\infty$ almost surely. Hence $\sum_{i=1}^{\infty}\mu^{i}$
is a random, locally finite, Borel measure.

Let $f\in C_{c}(X)$ and $K=\mathrm{supp}(f)$. Let $\varepsilon>0$
be fixed and let $N\in\mathbb{N}$ such that $\sum_{i=N}^{\infty}E(\mu_{k}^{i}(K))<\varepsilon$
for every $k\in\mathbb{N}$. Similarly to (\ref{eq:sum vague}) we
can further assume that $E(\sum_{i=N}^{\infty}\mu^{i}(K))\leq\varepsilon$.
Thus
\[
\rho(\sum_{i=1}^{\infty}\intop_{X}f(x)\mathrm{d}\mu_{k}^{i}(x),\sum_{i=1}^{\infty}\intop_{X}f(x)\mathrm{d}\mu^{i}(x))\leq
\]
\[
\rho(\sum_{i=1}^{\infty}\intop_{X}f(x)\mathrm{d}\mu_{k}^{i}(x),\sum_{i=1}^{N-1}\intop_{X}f(x)\mathrm{d}\mu_{k}^{i}(x))+\rho(\sum_{i=1}^{N-1}\intop_{X}f(x)\mathrm{d}\mu_{k}^{i}(x),\sum_{i=1}^{N-1}\intop_{X}f(x)\mathrm{d}\mu^{i}(x))
\]
\[
+\rho(\sum_{i=1}^{N-1}\intop_{X}f(x)\mathrm{d}\mu^{i}(x),\sum_{i=1}^{\infty}\intop_{X}f(x)\mathrm{d}\mu^{i}(x))
\]
\[
\leq\varepsilon\left\Vert f\right\Vert _{\infty}+\rho(\sum_{i=1}^{N-1}\intop_{X}f(x)\mathrm{d}\mu_{k}^{i}(x),\sum_{i=1}^{N-1}\intop_{X}f(x)\mathrm{d}\mu^{i}(x))+\varepsilon\left\Vert f\right\Vert _{\infty}
\]
where we used the fact that $\rho(Y,Z)\leq E(\left|Y-Z\right|)$.
Thus
\[
\limsup_{k\rightarrow\infty}\rho(\sum_{i=1}^{\infty}\intop_{X}f(x)\mathrm{d}\mu_{k}^{i}(x),\sum_{i=1}^{\infty}\intop_{X}f(x)\mathrm{d}\mu^{i}(x))\leq2\varepsilon\left\Vert f\right\Vert _{\infty}
\]
since $\sum_{i=1}^{N-1}\intop_{X}f(x)\mathrm{d}\mu_{k}^{i}(x))$ converges
to $\sum_{i=1}^{N-1}\intop_{X}f(x)\mathrm{d}\mu^{i}(x)$ in probability.
By taking limit $\varepsilon$ goes to $0$ it follows that $\sum_{i=1}^{\infty}\intop_{X}f(x)\mathrm{d}\mu_{k}^{i}(x)$
converges to $\sum_{i=1}^{\infty}\intop_{X}f(x)\mathrm{d}\mu^{i}(x)$
in probability.
\end{proof}
\begin{thm}
\label{thm:ABS cont representation-Vague}Let $X$ be locally compact.
Let $\nu$ be a deterministic, locally finite, Borel measure on $X$.
Let $\mu_{k}$ be a sequence of random, locally finite, Borel measures
on $X$ such that $\mu_{k}\ll\nu$ almost surely for every $k$, there
exists $c>0$ such that $E\left(\frac{\mathrm{d}\mu_{k}}{\mathrm{d}\nu}(x)\right)\leq c$
for every $k\in\mathbb{N}$, $x\in X$ and $\mu_{k}(A)$ converges
to a random variable $\mu_{\infty}(A)$ in probability for every compact
set $A\subseteq X$. Then $\mu_{k}$ vaguely converges to a random,
locally finite, Borel measure $\mu$ in probability. Furthermore,
$\intop_{X}f(x)\mathrm{d}\mu_{k}(x)$ converges to a random variable
$S(f)$ in probability with $E(\left|S(f)\right|)\leq c\intop_{X}\left|f(x)\right|\mathrm{d}\nu(x)$
for every $f:X\longrightarrow\mathbb{R}$ Borel measurable function
such that $\intop_{X}\left|f(x)\right|\mathrm{d}\nu(x)<\infty$. For
every countable collection of Borel measurable functions $f_{n}:X\longrightarrow\mathbb{R}$
such that $\intop_{X}\left|f_{n}(x)\right|\mathrm{d}\mu(x)<\infty$
we have that $\intop_{X}f_{n}(x)\mathrm{d}\mu(x)=S(f_{n})$ almost
surely for every $n$.
\end{thm}

\begin{proof}
Due to Remark \ref{rem:countab exhaust vague} there exists a sequence
of disjoint compact sets $A_{i}$ such that $\nu^{i}=\nu\vert_{A_{i}}$
is a finite measure for every $i\in\mathbb{N}$ and $\nu\left(X\setminus\left(\cup_{i\in\mathbb{N}}A_{i}\right)\right)=0$.
Let $\mu_{k}^{i}=\mu_{k}\vert_{A_{i}}$. Then $\mu_{k}=\sum_{i=1}^{\infty}\mu_{k}^{i}$
and $\mu_{k}^{i}$ vaguely converges to a random, finite Borel measure
$\mu^{i}$ in probability by Remark \ref{rem:weak imply vague} and
Theorem \ref{thm:ABS cont representation}. Since $\sum_{i=N}^{\infty}E\left(\mu_{k}^{i}(K)\right)\leq c\sum_{i=N}^{\infty}\nu(K)<\infty$
for every compact set $K$ it follows that the conditions of Proposition
\ref{prop:SUM cov Vague} are satisfied and so $\mu=\sum_{i=1}^{\infty}\mu^{i}$
is a random, locally finite, Borel measure and $\mu_{k}$ vaguely
converges to $\mu$ in probability.

By Proposition \ref{Prop:single integral prob conv-Vague} it follows
that $\intop_{X}f(x)\mathrm{d}\mu_{k}(x)$ converges to a random variable
$S(f)$ in probability with $E(\left|S(f)\right|)\leq c\intop_{X}\left|f(x)\right|\mathrm{d}\nu(x)$
for every $f:X\longrightarrow\mathbb{R}$ Borel measurable function
such that $\intop_{X}\left|f(x)\right|\mathrm{d}\nu(x)<\infty$.

Let $n\in\mathbb{N}$ be fixed. For every $\varepsilon>0$ we can
find $g\in C_{c}(X)$ such that $\intop_{X}\left|f_{n}(x)-g(x)\right|\mathrm{d}\nu(x)<\varepsilon$
by \cite[Theorem 3.14]{Rudin}. The rest of the proof proceeds similarly
to the proof of Theorem \ref{thm:ABS cont representation}.
\end{proof}
\begin{rem}
\label{rem:convergenc for sets to the measure-Vague}In particular,
if we take $f_{n}=\chi_{A_{n}}$ in Theorem \ref{thm:ABS cont representation-Vague}
for a countable collection of Borel sets $A_{n}$ such that $\nu(A_{n})<\infty$
then it follows that $\mu(A_{n})=S(\chi_{A_{n}})$ almost surely for
every $n$ where $\mu_{k}(A_{n})$ converges to $S(\chi_{A_{n}})$
in probability as $k$ goes to infinity.
\end{rem}

\begin{rem}
In Theorem \ref{thm:ABS cont representation-Vague} we can relax the
condition $E\left(\frac{\mathrm{d}\mu_{k}}{\mathrm{d}\nu}(x)\right)\leq c$
for every $k$. It is enough to assume that there exists a nonnegative
Borel function $c:X\longrightarrow\mathbb{R}$ such that $\intop_{K}c(x)\mathrm{d}\nu(x)<\infty$
for every compact set $K$ and $E\left(\frac{\mathrm{d}\mu_{k}}{\mathrm{d}\nu}(x)\right)\leq c(x)$
for $\nu$ almost every $x\in X$ for every $k$ . It can be seen
by replacing $\mathrm{d}\nu(x)$ by $c(x)\mathrm{d}\nu(x)$.
\end{rem}

\begin{prop}
\label{prop:randomsupport-Vague}Let $\mu_{k}$ be a sequence of random,
locally finite, Borel measures on $X$ such that there exists a sequence
of random closed sets $F_{1}\supseteq F_{2}\supseteq\dots$ such that
$\mathrm{supp}\mu_{k}\subseteq F_{k}$ almost surely. If $\mu_{k}$
vaguely converges to a random, locally finite, Borel measure $\mu$
in probability then $\mathrm{supp}\mu\subseteq\bigcap_{n=1}^{\infty}F_{n}$
almost surely.
\end{prop}

Proposition \ref{prop:randomsupport-Vague} can be shown similarly
to the proof Proposition \ref{prop:randomsupport} due to the equivalence
in Theorem \ref{thm:equiv of vague conv}.

\begin{prop}
\label{prop:change of meaasure-Vague}Let $X$ be locally compact.
Let $\nu$ be a deterministic, locally finite, Borel measure on $X$.
Let $\mu_{k}$ be a sequence of random, locally finite, Borel measures
on $X$ such that $\mu_{k}\ll\nu$ almost surely for every $k$, there
exists $c>0$ such that $E\left(\frac{\mathrm{d}\mu_{k}}{\mathrm{d}\nu}(x)\right)\leq c$
for every $k\in\mathbb{N}$, $x\in X$ and $\mu_{k}(A)$ converges
to a random variable $\mu_{\infty}(A)$ in probability for every compact
set $A\subseteq X$. Let $f:X\longrightarrow\mathbb{R}$ be a nonnegative
Borel measurable function such that for every $y\in X$ there exists
a neigbourhood $U$ of $y$ such that $\intop_{U}f(x)\mathrm{d}\nu(x)<\infty$.
Then the sequence of random, locally finite, Borel measures $f(x)\mathrm{d}\mu_{k}(x)$
vaguely converges to the random, locally finite, Borel measure $f(x)\mathrm{d}\mu(x)$
in probability where $\mu$ is the random, locally finite, Borel measure
such that $\mu_{k}$ weakly converges to $\mu$ subsequentially in
probability. Additionally, $\int_{A}f(x)\mathrm{d}\mu_{k}(x)$ converges
to $\int_{A}f(x)\mathrm{d}\mu(x)$ in probability for any compact
set $A\subseteq X$.
\end{prop}

\begin{proof}
Let $U$ be an open set such that $\intop_{U}f(x)\mathrm{d}\nu(x)<\infty$.
Then $E(\intop_{U}f(x)\mathrm{d}\mu_{k}(x))\leq c\intop_{U}f(x)\mathrm{d}\nu(x)<\infty$
and by Theorem \ref{thm:ABS cont representation-Vague} we have that
$E(\intop_{U}f(x)\mathrm{d}\mu(x))\leq c\intop_{U}f(x)\mathrm{d}\nu(x)<\infty$.
Hence $f(x)\mathrm{d}\mu_{k}(x)$ and $f(x)\mathrm{d}\mu(x)$ are
random, locally finite, Borel measures.

Let $\Psi\subseteq C_{c}(X)$ as in Lemma \ref{lem:deterministic con-Vague}.
By Theorem \ref{thm:ABS cont representation-Vague} it follows that
$\intop_{X}g(x)f(x)\mathrm{d}\mu_{k}(x)$ converges to $\intop_{X}g(x)f(x)\mathrm{d}\mu(x)$
in probability for every $g\in\Psi$.

By the application of Theorem \ref{thm:ABS cont representation-Vague}
to the reference measure $f(x)\mathrm{d}\nu(x)$ and the sequence
$f(x)\mathrm{d}\mu_{k}(x)$ it follows that the sequence of random,
locally finite, Borel measures $f(x)\mathrm{d}\mu_{k}(x)$ converges
to a random, locally finite, Borel measure $\tau$ and $\intop_{X}g(x)f(x)\mathrm{d}\mu_{k}(x)$
converges to $\intop_{X}g(x)\mathrm{d}\tau(x)$ in probability for
every $g\in\Psi$.

Thus $\intop_{X}g(x)f(x)\mathrm{d}\mu(x)=\intop_{X}g(x)\mathrm{d}\tau(x)$
for every $g\in\Psi$ almost surely. Hence $f(x)\mathrm{d}\mu(x)=\mathrm{d}\tau(x)$
almost surely by Lemma \ref{lem:countab agree}.

Applying Theorem \ref{thm:ABS cont representation-Vague} to the function
$\chi_{A}\cdot f$ for a compact set $A\subseteq X$, it follows that
$\int_{A}f(x)\mathrm{d}\mu_{k}(x)$ converges to $\int_{A}f(x)\mathrm{d}\mu(x)$
in probability.
\end{proof}

\subsection{General existence of the conditional measure\label{subsec:SUB-General-existance-of}}

In this section our main goal is to prove Theorem \ref{thm:Gen exist both}
and Proposition \ref{prop:two def are the same}. The proof of Proposition
\ref{prop:two def are the same} can be found at the end of the section
and Theorem \ref{thm:Gen exist both} follows from Theorem \ref{thm:General existence of conditional measure},
Theorem \ref{thm:General existence of conditional measure-Vague}
and Lemma \ref{lem:compact vanish lemma}.
\begin{lem}
\label{prop:L^1 limit for sum}Let $A\subseteq X$ be a Borel set
$\nu=\sum_{i=1}^{\infty}\nu^{i}$ be a finite Borel measure and $\mu_{k}^{i}$
be random finite Borel measures for every $i\in\mathbb{N}$ such that
$\mu_{k}^{i}(A)$ converges in $\mathcal{L}^{1}$ to a random variable
$\mu^{i}(A)$ for every $i\in\mathbb{N}$ and $E(\mu_{k}^{i}(A))=\nu^{i}(A)$.
Then $\sum_{i=1}^{\infty}\mu_{k}^{i}(A)$ converges in $\mathcal{L}^{1}$
to $\sum_{i=1}^{\infty}\mu^{i}(A)$ as $k$ goes to $\infty$.
\end{lem}

\begin{proof}
Since $\mu_{k}^{i}(A)$ converges to $\mu^{i}(A)$ in $\mathcal{L}^{1}$
it follows that $E(\mu^{i}(A))=E(\mu_{k}^{i}(A))=\nu^{i}(A)$. Thus
by Fubini`s theorem that
\[
E\left(\sum_{i=n}^{\infty}\mu_{k}^{i}(A)\right)=\sum_{i=n}^{\infty}\nu^{i}(A)=E\left(\sum_{i=n}^{\infty}\mu^{i}(A)\right)
\]
for every $n\in\mathbb{N}$ and $\sum_{i=1}^{\infty}\nu^{i}(A)=\nu(A)<\infty$.
Let $\eta>0$ and let $n\in\mathbb{N}$ be large enough that $\sum_{i=n}^{\infty}\nu^{i}(A)<\eta$.
By assumption $\sum_{i=1}^{n}\mu_{k}^{i}(A)$ converges in $\mathcal{L}^{1}$
to $\sum_{i=1}^{n}\mu^{i}(A)$, hence converges in probability. Thus
\[
\limsup_{k\rightarrow\infty}\rho\left(\sum_{i=1}^{\infty}\mu_{k}^{i}(A),\sum_{i=1}^{\infty}\mu^{i}(A)\right)\leq
\]
\[
\limsup_{k\rightarrow\infty}\rho\left(\sum_{i=1}^{\infty}\mu_{k}^{i}(A),\sum_{i=1}^{n}\mu_{k}^{i}(A)\right)+\limsup_{k\rightarrow\infty}\rho\left(\sum_{i=1}^{n}\mu_{k}^{i}(A),\sum_{i=1}^{n}\mu^{i}(A)\right)+\rho\left(\sum_{i=1}^{n}\mu^{i}(A),\sum_{i=1}^{\infty}\mu^{i}(A)\right)\leq
\]
\[
\limsup_{k\rightarrow\infty}E\left(\sum_{i=n}^{\infty}\mu_{k}^{i}(A)\right)+0+E\left(\sum_{i=n}^{\infty}\mu^{i}(A)\right)=2\sum_{i=n}^{\infty}\nu^{i}(A)\leq2\eta.
\]
Hence by taking limit $\eta$ goes to $0$ it follows that $\sum_{i=1}^{\infty}\mu_{k}^{i}(A)$
converges in probability to $\sum_{i=1}^{\infty}\mu^{i}(A)$. Thus
the statement follows from Lemma \ref{lem:Scheffe}.
\end{proof}
\begin{lem}
\label{thm:E(measure)}Let $\nu$ be a finite Borel measure on $X$
such that $\nu(X\setminus X_{0})=0$ and let $A\subseteq X$ be a
Borel set. Assume that $\mathcal{C}_{k}(\nu\vert_{D})(X)=\mathcal{C}_{k}(\nu)(D)$
converges in $\mathcal{L}^{1}$ for every compact set $D\subseteq X_{0}$
with $I_{\varphi}(\nu\vert_{D})<\infty$. Assume that $\mathcal{C}_{k}(\nu_{\perp})(X)$
converges to $0$ in probability. Then $\mathcal{C}_{k}(\nu_{R})(A)$
converges in $\mathcal{L}^{1}$ to a limit $\mu_{\infty}(A)$, and
both $\mathcal{C}_{k}(\nu)(A)$ and $\mathcal{C}_{k}(\nu_{R})(A)$
 converges in probability to the same limit $\mu_{\infty}(A)$ and
$E(\mu_{\infty}(A))=\nu_{R}(A)\leq\nu(A)$.
\end{lem}

\begin{proof}
Let us take a sequence $(A_{n})_{n\in\mathbb{N}}$ as in Proposition
\ref{decomposition} for the measure $\nu\vert_{A}$ in place of $\nu$.
Note that we can assume that $A_{n}\subseteq A\cap X_{0}$ for every
$n\in\mathbb{N}$ because $\nu\vert_{A}(X\setminus X_{0})=0$. By
Lemma \ref{lem:countab exhaustion} we can further assume that all
the $A_{n}$ are compact. Then $\nu\vert_{A}=\nu_{\bot}\vert_{A}+\sum_{n\in\mathbb{N}}\nu\vert_{A_{n}}$
and $I_{\varphi}(\nu\vert_{A_{n}})<\infty$. By the assumption of
the statement it follows that $\mathcal{C}_{k}(\nu\vert_{A_{n}})(X)=\mathcal{C}_{k}(\nu)(A_{n})$
converges in $\mathcal{L}^{1}$ for every $n\in\mathbb{N}$ as $k$
goes to $\infty$. By (\ref{eq:martingal expectation inequality})
it follows that $E(\mathcal{C}_{k}(\nu)(A_{n}))=\nu(A_{n})$. Thus
$\mathcal{C}_{k}(\nu_{R})(A)=\sum_{n\in\mathbb{N}}\mathcal{C}_{k}(\nu\vert_{A_{n}})(X)$
converges in $\mathcal{L}^{1}$ to a limit $\mu_{\infty}(A)$ by Lemma
\ref{prop:L^1 limit for sum}. Thus $\mathcal{C}_{k}(\nu_{R})(A)$
converges in probability to $\mu_{\infty}(A)$ and $E(\mu_{\infty}(A))=\nu_{R}(A)$.
Since $\mathcal{C}_{k}(\nu_{\perp})(A)\leq\mathcal{C}_{k}(\nu_{\perp})(X)$
we have that $\mathcal{C}_{k}(\nu_{\perp})(A)$ converges to $0$
in probability. Thus $\mathcal{C}_{k}(\nu)(A)$ converges in probability
to $\mu_{\infty}(A)$.
\end{proof}
\begin{prop}
\label{prop:equiv of egsistance}Let $\nu$ be a finite Borel measure
on $X$ such that $\nu(X\setminus X_{0})=0$. Assume that $\mathcal{C}_{k}(\nu\vert_{D})(X)=\mathcal{C}_{k}(\nu)(D)$
converges in $\mathcal{L}^{1}$ for every compact set $D\subseteq X_{0}$
with $I_{\varphi}(\nu\vert_{D})<\infty$. Assume that $\mathcal{C}_{k}(\nu_{\perp})(X)$
converges to $0$ in probability. Then the conditional measure $\mathcal{C}(\nu)$
of $\nu$ on $B$ exists with respect to $\mathcal{Q}_{k}$ ($k\geq1$)
with regularity kernel $\varphi$.
\end{prop}

\begin{proof}
From the definition of $\mathcal{C}_{k}(\nu)$ it follows that that
$E\left(\frac{\mathrm{d}\mathcal{C}_{k}(\nu)}{\mathrm{d}\nu}(x)\right)=1$
for every $x\in X_{0}$, i.e. for $\nu$ almost every $x\in X$. The
assumptions of Theorem \ref{thm:ABS cont representation} are satisfied
by Lemma \ref{thm:E(measure)} for $\mu_{k}=\mathcal{C}_{k}(\nu)$
and for the reference measure $\nu$. Hence Property \textit{i.)},
\textit{iii.), iv.)} of Definition \ref{def:def of cond meas} hold
for $\tau=\nu$ by Theorem \ref{thm:ABS cont representation} and
Remark \ref{rem:convergenc for sets to the measure}. Property \textit{v.),
vi.)} of Definition \ref{def:def of cond meas} hold for $\tau=\nu$
by Lemma \ref{thm:E(measure)}, and Remark \ref{rem:convergenc for sets to the measure}.
Property \textit{vii.), viii.)} of Definition \ref{def:def of cond meas}
hold for $\tau=\nu$ by the fact that $\mathcal{C}_{k}(\nu_{\perp})(X)$
converges to $0$ in probability. Property \textit{x.)} of Definition
\ref{def:def of cond meas} holds for $\tau=\nu$ by Proposition \ref{prop:change of meas prop}.
Property \textit{xi.)} of Definition \ref{def:def of cond meas} holds
for $\tau=\nu$ by Proposition \ref{prop:randomsupport}.

Now we show that Property \textit{ii.)} of Definition \ref{def:def of cond meas}
holds for $\tau=\nu$. Let $f:X\longrightarrow\mathbb{R}$ be a Borel
measurable function such that $\intop_{X}\left|f(x)\right|\mathrm{d}\nu(x)<\infty$.
It follows from Theorem \ref{thm:ABS cont representation} that $\intop_{X}f(x)\mathrm{d}\mathcal{C}_{k}(\nu)(x)$
converges to a random variable $S(f)$ in probability with $E(\left|S(f)\right|)\leq\intop_{X}\left|f(x)\right|\mathrm{d}\nu(x)$.
Since \textit{v.)} holds for $\tau=\nu$ it follows that $\nu_{R}(A)=E(\mathcal{C}(\nu)(A))$
for every Borel set $A\subseteq X$ and so $E(\intop_{X}f(x)\mathrm{d}\mathcal{C}(\nu)(x))=\intop_{X}f(x)\mathrm{d}\nu_{R}(x)$
holds when $f$ is a simple function and so by the monotone convergence
theorem it holds for Borel measurable $f$. On the other hand $\intop_{X}f(x)\mathrm{d}\mathcal{C}(\nu)(x)=S(f)$
almost surely by \textit{iii.)}. Hence \textit{ii.)} holds for $\tau=\nu$.
In particular, applied to $\chi_{A}\cdot f$ for a nonnegative Borel
function $f$ with $\intop_{X}f(x)\mathrm{d}\nu(x)<\infty$ and a
Borel set $A\subseteq X$ we get that
\begin{equation}
E(\intop_{A}f(x)\mathrm{d}\mathcal{C}(\nu)(x))=\intop_{A}f(x)\mathrm{d}\nu_{R}(x),\label{eq:set l1 lim for f}
\end{equation}
the sequence $\intop_{A}f(x)\mathrm{d}\mathcal{C}_{k}(\nu)(x)$ converges
to $S(\chi_{A}f)$ in probability and $\intop_{A}f(x)\mathrm{d}\mathcal{C}(\nu)(x)=S(\chi_{A}f)$
almost surely. Moreover, $\intop_{A}f(x)\mathrm{d}\mathcal{C}(\nu)(x)=\intop_{A}f(x)\mathrm{d}\mathcal{C}(\nu_{R})(x)$
almost surely since Property \textit{viii.)} of Definition \ref{def:def of cond meas}
holds for $\tau=\nu$. There exists a Borel set $A_{0}\subseteq X$
such that $\nu_{R}(A_{0})=0$ and $\nu_{\perp}(X\setminus A_{0})=0$,
then by Property \textit{ii.) }of Definition \ref{def:def of cond meas}
for $\tau=\nu$ 
\begin{equation}
\intop_{A}f(x)\mathrm{d}\mathcal{C}_{k}(\nu_{\perp})(x)=\intop_{X}\chi_{A}\chi_{A_{0}}f(x)\mathrm{d}\mathcal{C}_{k}(\nu)(x)\label{eq:0ba tar az abs}
\end{equation}
converges to $S(\chi_{A}\chi_{A_{0}}f)$ in probability and $E(S(\intop_{X}\chi_{A}\chi_{A_{0}}f))=\int_{A\cap A_{0}}f(x)\mathrm{d}\nu_{R}(x)=0$,
so $S(\chi_{A}\chi_{A_{0}}f)=0$ almost surely since it is nonnegative.
Hence $\intop_{A}f(x)\mathrm{d}\mathcal{C}_{k}(\nu_{R})(x)$ converges
to $\intop_{A}f(x)\mathrm{d}\mathcal{C}(\nu)(x)=S(\chi_{A}f)=\intop_{A}f(x)\mathrm{d}\mathcal{C}(\nu_{R})(x)$
in $\mathcal{L}^{1}$ by Lemma \ref{lem:Scheffe} due to (\ref{eq:set l1 lim for f}).
In particular, if $\nu_{\perp}(A)=0$ for a Borel set $A\subseteq X$
then
\begin{equation}
\intop_{A}f(x)\mathrm{d}\mathcal{C}_{k}(\nu)(x)\label{eq:abd l1 conv set}
\end{equation}
converges to $\intop_{A}f(x)\mathrm{d}\mathcal{C}(\nu)(x)=S(\chi_{A}f)=\intop_{A}f(x)\mathrm{d}\mathcal{C}(\nu_{R})(x)$
in $\mathcal{L}^{1}$.

Let $\tau\ll\nu$ be a deterministic finite Borel measure on $X$.
If $I_{\varphi}(\tau\vert_{D})<\infty$ for a Borel set $D\subseteq X$
then $\tau\vert_{D}=\tau_{R}\vert_{D}$ and, by Proposition \ref{prop:abs cont decomposition},
there exists $A\subseteq D$ such that $\tau(D\setminus A)=0$ and
$\nu_{\perp}(A)=0$. Thus, by the end of the last paragraph (\ref{eq:abd l1 conv set}),
we have that $\mathcal{C}_{k}(\tau)(D)=\intop_{D}\frac{\mathrm{d}\tau}{\mathrm{d}\nu}(x)\mathrm{d}\mathcal{C}_{k}(\nu)(x)=\intop_{A}\frac{\mathrm{d}\tau}{\mathrm{d}\nu}(x)\mathrm{d}\mathcal{C}_{k}(\nu)(x)$
converges to $\intop_{A}f(x)\mathrm{d}\mathcal{C}(\nu)(x)$ in $\mathcal{L}^{1}$.
By (\ref{eq:0ba tar az abs}) and Proposition \ref{prop:abs cont decomposition}
it follows that $\mathcal{C}_{k}(\tau_{\perp})(X)=\intop_{X}\frac{\mathrm{d}\tau}{\mathrm{d}\nu}(x)\mathrm{d}\mathcal{C}_{k}(\nu_{\perp})(x)$
converges to $0$ in probability. Hence the assumptions of the statement
not only hold for $\nu$ but the same assumptions also hold for $\tau$.
Thus similarly to the way we proved that Property \textit{i.)-viii.),
x.)-xi.)} of Definition \ref{def:def of cond meas} hold for $\tau=\nu$,
we can also prove the same properties for arbitrary $\tau\ll\nu$.
Hence it is left to prove that \textit{ix.)} holds.

Property \textit{ix.)} of Definition \ref{def:def of cond meas} holds
by Proposition \ref{prop:SUM lem subseq} and the fact that the limit
is unique by Proposition \ref{prop:unique random lim}.
\end{proof}
\begin{thm}
\label{thm:General existence of conditional measure}Let $\nu$ be
a finite Borel measure on $X$ such that $\nu(X\setminus X_{0})=0$.
The conditional measure $\mathcal{C}(\nu)$ of $\nu$ on $B$ exists
with respect to $\mathcal{Q}_{k}$ ($k\geq1$) with regularity kernel
$\varphi$ if and only if $\mathcal{C}_{k}(\nu\vert_{D})(X)=\mathcal{C}_{k}(\nu)(D)$
converges in $\mathcal{L}^{1}$ for every compact set $D\subseteq X_{0}$
with $I_{\varphi}(\nu\vert_{D})<\infty$ and $\mathcal{C}_{k}(\nu_{\perp})(X)$
converges to $0$ in probability.
\end{thm}

\begin{proof}
Assume that the conditional measure $\mathcal{C}(\nu)$ of $\nu$
on $B$ exists with respect to $\mathcal{Q}_{k}$ ($k\geq1$) with
regularity kernel $\varphi$. Then by\textit{ }Property \textit{vii.)}
of Definition \ref{def:def of cond meas} it follows that $\mathcal{C}_{k}(\nu_{\perp})(X)$
converges to $0$ in probability. Let $D\subseteq X_{0}$ be a compact
set such that $I_{\varphi}(\nu\vert_{D})<\infty$. Then $\nu_{R}\vert_{D}=\nu\vert_{D}$
and so $\mathcal{C}_{k}(\nu)(D)=\mathcal{C}_{k}(\nu_{R})(D)$. Thus
by Property \textit{vi.)} of Definition \ref{def:def of cond meas}
it follows that $\mathcal{C}_{k}(\nu)(D)$ converges to $\mathcal{C}(\nu)(D)$
in $\mathcal{L}^{1}$.

The other direction of the equivalence follows from Proposition \ref{prop:equiv of egsistance}.
\end{proof}
\begin{lem}
\label{lem:compact vanish lemma}Let $\nu$ be a finite Borel measure
on $X$ such that $\nu(X\setminus X_{0})=0$. Then $\mathcal{C}_{k}(\nu_{\perp})(X)$
converges to $0$ in probability if and only if $\mathcal{C}_{k}(\nu_{\perp})(D)$
converges to $0$ in probability for every compact set $D\subseteq X_{0}$.
\end{lem}

\begin{proof}
If $\mathcal{C}_{k}(\nu_{\perp})(X)$ converges to $0$ in probability
then of course $\mathcal{C}_{k}(\nu_{\perp})(D)$ converges to $0$
in probability for every compact set $D\subseteq X$.

Assume that $\mathcal{C}_{k}(\nu_{\perp})(D)$ converges to $0$ in
probability for every compact set $D\subseteq X_{0}$. Let $\varepsilon>0$
be fixed. Let $D\subseteq X_{0}$ be a compact set such that $\nu(X\setminus D)<\varepsilon$,
which we can choose by inner regularity. Then by the fact that $\rho(Y,Z)\leq E\left(\left|Y-Z\right|\right)$,
by (\ref{eq:martingal expectation inequality}) and Lemma \ref{lem:abs cont weak prob conv fatou for open sets}
it follows that
\[
\rho\left(\mathcal{C}_{k}(\nu_{\perp})(X),0\right)\leq\rho\left(\mathcal{C}_{k}(\nu_{\perp})(X),\mathcal{C}_{k}(\nu_{\perp})(D)\right)+\rho\left(\mathcal{C}_{k}(\nu_{\perp})(D),0\right)\leq
\]
\[
E\left(\mathcal{C}_{k}(\nu_{\perp})(X)-\mathcal{C}_{k}(\nu_{\perp})(D)\right)+\rho\left(\mathcal{C}_{k}(\nu_{\perp})(D),0\right)\leq\nu(X\setminus D)+\rho\left(\mathcal{C}_{k}(\nu_{\perp})(D),0\right).
\]
Hence
\[
\limsup_{k\rightarrow\infty}\rho\left(\mathcal{C}_{k}(\nu_{\perp})(X),0\right)\leq\nu(X\setminus D)\leq\varepsilon
\]
because $\mathcal{C}_{k}(\nu_{\perp})(D)$ converges to $0$ in probability.
It holds for every $\varepsilon>0$ thus $\mathcal{C}_{k}(\nu_{\perp})(X)$
converges to $0$ in probability.
\end{proof}
\begin{rem}
\label{rem:0vagynem nulla}Let $X$ be locally compact and let $\nu$
be a locally finite Borel measure on $X$ such that $\nu(X\setminus X_{0})=0$.
Then $\mathcal{C}_{k}(\nu_{\perp})(D)$ converges to $0$ in probability
for every compact set $D\subseteq X_{0}$ if and only if $\mathcal{C}_{k}(\nu_{\perp})(D)$
converges to $0$ in probability for every compact set $D\subseteq X$.
It can easily be deduced from Lemma \ref{lem:compact vanish lemma}.
\end{rem}

\begin{prop}
\label{prop:equiv of egsistance-Vague}Let $X$ be locally compact
and let $\nu$ be a locally finite Borel measure on $X$ such that
$\nu(X\setminus X_{0})=0$. Assume that $\mathcal{C}_{k}(\nu\vert_{D})(X)=\mathcal{C}_{k}(\nu)(D)$
converges in $\mathcal{L}^{1}$ for every compact set $D\subseteq X_{0}$
with $I_{\varphi}(\nu\vert_{D})<\infty$. Assume that $\mathcal{C}_{k}(\nu_{\perp})(D)$
converges to $0$ in probability for every compact set $D\subseteq X$.
Then the conditional measure $\mathcal{C}(\nu)$ of $\nu$ on $B$
exists with respect to $\mathcal{Q}_{k}$ ($k\geq1$) with regularity
kernel $\varphi$.
\end{prop}

\begin{proof}
From the definition of $\mathcal{C}_{k}(\nu)$ it follows that that
$E\left(\frac{\mathrm{d}\mathcal{C}_{k}(\nu)}{\mathrm{d}\nu}(x)\right)=1$
for every $x\in X_{0}$, i.e. for $\nu$ almost every $x\in X$. The
assumptions of Theorem \ref{thm:ABS cont representation-Vague} are
satisfied for $\mu_{k}=\mathcal{C}_{k}(\nu)$ and for the reference
measure $\nu$ by Lemma  \ref{thm:E(measure)} because $\nu\vert_{A}$
is a finite Borel measure for every Borel set $A\subseteq X$ with
$\nu(A)<\infty$, in particular if $A$ is compact. Hence Property
\textit{i{*}.)}, \textit{iii.), iv.)} of Definition \ref{def:def of cond meas}
hold for $\tau=\nu$ by Theorem \ref{thm:ABS cont representation-Vague}
and Remark \ref{rem:convergenc for sets to the measure-Vague}. Property
\textit{v.), vi.)} of Definition \ref{def:def of cond meas} hold
for $\tau=\nu$ by Lemma \ref{thm:E(measure)}, and Remark \ref{rem:convergenc for sets to the measure-Vague}.
Property \textit{vii.), viii.)} of Definition \ref{def:def of cond meas}
hold for $\tau=\nu$ by the fact that
\[
0\leq\left|\int f(x)\mathrm{d}\mathcal{C}_{k}(\nu_{\perp})(x)\right|\leq\left\Vert f\right\Vert _{\infty}\mathcal{C}_{k}(\nu_{\perp})(\mathrm{supp}(f))
\]
converges to $0$ in probability for every $f\in C_{c}(X)$. Property
\textit{x{*}.)} of Definition \ref{def:def of cond meas} holds for
$\tau=\nu$ by Proposition \ref{prop:change of meaasure-Vague}. Property
\textit{xi.)} of Definition \ref{def:def of cond meas} holds for
$\tau=\nu$ by Proposition \ref{prop:randomsupport-Vague}.

Now we show that Property \textit{ii.)} of Definition \ref{def:def of cond meas}
hold for $\tau=\nu$. Let $f:X\longrightarrow\mathbb{R}$ be a Borel
measurable function such that $\intop_{X}\left|f(x)\right|\mathrm{d}\nu(x)<\infty$.
It follows from Theorem \ref{thm:ABS cont representation-Vague} that
$\intop_{X}f(x)\mathrm{d}\mathcal{C}_{k}(\nu)(x)$ converges to a
random variable $S(f)$ in probability with $E(\left|S(f)\right|)\leq\intop_{X}\left|f(x)\right|\mathrm{d}\nu(x)$.
Since \textit{v.)} holds for $\tau=\nu$ it follows that $\nu_{R}(A)=E(\mathcal{C}(\nu)(A))$
for every Borel set $A\subseteq X$ with $\nu(A)<\infty$ and so $E(\intop_{X}f(x)\mathrm{d}\mathcal{C}(\nu)(x))=\intop_{X}f(x)\mathrm{d}\nu_{R}(x)$
holds when $f$ is a simple function and so by the monotone convergence
theorem it holds for Borel measurable $f$. On the other hand $\intop_{X}f(x)\mathrm{d}\mathcal{C}(\nu)(x)=S(f)$
almost surely by \textit{iii.)}. Hence \textit{ii.)} holds for $\tau=\nu$.
In particular, applied to $\chi_{A}\cdot f$ for a nonnegative, Borel
measurable function $f$ and a Borel set $A\subseteq X$ with compact
closure such that $\int_{A}f(x)\mathrm{d}\nu(x)<\infty$ we get that
\begin{equation}
E(\intop_{A}f(x)\mathrm{d}\mathcal{C}(\nu)(x))=\intop_{A}f(x)\mathrm{d}\nu_{R}(x),\label{eq:set l1 lim for f-1}
\end{equation}
the sequence $\intop_{A}f(x)\mathrm{d}\mathcal{C}_{k}(\nu)(x)$ converges
to $S(\chi_{A}f)$ in probability and $\intop_{A}f(x)\mathrm{d}\mathcal{C}(\nu)(x)=S(\chi_{A}f)$
almost surely. Moreover, $\intop_{A}f(x)\mathrm{d}\mathcal{C}(\nu)(x)=\intop_{A}f(x)\mathrm{d}\mathcal{C}(\nu_{R})(x)$
almost surely since Property \textit{viii.)} of Definition \ref{def:def of cond meas}
holds for $\tau=\nu$. There exists a Borel set $A_{0}\subseteq X$
such that $\nu_{R}(A_{0})=0$ and $\nu_{\perp}(X\setminus A_{0})=0$,
then by Property \textit{ii.) }of Definition \ref{def:def of cond meas}
for $\tau=\nu$ 
\begin{equation}
\intop_{A}f(x)\mathrm{d}\mathcal{C}_{k}(\nu_{\perp})(x)=\intop_{X}\chi_{A}\chi_{A_{0}}f(x)\mathrm{d}\mathcal{C}_{k}(\nu)(x)\label{eq:0ba tar az abs-1}
\end{equation}
converges to $S(\chi_{A}\chi_{A_{0}}f)$ in probability and $E(S(\intop_{X}\chi_{A}\chi_{A_{0}}f))=\int_{A\cap A_{0}}f(x)\mathrm{d}\nu_{R}(x)=0$,
so $S(\chi_{A}\chi_{A_{0}}f)=0$ almost surely since it is nonnegative.
Hence $\intop_{A}f(x)\mathrm{d}\mathcal{C}_{k}(\nu_{R})(x)$ converges
to $\intop_{A}f(x)\mathrm{d}\mathcal{C}(\nu)(x)=S(\chi_{A}f)=\intop_{A}f(x)\mathrm{d}\mathcal{C}(\nu_{R})(x)$
in $\mathcal{L}^{1}$ by Lemma \ref{lem:Scheffe} due to (\ref{eq:set l1 lim for f-1}).
In particular, if $\nu_{\perp}(A)=0$ for a Borel set $A\subseteq X$
with compact closure then
\begin{equation}
\intop_{A}f(x)\mathrm{d}\mathcal{C}_{k}(\nu)(x)\label{eq:abd l1 conv set-1}
\end{equation}
converges to $\intop_{A}f(x)\mathrm{d}\mathcal{C}(\nu)(x)=S(\chi_{A}f)=\intop_{A}f(x)\mathrm{d}\mathcal{C}(\nu_{R})(x)$
in $\mathcal{L}^{1}$.

Let $\tau\ll\nu$ be a deterministic, locally finite, Borel measure
on $X$. If $I_{\varphi}(\tau\vert_{D})<\infty$ for a compact set
$D\subseteq X$ then $\tau\vert_{D}=\tau_{R}\vert_{D}$ and, by Proposition
\ref{prop:abs cont decomposition}, there exists $A\subseteq D$ such
that $\tau(D\setminus A)=0$ and $\nu_{\perp}(A)=0$. Thus, by the
end of the last paragraph (\ref{eq:abd l1 conv set-1}), we have that
$\mathcal{C}_{k}(\tau)(D)=\intop_{D}\frac{\mathrm{d}\tau}{\mathrm{d}\nu}(x)\mathrm{d}\mathcal{C}_{k}(\nu)(x)=\intop_{A}\frac{\mathrm{d}\tau}{\mathrm{d}\nu}(x)\mathrm{d}\mathcal{C}_{k}(\nu)(x)$
converges to $\intop_{A}f(x)\mathrm{d}\mathcal{C}(\nu)(x)$ in $\mathcal{L}^{1}$.
By (\ref{eq:0ba tar az abs-1}) and Proposition \ref{prop:abs cont decomposition}
it follows that $\mathcal{C}_{k}(\tau_{\perp})(D)=\intop_{D}\frac{\mathrm{d}\tau}{\mathrm{d}\nu}(x)\mathrm{d}\mathcal{C}_{k}(\nu_{\perp})(x)$
converges to $0$ in probability for every compact set $D\subseteq X$.
Hence the assumptions of the statement not only hold for $\nu$ but
the same assumptions also hold for $\tau$. Thus similarly to the
way we proved that Property \textit{i.)-viii.), x.)-xi.)} of Definition
\ref{def:def of cond meas} hold for $\tau=\nu$, we can also prove
the same properties for arbitrary $\tau\ll\nu$. Hence it is left
to prove that \textit{ix.)} holds.

Property \textit{ix.)} of Definition \ref{def:def of cond meas} holds
by Proposition \ref{prop:SUM cov Vague} and the fact that the limit
is unique by Proposition \ref{prop:unique random lim-Vague}.
\end{proof}
\begin{thm}
\label{thm:General existence of conditional measure-Vague}Let $X$
be locally compact. Let $\nu$ be a locally finite Borel measure on
$X$ such that $\nu(X\setminus X_{0})=0$. Then the conditional measure
$\mathcal{C}(\nu)$ of $\nu$ on $B$ exists with respect to $\mathcal{Q}_{k}$
($k\geq1$) with regularity kernel $\varphi$ if and only if $\mathcal{C}_{k}(\nu\vert_{D})(X)=\mathcal{C}_{k}(\nu)(D)$
converges in $\mathcal{L}^{1}$ for every compact set $D\subseteq X_{0}$
with $I_{\varphi}(\nu\vert_{D})<\infty$ and $\mathcal{C}_{k}(\nu_{\perp})(D)$
converges to $0$ in probability for every compact set $D\subseteq X_{0}$.
\end{thm}

\begin{proof}
Assume that the conditional measure $\mathcal{C}(\nu)$ of $\nu$
on $B$ exists with respect to $\mathcal{Q}_{k}$ ($k\geq1$) with
regularity kernel $\varphi$. Then by\textit{ }Property \textit{ii.),
iv.)} and \textit{vii.)} of Definition \ref{def:def of cond meas}
it follows that $\mathcal{C}_{k}(\nu_{\perp})(D)$ converges to $0$
in probability for every compact set $D\subseteq X_{0}$. Let $D\subseteq X_{0}$
be a compact subset such that $I_{\varphi}(\nu\vert_{D})<\infty$.
Then $\nu_{R}\vert_{D}=\nu\vert_{D}$ and so $\mathcal{C}_{k}(\nu)(D)=\mathcal{C}_{k}(\nu_{R})(D)$.
Thus by Property \textit{vi.)} of Definition \ref{def:def of cond meas}
it follows that $\mathcal{C}_{k}(\nu)(D)$ converges to $\mathcal{C}(\nu)(D)$
in $\mathcal{L}^{1}$.

The other direction of the equivalence follows from Proposition \ref{prop:equiv of egsistance-Vague}
and Remark \ref{rem:0vagynem nulla}.
\end{proof}
\selectlanguage{english}%
\noindent \textit{Proof of Proposition }\foreignlanguage{british}{\ref{prop:two def are the same}.
By Theorem \ref{thm:General existence of conditional measure}, Theorem
\ref{thm:General existence of conditional measure-Vague} and Lemma
\ref{lem:compact vanish lemma} according to both of the two definition
of the conditional measure in Definition \ref{def:def of cond meas}
the conditional measure exists if and only if $\mathcal{C}_{k}(\nu\vert_{D})(X)=\mathcal{C}_{k}(\nu)(D)$
converges in $\mathcal{L}^{1}$ for every compact set $D\subseteq X_{0}$
with $I_{\varphi}(\nu\vert_{D})<\infty$ and $\mathcal{C}_{k}(\nu_{\perp})(D)$
converges to $0$ in probability for every compact set $D\subseteq X$.
The limit in Property\textit{ i.)} and in Property \textit{i{*}.)}
are the same almost surely because if $\mathcal{C}_{k}(\nu)$ weakly
converges to $\mathcal{C}(\nu)$ subsequentially in probability then
$\mathcal{C}_{k}(\nu)$ vaguely converges to $\mathcal{C}(\nu)$ in
probability by Remark \ref{rem:weak imply vague} and Proposition
\ref{prop:trivi direct} and since the limit is unique by Proposition
\ref{prop:unique random lim-Vague} the two limits are the same almost
surely.$\hfill\square$}
\selectlanguage{british}%

\section{Decomposition of measure\label{sec:Decomposition-of-measure}}

We prove Proposition \ref{decomposition}, Proposition \ref{prop:abs cont decomposition}
and Proposition \ref{prop:singular felbontas loc fin intro} in this
section.
\begin{prop}
\label{prop:decomp peop}Let $\nu$ be a finite, Borel measure on
$X$. There exist two finite, Borel measures $\nu_{\varphi R}=\nu_{R}$
and $\nu_{\varphi\perp}=\nu_{\perp}$ with the following properties:

i) $\nu=\nu_{R}+\nu_{\perp}$

ii) $\nu_{R}\perp\nu_{\perp}$

iii) $\nu_{\perp}$ is singular to every finite Borel measure with
finite $\varphi$-energy

iv) there exists a sequence of disjoint Borel sets $(A_{n})_{n\in\mathbb{N}}$
such that $\nu_{R}=\nu\vert_{\cup_{n\in\mathbb{N}}A_{n}}=\sum_{n\in\mathbb{N}}\nu\vert_{A_{n}}$
and $I_{\varphi}(\nu\vert_{A_{n}})=I_{\varphi}(\nu_{R}\vert_{A_{n}})<\infty$.
\end{prop}

\begin{proof}
Let
\[
c_{\max}=\sup\left\{ \nu(B):B\subseteq X\,\mathrm{is\,a\,Borel\,set\,},B\subseteq\cup B_{n},\,B_{n}\subseteq X\,\mathrm{is\,a\,Borel\,set\,with\,}I_{\alpha}(\nu\vert_{B_{n}})<\infty\,\mathrm{for\,all\,n}\right\} .
\]
We can find $A$ and $(A_{n})_{n\in\mathbb{N}}$ such that $A=\cup A_{n}$,
$\nu(A)=c_{\max}$ and $I_{\varphi}(\nu\vert_{A_{n}})<\infty$ for
all $n\in\mathbb{N}$. Without the loss of generality we can further
assume that all the $A_{n}$ are disjoint for different $n$. Let
$\nu_{R}=\nu\vert_{A}$ and $\nu_{\bot}=\nu\vert_{X\setminus A}$.
Then \textit{i)}, \textit{ii)} and \textit{iv)} are satisfied. Assume
for a contradiction that \textit{iii)} is not satisfied that is there
exists a Borel probability measure $\tau$ with $I_{\varphi}(\tau)<\infty$
such that $\tau\ll\nu_{\bot}$. Then there exists $N>0$ such that
$\nu_{\bot}(C_{N})>0$ where $C_{N}=\left\{ x:\frac{\mathrm{d}\tau}{\mathrm{d}\nu_{\bot}}(x)\geq\frac{1}{N}\right\} $.
If $D\subseteq C_{N}$ then
\[
\tau(D)=\int_{D}\frac{\mathrm{d}\tau}{\mathrm{d}\nu_{\bot}}(x)\mathrm{d}\nu_{\bot}(x)\geq\int_{D}\frac{1}{N}\mathrm{d}\nu_{\bot}(x)=\frac{1}{N}\nu_{\bot}(D).
\]
Thus $I_{\varphi}(\nu_{\bot}\vert_{C_{N}})\leq N^{2}I_{\varphi}(\tau\vert_{C_{N}})\leq N^{2}I_{\varphi}(\tau)<\infty$.
This contradicts with the maximality of $c_{\max}$.
\end{proof}
\begin{prop}
\label{prop:abs cont parts are inheriting}Let $\nu$ be a finite
Borel measure on $X$ and let $\tau\ll\nu$ be a finite Borel measure
on $X$. Then $\frac{\mathrm{d}\tau}{\mathrm{d}\nu}(x)\mathrm{d}\nu_{\varphi R}=\tau_{\varphi R}$
and $\frac{\mathrm{d}\tau}{\mathrm{d}\nu}(x)\mathrm{d}\nu_{\varphi\perp}=\tau_{\varphi\perp}$.
\end{prop}

\begin{proof}
It is easy to see that the statement holds for $\tau\vert_{A_{n}}$
in place of $\tau$ where $A_{n}=\left\{ x\in X:1/n<\frac{\mathrm{d}\tau}{\mathrm{d}\nu}(x)<n\right\} $
and so holds for $\tau\vert_{\cup_{n\in\mathbb{N}}A_{n}}$. Since
$\tau$ is a finite Borel measure $\tau(X\setminus\cup_{n\in\mathbb{N}}A_{n})=0$
and so the statement holds.
\end{proof}
\begin{rem}
\label{rem:locallyn finite decomp}Proposition \ref{prop:decomp peop}
and Proposition \ref{prop:abs cont parts are inheriting} hold for
locally finite, Borel measures. This can easily be deduced from Proposition
\ref{prop:decomp peop}, Proposition \ref{prop:abs cont parts are inheriting}
and the Lindel\"of property of $X$.
\end{rem}

\begin{prop}
\label{prop:boundedness condition}Let $\tau$ be a finite Borel measure
on $X$ such that $\intop\intop\varphi(x,y)\mathrm{d}\tau(y)\mathrm{d}\tau(x)<\infty$.
Then for every $\varepsilon>0$ there exist $0<M<\infty$ and a compact
set $F\subseteq X$ such that $\tau(X\setminus F)<\varepsilon$ and
$\intop_{F}\varphi(x,y)\mathrm{d}\tau(y)<M$ for every $x\in X$.
\end{prop}

\begin{proof}
Since $\intop\intop\varphi(x,y)\mathrm{d}\tau(y)\mathrm{d}\tau(x)<\infty$
it follows that $\intop\varphi(x,y)\mathrm{d}\tau(y)<\infty$ for
$\tau$ almost every $x\in X$. Thus there exists a compact set $F_{0}\subseteq X$
and $M_{0}>0$ such that $\tau(X\setminus F_{0})<\varepsilon/2$ and
$\intop\varphi(x,y)\mathrm{d}\tau(y)<M_{0}$ for every $x\in F_{0}$.
In the proof of \cite[Chapter III. Thm 1, page 15]{Carleson-selected problems on exceptional sets}
it is shown that there exists a compact set $F\subseteq F_{0}$ such
that $\tau(X\setminus F)<\varepsilon$ and $\lim_{x\rightarrow x_{0}}\intop_{F}\varphi(x,y)\mathrm{d}\tau(y)=\intop_{F}\varphi(x_{0},y)\mathrm{d}\tau(y)<M_{0}$
for every $x_{0}\in F$. Hence there exists $r>0$ such that if $\mathrm{dist}(F,x)<r$
then $\intop_{F}\varphi(x,y)\mathrm{d}\tau(y)<M_{0}$. Whenever $\mathrm{dist}(F,x)\geq r$
then $\intop_{F}\varphi(x,y)\mathrm{d}\tau(y)\leq\varphi(r)\tau(X)$
because $\varphi$ is monotone decreasing. Hence the statement follows.
\end{proof}
\begin{prop}
\label{prop:singular extinction}If $\nu$ is a finite Borel measure
that is singular to every finite Borel measure with finite $\varphi$-energy
then there exits a Borel set $Z\subseteq X$ such that $\nu(X\setminus Z)=0$
and $C_{\varphi}(Z)=0$.
\end{prop}

\begin{proof}
Let $A_{n}=\left\{ x\in X:\int\varphi(x,y)\mathrm{d}\nu(y)\leq n\right\} $.
Then
\[
I_{\varphi}(\nu\vert_{A_{n}})=\int_{A_{n}}\int_{A_{n}}\varphi(x,y)\mathrm{d}\nu(y)\mathrm{d}\nu(x)\leq\int_{A_{n}}\int_{X}\varphi(x,y)\mathrm{d}\nu(y)\mathrm{d}\nu(x)\leq n\cdot\nu(A_{n})<\infty.
\]
Thus $\nu(A_{n})=0$ by the assumption and hence $\nu(X\setminus Z)=0$
for $Z=\left\{ x\in X:\int\varphi(x,y)\mathrm{d}\nu(y)=\infty\right\} $.

Assume for a contradiction that $C_{\varphi}(Z)>0$. Then there exists
a probability measure $\tau$ on $Z$ such that $\intop\intop\varphi(x,y)\mathrm{d}\tau(y)\mathrm{d}\tau(x)<\infty$.
Then by Proposition \ref{prop:boundedness condition} there exist
$0<M<\infty$ and $F\subseteq Z$ such that $\tau(F)>0$ and $\intop_{F}\varphi(x,y)\mathrm{d}\tau(y)<M$
for every $x\in X$. Thus
\[
\int_{X}\left(\int_{F}\varphi(x,y)\mathrm{d}\tau(y)\right)\mathrm{d}\nu(x)\leq\int_{X}M\mathrm{d}\nu(x)\leq M\cdot\nu(X)<\infty
\]
contradicting with that
\[
\int_{X}\left(\int_{F}\varphi(x,y)\mathrm{d}\tau(y)\right)\mathrm{d}\nu(x)=\int_{F}\left(\int_{X}\varphi(x,y)\mathrm{d}\nu(x)\right)\mathrm{d}\tau(y)=\int_{F}\infty\mathrm{d}\tau(x)=\infty
\]
where we used Fubini`s theorem. Hence $C_{\varphi}(Z)=0$.
\end{proof}
\begin{rem}
\label{rem:sing felbontas loc fin}Proposition \ref{prop:singular extinction}
holds for locally finite, Borel measures. This can easily be deduced
from Proposition \ref{prop:singular extinction}, the Lindel\"of
property of $X$ and the fact that $C_{\varphi}(\cup_{n=1}^{\infty}A_{n})=0$
for a sequence of Borel sets $A_{n}$ with $C_{\varphi}(A_{n})=0$.
\end{rem}

\section{Degenerate case\label{sec:Degenerate-case}}

In this section we discuss why the conditional measure of the $\varphi$-singular
part $\nu_{\perp}$ vanishes.
\begin{lem}
\label{lem:compact separation lemma}Let $D\subseteq X$ be a compact
set and let $B_{\omega}$ be a closed realisation of the random closed
set $B$ such that $D\cap B_{\omega}=\emptyset$. Let $\nu$ be a
finite Borel measure such that $\mathrm{supp}\nu\subseteq D$. Then
$\mathcal{C}_{k}(\nu)(X)=\mathcal{C}_{k,\omega}(\nu)(X)$ converges
to $0$ for that realisation $B_{\omega}$.
\end{lem}

\begin{proof}
Since $B\cap D=\emptyset$ then $\mathrm{dist}(B,D)>0$. Let $k_{0}$
be such that $\sup\left\{ \mathrm{diam}(Q):Q\in\mathcal{Q}_{k}\right\} <\mathrm{dist}(B,D)$
for ever $k\geq k_{0}$ (we note that $k_{0}$ depends on the realisation
$B_{\omega}$ but exists nevertheless). Then $\mathcal{C}_{k,\omega}(\nu)(X)=0$
by the definition of $\mathcal{C}_{k}$ for $k\geq k_{0}$.
\end{proof}
\begin{thm}
\label{prop:dieing singular part}Assume that if $C_{\varphi}(D)=0$
for some compact set $D\subseteq X$ then $B\cap D=\emptyset$ almost
surely. If $\nu$ is a finite Borel measure that is singular to every
finite Borel measure with finite $\varphi$-energy then $\mathcal{C}_{k}(\nu)(X)$
converges to $0$ in probability.
\end{thm}

\begin{proof}
By Proposition \ref{prop:singular extinction} there exists $Z\subseteq X$
such that $\nu(X\setminus Z)=0$ and $C_{\varphi}(Z)=0$. Let $\eta>0$
be fixed. There exists a compact set $D\subseteq Z$ such that $\nu(X\setminus D)<\eta$
by inner regularity. Then by assumption $B\cap D=\emptyset$ almost
surely. Thus $\mathcal{C}_{k}(\nu\vert_{D})(X)$ converges to $0$
almost surely by Lemma \ref{lem:compact separation lemma} and hence
converges to $0$ in probability. Since $\mathcal{C}_{k}(\nu)=\mathcal{C}_{k}(\nu\vert_{D})+\mathcal{C}_{k}(\nu\vert_{X\setminus D})$
it follows by (\ref{eq:martingal expectation inequality}) that
\[
\limsup_{k\rightarrow\infty}\rho\left(\mathcal{C}_{k}(\nu)(X),0\right)\leq\limsup_{k\rightarrow\infty}\rho\left(\mathcal{C}_{k}(\nu)(X),\mathcal{C}_{k}(\nu\vert_{D})(X)\right)+\rho\left(\mathcal{C}_{k}(\nu\vert_{D})(X),0\right)
\]
\[
\leq\limsup_{k\rightarrow\infty}E\left(\mathcal{C}_{k}(\nu\vert_{X\setminus D})(X)\right)+0\leq\nu(X\setminus D)<\eta.
\]
Since we can choose $\eta$ to be arbitrarily small it follows that
$\mathcal{C}_{k}(\nu)(X)$ converges to $0$ in probability.
\end{proof}
\begin{rem}
\label{rem:dieing only on X0} Let $\nu$ be such that $\nu(X\setminus X_{0})=0$.
Then for the conclusion of Theorem \ref{prop:dieing singular part}
it is enough to assume that if $C_{\varphi}(D)=0$ for some compact
set $D\subseteq X_{0}$ (rather than for compact sets $D\subseteq X$)
then $B\cap D=\emptyset$ almost surely. In the proof we can choose
$D\subseteq Z\cap X_{0}$.
\end{rem}

\begin{lem}
\label{lem:complem X0 is dead}Let $\nu$ be a finite Borel measure
on $X$ such that $\nu(X_{0})=0$. Then $\mathcal{C}_{k}(\nu)(X)$
converges to $0$ in $\mathcal{L}^{1}$ and so in probability.
\end{lem}

\begin{proof}
It follows from definition $\mathcal{C}_{k}(\nu)$ that
\begin{equation}
E(\mathcal{C}_{k}(\nu)(X))=\sum_{Q\in\mathcal{Q}_{k}}\nu(Q)=\nu(\bigcup_{Q\in\mathcal{Q}_{k}}Q).\label{eq:deg martingal eq}
\end{equation}
Since $\nu(X_{0})=0$ it follows that $\lim_{k\rightarrow\infty}\nu(\bigcup_{\mathcal{Q}_{k}}Q)=0$
and so the statement follows.
\end{proof}
\begin{rem}
Let $X$ be locally compact and $\nu$ be a locally finite Borel measure.
If $\nu(X_{0})=0$ then $\mathcal{C}_{k}(\nu)(K)$ converges to $0$
in $\mathcal{L}^{1}$ and so in probability for every compact set
$K$. It follows easily from Lemma \ref{lem:complem X0 is dead}.
That implies that $\mathcal{C}_{k}(\nu)(K)$ converges to $0$ vaguely
in probability.
\end{rem}

\section{$\mathcal{L}^{2}$-boundedness\label{sec:-boundedness}}

In this section we show that under the assumptions of Section \ref{subsec:Special-assumptions}
if $\nu$ is a finite Borel measure with $I_{\varphi}(\nu)<\infty$
then the sequence $\mathcal{C}_{k}(\nu)(X)$ is $\mathcal{L}^{2}$-bounded.

\begin{lem}
\label{lem:E(mu_n)}Let $Q,S\in\mathcal{Q}_{k}$, $Q\neq S$ and $\nu_{1}$
and $\nu_{2}$ be finite Borel measures on $X$. If (\ref{eq:capacity-independence})
holds for $Q$ and $S$ then
\[
E\left(\mathcal{C}_{k}(\nu_{1})(Q)\cdot\mathcal{C}_{k}(\nu_{2})(S)\right)\leq c\cdot\varphi(\mathrm{dist}(Q,S)\cdot\nu_{1}(Q)\cdot\nu_{2}(S).
\]
\end{lem}

\begin{proof}
By the definition of $\mathcal{C}_{k}(\nu)$ and (\ref{eq:capacity-independence})
it follows that
\[
E\left(\mathcal{C}_{k}(\nu_{1})(Q)\cdot\mathcal{C}_{k}(\nu_{2})(S)\right)
\]
\[
=E\left(P(Q\cap B\neq\emptyset)^{-1}\cdot I_{Q\cap B\neq\emptyset}\cdot\nu_{1}(Q)\cdot P(S\cap B\neq\emptyset)^{-1}\cdot I_{S\cap B\neq\emptyset}\cdot\nu_{2}(S)\right)
\]
\[
=P(Q\cap B\neq\emptyset)^{-1}\cdot P(S\cap B\neq\emptyset)^{-1}\cdot P(Q\cap B\ne\emptyset\,and\,S\cap B\ne\emptyset)\cdot\nu_{1}(Q)\cdot\nu_{2}(S)
\]
\[
\leq c\cdot\varphi(\mathrm{dist}(Q,S))\cdot\nu_{1}(Q)\cdot\nu_{2}(S).
\]
\end{proof}
\begin{lem}
\label{lem:dist jump lem}Let $\nu_{1}$ and $\nu_{2}$ be finite
Borel measures on $X$. Assume that $\delta>0$ is such that (\ref{eq:Kernel restriction})
holds and $Q,S\in\mathcal{Q}_{k}$ such that $\max\left\{ \mathrm{diam}(Q),\mathrm{diam}(S)\right\} <\delta\cdot\mathrm{dist}(Q,S)$.
Then
\[
\varphi(\mathrm{dist}(Q,S))\cdot\nu_{1}(Q)\cdot\nu_{2}(S)\leq c_{2}\intop_{S}\left(\intop_{Q}\varphi(x,y)\mathrm{d}\nu_{1}(x)\right)\mathrm{d}\nu_{2}(y)+c_{3}\nu_{1}(Q)\nu_{2}(S).
\]
\end{lem}

\begin{proof}
If $x\in Q$ and $y\in S$ then
\[
d(x,y)\leq\mathrm{dist}(Q,S)+\mathrm{diam}(Q)+\mathrm{diam}(S)<\mathrm{dist}(Q,S)\cdot(1+2\delta).
\]
Then by (\ref{eq:Kernel restriction}) and the monotonicity of $\varphi$
it follows that
\[
\varphi(\mathrm{dist}(Q,S))\leq c_{2}\varphi(\mathrm{dist}(Q,S)\cdot(1+2\delta))+c_{3}\leq c_{2}\varphi(d(x,y))+c_{3}.
\]
Integrating over $Q\times S$ with respect to $\nu_{1}\times\nu_{2}$
the statement follows.
\end{proof}
\begin{lem}
\label{lem:diagonal expect bound}Assume that (\ref{eq:lower hitting prob})
holds. Let $\nu_{1}$ and $\nu_{2}$ be finite Borel measures on $X$.
If $Q,S\in\mathcal{Q}_{k}$ then
\[
E\left(\mathcal{C}_{k}(\nu_{1})(Q)\cdot\mathcal{C}_{k}(\nu_{2})(S)\right)\leq a^{-1}\left(\intop_{Q}\intop_{Q}\varphi(x,y)\mathrm{d}\nu_{1}(x)\mathrm{d}\nu_{1}(y)+\intop_{S}\intop_{S}\varphi(x,y)\mathrm{d}\nu_{2}(x)\mathrm{d}\nu_{2}(y)\right).
\]
\end{lem}

\begin{proof}
Due to symmetry without the loss of generality we can assume that
$\nu_{1}(Q)\leq\nu_{2}(S)$. Also we can assume that $0<\nu_{1}(Q)\leq\nu_{2}(S)$
otherwise the proof is trivial. By the definition of $\mathcal{C}_{k}$
and (\ref{eq:lower hitting prob}) it follows that
\[
E\left(\mathcal{C}_{k}(\nu_{1})(Q)\cdot\mathcal{C}_{k}(\nu_{2})(S)\right)
\]
\[
=E\left(P(Q\cap B\neq\emptyset)^{-1}\cdot I_{Q\cap B\neq\emptyset}\cdot\nu_{1}(Q)\cdot P(S\cap B\neq\emptyset)^{-1}\cdot I_{S\cap B\neq\emptyset}\cdot\nu_{2}(S)\right)
\]
\[
=P(Q\cap B\neq\emptyset)^{-1}\cdot P(S\cap B\neq\emptyset)^{-1}\cdot P(Q\cap B\ne\emptyset\,and\,S\cap B\ne\emptyset)\cdot\nu_{1}(Q)\cdot\nu_{2}(S)
\]
\[
\leq P(Q\cap B\neq\emptyset)^{-1}\cdot P(S\cap B\neq\emptyset)^{-1}\cdot P(Q\cap B\ne\emptyset)\cdot\nu_{2}(S)\cdot\nu_{2}(S)
\]

\[
=P(S\cap B\neq\emptyset)^{-1}\cdot\nu_{2}(S)^{2}\leq a^{-1}\cdot C_{\varphi}(S)^{-1}\cdot\nu_{2}(S)^{2}\leq a^{-1}\cdot\left(I_{\varphi}(\nu_{2}\vert_{S})\cdot\nu_{2}(S)^{-2}\right)\cdot\nu_{2}(S)^{2}
\]
\[
=a^{-1}\intop_{S}\intop_{S}\varphi(x,y)\mathrm{d}\nu_{2}(x)\mathrm{d}\nu_{2}(y)\leq a^{-1}\left(\intop_{Q}\intop_{Q}\varphi(x,y)\mathrm{d}\nu_{1}(x)\mathrm{d}\nu_{1}(y)+\intop_{S}\intop_{S}\varphi(x,y)\mathrm{d}\nu_{2}(x)\mathrm{d}\nu_{2}(y)\right).
\]
\end{proof}
\begin{prop}
\label{prop:2bounded martingale}Let $\nu$ be a finite Borel measure
on $X$. Assume that (\ref{eq:lower hitting prob}) holds and there
exists $\delta>0$ such that (\ref{eq:Kernel restriction}), (\ref{eq:capacity-independence})
and (\ref{eq:bounded sundivision}) hold. Then
\[
E\left(\mathcal{C}_{k}(\nu)(X)\mathcal{C}_{k}(\nu)(X)\right)\leq\left(cc_{2}+2a^{-1}M_{\delta}\right)I_{\varphi}(\nu)+cc_{3}\nu(X)^{2}
\]
for every $k\in\mathbb{N}$.
\end{prop}

\begin{proof}
We say that a pair $(Q,S)\in\mathcal{Q}_{k}\times\mathcal{Q}_{k}$
is a `good' pair if $\max\left\{ \mathrm{diam}(Q),\mathrm{diam}(S)\right\} <\delta\cdot\mathrm{dist}(Q,S)$
and is a bad pair if $\max\left\{ \mathrm{diam}(Q),\mathrm{diam}(S)\right\} \geq\delta\cdot\mathrm{dist}(Q,S)$.
Combining Lemma \ref{lem:E(mu_n)} and Lemma \ref{lem:dist jump lem}
it follows that
\[
\sum_{(Q,S)\mathrm{\,is\,good}}E\left(\mathcal{C}_{k}(\nu)(Q)\mathcal{C}_{k}(\nu)(S)\right)\leq\sum_{(Q,S)\mathrm{\,is\,good}}c\cdot\varphi(\mathrm{dist}(Q,S)\cdot\nu(Q)\cdot\nu(S)
\]
\[
\leq c\sum_{(Q,S)\mathrm{\,is\,good}}\left(c_{2}\intop_{S}\left(\intop_{Q}\varphi(x,y)\mathrm{d}\nu(x)\right)\mathrm{d}\nu(y)+c_{3}\nu(Q)\nu(S)\right)
\]
\[
\leq c\sum_{Q,S\in\mathcal{Q}_{k}}\left(c_{2}\intop_{S}\left(\intop_{Q}\varphi(x,y)\mathrm{d}\nu(x)\right)\mathrm{d}\nu(y)+c_{3}\nu(Q)\nu(S)\right)\leq c\left(c_{2}I_{\varphi}(\nu)+c_{3}\nu(X)^{2}\right).
\]
By Lemma \ref{lem:diagonal expect bound} and (\ref{eq:bounded sundivision})
\[
\sum_{(Q,S)\mathrm{\,is\,bad}}E\left(\mathcal{C}_{k}(\nu)(Q)\mathcal{C}_{k}(\nu)(S)\right)
\]
\[
\leq\sum_{(Q,S)\mathrm{\,is\,bad}}a^{-1}\left(\intop_{Q}\intop_{Q}\varphi(x,y)\mathrm{d}\nu(x)\mathrm{d}\nu(y)+\intop_{S}\intop_{S}\varphi(x,y)\mathrm{d}\nu(x)\mathrm{d}\nu(y)\right)
\]
\[
=2a^{-1}\sum_{Q\in\mathcal{Q}_{k}}\left(\sum_{\begin{array}{c}
S\in\mathcal{Q}_{k}\\
(Q,S)\mathrm{\,is\,bad}
\end{array}}\intop_{Q}\intop_{Q}\varphi(x,y)\mathrm{d}\nu(x)\mathrm{d}\nu(y)\right)
\]
\[
\leq2a^{-1}\sum_{Q\in\mathcal{Q}_{k}}M_{\delta}\intop_{Q}\intop_{Q}\varphi(x,y)\mathrm{d}\nu(x)\mathrm{d}\nu(y)\leq2a^{-1}M_{\delta}I_{\varphi}(\nu).
\]
Hence the statement follows by
\[
E\left(\mathcal{C}_{k}(\nu)(X)^{2}\right)=\sum_{Q,S\in\mathcal{Q}_{k}}E\left(\mathcal{C}_{k}(\nu)(Q)\mathcal{C}_{k}(\nu)(S)\right)
\]
\[
=\sum_{(Q,S)\mathrm{\,is\,good}}E\left(\mathcal{C}_{k}(\nu)(Q)\mathcal{C}_{k}(\nu)(S)\right)+\sum_{(Q,S)\mathrm{\,is\,bad}}E\left(\mathcal{C}_{k}(\nu)(Q)\mathcal{C}_{k}(\nu)(S)\right)
\]
\[
\leq c\left(c_{2}I_{\varphi}(\nu)+c_{3}\nu(X)^{2}\right)+2a^{-1}M_{\delta}I_{\varphi}(\nu).
\]
\end{proof}

\section{\label{sec:Non-degenerate-percolation-of}Non-degenerate limit}

In this section our main goal is to show that $\mathcal{C}_{k}(\nu)(A)$
converges in $\mathcal{L}^{2}$ if $I_{\varphi}(\nu)<\infty$ and
$\underline{F}(x,y)=\overline{F}(x,y)$. Recall that the definition
of $F_{k,n}$, $\underline{F}(x,y)$ and $\overline{F}(x,y)$ can
be found in Section \ref{subsec:Conditional-measure}. A key observation
is that $E\left(\mathcal{C}_{k}(\nu)(X)^{2}\right)=\intop_{X}\intop_{X}F_{k,k}(x,y)\mathrm{d}\nu(x)\mathrm{d}\nu(y)$.
Using this and the assumption that $\underline{F}(x,y)=\overline{F}(x,y)$
we prove the $\mathcal{L}^{2}$ convergence in two steps. We divide
the double integral into two parts, one part is the double integral
on a domain that is bounded away from the diagonal and approximates
the double integral uniformly, the other part is around the diagonal
that is small. Then from this we deduce the convergence in $\mathcal{L}^{2}$.
At the end of the section we show that if $\mathcal{C}_{k}(\nu)(A)$
is a martingale then we do not even need the assumption that $\underline{F}(x,y)=\overline{F}(x,y)$
because then the convergence in $\mathcal{L}^{2}$ is automatic by
the $\mathcal{L}^{2}$-boundedness.
\begin{lem}
\label{lem:Fkn lem sequance exact}Let $\nu_{1}$ and $\nu_{2}$ be
finite Borel measures on $X$. For $k,n\in\mathbb{N}$
\[
E\left(\mathcal{C}_{k}(\nu_{1})(X)\cdot\mathcal{C}_{n}(\nu_{2})(X)\right)=\intop_{X}\intop_{X}F_{k,n}(x,y)\mathrm{d}\nu_{1}(x)\mathrm{d}\nu_{2}(y).
\]
\end{lem}

\begin{proof}
By the definition of $\mathcal{C}_{k}$ it follows that
\[
E\left(\mathcal{C}_{k}(\nu_{1})(X)\cdot\mathcal{C}_{n}(\nu_{2})(X)\right)=\sum_{Q\in\mathcal{Q}_{k}}\sum_{S\in\mathcal{Q}_{n}}\frac{P(Q\cap B\ne\emptyset\,and\,S\cap B\ne\emptyset)}{P(Q\cap B\ne\emptyset)\cdot P(S\cap B\ne\emptyset)}\cdot\nu_{1}(Q)\cdot\nu_{2}(S)
\]
\[
=\sum_{Q\in\mathcal{Q}_{k}}\sum_{S\in\mathcal{Q}_{n}}\intop_{Q}\intop_{S}F_{k,n}(x,y)\mathrm{d}\nu_{1}(x)\mathrm{d}\nu_{2}(y)=\intop_{X}\intop_{X}F_{k,n}(x,y)\mathrm{d}\nu_{1}(x)\mathrm{d}\nu_{2}(y).
\]
\end{proof}
\begin{lem}
\label{lem:Fkn mixed integral}Let $\nu_{1}$ and $\nu_{2}$ be finite
Borel measures on $X$. Let $Q,S\subseteq X$ be Borel sets. For $k,n\in\mathbb{N}$,
it follows that
\[
E\left(\mathcal{C}_{k}(\nu_{1})(Q)\cdot\mathcal{C}_{n}(\nu_{2})(S)\right)=\intop_{S}\left(\intop_{Q}F_{k,n}(x,y)\mathrm{d}\nu_{1}(x)\right)\mathrm{d}\nu_{2}(y).
\]
\end{lem}

Lemma \ref{lem:Fkn mixed integral} follows by the application of
Lemma \ref{lem:Fkn lem sequance exact} to the measures $\nu_{1}\vert_{Q}$
and $\nu_{2}\vert_{S}$.

\begin{lem}
\label{lem:nondiagonal integral bound}Let $\nu_{1}$ and $\nu_{2}$
be finite Borel measures on $X$. Assume that there exists $\delta>0$
such that (\ref{eq:capacity-independence}) holds. Let $\varepsilon,\eta>0$
and let $A_{\varepsilon}=\left\{ (x,y)\in X\times X:d(x,y)>\varepsilon\right\} $.
Then there exists $m\in\mathbb{N}$ such that for every $n,k\geq m$
\[
-\eta+\iintop_{A_{\varepsilon}}\underline{F}(x,y)\mathrm{d}\nu_{1}(x)\mathrm{d}\nu_{2}(y)\leq\iintop_{A_{\varepsilon}}F_{k,n}(x,y)\mathrm{d}\nu_{1}(x)\mathrm{d}\nu_{2}(y)\leq\eta+\iintop_{A_{\varepsilon}}\overline{F}(x,y)\mathrm{d}\nu_{1}(x)\mathrm{d}\nu_{2}(y)
\]
\end{lem}

\begin{proof}
Let $k_{0}\in\mathbb{N}$ be large enough that $\sup\left\{ \mathrm{diam}(Q):Q\in\mathcal{Q}_{k}\right\} <\varepsilon/3$
for every $k>k_{0}$, we can choose such $k_{0}$ due to (\ref{eq:diameter_goes_to0}).
Whenever $A_{\varepsilon}\cap Q\times S\neq\emptyset$ for $Q\in\mathcal{Q}_{k}$,
$S\in\mathcal{Q}_{n}$, $k,n\geq k_{0}$ then $\mathrm{dist}(Q,S)>\varepsilon/3$.
We can choose $k_{1}\geq k_{0}$ such that $\sup\left\{ \mathrm{diam}(Q):Q\in\mathcal{Q}_{k}\right\} <\delta\varepsilon/3$
for every $k\geq k_{1}$. Whenever $A_{\varepsilon}\cap Q\times S\neq\emptyset$
for $Q\in\mathcal{Q}_{k}$, $S\in\mathcal{Q}_{n}$, $k,n\geq k_{1}$
then $\max\left(\mathrm{diam}(Q),\mathrm{diam}(S)\right)<\delta\varepsilon/3<\delta\mathrm{dist}(Q,S)$.
Hence $F_{k,n}(x,y)\leq c\varphi(\mathrm{dist}(Q,S))$ for every $(x,y)\in Q\times S$
by (\ref{eq:capacity-independence}). Since $\varphi$ is a nonnegative,
monotone decreasing, continuous function it follows that $\varphi$
is absolutely continuous on $[\varepsilon/3,\infty)$, hence we can
choose $k_{2}\geq k_{1}$, due to (\ref{eq:diameter_goes_to0}), such
that $\varphi(\mathrm{dist}(Q,S))\leq\varphi(d(x,y))+\eta/c$ whenever
$A_{\varepsilon}\cap Q\times S\neq\emptyset$ for $x\in Q\in\mathcal{Q}_{k}$,
$y\in S\in\mathcal{Q}_{n}$, $k,n\geq k_{2}$. Hence
\[
F_{k,n}(x,y)\leq c\varphi(d(x,y))+\eta
\]
whenever $A_{\varepsilon}\cap Q\times S\neq\emptyset$ for $x\in Q\in\mathcal{Q}_{k}$,
$y\in S\in\mathcal{Q}_{n}$ $k,n\geq k_{2}$. Thus
\[
\underline{F}_{N}(x,y)\leq F_{k,n}(x,y)\leq\overline{F}_{N}(x,y)\leq c\cdot\varphi(d(x,y))+\eta\leq c\varphi(\varepsilon)+\eta
\]
for $(x,y)\in A_{\varepsilon}$, $N\geq k_{2}$ and $k,n\geq N$.
Since $\underline{F_{N}}$ and $\overline{F_{N}}$ converge as $N$
goes to $\infty$, due to the dominated convergence theorem, there
exists $m\geq k_{2}$ such that
\[
-\eta+\iintop_{A_{\varepsilon}}\underline{F}(x,y)\mathrm{d}\nu_{1}(x)\mathrm{d}\nu_{2}(y)\leq\iintop_{A_{\varepsilon}}\underline{F}_{N}(x,y)\mathrm{d}\nu_{1}(x)\mathrm{d}\nu_{2}(y)
\]
and
\[
\iintop_{A_{\varepsilon}}\overline{F}_{N}(x,y)\mathrm{d}\nu_{1}(x)\mathrm{d}\nu_{2}(y)\leq\eta+\iintop_{A_{\varepsilon}}\overline{F}(x,y)\mathrm{d}\nu_{1}(x)\mathrm{d}\nu_{2}(y)
\]
for every $N\geq m$. Thus the statement follows.
\end{proof}
\begin{notation}
\label{nota:r_epsilon }Assume that (\ref{eq:same size}) holds. For
$\varepsilon>0$ let $k_{\varepsilon}$ be the largest positive integer
such that $\inf\left\{ \mathrm{diam}(Q):Q\in\mathcal{Q}_{k_{\varepsilon}}\right\} >\varepsilon$,
it is well-defined by (\ref{eq:same size}) and (\ref{eq:diameter_goes_to0})
if $\varepsilon>0$ is small enough. Let $r_{\varepsilon}=\sup\left\{ \mathrm{diam}(Q):Q\in\mathcal{Q}_{k_{\varepsilon}}\right\} $.
Note that $r_{\varepsilon}$ converges to $0$ as $\varepsilon$ approaches
$0$ by (\ref{eq:diameter_goes_to0}) and (\ref{eq:same size}).
\end{notation}

\begin{lem}
\label{lem:A_epsilon reminder}Let $\nu_{1}$ and $\nu_{2}$ be finite
Borel measures on $X$ and let $\nu=\nu_{1}+\nu_{2}$. Let $\varepsilon>0$,
let $G_{\varepsilon}=\left\{ (x,y)\in X\times X:d(x,y)\leq\varepsilon\right\} $
and $H_{\varepsilon}=\left\{ (x,y)\in X\times X:d(x,y)\leq\varepsilon+2r_{\varepsilon}\right\} $.
Assume that (\ref{eq:same size}) and (\ref{eq:lower hitting prob})
hold and there exists $0<\delta<1$ such that (\ref{eq:Kernel restriction}),
(\ref{eq:capacity-independence}) and (\ref{eq:bounded sundivision}).
Then there exists $c_{4}>0$, depending on $a$, $c$, $c_{2}$, $c_{3}$
and $M_{\delta}$, such that
\[
\iintop_{G_{\varepsilon}}F_{n,n}(x,y)\mathrm{d}\nu_{1}(x)\mathrm{d}\nu_{2}(y)\leq c_{4}\iintop_{H_{\varepsilon}}\varphi(x,y)\mathrm{d}\nu(x)\mathrm{d}\nu(y)+c_{4}\cdot\nu\times\nu(H_{\varepsilon})
\]
for every $n\in\mathbb{N}$.
\end{lem}

\begin{proof}
Whenever $Q\times S\cap G_{\varepsilon}\neq\emptyset$ for $Q,S\in\mathcal{Q}_{k_{\varepsilon}}$
then $\mathrm{dist}(Q,S)\leq\varepsilon<\max\left\{ \mathrm{diam}(Q),\mathrm{diam}(F)\right\} $.
In particular, $\delta\cdot\mathrm{dist}(Q,S)<\max\left\{ \mathrm{diam}(Q),\mathrm{diam}(F)\right\} $.
Hence
\begin{equation}
\#\left\{ S\in\mathcal{Q}_{k_{\varepsilon}}:Q\times S\cap G_{\varepsilon}\neq\emptyset\right\} \leq M_{\delta}\label{eq:bounded diagonal cover}
\end{equation}
for every $Q\in\mathcal{Q}_{k_{\varepsilon}}$ by (\ref{eq:bounded sundivision}).
By Lemma \ref{lem:Fkn mixed integral} and Lemma \ref{prop:2bounded martingale}
it follows that
\[
\iintop_{G_{\varepsilon}}F_{n,n}(x,y)\mathrm{d}\nu_{1}(x)\mathrm{d}\nu_{2}(y)\leq\iintop_{G_{\varepsilon}}F_{n,n}(x,y)\mathrm{d}\nu(x)\mathrm{d}\nu(y)
\]
\[
\leq\sum_{\begin{array}{c}
Q,S\in\mathcal{Q}_{k_{\varepsilon}}\\
Q\times S\cap G_{\varepsilon}\neq\emptyset
\end{array}}\iintop_{Q\times S}F_{n,n}(x,y)\mathrm{d}\nu(x)\mathrm{d}\nu(y)=\sum_{\begin{array}{c}
Q,S\in\mathcal{Q}_{k_{\varepsilon}}\\
Q\times S\cap G_{\varepsilon}\neq\emptyset
\end{array}}E\left(\mathcal{C}_{n}(\nu)(Q)\cdot\mathcal{C}_{n}(\nu)(S)\right)
\]
\begin{equation}
\leq\sum_{\begin{array}{c}
Q,S\in\mathcal{Q}_{k_{\varepsilon}}\\
Q\times S\cap G_{\varepsilon}\neq\emptyset
\end{array}}\left(cc_{2}+2a^{-1}M_{\delta}\right)I_{\varphi}(\nu\vert_{Q}+\nu\vert_{S})+cc_{3}\nu(Q\cup S)^{2}.\label{eq:reminder eq sum}
\end{equation}
Whenever $Q\times S\cap G_{\varepsilon}\neq\emptyset$ for $Q,S\in\mathcal{Q}_{k_{\varepsilon}}$
then $Q\times S\subseteq H_{\varepsilon}$. Hence
\begin{equation}
\sum_{\begin{array}{c}
Q,S\in\mathcal{Q}_{k_{\varepsilon}}\\
Q\times S\cap G_{\varepsilon}\neq\emptyset
\end{array}}\iintop_{Q\times S}\varphi(x,y)\mathrm{d}\nu(x)\mathrm{d}\nu(y)\leq\iintop_{H_{\varepsilon}}\varphi(x,y)\mathrm{d}\nu(x)\mathrm{d}\nu(y).\label{eq:reminder eq 1}
\end{equation}
By (\ref{eq:bounded diagonal cover})
\[
\sum_{\begin{array}{c}
Q,S\in\mathcal{Q}_{k_{\varepsilon}}\\
Q\times S\cap G_{\varepsilon}\neq\emptyset
\end{array}}\iintop_{Q\times Q}\varphi(x,y)\mathrm{d}\nu(x)\mathrm{d}\nu(y)\leq M_{\delta}\sum_{Q\in\mathcal{Q}_{k_{\varepsilon}}}\iintop_{Q\times Q}\varphi(x,y)\mathrm{d}\nu(x)\mathrm{d}\nu(y)
\]
\begin{equation}
\leq M_{\delta}\iintop_{H_{\varepsilon}}\varphi(x,y)\mathrm{d}\nu(x)\mathrm{d}\nu(y)\label{eq:reminder eq 2}
\end{equation}
and similarly
\begin{equation}
\sum_{\begin{array}{c}
Q,S\in\mathcal{Q}_{k_{\varepsilon}}\\
Q\times S\cap G_{\varepsilon}\neq\emptyset
\end{array}}\nu(Q)^{2}\leq M_{\delta}\sum_{Q\in\mathcal{Q}_{k_{\varepsilon}}}\nu(Q)^{2}\leq M_{\delta}\cdot\nu\times\nu(H_{\varepsilon})\label{eq:remindereq 3}
\end{equation}
It is easy to see that
\begin{equation}
\sum_{\begin{array}{c}
Q,S\in\mathcal{Q}_{k_{\varepsilon}}\\
Q\times S\cap G_{\varepsilon}\neq\emptyset
\end{array}}\nu(Q)\nu(S)\leq\nu\times\nu(H_{\varepsilon}).\label{eq:reminder eq 4}
\end{equation}
Since
\[
I_{\varphi}(\nu\vert_{Q}+\nu\vert_{S})=2\iintop_{Q\times S}\varphi(x,y)\mathrm{d}\nu(x)\mathrm{d}\nu(y)+\iintop_{Q\times Q}\varphi(x,y)\mathrm{d}\nu(x)\mathrm{d}\nu(y)+\iintop_{S\times S}\varphi(x,y)\mathrm{d}\nu(x)\mathrm{d}\nu(y)
\]
and
\[
\nu(Q\cup S)^{2}\leq\nu(Q)^{2}+\nu(S)^{2}+2\nu(Q)\nu(S),
\]
the statement follows by combining (\ref{eq:reminder eq sum}), (\ref{eq:reminder eq 1}),
(\ref{eq:reminder eq 2}), (\ref{eq:remindereq 3}) and (\ref{eq:reminder eq 4}).
\end{proof}

\subsection{Limit in the presence of the kernel $F(x,y)$}
\begin{lem}
\label{lem:Lim exist prop nontouchable condiition}Let $\nu$ be finite
Borel measure. Then
\begin{equation}
\intop_{X}\intop_{X}F_{n,n}(x,y)+F_{k,k}(x,y)-2F_{k,n}(x,y)\mathrm{d}\nu(x)\mathrm{d}\nu(y)\geq0\label{eq:lim manko}
\end{equation}
for every $k,n\in\mathbb{N}$. If
\[
\limsup_{n\rightarrow\infty}\limsup_{k\rightarrow\infty}\intop_{X}\intop_{X}F_{n,n}(x,y)+F_{k,k}(x,y)-2F_{k,n}(x,y)\mathrm{d}\nu(x)\mathrm{d}\nu(y)=0
\]
then $\mathcal{C}_{k}(\nu)(X)$ converges in $\mathcal{L}^{2}$.
\end{lem}

\begin{proof}
Since $\mathcal{L}^{2}$ is complete it is sufficient to show that
$\mathcal{C}_{k}(\nu)(X)$ is a Cauchy sequence in $\mathcal{L}^{2}$.
For the integers $n\leq k$ by Lemma \ref{lem:Fkn lem sequance exact}
\[
E\left(\left(\mathcal{C}_{k}(\nu)(X)-\mathcal{C}_{n}(\nu)(X)\right)^{2}\right)=E\left(\mathcal{C}_{k}(\nu)(X)^{2}+\mathcal{C}_{n}(\nu)(X)^{2}-2\mathcal{C}_{k}(\nu)(X)\mathcal{C}_{n}(\nu)(X)\right)
\]
\[
=\intop_{X}\intop_{X}F_{k,k}(x,y)+F_{n,n}(x,y)-2F_{k,n}(x,y)\mathrm{d}\nu(x)\mathrm{d}\nu(y),
\]
hence (\ref{eq:lim manko}) holds and by the assumption
\[
\limsup_{n\rightarrow\infty}\limsup_{k\rightarrow\infty}E\left(\mathcal{C}_{k}(\nu)(X)-\mathcal{C}_{n}(\nu)(X)\right)^{2}=0
\]
and so $\mathcal{C}_{k}(\nu)(X)$ is a Cauchy sequence in $\mathcal{L}^{2}$.
\end{proof}
\begin{prop}
\label{prop:L2 limsup=00003Dlimif cauchy}Let $\nu$ be a finite Borel
measure on $X$ with $I_{\varphi}(\nu)<\infty$. Assume that (\ref{eq:phi =00003Dinfty}),
(\ref{eq:same size}) and (\ref{eq:lower hitting prob}) hold and
there exists $0<\delta<1$ such that (\ref{eq:Kernel restriction}),
(\ref{eq:capacity-independence}) and (\ref{eq:bounded sundivision})
hold. If $\underline{F}(x,y)=\overline{F}(x,y)$ for $\nu\times\nu$
almost every $(x,y)$ then
\[
\limsup_{n\rightarrow\infty}\limsup_{k\rightarrow\infty}\intop_{X}\intop_{X}F_{n,n}(x,y)+F_{k,k}(x,y)-2F_{k,n}(x,y)\mathrm{d}\nu(x)\mathrm{d}\nu(y)=0.
\]
\end{prop}

\begin{proof}
Let $\eta,\varepsilon>0$. Let $m\in\mathbb{N}$ be as in Lemma \ref{lem:nondiagonal integral bound}
and $k,n\geq m$. Let $A_{\varepsilon}$, $G_{\varepsilon}$ and $H_{\varepsilon}$
as in Lemma \ref{lem:nondiagonal integral bound} and Lemma \ref{lem:A_epsilon reminder}.
Then
\[
\intop_{X}\intop_{X}F_{n,n}(x,y)+F_{k,k}(x,y)-2F_{k,n}(x,y)\mathrm{d}\nu(x)\mathrm{d}\nu(y)
\]
\[
\leq\intop_{X}\intop_{X}F_{n,n}(x,y)+F_{k,k}(x,y)\mathrm{d}\nu(x)\mathrm{d}\nu(y)-2\iintop_{A_{\varepsilon}}F_{k,n}(x,y)\mathrm{d}\nu(x)\mathrm{d}\nu(y)
\]
\[
\leq2\eta+2\iintop_{A_{\varepsilon}}\overline{F}(x,y)\mathrm{d}\nu(x)\mathrm{d}\nu(y)+\iintop_{G_{\varepsilon}}F_{n,n}(x,y)+F_{k,k}(x,y)\mathrm{d}\nu(x)\mathrm{d}\nu(y)+2\eta-2\iintop_{A_{\varepsilon}}\underline{F}(x,y)\mathrm{d}\nu(x)\mathrm{d}\nu(y)
\]
\[
\leq4\eta+2\cdot c_{4}\left(\iintop_{H_{\varepsilon}}\varphi(x,y)\mathrm{d}\nu(x)\mathrm{d}\nu(y)+\nu\times\nu(H_{\varepsilon})\right)
\]
where we used Lemma \ref{lem:A_epsilon reminder} for $\nu_{1}=\nu_{2}=\nu$
and the fact that $\underline{F}(x,y)=\overline{F}(x,y)$ for $\nu\times\nu$
almost every $(x,y)$. Thus
\[
\limsup_{n\rightarrow\infty}\limsup_{k\rightarrow\infty}\intop_{X}\intop_{X}F_{n,n}(x,y)+F_{k,k}(x,y)-2F_{k,n}(x,y)\mathrm{d}\nu(x)\mathrm{d}\nu(y)
\]
\begin{equation}
\leq4\eta+2\cdot c_{4}\left(\iintop_{H_{\varepsilon}}\varphi(x,y)\mathrm{d}\nu(x)\mathrm{d}\nu(y)+\nu\times\nu(H_{\varepsilon})\right).\label{eq:limisupliminf equation 1}
\end{equation}
By (\ref{eq:phi =00003Dinfty}) and Fubini`s theorem $\nu\times\nu\left((x,x):x\in X\right)=0.$
Since $\nu\times\nu$ and $\varphi(x,y)\mathrm{d}\nu(x)\mathrm{d}\nu(y)$
are finite measures it follows that
\[
\iintop_{H_{\varepsilon}}\varphi(x,y)\mathrm{d}\nu(x)\mathrm{d}\nu(y)+\nu\times\nu(H_{\varepsilon})
\]
converges to $0$ as $\varepsilon$ approaches $0$ because $r_{\varepsilon}$
goes to $0$ by Notation \ref{nota:r_epsilon }. Taking the limit
as $\eta$ and $\varepsilon$ go to $0$ in (\ref{eq:limisupliminf equation 1})
the statement follows by (\ref{eq:lim manko}).
\end{proof}
\begin{thm}
\label{L2 bounded}Let $\nu$ be a finite Borel measure on $X$ such
that $\nu(X\setminus X_{0})=0$ and $I_{\varphi}(\nu)<\infty$. Assume
that (\ref{eq:phi =00003Dinfty}), (\ref{eq:same size}) and (\ref{eq:lower hitting prob})
hold and there exists $0<\delta<1$ such that (\ref{eq:Kernel restriction}),
(\ref{eq:capacity-independence}) and (\ref{eq:bounded sundivision})
hold. If $\underline{F}(x,y)=\overline{F}(x,y)$ for $\nu\times\nu$
almost every $(x,y)$ then for every Borel set $A\subseteq X$ it
follows that $\mathcal{C}_{k}(\nu)(A)$ converges to a limit $\mu(A)$
in $\mathcal{L}^{2}$, in $\mathcal{L}^{1}$ and in probability and
$E(\mu(A))=E(\mathcal{C}_{k}(\nu)(A))=\nu(A)$.
\end{thm}

\begin{proof}
By applying Lemma \ref{lem:Lim exist prop nontouchable condiition}
and Proposition \ref{prop:L2 limsup=00003Dlimif cauchy} to the measure
$\nu\vert_{A}$ it follows that $\mathcal{C}_{k}(\nu)(A)$ converges
to a limit $\mu(A)$ in $\mathcal{L}^{2}$ and so in $\mathcal{L}^{1}$
and in probability. Thus by (\ref{eq:martingal expectation inequality})
we have that $E(\mu(A))=E(\mathcal{C}_{k}(\nu)(A))=\nu(A)$.
\end{proof}

\subsection{Limit in the presence of a martingale filtration}
\begin{thm}
\label{thm:martingal limit}Let $\nu$ be a finite Borel measure on
$X$ such that $\nu(X\setminus X_{0})=0$. Assume that for a Borel
set $A\subseteq X$ there exists a filtration $\mathcal{F}_{k}$ such
that $\mathcal{C}_{k}(\nu)(A)$ is a martingale with respect to the
filtration $\mathcal{F}_{k}$. Then $\mathcal{C}_{k}(\nu)(A)$ converges
to a limit $\mu(A)$ almost surely and $E(\mu(A))\leq E(\mathcal{C}_{k}(\nu)(A))=\nu(A)$.

If additionally $I_{\varphi}(\nu)<\infty$, (\ref{eq:lower hitting prob})
holds and there exists $\delta>0$ such that (\ref{eq:Kernel restriction}),
(\ref{eq:capacity-independence}) and (\ref{eq:bounded sundivision})
hold then $\mathcal{C}_{k}(\nu)(A)$ converges in $\mathcal{L}^{2}$,
in $\mathcal{L}^{1}$ and $E(\mu(A))=E(\mathcal{C}_{k}(\nu)(A))=\nu(A)$.
\end{thm}

\begin{proof}
Since $\mathcal{C}_{k}(\nu)(A)$ is a nonnegative martingale it converges
almost surely to a random limit $\mu(A)$ and $E(\mu(A))\leq E(\mathcal{C}_{k}(\nu)(A))=\nu(A)$
by the nonnegative martingale limit theorem \cite[Theorem 5.2.9]{Durrett}.
If $I_{\varphi}(\nu)<\infty$ then $\mathcal{C}_{k}(\nu)(A)$ is $\mathcal{L}^{2}$-bounded
by Proposition \ref{prop:2bounded martingale}, hence converges in
$\mathcal{L}^{2}$ and in $\mathcal{L}^{1}$ by \cite[Theorem 5.4.5]{Durrett}.
Thus by (\ref{eq:martingal expectation inequality}) we have that
$E(\mu(A))=E(\mathcal{C}_{k}(\nu)(A))=\nu(A)$.
\end{proof}

\section{Existence of the conditional measure\label{sec:Existence-of-the}}

We are now prepared to prove the existence of the conditional measure
$\mathcal{C}(\nu)$ when $F(x,y)=\underline{F}(x,y)=\overline{F}(x,y)$.

\begin{thm}
\label{thm:Special existence of conditional expectation finite+vague}Let
$\nu$ be a finite, Borel measure on $X$ or let $X$ be locally compact
and $\nu$ be a locally finite, Borel measure on $X$. Assume that
$\nu(X\setminus X_{0})=0$. Assume that if $C_{\varphi}(D)=0$ for
some compact set $D\subseteq X_{0}$ then $B\cap D=\emptyset$ almost
surely. Assume that (\ref{eq:phi =00003Dinfty}), (\ref{eq:same size})
and (\ref{eq:lower hitting prob}) hold and there exists $0<\delta<1$
such that (\ref{eq:Kernel restriction}), (\ref{eq:capacity-independence})
and (\ref{eq:bounded sundivision}) hold. Assume that at least one
of the following conditions hold:

A.) $\underline{F}(x,y)=\overline{F}(x,y)$ for $\nu\times\nu$ almost
every $(x,y)$,

B.) for every Borel set $A\subseteq X$ with $\nu(A)<\infty$ there
exists a filtration $\mathcal{F}_{k}$ such that $\mathcal{C}_{k}(\nu)(A)$
is a martingale with respect to the filtration $\mathcal{F}_{k}$.

\noindent Then the conditional measure $\mathcal{C}(\nu)$ of $\nu$
on $B$ exists with respect to $\mathcal{Q}_{k}$ ($k\geq1$) with
regularity kernel $\varphi$. Moreover, if $A\subseteq X$ is a Borel
set such that $\nu(A)<\infty$ and $I_{\varphi}(\nu\vert_{A})<\infty$
then $\mathcal{C}_{k}(\nu)(A)$ converges to $\mathcal{C}(\nu)(A)$
in $\mathcal{L}^{2}$.
\end{thm}

\begin{proof}
By Theorem \ref{L2 bounded} and Theorem \ref{thm:martingal limit}
if $D\subseteq X$ is a Borel  set such that $I_{\varphi}(\nu\vert_{D})<\infty$
and $\nu(D)<\infty$ then $\mathcal{C}_{k}(\nu)(D)$ converges in
$\mathcal{L}^{2}$ and in $\mathcal{L}^{1}$ to a random variable
$\mu(D)$. By Theorem \ref{prop:dieing singular part} and Remark
\ref{rem:dieing only on X0} we have that $\mathcal{C}_{k}(\nu_{\perp})(D)$
converges to $0$ in probability for every compact set $D\subseteq X$.
Thus it follows by Theorem \ref{thm:Gen exist both} that the conditional
measure $\mathcal{C}(\nu)$ of $\nu$ on $B$ exists with respect
to $\mathcal{Q}_{k}$ ($k\geq1$) with regularity kernel $\varphi$.

Let $A\subseteq X$ be a Borel set such that $\nu(A)<\infty$ and
$I_{\varphi}(\nu\vert_{A})<\infty$. As we established at the beginning
of the proof $\mathcal{C}_{k}(\nu)(A)$ converges to $\mu(A)$ in
$\mathcal{L}^{2}$. By Property \textit{iv.)} of Definition \ref{def:def of cond meas}
it follows that $\mu(A)=\mathcal{C}(\nu)(A)$ almost surely. Hence
it follows that $\mathcal{C}_{k}(\nu)(A)$ converges to $\mathcal{C}(\nu)(A)$
in $\mathcal{L}^{2}$.
\end{proof}
\begin{rem}
\label{rem:regular measure no need 0 cap}If we assume that $\nu_{R}=\nu$
in Theorem \ref{thm:Special existence of conditional expectation finite+vague}
then for the conclusion to hold we do not even need the assumption
that if $C_{\varphi}(D)=0$ for some compact set $D\subseteq X_{0}$
then $B\cap D=\emptyset$ almost surely. It is because we only use
this assumption to ensure that $\mathcal{C}_{k}(\nu_{\perp})(D)$
converges to $0$ in probability. However, if $\nu_{R}=\nu$ then
$\mathcal{C}_{k}(\nu_{\perp})(X)=0$ for every $k$.
\end{rem}

\section{Double integration\label{sec:Double-integration}}

In this section we prove the double integration formula (\ref{eq:double int form intro}).

\begin{prop}
\label{prop:Double integ BOXES upper}Assume that (\ref{eq:phi =00003Dinfty}),
(\ref{eq:same size}) and (\ref{eq:lower hitting prob}) hold and
there exists $0<\delta<1$ such that (\ref{eq:Kernel restriction}),
(\ref{eq:capacity-independence}) and (\ref{eq:bounded sundivision})
hold. Let $\nu_{1}$ and $\nu_{2}$ be finite Borel measures on $X$
with $I_{\varphi}(\nu_{1}+\nu_{2})<\infty$. Assume that the conditional
measure $\mathcal{C}(\nu_{i})$ of $\nu_{i}$ on $B$ exist with respect
to $\mathcal{Q}_{k}$ ($k\geq1$) with regularity kernel $\varphi$
for $i=1,2$. Assume that if $A\subseteq X$ is a Borel set then $\mathcal{C}_{k}(\nu_{i})(A)$
converges to $\mathcal{C}(\nu_{i})(A)$ in $\mathcal{L}^{2}$ for
$i=1,2$. Then
\begin{equation}
\intop_{A_{2}}\left(\intop_{A_{1}}\underline{F}(x,y)\mathrm{d}\nu_{1}(x)\right)\mathrm{d}\nu_{2}(y)\leq E\left(\mathcal{C}(\nu_{1})(A_{1})\cdot\mathcal{C}(\nu_{2})(A_{2})\right)\leq\intop_{A_{2}}\left(\intop_{A_{1}}\overline{F}(x,y)\mathrm{d}\nu_{1}(x)\right)\mathrm{d}\nu_{2}(y)\label{eq:Double integral BOX W}
\end{equation}
\[
\leq c\intop_{A_{2}}\left(\intop_{A_{1}}\varphi(x,y)\mathrm{d}\nu_{1}(x)\right)\mathrm{d}\nu_{2}(y)<\infty
\]
for every Borel sets $A_{1},A_{2}\subseteq X$.
\end{prop}

\begin{proof}
Without the loss of generality we can assume that $A_{1}$ and $A_{2}$
are both $X$ by restricting the measures to $\nu_{1}\vert_{A_{1}}$
and $\nu_{2}\vert_{A_{2}}$. We have that $\mathcal{C}_{k}(\nu_{1})(X)$
converges to $\mathcal{C}(\nu_{1})(X)$ in in $\mathcal{L}^{2}$ and
$\mathcal{C}_{k}(\nu_{2})(X)$ converges to $\mathcal{C}(\nu_{2})(X)$
in $\mathcal{L}^{2}$. Thus $\mathcal{C}_{k}(\nu_{1})(X)\cdot\mathcal{C}_{k}(\nu_{2})(X)$
converges to $\mathcal{C}(\nu_{1})(X)\cdot\mathcal{C}(\nu_{2})(X)$
in $\mathcal{L}^{1}$ and in particular,
\begin{equation}
E\left(\mathcal{C}(\nu_{1})(X)\cdot\mathcal{C}(\nu_{2})(X)\right)=\lim_{k\rightarrow\infty}E\left(\mathcal{C}_{k}(\nu_{1})(X)\cdot\mathcal{C}_{k}(\nu_{2})(X)\right)=\lim_{k\rightarrow\infty}\intop_{X}\intop_{X}F_{k,k}(x,y)\mathrm{d}\nu_{1}(x)\mathrm{d}\nu_{2}(y)\label{eq:L_1 eq biz}
\end{equation}
by Lemma \ref{lem:Fkn lem sequance exact}.

Let $\varepsilon,\eta>0$ be fixed. Then, by Lemma \ref{lem:nondiagonal integral bound}
and by Lemma \ref{lem:A_epsilon reminder}, for large enough $k\in\mathbb{N}$
\[
\intop_{X}\intop_{X}F_{k,k}(x,y)\mathrm{d}\nu_{1}(x)\mathrm{d}\nu_{2}(y)\leq\eta+\iintop_{A_{\varepsilon}}\overline{F}(x,y)\mathrm{d}\nu_{1}(x)\mathrm{d}\nu_{2}(y)+c_{4}\iintop_{H_{\varepsilon}}\varphi(x,y)\mathrm{d}\nu(x)\mathrm{d}\nu(y)+c_{4}\cdot\nu\times\nu(H_{\varepsilon})
\]
\begin{equation}
\leq\eta+\intop_{X}\intop_{X}\overline{F}(x,y)\mathrm{d}\nu_{1}(x)\mathrm{d}\nu_{2}(y)+c_{4}\iintop_{H_{\varepsilon}}\varphi(x,y)\mathrm{d}\nu(x)\mathrm{d}\nu(y)+c_{4}\cdot\nu\times\nu(H_{\varepsilon}).\label{eq:BOX eq 1}
\end{equation}
Let $D=\left\{ (x,x):x\in X\right\} $. Then $\nu(D)=0$ by (\ref{eq:phi =00003Dinfty})
and Fubini`s theorem. Hence by the fact that $\intop_{X}\intop_{X}\varphi(x,y)\mathrm{d}\nu(x)\mathrm{d}\nu(y)<\infty$
it follows that
\begin{equation}
\lim_{\varepsilon\rightarrow0}\iintop_{H_{\varepsilon}}\varphi(x,y)\mathrm{d}\nu(x)\mathrm{d}\nu(y)=0\label{eq:integral heps lim}
\end{equation}
by Notation \ref{nota:r_epsilon }. Similarly
\begin{equation}
\lim_{\varepsilon\rightarrow0}\nu\times\nu(H_{\varepsilon})=0.\label{eq:nuHepslim}
\end{equation}
Hence by (\ref{eq:L_1 eq biz}), (\ref{eq:BOX eq 1}), (\ref{eq:integral heps lim}),
(\ref{eq:nuHepslim}) and Remark \ref{rem:controlled F}
\begin{equation}
E\left(\mathcal{C}(\nu_{1})(X)\cdot\mathcal{C}(\nu_{2})(X)\right)\leq\intop_{X}\intop_{X}\overline{F}(x,y)\mathrm{d}\nu_{1}(x)\mathrm{d}\nu_{2}(y)\leq c\intop_{X}\intop_{X}\varphi(x,y)\mathrm{d}\nu_{1}(x)\mathrm{d}\nu_{2}(y)<\infty.\label{eq:fo1 box}
\end{equation}
By Lemma \ref{lem:nondiagonal integral bound}, for large enough $k\in\mathbb{N}$
\begin{equation}
-\eta+\iintop_{A_{\varepsilon}}\underline{F}(x,y)\mathrm{d}\nu_{1}(x)\mathrm{d}\nu_{2}(y)\leq\iintop_{A_{\varepsilon}}F_{k,k}(x,y)\mathrm{d}\nu_{1}(x)\mathrm{d}\nu_{2}(y)\leq\intop_{X}\intop_{X}F_{k,k}(x,y)\mathrm{d}\nu_{1}(x)\mathrm{d}\nu_{2}(y).\label{eq:lower box eq}
\end{equation}
By Remark \ref{rem:controlled F} we have that $\underline{F}(x,y)\leq c\cdot\varphi(x,y)$
for $\nu\times\nu$ almost every $(x,y)$. Thus similarly to (\ref{eq:integral heps lim})
\begin{equation}
\lim_{\varepsilon\rightarrow0}\iintop_{A_{\varepsilon}}\underline{F}(x,y)\mathrm{d}\nu_{1}(x)\mathrm{d}\nu_{2}(y)=\intop_{X}\intop_{X}\underline{F}(x,y)\mathrm{d}\nu_{1}(x)\mathrm{d}\nu_{2}(y).\label{eq:lower int heps}
\end{equation}
Then it follows from (\ref{eq:L_1 eq biz}), (\ref{eq:lower box eq})
and (\ref{eq:lower int heps}) that
\begin{equation}
\intop_{X}\intop_{X}\underline{F}(x,y)\mathrm{d}\nu_{1}(x)\mathrm{d}\nu_{2}(y)\leq E\left(\mathcal{C}(\nu_{1})(X)\cdot\mathcal{C}(\nu_{2})(X)\right).\label{eq:fo2 box}
\end{equation}
So the statement follows from (\ref{eq:fo1 box}) and (\ref{eq:fo2 box}).
\end{proof}
\begin{prop}
\label{thm:Double integration}Assume that (\ref{eq:phi =00003Dinfty}),
(\ref{eq:same size}) and (\ref{eq:lower hitting prob}) hold and
there exists $0<\delta<1$ such that (\ref{eq:Kernel restriction}),
(\ref{eq:capacity-independence}) and (\ref{eq:bounded sundivision})
hold. Let $\nu_{1}$ and $\nu_{2}$ be finite Borel measures on $X$
with $I_{\varphi}(\nu_{1}+\nu_{2})<\infty$. Assume that the conditional
measure $\mathcal{C}(\nu_{i})$ of $\nu_{i}$ on $B$ exist with respect
to $\mathcal{Q}_{k}$ ($k\geq1$) with regularity kernel $\varphi$
for $i=1,2$. Assume that if $A\subseteq X$ is a Borel set then $\mathcal{C}_{k}(\nu_{i})(A)$
converges to $\mathcal{C}(\nu_{i})(A)$ in $\mathcal{L}^{2}$ for
$i=1,2$. Let $f:X\times X\longrightarrow\mathbb{R}$ be a nonnegative
Borel function. Then
\[
\intop\intop\underline{F}(x,y)f(x,y)\mathrm{d}\nu_{1}(x)\mathrm{d}\nu_{2}(y)\leq E\left(\int\int f(x,y)\mathrm{d}\mathcal{C}(\nu_{1})(x)\mathrm{d}\mathcal{C}(\nu_{2})(y)\right)
\]
\begin{equation}
\leq\intop\intop\overline{F}(x,y)f(x,y)\mathrm{d}\nu_{1}(x)\mathrm{d}\nu_{2}(y).\label{eq: E(mu^2)}
\end{equation}
\end{prop}

\begin{proof}
It follows from Proposition \ref{prop:Double integ BOXES upper} that
(\ref{eq: E(mu^2)}) holds for functions of the form $f(x,y)=\chi_{A_{1}}(x)\cdot\chi_{A_{2}}(y)$
for Borel sets $A_{1},A_{2}\subseteq X$. Hence, by the fact that
the sets of the form $A_{1}\times A_{2}$ form a semi-ring generating
the Borel $\sigma$-algebra of $X\times X$ we can deduce that (\ref{eq: E(mu^2)})
holds for $f(x,y)=I_{A}(x,y)$ for Borel sets $A\subseteq X\times X$
by Proposition \ref{lem:charateodory ineq}. It follows that (\ref{eq: E(mu^2)})
holds for non-negative simple functions on $X\times X$ and so we
can deduce (\ref{eq: E(mu^2)}) for every nonnegative Borel function
on $X\times X$ using the monotone convergence theorem.
\end{proof}
\begin{lem}
\label{lem:kereszt int sum}Let $\nu$ and $\tau$ be locally finite
Borel measures on $X$. Then there exist a sequence of finite measures
$\left\{ \nu_{i}\right\} _{i=1}^{\infty}$ and another sequence of
finite measures $\left\{ \tau_{i}\right\} _{i=1}^{\infty}$ such that
$\nu_{R}=\sum_{i=1}^{\infty}\nu_{i}$, $\tau_{R}=\sum_{i=1}^{\infty}\tau_{i}$
and $I_{\varphi}(\nu_{i}+\tau_{i})<\infty$ for every $i,j\in\mathbb{N}$.
\end{lem}

\begin{proof}
By the Lebesgue decomposition \cite[Theorem A.4.5]{Durrett} there
exist a nonnegative Borel function $g$ and a locally finite Borel
measure $\nu_{s}$ such that
\[
\nu_{R}(A)=\nu_{s}(A)+\intop_{A}g(x)\mathrm{d}\tau_{R}(x)
\]
and $\nu_{s}$ is singular to $\tau_{R}$. Let $G\subseteq X$ be
a Borel set such that $\tau_{R}(X\setminus G)=0$ and $\nu_{s}(G)=0$.
Since $\nu_{R}$ is locally finite and $X$ is separable metric space
it follows that $g(x)<\infty$ for $\tau_{R}$ almost all $x$. Let
$G_{i}=\left\{ x\in G:i-1\leq g(x)<i\right\} $ for every $i\in\mathbb{N}$.
By Proposition \ref{decomposition} for every $i\in\mathbb{N}$ we
can find a sequence $\left\{ E_{i,j}\right\} _{j=1}^{\infty}$ of
disjoint Borel subsets of $G_{i}$ such that $I_{\varphi}(\tau_{R}\vert_{E_{i,j}})<\infty$
for every $j\in\mathbb{N}$ and $\tau_{R}\vert_{G_{i}}=\sum_{j=1}^{\infty}\tau_{R}\vert_{E_{i,j}}$.
By Lemma \ref{lem:countab exhaustion} and Remark \ref{rem:countab exhaust vague}
we can further assume that $\left\{ E_{i,j}\right\} _{j=1}^{\infty}$
is a collection of disjoint compact sets and so $\tau_{R}(E_{i,j})<\infty$
for every $j\in\mathbb{N}$. Similarly, we can also find a collection
of disjoint compact sets $\left\{ A_{j}\right\} _{j=1}^{\infty}$
such that $I_{\varphi}(\nu_{s}\vert_{A_{j}})<\infty$, $\nu_{s}(A_{j})<\infty$,
$A_{j}\subseteq X\setminus G$ for every $j\in\mathbb{N}$ and $\nu_{s}=\sum_{j=1}^{\infty}\nu_{s}\vert_{A_{j}}$.

We define the following decomposition of $\tau_{R}$ and $\nu_{R}$
to obtain the desired decomposition of the statement:
\begin{equation}
\tau_{R}=\sum_{i=1}^{\infty}\sum_{j=1}^{\infty}\tau_{R}\vert_{E_{i,j}}\label{eq:tau dec}
\end{equation}
and 
\begin{equation}
\mathrm{d}\nu_{R}(x)=\sum_{j=1}^{\infty}\mathrm{d}\nu_{s}\vert_{A_{j}}(x)+\sum_{i=1}^{\infty}\sum_{j=1}^{\infty}g(x)\mathrm{d}\tau_{R}\vert_{E_{i,j}}(x).\label{eq:nu dec}
\end{equation}
Then
\[
I_{\varphi}(g(x)\mathrm{d}\tau_{R}\vert_{E_{i,j}}(x)+\mathrm{d}\tau_{R}\vert_{E_{i,j}}(x))\leq(i+1)^{2}I_{\varphi}(\tau_{R}\vert_{E_{i,j}})<\infty
\]
for every $i,j\in\mathbb{N}$. Let $i,j,k,l\in\mathbb{N}$ such that
$E_{i,j}\neq E_{k,l}$. Then $E_{i,j}$ and $E_{k,l}$ are disjoint
compact sets and let $r=\mathrm{dist}(E_{i,j},E_{k,l})$. Then
\[
I_{\varphi}(g(x)\mathrm{d}\tau_{R}\vert_{E_{i,j}}(x)+\mathrm{d}\tau_{R}\vert_{E_{k,l}}(x))
\]
\[
=I_{\varphi}(g(x)\mathrm{d}\tau_{R}\vert_{E_{i,j}}(x))+I_{\varphi}(\tau_{R}\vert_{E_{k,l}})+2\intop_{E_{i,j}}g(x)\left(\intop_{E_{k,l}}\varphi(x,y)\mathrm{d}\tau_{R}(y)\right)\mathrm{d}\tau_{R}(x)
\]
\[
\leq i^{2}I_{\varphi}(\tau_{R}\vert_{E_{i,j}})+I_{\varphi}(\tau_{R}\vert_{E_{k,l}})+2i\varphi(r)\cdot\tau_{R}(E_{i,j})\cdot\tau_{R}(E_{k,l})<\infty.
\]
Finally if $j,k,l\in\mathbb{N}$ then $A_{j}$ and $E_{k,l}$ are
disjoint compact sets and let $r=\mathrm{dist}(A_{j},E_{k,l})$. Then
\[
I_{\varphi}(\nu_{s}\vert_{A_{j}}+\tau_{R}\vert_{E_{k,l}})=I_{\varphi}(\nu_{s}\vert_{A_{j}})+I_{\varphi}(\tau_{R}\vert_{E_{k,l}})+2\intop_{A_{j}}\left(\intop_{E_{k,l}}\varphi(x,y)\mathrm{d}\tau_{R}(y)\right)\mathrm{d}\nu_{s}(x)
\]
\[
\leq I_{\varphi}(\nu_{s}\vert_{A_{j}})+I_{\varphi}(\tau_{R}\vert_{E_{k,l}})+2\varphi(r)\cdot\nu_{s}(A_{j})\cdot\tau_{R}(E_{k,l})<\infty.
\]
Hence decompositions of $\tau_{R}$ and $\nu_{R}$ in (\ref{eq:tau dec})
and (\ref{eq:nu dec}) satisfy the statement.
\end{proof}
\begin{thm}
\label{thm:Double integration-summed}Assume that (\ref{eq:phi =00003Dinfty}),
(\ref{eq:same size}) and (\ref{eq:lower hitting prob}) hold and
there exists $0<\delta<1$ such that (\ref{eq:Kernel restriction}),
(\ref{eq:capacity-independence}) and (\ref{eq:bounded sundivision})
hold. Let either $\nu$ and $\tau$ be finite Borel measures or $X$
be locally compact and $\nu$ and $\tau$ be locally finite Borel
measures. Assume that $\nu(X\setminus X_{0})=0$ and $\tau(X\setminus X_{0})=0$.
Assume that the conditional measure $\mathcal{C}(\nu)$ of $\nu$
and $\mathcal{C}(\tau)$ of $\tau$ on $B$ exist with respect to
$\mathcal{Q}_{k}$ ($k\geq1$) with regularity kernel $\varphi$.
Assume that if $A\subseteq X$ is a Borel set such that $\nu(A)<\infty$
and $I_{\varphi}(\nu\vert_{A})<\infty$ then $\mathcal{C}_{k}(\nu)(A)$
converges to $\mathcal{C}(\nu)(A)$ in $\mathcal{L}^{2}$ and if $A\subseteq X$
is a Borel set such that $\tau(A)<\infty$ and $I_{\varphi}(\tau\vert_{A})<\infty$
then $\mathcal{C}_{k}(\tau)(A)$ converges to $\mathcal{C}(\tau)(A)$
in $\mathcal{L}^{2}$. Let $f:X\times X\longrightarrow\mathbb{R}$
be a nonnegative Borel function. Then
\[
\intop\intop\underline{F}(x,y)f(x,y)\mathrm{d}\nu_{R}(x)\mathrm{d}\tau_{R}(y)\leq E\left(\int\int f(x,y)\mathrm{d}\mathcal{C}(\nu)(x)\mathrm{d}\mathcal{C}(\tau)(y)\right)
\]
\begin{equation}
\leq\intop\intop\overline{F}(x,y)f(x,y)\mathrm{d}\nu_{R}(x)\mathrm{d}\tau_{R}(y).\label{eq: E(mu^2)-summed}
\end{equation}
\end{thm}

\begin{proof}
By Lemma \ref{lem:kereszt int sum} we can decompose $\nu_{R}$ and
$\tau_{R}$ into sum of finite measures such that $\nu_{R}=\sum_{i=1}^{\infty}\nu_{i}$,
$\tau_{R}=\sum_{i=1}^{\infty}\tau_{i}$ and $I_{\varphi}(\nu_{i}+\tau_{i})<\infty$
for every $i,j\in\mathbb{N}$. By Theorem \ref{thm:Double integration}
\[
\intop\intop\underline{F}(x,y)f(x,y)\mathrm{d}\nu_{i}(x)\mathrm{d}\tau_{j}(y)\leq E\left(\int\int f(x,y)\mathrm{d}\mathcal{C}(\nu_{i})(x)\mathrm{d}\mathcal{C}(\tau_{j})(y)\right)
\]
\[
\leq\intop\intop\overline{F}(x,y)f(x,y)\mathrm{d}\nu_{i}(x)\mathrm{d}\tau_{j}(y)
\]
for every $i,j\in\mathbb{N}$. By summing over all $i,j\in\mathbb{N}$
the statement follows by Property \textit{vii.), viii.)} and \textit{ix.)}
in Definition \ref{def:def of cond meas}.
\end{proof}
\selectlanguage{english}%
\noindent \textit{Proof of Theorem }\foreignlanguage{british}{\ref{cor:Double integral when W exists-summed}.
By Remark \ref{rem:controlled F} $\intop_{X}\intop_{X}\varphi(x,y)\left|f(x,y)\right|\mathrm{d}\nu(x)\mathrm{d}\tau(y)<\infty$
implies that $\intop F(x,y)\left|f(x,y)\right|\mathrm{d}\nu_{R}(x)\mathrm{d}\tau_{R}(y)<\infty$.}

\selectlanguage{british}%
The conditional measure $\mathcal{C}(\nu)$ of $\nu$ and $\mathcal{C}(\tau)$
of $\tau$ on $B$ exist with respect to $\mathcal{Q}_{k}$ ($k\geq1$)
with regularity kernel $\varphi$ by Theorem \ref{thm:Special existence of conditional expectation finite+vague}.
The conditions of Theorem \ref{thm:Double integration-summed} are
satisfied by Theorem \ref{thm:Special existence of conditional expectation finite+vague}.
The statement follows by applying Theorem \ref{thm:Double integration-summed}
to $f^{+}$ and $f^{-}$.$\hfill\square$

\section{\label{sec:increaasing cond meas}Conditional measure on an increasing
union}

Our aim in this section is to establish the extension of the conditional
measure with respect to an increasing union of $\mathcal{Q}_{k}^{i}$.

For every $i\in\mathbb{N}$ let $\mathcal{Q}_{k}^{i}$ be a sequence
of countable families of Borel subsets of $X$ for $k\geq n_{i}$,
for some $n_{i}\in\mathbb{N}$, such that $Q\cap S=\emptyset$ for
$Q,S\in\mathcal{Q}_{k}^{i}$, for all $k\in\mathbb{N}$. Assume that
(\ref{eq:diameter_goes_to0}), (\ref{eq:unique-subset}) and (\ref{eq:psoitivity of probability})
hold. Assume further that if $i<j$ and $k\geq\max\{n_{i},n_{j}\}$
then 
\begin{equation}
\mathcal{Q}_{k}^{i}\subseteq\mathcal{Q}_{k}^{j}.\label{eq:tartalmazo web}
\end{equation}
See Example \ref{exa:example Q_k^i}. Let $X_{0}^{i}=\bigcap_{k=n_{i}}^{\infty}(\bigcup_{Q\in\mathcal{Q}_{k}^{i}}Q)$
and let $X_{0}^{\infty}=\bigcup_{i=1}^{\infty}X_{0}^{i}$, note that
it is an increasing union. Let either $\nu$ be a finite Borel measure
on $X$ or $X$ be locally compact and $\nu$ be a locally finite
Borel measure on $X$. Assume that
\begin{equation}
\nu(X\setminus X_{0}^{\infty})=0\label{eq:kulonseg 0}
\end{equation}
and that the conditional measure $\mathcal{C}^{i}(\nu\vert_{X_{0}^{i}})$
of $\nu\vert_{X_{0}^{i}}$ on $B$ exists with respect to $\mathcal{Q}_{k}^{i}$
($k\geq n_{i}$) with regularity kernel $\varphi$. Let
\begin{equation}
\mu^{\nu}:=\sum_{i=1}^{\infty}\mathcal{C}^{i}(\nu\vert_{X_{0}^{i}\setminus X_{0}^{i-1}})\label{eq:def sum cond meas}
\end{equation}
with the convention that $X_{0}^{0}=\emptyset$.
\begin{prop}
\label{thm:increasing conditional measure}Let $X_{0}^{i}$, $X_{0}^{\infty}$,
$\mathcal{Q}_{k}^{i}$, $\nu$, $\mathcal{C}^{i}(\nu\vert_{X_{0}^{i}})$
and $\mu^{\nu}$ be as above. Then the following hold:

1.) $E(\intop_{X}f(x)\mathrm{d}\mu^{\nu}(x))=\intop_{X}f(x)\mathrm{d}\nu_{R}(x)$
for every $f:X\longrightarrow\mathbb{R}$ Borel measurable function
such that $\intop_{X}\left|f(x)\right|\mathrm{d}\nu(x)<\infty$,

2.) $E(\mu^{\nu}(A))=\nu_{R}(A)\leq\nu(A)$ for every Borel set $A\subseteq X$
with $\nu(A)<\infty$,

3.) $\mu^{\nu_{\bot}}=0$ almost surely,

4.) $\mu^{\nu}=\mu^{\nu_{R}}$ almost surely,

5.) if $f:X\longrightarrow\mathbb{R}$ is a nonnegative Borel function
such that $\intop_{X}f(x)\mathrm{d}\nu(x)<\infty$ then $\mu^{f(x)\mathrm{d}\nu(x)}=f(x)\mathrm{d}\mu^{\nu}(x)$
almost surely,

6.) $\mathrm{supp}\mu^{\nu}\subseteq\mathrm{supp}\nu\cap B$ almost
surely,
\end{prop}

\begin{proof}
Let $f:X\longrightarrow\mathbb{R}$ be a nonnegative Borel measurable
function such that $\intop_{X}\left|f(x)\right|\mathrm{d}\nu(x)<\infty$,
then by Property \textit{ii.)} of Definition \ref{def:def of cond meas}
we have that $E(\intop_{X}f(x)\mathrm{d}\mathcal{C}^{i}(\nu\vert_{X_{0}^{i}\setminus X_{0}^{i-1}})(x))=\intop_{X}f(x)\mathrm{d}\nu_{R}\vert_{X_{0}^{i}\setminus X_{0}^{i-1}}(x)$
for every $i\in\mathbb{N}$. Then summing over $i\in\mathbb{N}$ and
by Fubini`s theorem we have that $E(\intop_{X}f(x)\mathrm{d}\mu^{\nu}(x))=\intop_{X}f(x)\mathrm{d}\nu_{R}(x)$,
i.e. \textit{1.)} holds for nonnegative $f$. This implies \textit{1.)}
in the general case since we know it for $f^{+}$ and $f^{-}$. Property
\textit{2.)} is a spacial case of property \textit{1.)}.

Property \textit{3.), 4.), 5.)} and \textit{6.)} follows by the fact
that the analogous properties of the conditional measure in Definition
\ref{def:def of cond meas} hold for the summands $\mathcal{C}^{i}(\nu\vert_{X_{0}^{i}\setminus X_{0}^{i-1}})$
(for Property \textit{6.)} note that $B$ is almost surely closed).
\end{proof}
Let $X_{0}^{i}$, $X_{0}^{\infty}$, $\mathcal{Q}_{k}^{i}$, $\nu$,
$\mathcal{C}^{i}(\nu\vert_{X_{0}^{i}})$ and $\mu^{\nu}$ be as above.
Let $\mathcal{Q}_{k}$ be a sequence of countable families of Borel
subsets of $X$ such that $Q\cap S=\emptyset$ for $Q,S\in\mathcal{Q}_{k}$,
for all $k\in\mathbb{N}$. Assume that (\ref{eq:unique-subset}),
(\ref{eq:psoitivity of probability}) hold and assume that
\begin{equation}
\mathcal{Q}_{k}^{i}\subseteq\mathcal{Q}_{k}\label{eq:minden tartalmaz}
\end{equation}
for $k\geq n_{i}$ for every $i$. Let $\mathcal{C}_{k}(\nu)$ be
as in (\ref{eq:C_k def}).

Recall, that $X_{0}=\bigcap_{k=0}^{\infty}(\bigcup_{Q\in\mathcal{Q}_{k}}Q)$.
Note that, if
\begin{equation}
X_{0}=X_{0}^{\infty}\label{eq:X_0=00003Dinfty}
\end{equation}
and $\nu(X\setminus X_{0})=0$ then (\ref{eq:kulonseg 0}) holds.

By (\ref{eq:minden tartalmaz}), Property \textit{ix.)} of Definition
\ref{def:def of cond meas} and (\ref{eq:tartalmazo web}) it follows
that
\begin{equation}
\mathcal{C}_{k}(\nu\vert_{X_{0}^{i}})=\mathcal{C}_{k}^{i}(\nu\vert_{X_{0}^{i}})=\sum_{j=1}^{i}\mathcal{C}_{k}^{i}(\nu\vert_{X_{0}^{j}\setminus X_{0}^{j-1}})=\sum_{j=1}^{i}\mathcal{C}_{k}^{j}(\nu\vert_{X_{0}^{j}\setminus X_{0}^{j-1}}).\label{eq:ni sum Ck}
\end{equation}

\begin{prop}
\label{S(f) convergence of extended measures}Let $X_{0}^{i}$, $X_{0}^{\infty}$,
$\mathcal{Q}_{k}^{i}$, $\mathcal{Q}_{k}$ $\nu$, $\mu^{\nu}$ and
$\mathcal{C}_{k}(\nu)$ be as above, (note that we assume (\ref{eq:kulonseg 0})).
Then $\intop_{X}f(x)\mathrm{d}\mathcal{C}_{k}(\nu)(x)$ converges
to $\int f(x)\mathrm{d}\mu^{\nu}(x)$ in probability for every Borel
measurable function $f:X\longrightarrow\mathbb{R}$ such that $\intop_{X}\left|f(x)\right|\mathrm{d}\nu(x)<\infty$.
\end{prop}

\begin{proof}
Let $f$ be as in the statement. Let $\varepsilon>0$ be fixed. Since
$\intop_{X}\left|f(x)\right|\mathrm{d}\nu(x)<\infty$ and we assume
(\ref{eq:kulonseg 0}), we can find $i\in\mathbb{N}$ such that $\intop_{X\setminus X_{0}^{i}}\left|f(x)\right|\mathrm{d}\nu(x)<\varepsilon$.
By (\ref{eq:martingal expectation inequality}) and (\ref{eq:ni sum Ck})
it follows that
\begin{equation}
E\left|\int f(x)\mathrm{d}\mathcal{C}_{k}(\nu)(x)-\int f(x)\mathrm{d}\mathcal{C}_{k}^{i}(\nu\vert_{X_{0}^{i}})(x)\right|\leq E\left(\int\left|f(x)\right|\mathrm{d}\mathcal{C}_{k}(\nu\vert_{X\setminus X_{0}^{i}})(x)\right)=\intop_{X\setminus X_{0}^{i}}\left|f(x)\right|\mathrm{d}\nu(x)<\varepsilon.\label{eq:segg1}
\end{equation}
Since the conditional measure $\mathcal{C}^{j}(\nu\vert_{X_{0}^{j}})$
of $\nu\vert_{X_{0}^{j}}$ on $B$ exists with respect to $\mathcal{Q}_{k}^{j}$
($k\geq n_{i}$) with regularity kernel $\varphi$ it follows by Property
\textit{ii.), x.)} and \textit{x{*}.)} of Definition \ref{def:def of cond meas}
that
\[
E\left|\int f(x)\mathrm{d}\mathcal{C}^{j}(\nu\vert_{X_{0}^{j}\setminus X_{0}^{j-1}})(x)\right|\leq\intop_{X_{0}^{j}\setminus X_{0}^{j-1}}\left|f(x)\right|\mathrm{d}\nu(x).
\]
Thus by (\ref{eq:def sum cond meas}) and (\ref{eq:ni sum Ck}) it
follows that
\[
E\left|\int f(x)\mathrm{d}\mu^{\nu}(x)-\int f(x)\mathrm{d}\mathcal{C}^{i}(\nu\vert_{X_{0}^{i}})(x)\right|=E\left|\sum_{j=i+1}^{\infty}\int f(x)\mathrm{d}\mathcal{C}^{j}(\nu\vert_{X_{0}^{j}\setminus X_{0}^{j-1}})(x)\right|
\]
\begin{equation}
\leq\intop_{X\setminus X_{0}^{i}}\left|f(x)\right|\mathrm{d}\nu(x)<\varepsilon.\label{eq:segg2}
\end{equation}

We have that $\intop_{X}f(x)\mathrm{d}\mathcal{C}_{k}^{i}(\nu\vert_{X_{0}^{i}})(x)$
converges to $\intop_{X}f(x)\mathrm{d}\mathcal{C}^{i}(\nu\vert_{X_{0}^{i}})(x)$
in probability by Property \textit{ii.)} and \textit{iii.)} of Definition
\ref{def:def of cond meas}. Thus using (\ref{eq:segg1}) and (\ref{eq:segg2})
it follows that
\[
\rho\left(\int f(x)\mathrm{d}\mathcal{C}_{k}(\nu)(x),\int f(x)\mathrm{d}\mu^{\nu}(x)\right)
\]
\[
\leq\rho\left(\int f(x)\mathrm{d}\mathcal{C}_{k}(\nu)(x),\int f(x)\mathrm{d}\mathcal{C}_{k}^{i}(\nu\vert_{X_{0}^{i}})(x)\right)+\rho\left(\int f(x)\mathrm{d}\mathcal{C}_{k}^{i}(\nu\vert_{X_{0}^{i}})(x),\int f(x)\mathrm{d}\mathcal{C}^{i}(\nu\vert_{X_{0}^{i}})(x)\right)
\]
\[
+\rho\left(\int f(x)\mathrm{d}\mathcal{C}^{i}(\nu\vert_{X_{0}^{i}})(x),\int f(x)\mathrm{d}\mu^{\nu}(x)\right)
\]
\[
\leq E\left|\int f(x)\mathrm{d}\mathcal{C}_{k}(\nu)(x)-\int f(x)\mathrm{d}\mathcal{C}_{k}^{i}(\nu\vert_{X_{0}^{i}})(x)\right|+\rho\left(\int f(x)\mathrm{d}\mathcal{C}_{k}^{i}(\nu\vert_{X_{0}^{i}})(x),\int f(x)\mathrm{d}\mathcal{C}^{i}(\nu\vert_{X_{0}^{i}})(x)\right)
\]
\[
+E\left|\int f(x)\mathrm{d}\mu^{\nu}(x)-\int f(x)\mathrm{d}\mathcal{C}^{i}(\nu\vert_{X_{0}^{i}})(x)\right|
\]
\[
\leq\varepsilon+\rho\left(\int f(x)\mathrm{d}\mathcal{C}_{k}^{i}(\nu\vert_{X_{0}^{i}})(x),\int f(x)\mathrm{d}\mathcal{C}^{i}(\nu\vert_{X_{0}^{i}})(x)\right)+\varepsilon.
\]
It follows that
\[
\limsup_{k\rightarrow\infty}\rho\left(\int f(x)\mathrm{d}\mathcal{C}_{k}(\nu)(x),\int f(x)\mathrm{d}\mu^{\nu}(x)\right)\leq2\varepsilon
\]
because $\intop_{X}f(x)\mathrm{d}\mathcal{C}_{k}^{i}(\nu\vert_{X_{0}^{i}})(x)$
converges to $\intop_{X}f(x)\mathrm{d}\mathcal{C}^{i}(\nu\vert_{X_{0}^{i}})(x)$
in probability. Since $\varepsilon>0$ can be arbitrary the statement
follows.
\end{proof}
\begin{thm}
\label{thm:extended cond measure increasing}Let $X_{0}^{i}$, $X_{0}^{\infty}$,
$\mathcal{Q}_{k}^{i}$, $\mathcal{Q}_{k}$ $\nu$, $\mu^{\nu}$ and
$\mathcal{C}_{k}(\nu)$ be as above. Then the conditional measure
$\mathcal{C}(\nu)$ of $\nu$ on $B$ exists with respect to $\mathcal{Q}_{k}$
($k\geq1$) with regularity kernel $\varphi$, and $\mathcal{C}(\nu)=\mu^{\nu}$
almost surely.
\end{thm}

\begin{proof}
To show the existence of the conditional measure we wish to apply
Theorem \ref{eq:doub int brown}. Hence we need to show that $\mathcal{C}_{k}(\nu)(D)$
converges in $\mathcal{L}^{1}$ for every compact set $D\subseteq X_{0}$
with $I_{\varphi}(\nu\vert_{D})<\infty$ and $\mathcal{C}_{k}(\nu_{\perp})(D)$
converges to $0$ in probability for every compact set $D\subseteq X$.

Let $D\subseteq X_{0}$ be a compact set such that $I_{\varphi}(\nu\vert_{D})<\infty$
which implies that $\nu_{R}(D)=\nu(D)$. Then $\mathcal{C}_{k}(\nu)(D)$
converges to $\mu^{\nu}(D)$ in probability by Proposition \ref{S(f) convergence of extended measures}.
By Property \textit{2.)} of Theorem \ref{thm:increasing conditional measure}
it follows that $E(\mu^{\nu}(D))=\nu_{R}(D)=\nu(D)$. It follows from
(\ref{eq:martingal expectation inequality}) that $E(\mathcal{C}_{k}(\nu)(D))=\nu(D)$.
Thus $\mathcal{C}_{k}(\nu)(D)$ converges to $\mu^{\nu}(D)$ in $\mathcal{L}^{1}$
by Lemma \ref{lem:Scheffe}.

Let now $D\subseteq X$ be an arbitrary compact subset. Then $\mathcal{C}_{k}(\nu_{\perp})(D)$
converges to $\mu^{\nu_{\perp}}(D)$ in probability by Proposition
\ref{S(f) convergence of extended measures}. On the other hand by
Property \textit{3.)} of Theorem \ref{thm:increasing conditional measure}
it follows that $\mu^{\nu_{\perp}}(D)=0$ almost surely. Thus $\mathcal{C}_{k}(\nu_{\perp})(D)$
converges to $0$ in probability.

So we can conclude by the application of Theorem \ref{thm:Gen exist both}
that the conditional measure $\mathcal{C}(\nu)$ of $\nu$ on $B$
exist with respect to $\mathcal{Q}_{k}$ ($k\geq1$) with regularity
kernel $\varphi$. It remains to show that $\mu^{\nu}=\mathcal{C}(\nu)$
almost surely.

Assume first that $\nu$ is a finite Borel measure. Then $\mathcal{C}_{k}(\nu)(G)$
converges to $\mathcal{C}(\nu)(G)$ in probability for every open
set $G\subseteq X$ by Property \textit{ii.) }and \textit{iv.)} of
Definition \ref{def:def of cond meas}. On the other hand, $\mathcal{C}_{k}(\nu)(G)$
converges to $\mu^{\nu}(G)$ in probability for every open set $G\subseteq X$
by Proposition \ref{S(f) convergence of extended measures}. Hence
$\mu^{\nu}=\mathcal{C}(\nu)$ almost surely by Lemma \ref{prop:weak unique weak conv}.

Now let $\nu$ be a locally finite Borel measure and let $X$ be locally
compact. Then $\mathcal{C}_{k}(\nu)$ vaguely converges to $\mathcal{C}(\nu)$
in probability by Property \textit{i{*}.) }of Definition \ref{def:def of cond meas}.
On the other hand, $\mathcal{C}_{k}(\nu)$ vaguely converges to $\mu^{\nu}$
in probability by Proposition \ref{S(f) convergence of extended measures}.
Since the limit is unique by Proposition \ref{prop:unique random lim-Vague}
it follows that $\mu^{\nu}=\mathcal{C}(\nu)$ almost surely.
\end{proof}
Let $\overline{F^{i}}(x,y)$ and $\underline{F^{i}}(x,y)$ be defined
as in Definition \ref{def:upper and loweer F(x,y)} for the sequence
$\mathcal{Q}_{k}^{i}$ in place of $\mathcal{Q}_{k}$. If $x\notin X_{0}^{i}$
or $y\notin X_{0}^{i}$ then $\underline{F^{i}}(x,y)=\overline{F^{i}}(x,y)=0$.
If $x,y\in X_{0}^{i}$ and $i<j$ then $\underline{F^{i}}(x,y)=\underline{F^{j}}(x,y)$
and $\overline{F^{i}}(x,y)=\overline{F^{j}}(x,y)$ by (\ref{eq:tartalmazo web}).
Thus
\[
\underline{F^{\infty}}(x,y)=\begin{cases}
\underline{F^{i}}(x,y) & x,y\in X_{0}^{i}\mathrm{\,for\,some}\,i\\
0 & x\notin X_{0}^{\infty}\mathrm{\,or}\,y\notin X_{0}^{\infty}
\end{cases}
\]
and
\[
\overline{F^{\infty}}(x,y)=\begin{cases}
\overline{F^{i}}(x,y) & x,y\in X_{0}^{i}\mathrm{\,for\,some}\,i\\
0 & x\notin X_{0}^{\infty}\mathrm{\,or}\,y\notin X_{0}^{\infty}
\end{cases}
\]
are well-defined on $X\times X$. It is easy to see, by (\ref{eq:minden tartalmaz}),
that $\underline{F^{\infty}}(x,y)=\underline{F}(x,y)$ and $\overline{F^{\infty}}(x,y)=\overline{F}(x,y)$
for $(x,y)\in X_{0}^{\infty}\times X_{0}^{\infty}$. In particular,
if (\ref{eq:X_0=00003Dinfty}) holds then $\underline{F^{\infty}}(x,y)=\underline{F}(x,y)$
and $\overline{F^{\infty}}(x,y)=\overline{F}(x,y)$ for every $(x,y)\in X\times X$
where $\underline{F}$ and $\overline{F}$ are defined in (\ref{def:upper and loweer F(x,y)}).
\begin{thm}
\label{thm:increasing double int}Let $\mathcal{Q}_{k}$ be a sequence
of countable families of Borel subsets of $X$ such that $Q\cap S=\emptyset$
for $Q,S\in\mathcal{Q}_{k}$, for all $k\in\mathbb{N}$. Assume that
if $C_{\varphi}(D)=0$ for some compact set $D\subseteq X_{0}$ then
$B\cap D=\emptyset$ almost surely. For every $i\in\mathbb{N}$ let
$\mathcal{Q}_{k}^{i}$ be a sequence of countable families of Borel
subsets of $X$ for $k\geq n_{i}$, for some $n_{i}\in\mathbb{N}$,
such that $Q\cap S=\emptyset$ for $Q,S\in\mathcal{Q}_{k}^{i}$, for
all $k\in\mathbb{N}$. Assume that (\ref{eq:phi =00003Dinfty}), (\ref{eq:same size}),
(\ref{eq:lower hitting prob}), (\ref{eq:Kernel restriction}), (\ref{eq:capacity-independence})
and (\ref{eq:bounded sundivision}) hold for sufficient constants
$0<\delta^{i}<1$, $M_{\delta}^{i}<\infty$, $0<M^{i}<\infty$, $0<a^{i}<\infty$,
$0<c^{i}<\infty$, $c_{2}^{i}<\infty$, $c_{3}^{i}<\infty$ depending
on $i$. Assume further that if $i<j$ and $k\geq\max\{n_{i},n_{j}\}$
then $\mathcal{Q}_{k}^{i}\subseteq\mathcal{Q}_{k}^{j}.$ Assume that
(\ref{eq:unique-subset}), (\ref{eq:psoitivity of probability}) hold
and assume that $\mathcal{Q}_{k}^{i}\subseteq\mathcal{Q}_{k}$ for
$k\geq n_{i}$ for every $i$. Let $X_{0}^{i}=\bigcap_{k=n_{i}}^{\infty}(\bigcup_{Q\in\mathcal{Q}_{k}^{i}}Q)$
and $X_{0}^{\infty}=\bigcup_{i=1}^{\infty}X_{0}^{i}$. Assume that
$F(x,y)=\underline{F}(x,y)=\overline{F}(x,y)$ for every $x,y\in X$.
Let either $\nu$ and $\tau$ be finite Borel measures on $X$ or
$X$ be locally compact and $\nu$ and $\tau$ be locally finite Borel
measures on $X$. Assume that $\nu(X\setminus X_{0}^{\infty})=0$
and $\tau(X\setminus X_{0}^{\infty})=0$. Then the conditional measure
$\mathcal{C}(\nu)$ of $\nu$ and $\mathcal{C}(\tau)$ of $\tau$
on $B$ exist with respect to $\mathcal{Q}_{k}$ ($k\geq1$) with
regularity kernel $\varphi$ and 
\[
E\left(\int\int f(x,y)\mathrm{d}\mathcal{C}(\nu)(x)\mathrm{d}\mathcal{C}(\tau)(y)\right)=\intop\intop F(x,y)f(x,y)\mathrm{d}\nu_{R}(x)\mathrm{d}\tau_{R}(y)
\]
for every $f:X\times X\longrightarrow\mathbb{R}$ Borel function with
$\intop\intop F(x,y)\left|f(x,y)\right|\mathrm{d}\nu_{R}(x)\mathrm{d}\tau_{R}(y)<\infty$.
\end{thm}

\begin{proof}
We have that $\overline{F^{\infty}}(x,y)=\overline{F^{i}}(x,y)$ and
$\underline{F^{\infty}}(x,y)=\underline{F^{i}}(x,y)$ for $x,y\in X_{0}^{i}$
and that $F(x,y)=\underline{F^{\infty}}(x,y)=\overline{F^{\infty}}(x,y)$
for $x,y\in X_{0}^{\infty}$. It follows by Theorem \ref{thm:Special existence of conditional expectation finite+vague}
that the conditional measure $\mathcal{C}^{i}(\nu\vert_{X_{0}^{i}})$
of $\nu\vert_{X_{0}^{i}}$ on $B$ exists with respect to $\mathcal{Q}_{k}^{i}$
($k\geq n_{i}$) with regularity kernel $\varphi$ for every $i$.
Thus it follows by Theorem \ref{thm:extended cond measure increasing}
that the conditional measure $\mathcal{C}(\nu)$ of $\nu$ on $B$
exists with respect to $\mathcal{Q}_{k}$ ($k\geq1$) with regularity
kernel $\varphi$. Similarly, the conditional measure $\mathcal{C}(\tau)$
of $\tau$ on $B$ exists with respect to $\mathcal{Q}_{k}$ ($k\geq1$)
with regularity kernel $\varphi$.

Let $f:X\times X\longrightarrow\mathbb{R}$ be a Borel function with
$\intop\intop F(x,y)\left|f(x,y)\right|\mathrm{d}\nu(x)\mathrm{d}\tau(y)<\infty$.
By Property \textit{ix.) }of Definition \ref{def:def of cond meas}
it follows that $\mathcal{C}(\nu)=\sum_{i=1}^{\infty}\mathcal{C}(\nu\vert_{X_{0}^{i}\setminus X_{0}^{i-1}})$
and $\mathcal{C}(\tau)=\sum_{i=1}^{\infty}\mathcal{C}(\tau\vert_{X_{0}^{i}\setminus X_{0}^{i-1}})$.
Let $i,j\in\mathbb{N}$, $i\leq j$. By (\ref{eq:tartalmazo web})
we have that $\mathcal{C}_{k}^{i}(\nu\vert_{X_{0}^{i}\setminus X_{0}^{i-1}})=\mathcal{C}_{k}^{j}(\nu\vert_{X_{0}^{i}\setminus X_{0}^{i-1}})$
and so $\mathcal{C}^{i}(\nu\vert_{X_{0}^{i}\setminus X_{0}^{i-1}})=\mathcal{C}^{j}(\nu\vert_{X_{0}^{i}\setminus X_{0}^{i-1}})$
almost surely. By  applying Theorem \ref{cor:Double integral when W exists-summed}
to the measures $\nu\vert_{X_{0}^{i}\setminus X_{0}^{i-1}}$ and $\tau\vert_{X_{0}^{j}\setminus X_{0}^{j-1}}$
with respect to the sequence $\mathcal{Q}_{k}^{j}$ it follows that
\[
E\left(\int\int f(x,y)\mathrm{d}\mathcal{C}^{i}(\nu\vert_{X_{0}^{i}\setminus X_{0}^{i-1}})(x)\mathrm{d}\mathcal{C}^{j}(\tau\vert_{X_{0}^{j}\setminus X_{0}^{j-1}})(y)\right)
\]
\begin{equation}
=\intop_{X_{0}^{i}\setminus X_{0}^{i-1}}\left(\intop_{X_{0}^{j}\setminus X_{0}^{j-1}}F(x,y)f(x,y)\mathrm{d}\tau_{R}(y)\right)\mathrm{d}\nu_{R}(x).\label{eq:segg dub}
\end{equation}
Similarly we can show that (\ref{eq:segg dub}) holds when $i>j$.
By (\ref{eq:minden tartalmaz}) we have that $\mathcal{C}_{k}(\nu\vert_{X_{0}^{i}\setminus X_{0}^{i-1}})=\mathcal{C}_{k}^{i}(\nu\vert_{X_{0}^{i}\setminus X_{0}^{i-1}})$
and so $\mathcal{C}(\nu\vert_{X_{0}^{i}\setminus X_{0}^{i-1}})=\mathcal{C}^{i}(\nu\vert_{X_{0}^{i}\setminus X_{0}^{i-1}})$
and also $\mathcal{C}(\tau\vert_{X_{0}^{j}\setminus X_{0}^{j-1}})=\mathcal{C}^{i}(\tau\vert_{X_{0}^{j}\setminus X_{0}^{j-1}})$.
Then summing (\ref{eq:segg dub}) over all $i,j\in\mathbb{N}$ and
by Fubini`s theorem we can conclude that
\[
E\left(\int\int f(x,y)\mathrm{d}\mathcal{C}(\nu)(x)\mathrm{d}\mathcal{C}(\tau)(y)\right)=\intop\intop F(x,y)f(x,y)\mathrm{d}\nu_{R}(x)\mathrm{d}\tau_{R}(y).
\]
\end{proof}

\section{Probability of non-extinction\label{sec:Probability-of-non-extinction}}

In this section we estimate what is the probability that the conditional
measure $\mathcal{C}(\nu)$ has positive total mass. Among other estimates
we show Theorem \ref{thm:intro non extinction}. The upper bound in
Theorem \ref{thm:intro non extinction} is Corollary \ref{cor:upper non ex cor}
and the lower bound follows from Corollary \ref{thm:nob ext color lower},
Remark \ref{rem:controlled F} and the double integration formula
(\ref{eq:double int form intro}).

Let $K:X\times X\longrightarrow\mathbb{R}$ be a nonnegative Borel
function. Recall the definition of the $K$-energy of a measure (\ref{eq:energy})
and the $K$-capacity of a set (\ref{eq:capacity}).
\begin{defn}
\label{def:lower cap}The \textit{$K$-capacity} of a Borel measure
$\nu$ is
\[
C_{K}(\nu)=\sup\left\{ \frac{1}{I_{K}(\tau)}:\tau\ll\nu,\tau(X)=1\right\} .
\]
\end{defn}

\begin{defn}
\label{def:upper cap}The \textit{upper $K$-capacity} of a Borel
measure $\nu$ is
\[
\overline{C_{K}}(\nu)=\inf\left\{ C_{K}(A):\nu(X\setminus A)=0,A\subseteq X\mathrm{\,is\,Borel}\right\} .
\]
\end{defn}

It easily follows from the definitions that $C_{K}(\nu)\leq\overline{C_{K}}(\nu)$.
In the next theorem $\mathcal{C}(\tau)$ denotes an abstract random
measure that depends on $\tau$ and satisfies certain properties.
Of course, then we apply the theorem to the conditional measure which
satisfies the properties.
\begin{thm}
\label{thm:non-distinction}Let $\nu$ be a finite Borel measure and
$K:X\times X\longrightarrow\mathbb{R}$ be a nonnegative Borel function
such that
\begin{equation}
\tau_{\bot}(X)=\tau_{\varphi\bot}(X)=0\label{eq:taufi}
\end{equation}
whenever $I_{K}(\tau)<\infty$ for a finite Borel measure $\tau$
that satisfies $\tau\ll\nu$. Assume that $\mathcal{C}(\tau)$ is
a random finite Borel measure on $X$ for every finite Borel measure
$\tau$ that satisfies $\tau\ll\nu$ such that the following hold

1.) if $\nu_{i}$ are finite Borel measures and $\gamma_{i}\in\mathbb{R}$
are nonnegative such that $\sum_{i=1}^{\infty}\gamma_{i}\cdot\nu_{i}\leq\nu$
then $\sum_{i=1}^{\infty}\gamma_{i}\cdot\mathcal{C}(\nu_{i})\leq\mathcal{C}(\nu)$
almost surely,

2.) $E(\mathcal{C}(\tau)(X))\geq\tau_{R}(X)=\tau_{\varphi R}(X)$,

3.) $E\left(\mathcal{C}(\tau)(X)^{2}\right)\leq I_{K}(\tau)$.

\noindent Then
\[
C_{K}(\nu)\leq P(\mathcal{C}(\nu)(X)>0).
\]
\end{thm}

\begin{proof}
Let $\tau\ll\nu$ with $\tau(X)=1$ and $I_{K}(\tau)<\infty$ (if
there is no such $\tau$ then the proof is trivial). By \textit{2.)}
and (\ref{eq:taufi}) it follows that
\begin{equation}
E(\mathcal{C}(\tau)(X))\geq\tau_{\varphi R}(X)=\tau(X)=1.\label{eq:etau}
\end{equation}
Let $A_{i}=\left\{ x:i-1\leq\frac{\mathrm{d}\tau}{\mathrm{d}\nu}(x)<i\right\} $
for every $i\in\mathbb{N}$. Then $\nu=\sum_{i=1}^{\infty}\nu\vert_{A_{i}}$
and $\tau=\sum_{i=1}^{\infty}\tau\vert_{A_{i}}$. If $D\subseteq A_{i}$
is a Borel set then
\[
\tau(D)=\int_{D}\frac{\mathrm{d}\tau}{\mathrm{d}\nu}(x)\mathrm{d}\nu(x)\leq\int_{D}i\mathrm{d}\nu(x)=i\cdot\nu(D).
\]
Hence $\sum_{i=1}^{\infty}\frac{1}{i}\mathcal{C}(\tau\vert_{A_{i}})\leq\mathcal{C}(\nu)$
almost surely by \textit{1.)}. It follows that $\mathcal{C}(\nu)(X)=0$
implies $\mathcal{C}(\tau)(X)=0$. Thus
\begin{equation}
P(\mathcal{C}(\nu)(X)>0)\geq P(\mathcal{C}(\tau)(X)>0).\label{eq:nuta}
\end{equation}
Using the second moment method \cite[Lemma 3.23]{Peres-Morters-Broanian motion},
(\ref{eq:etau}) and \textit{3.)} it follows that
\[
P\left(\mathcal{C}(\tau)(X)>0\right)\geq\frac{E\left(\mathcal{C}(\tau)(X)\right)^{2}}{E\left(\mathcal{C}(\tau)(X)^{2}\right)}\geq\frac{1}{I_{K}(\tau)}.
\]
Taking the supremum over all $\tau$ it follows by (\ref{eq:nuta})
that $P(\mathcal{C}(\nu)(X)>0)\geq C_{K}(\nu)$.
\end{proof}
\begin{rem}
\label{rem:phi equiv}Assume that $K:X\times X\longrightarrow\mathbb{R}$
is a nonnegative Borel function and $\nu\left(X\setminus\cup_{i=1}^{\infty}A_{i}\right)=0$
for some Borel sets $A_{i}$ such that there exists $0<a_{i}<\infty$
such that
\begin{equation}
a_{i}\cdot\varphi(x,y)\leq K(x,y)\label{eq:fi eq eq}
\end{equation}
for every $x,y\in A_{i}$. Then $\tau_{\varphi\bot}(A_{i})=0$ whenever
$I_{K}(\tau\vert_{A_{i}})<\infty$ for some finite Borel measure $\tau$
that satisfies $\tau\ll\nu$. If $I_{K}(\tau)<\infty$ for some $\tau$
then $I_{K}(\tau\vert_{A_{i}})\leq I_{K}(\tau)<\infty$ for every
$i\in\mathbb{N}$. Hence $\tau_{\varphi\bot}(A_{i})=0$ for every
$i\in\mathbb{N}$ and so (\ref{eq:taufi}) holds.
\end{rem}

\begin{cor}
\label{thm:nob ext color lower}Let $\nu$ be a finite Borel measure.
Assume that the conditional measure $\mathcal{C}(\nu)$ of $\nu$
on $B$ exists with respect to $\mathcal{Q}_{k}$ ($k\geq1$) with
regularity kernel $\varphi$ and there exists $0<c<\infty$ such that
\begin{equation}
E\left(\mathcal{C}(\tau)(X)^{2}\right)\leq c\cdot I_{\varphi}(\tau)\label{eq:color eq}
\end{equation}
for every finite Borel measure $\tau$ that satisfies $\tau\ll\nu$.
Then
\[
c^{-1}\cdot C_{\varphi}(\nu)\leq P(\mathcal{C}(\nu)(X)>0).
\]
\end{cor}

\begin{proof}
We show that the conditions of Theorem \ref{thm:non-distinction}
hold for $K=c\cdot\varphi$ so we can conclude the statement from
Theorem \ref{thm:non-distinction}. Clearly (\ref{eq:taufi}) holds.
Condition \textit{1.)} of Theorem \ref{thm:non-distinction} holds
by Property \textit{ix.)} and \textit{x.) }of Definition \ref{def:def of cond meas}.
Condition \textit{2.)} of Theorem \ref{thm:non-distinction} holds
by Property \textit{v.) }of Definition \ref{def:def of cond meas}.
Condition \textit{3.)} of Theorem \ref{thm:non-distinction} holds
by (\ref{eq:color eq}).
\end{proof}
Again, in the next theorem $\mathcal{C}(\nu\vert_{A})$ denotes an
abstract random measure that depends on $A\subseteq X$ and satisfies
certain properties and then we apply the theorem to the conditional
measure which satisfies the properties.
\begin{thm}
\label{thm:upper non extinction}Let $K:X\times X\longrightarrow\mathbb{R}$
be a nonnegative Borel function and let $\nu$ be a finite Borel measure
such that $\nu(X\setminus X_{0})=0$ for some Borel set $X_{0}\subseteq X$.
Let $B$ be a random closed set, assume that there exists $b>0$ such
that
\begin{equation}
P(D\cap B\neq\emptyset)\leq b\cdot C_{K}(D)\label{eq:prob}
\end{equation}
for every compact set $D\subseteq X_{0}$. Assume that $\mathcal{C}(\nu\vert_{A})$
is a random finite Borel measure for every Borel set $A\subseteq X$
such that the following hold

1.) if $\left\{ H_{i}\right\} _{i=1}^{\infty}$ is a sequence of disjoint
Borel subsets of $X$ such that $\sum_{i=1}^{\infty}\nu\vert_{H_{i}}=\nu$
then $\mathcal{C}(\nu)=\sum_{i=1}^{\infty}\mathcal{C}(\nu\vert_{H_{i}})$
almost surely,

2.) for a compact set $D\subseteq X_{0}$ conditional on the event
$D\cap B=\emptyset$ we have that $\mathcal{C}(\nu\vert_{D})(X)=0$
almost surely.

\noindent Then
\[
P(\mathcal{C}(\nu)(X)>0)\leq b\cdot\overline{C_{K}}(\nu).
\]
\end{thm}

\begin{proof}
Let $A_{n}\subseteq X_{0}$ be a sequence of Borel sets such that
$\nu(X\setminus A_{n})=0$ and $C_{K}(A_{n})\leq\overline{C_{K}}(\nu)+1/n$.
Then for $A:=\cap_{n=1}^{\infty}A_{n}$ we have that $\nu(X\setminus A)=0$
and $C_{K}(A)=\overline{C_{K}}(\nu)$. Let $D_{n}\subseteq A\subseteq X_{0}$
be an increasing sequence of compact sets such that $\nu(X\setminus D_{n})=\nu(A\setminus D_{n})<1/n$
(we can find such sequence by inner regularity). Conditional on $D_{n}\cap B=\emptyset$
we have that $\mathcal{C}(\nu\vert_{D_{n}})(X)=0$ almost surely by\textit{
2.)}. Hence by (\ref{eq:prob})
\[
P\left(\mathcal{C}(\nu\vert_{D_{n}})(X)>0\right)=P\left(\mathcal{C}(\nu\vert_{D_{n}})(X)>0\mathrm{\,and\,}D_{n}\cap B\neq\emptyset\right)
\]
\[
\leq P\left(D_{n}\cap B\neq\emptyset\right)\leq b\cdot C_{K}(D_{n})\leq b\cdot C_{K}(A)=b\cdot\overline{C_{K}}(\nu).
\]
Let $H_{1}=D_{1}$ and $H_{n}=D_{n}\setminus D_{n-1}$ for $n\geq2$.
Then $\nu\vert_{D_{n}}=\sum_{i=1}^{n}\nu\vert_{H_{i}}$ and $\nu=\sum_{i=1}^{\infty}\nu\vert_{H_{i}}$
since $\nu(X\setminus D_{n})=\nu(A\setminus D_{n})<1/n$. Hence $\mathcal{C}(\nu\vert_{D_{n}})=\sum_{i=1}^{n}\mathcal{C}(\nu\vert_{H_{i}})$
and $\mathcal{C}(\nu)=\sum_{i=1}^{\infty}\mathcal{C}(\nu\vert_{H_{i}})$
by  \textit{1.)}. Thus
\[
P\left(\mathcal{C}(\nu)(X)>0\right)=\lim_{n\rightarrow\infty}P\left(\mathcal{C}(\nu\vert_{D_{n}})(X)>0\right)\leq b\cdot\overline{C_{K}}(\nu).
\]
\end{proof}
\begin{cor}
\label{cor:upper non ex cor}Let $\nu$ be a finite Borel measure
such that $\nu(X\setminus X_{0})=0$ and assume that the conditional
measure $\mathcal{C}(\nu)$ of $\nu$ on $B$ exists with respect
to $\mathcal{Q}_{k}$ ($k\geq1$) with regularity kernel $\varphi$.
Assume that $P(D\cap B\neq\emptyset)\leq b\cdot C_{\varphi}(D)$ for
every compact set $D\subseteq X_{0}$. Then $P(\mathcal{C}(\nu)(X)>0)\leq b\cdot\overline{C_{\varphi}}(\nu)$.
\end{cor}

\begin{proof}
We show that conditions \textit{1.)-2.)} of Theorem \ref{thm:upper non extinction}
hold for $K=\varphi$. Condition \textit{1.)} holds by Property \textit{ix.)}
of Definition \ref{def:def of cond meas}. For a compact set $D\subseteq X_{0}$
conditional on $D\cap B=\emptyset$ we have that $\mathcal{C}(\nu\vert_{D})(X)=0$
by Lemma \ref{lem:compact separation lemma}, hence condition \textit{2.)}
holds.
\end{proof}
\begin{rem}
\label{rem:loc fin non ext prob}Let $X$ be locally compact and $\nu$
be a locally finite Borel measure on $X$. Then the conclusion of
Theorem \ref{thm:non-distinction}, Corollary \ref{thm:nob ext color lower},
Theorem \ref{thm:upper non extinction} and Corollary \ref{cor:upper non ex cor}
hold for $\nu$. The proofs are identical to the proofs of the corresponding
results.
\end{rem}

\begin{thm}
\label{thm:increasing non extinction}Let $\mathcal{Q}_{k}$, $\mathcal{Q}_{k}^{i}$
and $X_{0}^{\infty}$ be as in Theorem \ref{thm:increasing double int}.
Assume that  if $C_{\varphi}(D)=0$ for some compact set $D\subseteq X_{0}$
then $B\cap D=\emptyset$ almost surely. Assume that $F(x,y)=\underline{F}(x,y)=\overline{F}(x,y)$
for every $x,y\in X$ and there exists $b>0$ such that
\[
P(D\cap B\neq\emptyset)\leq b\cdot C_{F}(D)
\]
for every compact set $D\subseteq X_{0}$. Assume that $X_{0}=\cup_{i=1}^{\infty}A_{i}$
such that (\ref{eq:fi eq eq}) holds. Let either $\nu$ be a finite
Borel measure on $X$ or $X$ be locally compact and $\nu$ be a locally
finite Borel measure on $X$. Assume that $\nu(X\setminus X_{0}^{\infty})=0$.
Then the conditional measure $\mathcal{C}(\nu)$ of $\nu$ on $B$
exists with respect to $\mathcal{Q}_{k}$ ($k\geq1$) with regularity
kernel $\varphi$ and 
\[
C_{F}(\nu)\leq P(\mathcal{C}(\nu)(X)>0)\leq b\cdot\overline{C_{F}}(\nu).
\]
\end{thm}

\begin{proof}
Let $K(x,y)=F(x,y)$. Since (\ref{eq:fi eq eq}) holds it follows
that (\ref{eq:taufi}) holds by Remark \ref{rem:phi equiv}. The conditional
measure $\mathcal{C}(\nu)$ of $\nu$ on $B$ exists with respect
to $\mathcal{Q}_{k}$ ($k\geq1$) with regularity kernel $\varphi$
and Condition \textit{3.)} of Theorem \ref{thm:non-distinction} holds
by Theorem \ref{thm:increasing double int}. Conditions \textit{1.)-2.)}
of Theorem \ref{thm:non-distinction} hold by Property \textit{v.),
ix.), x.)} and \textit{x{*}.)} of Definition \ref{def:def of cond meas}.
Thus it follows from Theorem \ref{thm:non-distinction} and Remark
\ref{rem:loc fin non ext prob} that $C_{F}(\nu)\leq P(\mathcal{C}(\nu)(X)>0)$.

To prove the other inequality, we need to check that the conditions
of Theorem \ref{thm:upper non extinction} are satisfied to conclude
from Theorem \ref{thm:upper non extinction} and Remark \ref{rem:loc fin non ext prob}
that $P(\mathcal{C}(\nu)(X)>0)\leq b\cdot\overline{C_{F}}(\nu)$.
The conditions of Theorem \ref{thm:upper non extinction} can be checked
similarly to the proof of Corollary \ref{cor:upper non ex cor}.
\end{proof}

\section{\label{sec:Brownian-path}Conditional measure on the Brownian path}

Throughout this section let $B$ be the Brownian path in $\mathbb{R}^{d}$
($d\geq3$), i.e. the range of a Brownian motion which is started
at the origin unless stated otherwise. We prove Theorem \ref{thm:nonext prob for cond meas Brownian}
and Theorem \ref{thm:Brownian cond measure: Main} in this section.

\subsection{\label{subsec:strong markov}Strong Markov property}

Let $W(t)$ be a Brownian motion started at the origin. Usually the
strong Markov property is stated for stopping times that are almost
surely finite. We reformulate the strong Markov property in the case
when we only assume that the stopping time $T$ is finite with positive
probability. We show that $W(t+T)-W(T)$ is a Brownian motion conditional
on the event $\{T<\infty\}$ and is independent of what happens with
$W$ up till time $T$.

Let $\mathcal{F}^{0}(t)$ be the $\sigma$-algebra generated by the
random variables $\{W(s):0\leq s\leq t\}$, let $\mathcal{F}^{+}(t)=\bigcap_{\varepsilon>0}\mathcal{F}^{0}(t+\varepsilon)$
and for a stopping time $T$ let
\[
\mathcal{F}^{+}(T)=\left\{ A\in\mathcal{A}:A\cap\{T<t\}\in\mathcal{F}^{+}(t)\mathrm{\,for\,}\forall t\geq0\right\} .
\]

\begin{prop}
\label{prop:strong markov prop}Let $\left(W(t):t\geq0\right)$ be
a Brownian motion started at the origin and let $T$ be a stopping
time with respect to the filtration $\left(\mathcal{F}^{+}(t):t\geq0\right)$.
Then
\[
W_{T}(t)=W(t+T)-W(T)
\]
defines a Brownian motion conditional on $\{T<\infty\}$ that is independent
of $\mathcal{F}^{+}(T)$.
\end{prop}

\begin{proof}
Let $a_{n}\geq0$ be a monotone increasing deterministic sequence
such that $\lim_{n\rightarrow\infty}a_{n}=\infty$ and $P(T=a_{n})=0$
for every $n\in\mathbb{N}$. Then let $T_{N}=\min\{T,a_{N}\}$ be
stopping times and $W_{N}(t)=W(t+T_{N})-W(T_{N})$ be random processes
for every $N\in\mathbb{N}$. Then
\[
\{T_{N}<a_{n}\}\cap\{T<t\}=\{T<a_{n}\}\cap\{T<t\}=\{T<\min\{a_{n},t\}\}\in\mathcal{F}^{+}(t)
\]
for every $t\geq0$, i.e. $\{T_{N}<a_{n}\}\in\mathcal{F}^{+}(T)$.
Thus it follows by the strong Markov property (see for example \cite[Theorem 2.14]{Peres-Morters-Broanian motion})
for $T_{N}$ that
\[
P(W_{T}\in A\mid T<a_{N})=P(W_{N}\in A\mid T_{N}<a_{n})=P(W\in A)
\]
for large enough $N$ that $P(T<a_{N})>0$ and for every $A\in\mathcal{A}$.
Taking the limit as $N$ goes to $\infty$ it follows that $P(W_{T}\in A\mid T<\infty)=P(W\in A)$
thus $W_{T}$ is a Brownian motion.

For the proof of the independence let $H\in\mathcal{F}^{+}(T)$ and
$A\in\mathcal{A}$. Then by the strong Markov property for $T_{N}$
\[
P(W_{T}\in A,W\in H,T<a_{N})=P(W_{N}\in A,W\in H,T<a_{N})
\]
\[
=P(W_{N}\in A)P(W\in H,T<a_{N})=P(W_{N}\in A\mid T<a_{N})P(W\in H,T<a_{N})
\]
\[
=P(W_{T}\in A\mid T<a_{N})P(W\in H,T<a_{N}).
\]
Taking the limit as $N$ goes to $\infty$ it follows that
\[
P(W_{T}\in A,W\in H,T<\infty)=P(W_{T}\in A\mid T<\infty)P(W\in H,T<\infty)
\]
then dividing both sides by $P(T<\infty)$ it follows that $\{W_{T}\in A\}$
and $\{W\in H\}$ are independent conditional on $\{T<\infty\}$.
\end{proof}
In the following sections we use the strong Markov property in the
form of Proposition \ref{prop:strong markov prop} for stopping times
that are the first hitting of a compact set, hence the stopping time
is not necessarily finite.

\subsection{\label{subsec:Existance-of-the}Existence of the kernel $F(x,y)$}

This section is dedicated to show that $F(x,y)=\underline{F}(x,y)=\overline{F}(x,y)$
when $B$ is the Brownian path.
\begin{lem}
\label{lem:hitting prob of balls}Let $B$ be a Brownian path in $\mathbb{R}^{d}$
for $d\geq3$. Let $x\in\mathbb{R}^{d}$, $x\neq0$ and $r>0$ such
that $\left\Vert x\right\Vert >r$. Then
\[
P(B\cap B(x,r)\neq\emptyset)=\frac{r^{d-2}}{\left\Vert x\right\Vert ^{d-2}}.
\]
\end{lem}

See \cite[Corollary 3.19]{Peres-Morters-Broanian motion}.
\begin{prop}
\label{prop:ball double hitting}Let $B$ be a Brownian path in $\mathbb{R}^{d}$
for $d\geq3$. Then for $x,y\in\mathbb{R}^{d}\setminus\{0\}$, $x\neq y$
\[
\liminf_{R\rightarrow0}\liminf_{r\rightarrow0}\frac{P\left(B\cap B(x,R)\neq\emptyset,B\cap B(y,r)\neq\emptyset\right)}{P(B\cap B(x,R)\neq\emptyset)\cdot P(B\cap B(y,r)\neq\emptyset)}
\]
\begin{equation}
=\limsup_{R\rightarrow0}\limsup_{r\rightarrow0}\frac{P\left(B\cap B(x,R)\neq\emptyset,B\cap B(y,r)\neq\emptyset\right)}{P(B\cap B(x,R)\neq\emptyset)\cdot P(B\cap B(y,r)\neq\emptyset)}=\frac{\left\Vert x\right\Vert ^{d-2}+\left\Vert y\right\Vert ^{d-2}}{\left\Vert x-y\right\Vert ^{d-2}}.\label{eq:lower is upper}
\end{equation}
\end{prop}

\begin{proof}
Let $W(t)$ be a Brownian motion in $\mathbb{R}^{d}$ for some $d\geq3$,
so $B=\left\{ W(t):t\in[0,\infty)\right\} $. We denote by $P=P_{0}$
the probability measure that corresponds to the Brownian motion that
is started at the origin and by $P_{x}$ the probability measure that
corresponds to the Brownian motion that is started at $x\in\mathbb{R}^{d}$.
Let
\[
T_{x,r}=\inf\left\{ t\in[0,\infty):W(t)\in\partial B(x,r)\right\} 
\]
for $x\in\mathbb{R}^{d}$ and $r>0$. Let $x,y\in\mathbb{R}^{d}$
and $R,r>0$ such that $r+R<\left\Vert x-y\right\Vert $, $R<\left\Vert x\right\Vert $
and $r<\left\Vert y\right\Vert $. If $z\in\partial B(x,R)$ then
\[
P_{z}(T_{y,r}<\infty)=\frac{r^{d-2}}{\left\Vert x-z\right\Vert ^{d-2}}
\]
by Lemma \ref{lem:hitting prob of balls}, thus
\begin{equation}
\frac{r^{d-2}}{\left(\left\Vert x-y\right\Vert +R\right)^{d-2}}\leq P_{z}(T_{y,r}<\infty)\leq\frac{r^{d-2}}{\left(\left\Vert x-y\right\Vert -R\right)^{d-2}}\label{eq:x to y}
\end{equation}
and similarly for $z\in\partial B(y,r)$
\begin{equation}
\frac{R^{d-2}}{\left(\left\Vert x-y\right\Vert +r\right)^{d-2}}\leq P_{z}(T_{x,R}<\infty)\leq\frac{R^{d-2}}{\left(\left\Vert x-y\right\Vert -r\right)^{d-2}}.\label{eq:y to x}
\end{equation}
Let $U=W(T_{x,R})$ and $V=W(T_{y,r})$ be the stopped Brownian motions.
Then by Lemma \ref{lem:hitting prob of balls}, (\ref{eq:x to y})
and by the strong Markov property (Proposition \ref{prop:strong markov prop})
\[
P(W\mathrm{\,hits\,}B(x,R)\mathrm{\,and\,after\,that\,}W\mathrm{\,hits\,}B(y,r))
\]
\begin{equation}
=P(T_{x,R}<\infty)\cdot E(P_{U}(T_{y,r}<\infty)\mid T_{x,R}<\infty)\leq\frac{R^{d-2}}{\left\Vert x\right\Vert ^{d-2}}\cdot\frac{r^{d-2}}{\left(\left\Vert x-y\right\Vert -R\right)^{d-2}}\label{eq:kell egy jel}
\end{equation}
and similarly
\[
P(W\mathrm{\,hits\,}B(y,r)\mathrm{\,and\,after\,that\,}W\mathrm{\,hits\,}B(x,R))
\]
\[
=P(T_{y,r}<\infty)\cdot E(P_{V}(T_{x,R}<\infty)\mid T_{y,r}<\infty)\leq\frac{r^{d-2}}{\left\Vert y\right\Vert ^{d-2}}\cdot\frac{R^{d-2}}{\left(\left\Vert x-y\right\Vert -r\right)^{d-2}}.
\]
Hence
\[
P(B\cap B(x,R)\neq\emptyset\mathrm{\,and\,}B\cap B(y,r)\neq\emptyset)
\]
\[
\leq\frac{R^{d-2}}{\left\Vert x\right\Vert ^{d-2}}\cdot\frac{r^{d-2}}{\left(\left\Vert x-y\right\Vert -R\right)^{d-2}}+\frac{r^{d-2}}{\left\Vert y\right\Vert ^{d-2}}\cdot\frac{R^{d-2}}{\left(\left\Vert x-y\right\Vert -r\right)^{d-2}}
\]
\[
=r^{d-2}R^{d-2}\frac{\left\Vert y\right\Vert ^{d-2}\left(\left\Vert x-y\right\Vert -r\right)^{d-2}+\left\Vert x\right\Vert ^{d-2}\left(\left\Vert x-y\right\Vert -R\right)^{d-2}}{\left\Vert y\right\Vert ^{d-2}\left(\left\Vert x-y\right\Vert -r\right)^{d-2}\cdot\left\Vert x\right\Vert ^{d-2}\left(\left\Vert x-y\right\Vert -R\right)^{d-2}}
\]
\begin{equation}
\leq\frac{R^{d-2}}{\left\Vert x\right\Vert ^{d-2}}\cdot\frac{r^{d-2}}{\left\Vert y\right\Vert ^{d-2}}\cdot\frac{\left\Vert y\right\Vert ^{d-2}\left\Vert x-y\right\Vert ^{d-2}+\left\Vert x\right\Vert ^{d-2}\left\Vert x-y\right\Vert ^{d-2}}{\left(\left\Vert x-y\right\Vert -r\right)^{d-2}\cdot\left(\left\Vert x-y\right\Vert -R\right)^{d-2}}.\label{eq:hits both}
\end{equation}
Hence by Lemma \ref{lem:hitting prob of balls} it follows that
\[
\limsup_{R\rightarrow0}\limsup_{r\rightarrow0}\frac{P\left(B\cap B(x,R)\neq\emptyset,B\cap B(y,r)\neq\emptyset\right)}{P(B\cap B(x,R)\neq\emptyset)\cdot P(B\cap B(y,r)\neq\emptyset)}
\]
\begin{equation}
\leq\limsup_{R\rightarrow0}\limsup_{r\rightarrow0}\frac{\left\Vert y\right\Vert ^{d-2}\left\Vert x-y\right\Vert ^{d-2}+\left\Vert x\right\Vert ^{d-2}\left\Vert x-y\right\Vert ^{d-2}}{\left(\left\Vert x-y\right\Vert -r\right)^{d-2}\cdot\left(\left\Vert x-y\right\Vert -R\right)^{d-2}}=\frac{\left\Vert x\right\Vert ^{d-2}+\left\Vert y\right\Vert ^{d-2}}{\left\Vert x-y\right\Vert ^{d-2}}.\label{eq:upper hitting final}
\end{equation}

By (\ref{eq:hits both}) and Lemma \ref{lem:hitting prob of balls}
it follows that
\[
P(T_{x,R}<\infty,T_{x,R}\leq T_{y,r})\geq P(B\cap B(x,R)\neq\emptyset,B\cap B(y,r)=\emptyset)
\]
\begin{equation}
\geq P(B\cap B(x,R)\neq\emptyset)-P(B\cap B(x,R)\neq\emptyset,B\cap B(y,r)\neq\emptyset)\geq\frac{R^{d-2}}{\left\Vert x\right\Vert ^{d-2}}\cdot(1-O(r))\label{eq:egyik ,low}
\end{equation}
and similarly
\begin{equation}
P(T_{y,r}<\infty,T_{y,r}\leq T_{x,R})\geq\frac{r^{d-2}}{\left\Vert y\right\Vert ^{d-2}}\cdot(1-O(R)).\label{eq:masik low}
\end{equation}
Thus by Lemma \ref{lem:hitting prob of balls}, (\ref{eq:x to y}),
(\ref{eq:y to x}), (\ref{eq:egyik ,low}), (\ref{eq:masik low})
and by the strong Markov property (Proposition \ref{prop:strong markov prop})
\[
P(B\cap B(x,R)\neq\emptyset,B\cap B(y,r)\neq\emptyset)\geq P(T_{x,R}\leq T_{y,r}<\infty)+P(T_{y,r}\leq T_{x,R}<\infty)
\]
\[
\geq\frac{R^{d-2}}{\left\Vert x\right\Vert ^{d-2}}\cdot(1-O(r))\cdot\frac{r^{d-2}}{\left(\left\Vert x-y\right\Vert +R\right)^{d-2}}+\frac{r^{d-2}}{\left\Vert y\right\Vert ^{d-2}}\cdot(1-O(R))\cdot\frac{R^{d-2}}{\left(\left\Vert x-y\right\Vert +r\right)^{d-2}}
\]
\[
=R^{d-2}r^{d-2}\frac{(1-O(r))\left\Vert y\right\Vert ^{d-2}\left(\left\Vert x-y\right\Vert +r\right)^{d-2}+(1-O(R))\left\Vert x\right\Vert ^{d-2}\left(\left\Vert x-y\right\Vert +R\right)^{d-2}}{\left\Vert y\right\Vert ^{d-2}\left(\left\Vert x-y\right\Vert +r\right)^{d-2}\left\Vert x\right\Vert ^{d-2}\left(\left\Vert x-y\right\Vert +R\right)^{d-2}}
\]
\[
=\frac{R^{d-2}}{\left\Vert x\right\Vert ^{d-2}}\cdot\frac{r^{d-2}}{\left\Vert y\right\Vert ^{d-2}}\cdot\frac{(1-O(r))\left\Vert y\right\Vert ^{d-2}\left(\left\Vert x-y\right\Vert +r\right)^{d-2}+(1-O(R))\left\Vert x\right\Vert ^{d-2}\left(\left\Vert x-y\right\Vert +R\right)^{d-2}}{\left(\left\Vert x-y\right\Vert +r\right)^{d-2}\left(\left\Vert x-y\right\Vert +R\right)^{d-2}}.
\]
Hence by Lemma \ref{lem:hitting prob of balls} it follows that
\begin{equation}
\liminf_{R\rightarrow0}\liminf_{r\rightarrow0}\frac{P\left(B\cap B(x,R)\neq\emptyset,B\cap B(y,r)\neq\emptyset\right)}{P(B\cap B(x,R)\neq\emptyset)\cdot P(B\cap B(y,r)\neq\emptyset)}\geq\frac{\left\Vert x\right\Vert ^{d-2}+\left\Vert y\right\Vert ^{d-2}}{\left\Vert x-y\right\Vert ^{d-2}}.\label{eq:lower hitting final}
\end{equation}
So (\ref{eq:lower is upper}) holds by (\ref{eq:upper hitting final})
and (\ref{eq:lower hitting final}).
\end{proof}
\begin{lem}
\label{lem:hitting prob of compact sets}Let $B$ be a Brownian path
in $\mathbb{R}^{d}$ for $d\geq3$, let $0\in A\subseteq\mathbb{R}^{d}$
be a compact set with $\mathrm{diam}(A)>0$. Then for $x\in\mathbb{R}^{d}\setminus\{0\}$
and $0<r<2^{-1}\left\Vert x\right\Vert /\mathrm{diam}(A)$
\[
C_{G}(A)\frac{r^{d-2}}{\left\Vert x\right\Vert ^{d-2}}(1-r\cdot\mathrm{diam}(A)/\left\Vert x\right\Vert )^{d-2}\leq P\left((r\cdot A+x)\cap B\neq\emptyset\right)
\]
\[
\leq C_{G}(A)\frac{r^{d-2}}{\left\Vert x\right\Vert ^{d-2}}(1+2r\cdot\mathrm{diam}(A)/\left\Vert x\right\Vert )^{d-2},
\]
where $G(x,y)=c(d)\left\Vert x-y\right\Vert ^{2-d}$ is the Green`s
function of the Brownian motion for some constant $c(d)>0$ (see Remark
\ref{rem:green fn}).
\end{lem}

\begin{proof}
See \cite[Theorem 3.33]{Peres-Morters-Broanian motion}, that $G(x,y)$
is the Green`s function of the Brownian motion. By \cite[Corollary 8.12]{Peres-Morters-Broanian motion}
and \cite[Theorem 8.27]{Peres-Morters-Broanian motion}
\[
C_{G}(r\cdot A+x)\left(\left\Vert x\right\Vert +r\cdot\mathrm{diam}(A)\right)^{2-d}\leq P\left((r\cdot A+x)\cap B\neq\emptyset\right)
\]
\begin{equation}
\leq C_{G}(r\cdot A+x)\left((\left\Vert x\right\Vert -r\cdot\mathrm{diam}(A))\right)^{2-d}.\label{eq:hitting eq 1}
\end{equation}
On the other, hand by the scaling invariance of capacity it follows
that
\begin{equation}
C_{G}(r\cdot A+x)=r^{d-2}C_{G}(A).\label{eq:hitting eq 2}
\end{equation}
We have that $\left\Vert x\right\Vert ^{-1}(1-r\cdot\mathrm{diam}(A)/\left\Vert x\right\Vert )\leq\left(\left\Vert x\right\Vert +r\cdot\mathrm{diam}(A)\right)^{-1}$
and $\left(\left\Vert x\right\Vert -r\cdot\mathrm{diam}(A)\right)^{-1}\leq\left\Vert x\right\Vert ^{-1}(1+2r\cdot\mathrm{diam}(A)/\left\Vert x\right\Vert )$
because $0<r<2^{-1}\left\Vert x\right\Vert /\mathrm{diam}(A)$. Thus
the statement follows from (\ref{eq:hitting eq 1}) and (\ref{eq:hitting eq 2}).
\end{proof}
\begin{prop}
\label{prop:general double hitting}Let $B$ be a Brownian path in
$\mathbb{R}^{d}$ for $d\geq3$ , let $0\in A\subseteq\mathbb{R}^{d}$
be a compact set such that $C_{d-2}(A)>0$. Let $x,y\in\mathbb{R}^{d}\setminus\{0\}$,
$x\neq y$ and $x_{R},y_{r}\in\mathbb{R}^{d}$ for every $r,R>0$
be such that $\lim_{R\rightarrow0}x_{R}=x$ and $\lim_{R\rightarrow0}y_{R}=y$.
Then
\[
\liminf_{R\rightarrow0}\liminf_{r\rightarrow0}\frac{P\left((R\cdot A+x_{R})\cap B\neq\emptyset,(r\cdot A+y_{r})\cap B\neq\emptyset\right)}{P((R\cdot A+x_{R})\cap B\neq\emptyset)\cdot P((r\cdot A+y_{r})\cap B\neq\emptyset)}
\]
\[
=\limsup_{R\rightarrow0}\limsup_{r\rightarrow0}\frac{P\left((R\cdot A+x_{R})\cap B\neq\emptyset,(r\cdot A+y_{r})\cap B\neq\emptyset\right)}{P((R\cdot A+x_{R})\cap B\neq\emptyset)\cdot P((r\cdot A+y_{r})\cap B\neq\emptyset)}=\frac{\left\Vert x\right\Vert ^{d-2}+\left\Vert y\right\Vert ^{d-2}}{\left\Vert x-y\right\Vert ^{d-2}}.
\]
\end{prop}

For $x\in\mathbb{R}^{d}\setminus\{0\}$ let $Q_{k}(x)$ be the dyadic
cube $\left([\frac{i_{1}}{2^{k}},\frac{i_{1}+1}{2^{k}})\times\dots\times[\frac{i_{d}}{2^{k}},\frac{i_{d}+1}{2^{k}})\right)\setminus\{0\}$
for $i_{1},\dots,i_{d}\in\mathbb{Z}$ such that $x\in Q_{k}(x)$ and
let $\mathcal{Q}_{k}=\left\{ Q_{k}(x):x\in\mathbb{R}^{d}\right\} $.
\begin{prop}
\label{prop:double box hitting}Let $B$ be a Brownian path in $\mathbb{R}^{d}$
for $d\geq3$. Then
\[
\liminf_{n\rightarrow\infty}\liminf_{k\rightarrow\infty}\frac{P\left(B\cap Q_{n}(x)\neq\emptyset,B\cap Q_{k}(y)\neq\emptyset\right)}{P(B\cap Q_{n}(x)\neq\emptyset)\cdot P(B\cap Q_{k}(y)\neq\emptyset)}
\]
\[
=\limsup_{n\rightarrow\infty}\limsup_{k\rightarrow\infty}\frac{P\left(B\cap Q_{n}(x)\neq\emptyset,B\cap Q_{k}(y)\neq\emptyset\right)}{P(B\cap Q_{n}(x)\neq\emptyset)\cdot P(B\cap Q_{k}(y)\neq\emptyset)}=\frac{\left\Vert x\right\Vert ^{d-2}+\left\Vert y\right\Vert ^{d-2}}{\left\Vert x-y\right\Vert ^{d-2}}.
\]
\end{prop}

The proof of Proposition \ref{prop:general double hitting} goes similarly
to the proof of Proposition \ref{prop:ball double hitting} replacing
the use of Lemma \ref{lem:hitting prob of balls} by Lemma \ref{lem:hitting prob of compact sets}.
\begin{rem}
\label{rem:boundary dont count}Proposition \ref{prop:double box hitting}
is a special case of Proposition \ref{prop:general double hitting}.
Note that it is not an issue that the cubes are not compact because
the hitting probability does not depend on whether the boundary of
the cube is in the set or not. It can, for example, be deduced from
Lemma \ref{lem:hitting prob of compact sets} by approximating the
cube by compact cubes from inside and outside.
\end{rem}

\begin{thm}
\noindent \label{thm:F exist for Brownian}Let $B$ be a Brownian
path in $\mathbb{R}^{d}$ for $d\geq3$. Let $\mathcal{Q}_{k}$ be
as in Example \ref{exa:example Q_k} then
\[
F(x,y)=\underline{F}(x,y)=\overline{F}(x,y)=\frac{\left\Vert x\right\Vert ^{d-2}+\left\Vert y\right\Vert ^{d-2}}{\left\Vert x-y\right\Vert ^{d-2}}
\]
for $x,y\in\mathbb{R}^{d}\setminus\{0\}$, $x\neq y$.
\end{thm}

Theorem \ref{thm:F exist for Brownian} is a reformulation of Proposition
\ref{prop:double box hitting}.

\subsection{\label{subsec:Conditions-on}Conditions on $\mathcal{Q}_{k}^{i}$}

For the rest of Section \ref{sec:Brownian-path} let $\mathcal{Q}_{k}$
be as in Example \ref{exa:example Q_k}, $\mathcal{Q}_{k}^{i}$ be
as in Example \ref{exa:example Q_k^i} and let $\varphi(r)=r^{-(d-2)}$.
In Section \ref{subsec:Conditions-on} we show that the assumptions
of Section \ref{subsec:Special-assumptions} hold for the sequence
$\mathcal{Q}_{k}^{i}$ for sufficient constants.

We have that (\ref{eq:Kernel restriction}) holds for every $\delta>0$
for sufficient $c_{2}$ depending on $\delta$ and for $c_{3}=0$.
We have that (\ref{eq:phi =00003Dinfty}) holds. Clearly (\ref{eq:diameter_goes_to0}),
(\ref{eq:unique-subset}), (\ref{eq:psoitivity of probability}),
(\ref{eq:same size}) hold for both $\mathcal{Q}_{k}$ and $\mathcal{Q}_{k}^{i}$
for every $i\in\mathbb{N}$ with $M$=1. We have that $Q\cap S=\emptyset$
whenever $Q\neq S$ and $Q,S\in\mathcal{Q}_{k}$ or $Q,S\in\mathcal{Q}_{k}^{i}$.
It is easy to see that (\ref{eq:bounded sundivision}) hold for $\mathcal{Q}_{k}^{i}$
for every $i\in\mathbb{N}$ for every $\delta>0$ for constant $M_{\delta}>0$
that only depends on $\delta$ and $d$.

It follows from Lemma \ref{lem:hitting prob of compact sets}, Remark
\ref{rem:boundary dont count} and the scaling invariance of capacity
(\ref{eq:hitting eq 2}) that (\ref{eq:lower hitting prob}) holds
for $\mathcal{Q}_{k}^{i}$ for some sufficient constant $a^{i}<\infty$.
\begin{lem}
\label{prop: c ineq holds}Let $B$ be a Brownian path in $\mathbb{R}^{d}$
for $d\geq3$ and let $i$ be a positive integer. Then there exist
$0<\delta=\delta^{i}<1$ and $0<c=c^{i}<\infty$ such that (\ref{eq:capacity-independence})
holds for $\mathcal{Q}_{k}^{i}$ $(k\in\mathbb{N},k\geq i)$ and $c$
and $\delta$ depends only on $i$ and $d$.
\end{lem}

\begin{proof}
Let $W(t)$, $P_{x}$ and $T_{x,r}$ be as in the proof of Proposition
\ref{prop:ball double hitting}. Let $A=[0,1)^{d}.$ Let $N\in\mathbb{N}$
be large enough that
\begin{equation}
2^{-N}\leq2^{-i}/(2\sqrt{d}).\label{eq:majdkelkl}
\end{equation}
 Let $\delta\leq1/(2\sqrt{d})$, $k,n\geq\max\{i,N\}$ and $x\in Q\in\mathcal{Q}_{k}^{i}$,
$y\in S\in\mathcal{Q}_{n}^{i}$ such that $\max\left\{ \mathrm{\mathrm{diam}}(Q),\mathrm{\mathrm{diam}}(S)\right\} <\delta\cdot\mathrm{dist}(Q,S)$.
Let $r=2^{-k}$ and $R=2^{-n}$. Then
\[
\max(r,R)\leq\max\left\{ \mathrm{\mathrm{diam}}(Q),\mathrm{\mathrm{diam}}(S)\right\} <(2\sqrt{d})^{-1}\cdot\mathrm{dist}(Q,S)\leq\left\Vert z-y\right\Vert /(2\sqrt{d})
\]
for every $z\in\overline{Q}$. So
\[
1+2\frac{R}{\left\Vert z-y\right\Vert }\mathrm{diam}(A)\leq2
\]
for every $z\in\overline{Q}$ and. Thus by Lemma \ref{lem:hitting prob of compact sets}
and Remark \ref{rem:boundary dont count}
\begin{equation}
P_{z}\left(S\cap B\neq\emptyset\right)\leq a_{0}\frac{R^{d-2}}{\left\Vert z-y\right\Vert ^{d-2}}\leq a_{0}\frac{R^{d-2}}{\mathrm{dist}(Q,S)^{d-2}}\label{eq:zbol}
\end{equation}
for $a_{0}=C_{G}(A)\cdot2^{d-2}$ for every $z\in\overline{Q}$.

By (\ref{eq:majdkelkl}) it follows that
\[
r\leq2^{-i}/(2\sqrt{d})\leq\frac{\left\Vert x\right\Vert }{2\mathrm{diam}(A)}
\]
because $x\in Q\in\mathcal{Q}_{k}^{i}$. Hence by Lemma \ref{lem:hitting prob of compact sets}
\begin{equation}
C_{G}(A)\frac{r^{d-2}}{\left\Vert x\right\Vert ^{d-2}}2^{2-d}\leq P\left(Q\cap B\neq\emptyset\right)\leq C_{G}(A)\frac{r^{d-2}}{\left\Vert x\right\Vert ^{d-2}}2^{d-2}.\label{eq:xbe}
\end{equation}
Similarly
\begin{equation}
C_{G}(A)\frac{R^{d-2}}{\left\Vert y\right\Vert ^{d-2}}2^{2-d}\leq P\left(S\cap B\neq\emptyset\right)\leq C_{G}(A)\frac{R^{d-2}}{\left\Vert y\right\Vert ^{d-2}}2^{d-2}.\label{eq:yba}
\end{equation}

Then similarly to (\ref{eq:kell egy jel}) it follows from (\ref{eq:xbe})
and (\ref{eq:zbol}) that
\[
P(W\mathrm{\,hits\,}Q\mathrm{\,and\,after\,that\,}W\mathrm{\,hits\,}S)\leq b\cdot\frac{r^{d-2}}{\left\Vert x\right\Vert ^{d-2}}R^{d-2}\mathrm{dist}(Q,S)^{-(d-2)}
\]
\[
=\left\Vert y\right\Vert ^{d-2}b\cdot\frac{r^{d-2}}{\left\Vert x\right\Vert ^{d-2}}\cdot\frac{R^{d-2}}{\left\Vert y\right\Vert ^{d-2}}\cdot\mathrm{dist}(Q,S)^{-(d-2)}\leq(\sqrt{d}2^{i})^{(d-2)}b\frac{r^{d-2}}{\left\Vert x\right\Vert ^{d-2}}\cdot\frac{R^{d-2}}{\left\Vert y\right\Vert ^{d-2}}\cdot\mathrm{dist}(Q,S)^{-(d-2)}
\]
for $b=C_{G}(A)2^{d-2}a_{0}$. Similarly we can show that
\[
P(W\mathrm{\,hits\,}S\mathrm{\,and\,after\,that\,}W\mathrm{\,hits\,}Q)\leq(\sqrt{d}2^{i})^{(d-2)}b\frac{r^{d-2}}{\left\Vert x\right\Vert ^{d-2}}\cdot\frac{R^{d-2}}{\left\Vert y\right\Vert ^{d-2}}\cdot\mathrm{dist}(Q,S)^{-(d-2)}.
\]
Thus
\[
P(Q\cap B\ne\emptyset\,and\,S\cap B\ne\emptyset)
\]
\[
\leq P(W\mathrm{\,hits\,}Q\mathrm{\,and\,after\,that\,}W\mathrm{\,hits\,}D)+P(W\mathrm{\,hits\,}D\mathrm{\,and\,after\,that\,}W\mathrm{\,hits\,}Q)
\]
\[
\leq2(\sqrt{d}2^{i})^{(d-2)}b\frac{r^{d-2}}{\left\Vert x\right\Vert ^{d-2}}\cdot\frac{R^{d-2}}{\left\Vert y\right\Vert ^{d-2}}\cdot\mathrm{dist}(Q,S)^{-(d-2)}.
\]
Hence it follows from (\ref{eq:xbe}) and (\ref{eq:yba}) that
\[
P(Q\cap B\ne\emptyset\,and\,S\cap B\ne\emptyset)\leq cP(Q\cap B\ne\emptyset)P(S\cap B\ne\emptyset)(\mathrm{dist}(Q,S)^{-(d-2)}
\]
for $c=2(\sqrt{d}2^{i})^{(d-2)}b(2^{d-2})^{2}/C_{G}(A)^{2}$.

So (\ref{eq:capacity-independence}) holds if $k,n\geq\max\{i,N\}$
and $\delta\leq1/(2\sqrt{d})$. We can choose $\delta>0$ to be small
enough such that if either $k<\max\{i,N\}$ or $n<\max\{i,N\}$ then
$\max\left\{ \mathrm{diam}(Q),\mathrm{diam}(S)\right\} <\delta\cdot\mathrm{dist}(Q,S)$
does not hold for every pair of $Q\in\mathcal{Q}_{k}^{i}$ and $D\in\mathcal{Q}_{n}^{i}$.
\end{proof}

\subsection{\label{subsec:Existance-of-the-cond meas on Brownian path}Existence
of the conditional measure on the Brownian path}

We show Theorem \ref{thm:Brownian cond measure: Main} in this section.
\begin{lem}
\label{lem:degeneration lem for Brownian}If $A\subseteq\mathbb{R}^{d}\setminus\{0\}$
is a compact set and $C_{d-2}(A)=0$ then $P(B\cap A\neq\emptyset)=0$.
\end{lem}

Lemma \ref{lem:degeneration lem for Brownian} is a direct corollary
of Proposition \ref{thm:intersection capacity equivalence} \\

\selectlanguage{english}%
\noindent \textit{Proof of Theorem }\foreignlanguage{british}{\ref{thm:Brownian cond measure: Main}.
In Section \ref{subsec:Conditions-on} we establish that the assumptions
of Section \ref{subsec:Special-assumptions} hold for sufficient constants
$c_{2}^{i}<\infty$, $c_{3}=0$, $0<\delta^{i}<1$, $M_{\delta}^{i}<\infty$,
$M=1$, $a^{i}<\infty$ and $0<c^{i}<\infty$ for the sequence $\mathcal{Q}_{k}^{i}$.
We have that $\mathcal{Q}_{k}^{i}\subseteq\mathcal{Q}_{k}^{j}\subseteq\mathcal{Q}_{k}$
for $i\leq j\leq k$. It is easy to see that $X_{0}^{\infty}=\cup_{i\in\mathbb{N}}X_{0}^{i}=\mathbb{R}^{d}\setminus\{0\}$.
In Theorem \ref{thm:F exist for Brownian} we prove that
\[
F(x,y)=\underline{F}(x,y)=\overline{F}(x,y)=\frac{\left\Vert x\right\Vert ^{d-2}+\left\Vert y\right\Vert ^{d-2}}{\left\Vert x-y\right\Vert ^{d-2}}.
\]
Thus along with Lemma \ref{lem:degeneration lem for Brownian} the
conditions of Theorem \ref{thm:increasing double int} are satisfied.
Hence the statement follows from Theorem \ref{thm:increasing double int}.$\hfill\square$}
\selectlanguage{british}%

\subsection{\label{subsec:Probability-of-non Brown}Probability of non-extinction
of the conditional measure on the Brownian path}

Theorem \ref{thm:nonext prob for cond meas Brownian} states an analogous
result to Proposition \ref{thm:intersection capacity equivalence}
for measures.\\

\selectlanguage{english}%
\noindent \textit{Proof of Theorem }\foreignlanguage{british}{\ref{thm:nonext prob for cond meas Brownian}.
We wish to apply Theorem \ref{thm:increasing non extinction} for
$K(x,y)=F(x,y)$ to conclude Theorem \ref{thm:nonext prob for cond meas Brownian}.
In the proof of \ref{thm:Brownian cond measure: Main} we show that
the conditions of Theorem \ref{thm:increasing double int} are satisfied.
Obviously we can find a decomposition $\mathbb{R}^{d}\setminus\{0\}=\cup_{i=1}^{\infty}A_{i}$
such that (\ref{eq:fi eq eq}) holds, namely $A_{i}=X_{0}^{i}$. So,
along with Proposition \ref{thm:intersection capacity equivalence},
the conditions of Theorem \ref{thm:increasing non extinction} are
satisfied. The explicit formula for $F(x,y)$ is given in Theorem
\ref{thm:F exist for Brownian}.$\hfill\square$}
\selectlanguage{british}%

\section{Conditional measure of the Lebesgue measure on the Brownian path
and the occupation measure\label{sec:Conditional-ofleb and ocup}}

In this Section we prove Theorem \ref{thm:intro ocup} that is the
union of Theorem \ref{thm:ocup limit} and Theorem \ref{thm:conditional measure  of lebesgue}.
The proof consists many steps that we sorted in three sections. In
Section \ref{subsec:Ergod} we establish an asymptotic result on the
number of cubes of side length $1$ that is intersected by the Brownian
path. We prove this by the application of the ergodic theorem. The
major part of the proof is to show the ergodicity of the invariant
measure that we define. Then in Section \ref{subsec:small cubes}
we use the scaling invariance of the Brownian motion to deduce an
asymptotic result on the number of small cubes that is intersected
by the Brownian path. Finally, in Section \ref{subsec:Occupation-measure-and}
we use the result on the number of small cubes that is intersected
by the Brownian path and an approximation argument to finish the proof.

For a random variable $Y$ and a $\sigma$-algebra $\mathcal{F}\subseteq\mathcal{A}$
we denote the \textit{conditional expectation of $Y$ with respect
to $\mathcal{F}$} by $E\left(Y\mid\mathcal{F}\right)$. For random
variables $Y$ and $Z$ we write $E\left(Y\mid Z\right)$ for $E\left(Y\mid\mathcal{F}\right)$
where $\mathcal{F}$ is the $\sigma$-algebra generated by $Z$. In
that case there exists a deterministic function $f$ such that $E\left(Y\mid Z\right)=f(Z)$
almost surely and we write $E\left(Y\mid Z=z\right)$ for $f(z)$.
We write $E\left(Y\mid Z_{1},\dots,Z_{n}\right)$ when $\mathcal{F}$
is the $\sigma$-algebra generated by $Z_{1},\dots,Z_{n}$. In the
rest of the paper we use many basic properties of the conditional
expectation without reference. For an overview of the conditional
expectation see for example \cite{Dudley,Durrett}.

Throughout this section let $\mathcal{Q}_{k}$ be as in Example \ref{exa:example Q_k}
and let
\[
\mathcal{Q}_{k}^{*}=\left\{ [\frac{i_{1}}{2^{k}},\frac{i_{1}+1}{2^{k}})\times\dots\times[\frac{i_{d}}{2^{k}},\frac{i_{d}+1}{2^{k}}):i_{1},\dots,i_{d}\in\mathbb{Z}\right\} 
\]
for $k\in\mathbb{N}$. Recall that we denote the Lebesgue measure
by $\lambda$.

\subsection{Application of the ergodic theorem\label{subsec:Ergod}}

Throughout this subsection let $B_{i}$ ($i\in\mathbb{Z}$) be an
i.i.d. sequence of Brownian motions started at $0\in\mathbb{R}^{d}$
with domain $[0,1]$. Let $X$ be a random variable uniformly distributed
on $[0,1)^{d}$ such that $X$ and $B_{i}$ ($i\in\mathbb{Z}$) are
mutually independent. Let $W_{0}=W_{0}^{\omega}$ be the random function
defined by
\[
W_{0}(t)=\begin{cases}
B_{0}(t) & \mathrm{if}\,0\leq t<1\\
B_{n}(t-n)+\sum_{i=0}^{n-1}B_{i}(1) & \mathrm{if}\,1\leq n\leq t<n+1\,\mathrm{for\,some}\,n\in\mathbb{Z}\\
B_{n}(t-n)-\sum_{i=n}^{-1}B_{i}(1) & \mathrm{if}\,n\leq t<n+1\leq0\,\mathrm{for\,some}\,n\in\mathbb{Z}
\end{cases}
\]
and let $W(t)=X+W_{0}(t)$, i.e. $W$ is a two sided Brownian motion
started at $X$ (it is due to the independent increments of the Brownian
motion). Note, that $B_{i}(.)=\left(W(.+i)-W(i)\right)\vert_{[0,1]}$.
Thus $W$ determines $X$ and the sequence $B_{i}$ and vice versa.

For a vector $v=(v_{1},\dots,v_{d})\in\mathbb{R}^{d}$ we denote by
$\{v\}$ the equivalence class of $v$ in $\mathbb{R}^{d}/\mathbb{Z}^{d}$
and we denote by $\{v\}_{0}$ the element of the equivalence class
of $v$ in $\mathbb{R}^{d}/\mathbb{Z}^{d}$ that is contained in $[0,1)^{d}$.
For an equivalence class $w\in\mathbb{R}^{d}/\mathbb{Z}^{d}$ we also
denote by $\{w\}_{0}$ the element of the equivalence class $w$ that
is contained in $[0,1)^{d}$. Let $X_{n}=\{W(n)\}$.

Let $S$ be the right shift map on $C([0,1],d)^{\mathbb{Z}}\times(\mathbb{R}^{d}/\mathbb{Z}^{d})^{\mathbb{Z}}$,
where $C(K,d)$ denotes the space of continuous functions from the
compact set $K$ to $\mathbb{R}^{d}$ equipped with the supremum norm.
So $S\left((f_{i})_{i=-\infty}^{\infty},(x_{i})_{i=-\infty}^{\infty}\right)=\left((f_{i+1})_{i=-\infty}^{\infty},(x_{i+1})_{i=-\infty}^{\infty}\right)$
for $\left((f_{i})_{i=-\infty}^{\infty},(x_{i})_{i=-\infty}^{\infty}\right)\in C([0,1],d)^{\mathbb{Z}}\times\left(\mathbb{R}^{d}/\mathbb{Z}^{d}\right)^{\mathbb{Z}}$.
Let $R$ be the probability distribution of $\left((B_{i})_{i=-\infty}^{\infty},(X_{i})_{i=-\infty}^{\infty}\right)$.
Then $R$ is a Borel measure. We show that $R$ is a shift invariant
ergodic measure.

\begin{lem}
\label{lem:shift invariance}We have that $R$ is a shift invariant
measure, i.e. $R(.)=R(S^{-1}(.))$.
\end{lem}

\begin{proof}
Since $X_{i}=\{W(i)\}$ and $B_{i}(.)=\left(W(.+i)-W(i)\right)\vert_{[0,1]}$
we need to show that the random functions $W(.)$ and $W(.-1)-\left\lfloor W(-1)\right\rfloor $
have the same distribution, where $\left\lfloor .\right\rfloor $
denotes the floor function in every coordinates. We have that
\[
E\left(I_{X_{0}\in A+\left\{ B_{-1}(1)\right\} }\cdot I_{B_{-1}\in H}\mid B_{-1}\right)=I_{B_{-1}\in H}\cdot E\left(I_{X_{0}\in A+\left\{ B_{-1}(1)\right\} }\mid B_{-1}\right)=I_{B_{-1}\in H}\cdot P(X_{0}\in A)
\]
for ever Borel set $A\subseteq\mathbb{R}^{d}/\mathbb{Z}^{d}$ and
$H\subseteq C([0,1],d)$ because $X_{0}$ is uniformly distributed
on $\mathbb{R}^{d}/\mathbb{Z}^{d}$ and is independent of $B_{-1}$.
Hence
\[
P(X_{0}-\left\{ B_{-1}(1)\right\} \in A,B_{-1}\in H)=E(I_{B_{-1}\in H}\cdot P(X_{0}\in A))=
\]
\[
P(B_{-1}\in H)P(X_{0}\in A)=P(B_{-1}\in H)P(X_{0}-\left\{ B_{-1}(1)\right\} \in A)
\]
Since the collection of sets of the form $A\times H$ is a semi-ring
generating the Borel $\sigma$-algebra it follows from Proposition
\ref{lem:charateodory ineq} that $B_{-1}$ and $\left\{ W(-1)\right\} =X_{0}-\left\{ B_{-1}(1)\right\} $
are independent. Obviously we have that $B_{i}$ ($i\in\mathbb{Z}\setminus\{-1\}$)
are mutually independent of $B_{-1}$ and $\left\{ W(-1)\right\} =X_{0}-\left\{ B_{-1}(1)\right\} $
hence $W_{0}(.-1)-W_{0}(-1)$ is a two sided Brownian motion started
at $0$ such that $\left\{ W(-1)\right\} $ and $W_{0}(.-1)-W_{0}(-1)$
are independent. Hence $W(.)$ and $W(.-1)-\left\lfloor W(-1)\right\rfloor =\{W(-1)\}_{0}+W_{0}(.-1)-W_{0}(-1)$
have the same distribution.
\end{proof}
\begin{lem}
\label{lem:Rieman limit}Let $Y$ be a random variable that takes
values in $[0,1)$ almost surely, the distribution of $Y$ is absolutely
continuous with respect to the Lebesgue measure and the density function
$f$ is a bounded Riemann-integrable function. Then for every Borel
set $A\subseteq[0,1)$
\[
\lim_{k\rightarrow\infty}P\left(\{kY\}_{0}\in A\right)=P(U\in A)
\]
where $U$ is a random variable uniformly distributed on $[0,1)$.
Let $M<\infty$ be such that $0\leq f(x)\leq M$ for Lebesgue almost
every $x\in\mathbb{R}$. Then the density function of $\{kY\}_{0}$
is bounded by $M$ for Lebesgue almost every $x\in\mathbb{R}$.
\end{lem}

\begin{proof}
The density function of $kY$ is $k^{-1}f(x/k)$ for Lebesgue almost
every $x$ and so the density function of $\{kY\}_{0}$ is $k^{-1}\sum_{i=0}^{k-1}f((x+i)/k)$
for Lebesgue almost every $x\in[0,1)$ where we used the fact that
$f(x)=0$ for $x\notin[0,1)$. For every $x\in[0,1)$ we have that
$k^{-1}\sum_{i=0}^{k}f((x+i)/k)$ converges to $1$ since $f$ is
Riemann integrable. The sequence $k^{-1}\sum_{i=0}^{k-1}f((x+i)/k)$
is uniformly bounded by $M$ for Lebesgue almost every $x\in\mathbb{R}$.
Hence by the dominated convergence theorem $\lim_{k\rightarrow\infty}P\left(\{kY\}_{0}\in A\right)=P(U\in A)$
for every Borel set $A\subseteq[0,1)$.
\end{proof}
\begin{lem}
\label{lem:kY}Let $Y$ be a normally distributed vector in $\mathbb{R}^{d}$
such that the coordinates are independent. Then for Borel sets $A_{1},\dots,A_{d}\subseteq\mathbb{R}/\mathbb{Z}$
\[
\lim_{k\rightarrow\infty}P\left(\{kY\}\in A_{1}\times\dots\times A_{d}\right)=P(U\in A_{1}\times\dots\times A_{d})
\]
where $U$ is a random variable uniformly distributed on $\mathbb{R}^{d}/\mathbb{Z}^{d}$.
Additionally, the density functions of $\{kY\}_{0}$ are uniformly
bounded.
\end{lem}

\begin{proof}
We prove the statement for standard normally distributed vector $Y$
but the proof goes similarly for the general $Y$. Let $Y=(Y_{1},\dots,Y_{d})$
and $U=(U_{1},\dots,U_{d})$. The sequence of functions $\sum_{i=-n}^{n}e^{-(x+i)^{2}/2}$
locally uniformly converges to
\[
f(x)=\sum_{i=-\infty}^{\infty}e^{-(x+i)^{2}/2}
\]
hence $f$ is a continuous function and so
\[
g(x)=\begin{cases}
f(x) & \mathrm{if}\,x\in[0,1)\\
0 & \mathrm{otherwise}
\end{cases}
\]
is a bounded Riemann-integrable function. Since $Y_{j}$ is standard
normally distributed in $\mathbb{R}$ it follows that $\frac{1}{\sqrt{2\pi}}g$
is the density function of $\{Y_{j}\}_{0}$, thus
\[
\lim_{k\rightarrow\infty}P\left(\{kY_{j}\}\in A_{j}\right)=P(U_{j}\in A_{j})
\]
 and the density function of $\{kY_{j}\}_{0}$ is bounded by $\frac{1}{\sqrt{2\pi}}\sup_{x\in[0,1)}g(x)$
by Lemma \ref{lem:Rieman limit}. Since $Y_{j}$ are independent for
$j=1,\dots,d$ it follows that
\[
\lim_{k\rightarrow\infty}P\left(\{kY\}\in A_{1}\times\dots\times A_{d}\right)=\lim_{k\rightarrow\infty}P\left(\{kY_{1}\}\in A_{1}\right)\cdot\dots\cdot P\left(\{kY_{d}\}\in A_{d}\right)=
\]
\[
P\left(U_{1}\in A_{1}\right)\cdot\dots\cdot P\left(U_{d}\in A_{d}\right)=P(U\in A_{1}\times\dots\times A_{d})
\]
and the density function of $\{kY\}_{0}$ is uniformly bounded by
$\left(\frac{1}{\sqrt{2\pi}}\sup_{x\in[0,1)}g(x)\right)^{d}$.
\end{proof}
We say that $C\subseteq\mathbb{R}^{d}/\mathbb{Z}^{d}$ is a \textit{box
in $\mathbb{R}^{d}/\mathbb{Z}^{d}$} if there exists $0\leq a_{j}\leq b_{j}\leq1$
($j=1,\dots,d$) such that $C=\{A\}=\left\{ \{a\}:a\in A\right\} $
for $A=[a_{1},b_{1})\times\dots\times[a_{d},b_{d})$.
\begin{lem}
\label{lem:mixing}We have that $R$ is mixing, i.e.
\[
\lim_{k\rightarrow\infty}R(A_{1}\cap S^{-k}(A_{2}))=R(A_{1})R(A_{2})
\]
for Borel sets $A_{1},A_{2}\subseteq C([0,1],d)^{\mathbb{Z}}\times\left(\mathbb{R}^{d}/\mathbb{Z}^{d}\right)^{\mathbb{Z}}$.
\end{lem}

\begin{proof}
Since $R$ is shift invariant by Lemma \ref{lem:shift invariance},
due to \cite[Theorem 1.17]{Walters ergodic} it is enough to show
that the statement holds for every $A_{1}$ and $A_{2}$ taken from
a semi-ring of Borel sets that generates the Borel $\sigma$-algebra.
Hence it is enough to show the statement for Borel sets of the form
$(H_{i})_{i=-\infty}^{\infty}\times(C_{i})_{i=-\infty}^{\infty}$
where $H_{i}\subseteq C([0,1],d)$ is a Borel set, $C_{i}$ is a box
in $\mathbb{R}^{d}/\mathbb{Z}^{d}$ for every $i\in\mathbb{Z}$ and
there exists $n\in\mathbb{N}$ such that $H_{i}=C([0,1],d)$ and $C_{i}=\mathbb{R}^{d}/\mathbb{Z}^{d}$
for $i\in\mathbb{Z}\setminus[-n,n]$.

Let $A_{1}=(H_{i})_{i=-\infty}^{\infty}\times(C_{i})_{i=-\infty}^{\infty}$
and $A_{2}=(G_{i})_{i=-\infty}^{\infty}\times(D_{i})_{i=-\infty}^{\infty}$
where $H_{i},G_{i}\subseteq[0,1)^{d}$ are Borel sets, $C_{i}$, $D_{i}$
are boxes in $\mathbb{R}^{d}/\mathbb{Z}^{d}$ for every $i\in\mathbb{Z}$
and there exists $n\in\mathbb{N}$ such that $H_{i}=G_{i}=C([0,1],d)$
and $C_{i}=D_{i}=[0,1)^{d}$ for $i\in\mathbb{Z}\setminus[-n,n]$
(note, that without the loss of generality we can assume that $n$
is the same for both $A_{1}$ and $A_{2}$). Then
\begin{equation}
R(A_{1}\cap S^{-k}(A_{2}))=E\left(\prod_{i=-n}^{n}I_{B_{i}\in H_{i}}I_{X_{i}\in C_{i}}I_{B_{i-k}\in G_{i}}I_{X_{i-k}\in D_{i}}\right).\label{eq:mix1}
\end{equation}
Let $k>2n$. Then
\[
E\left(\prod_{i=-n}^{n}I_{B_{i}\in H_{i}}I_{X_{i}\in C_{i}}I_{B_{i-k}\in G_{i}}I_{X_{i-k}\in D_{i}}\mid X_{-n-k},(B_{i})_{i=-n}^{n},(B_{i})_{i=-n-k}^{n-k}\right)=
\]
\begin{equation}
\left(\prod_{i=-n}^{n}I_{B_{i}\in H_{i}}I_{B_{i-k}\in G_{i}}I_{X_{i-k}\in D_{i}}\right)E\left(\prod_{i=-n}^{n}I_{X_{i}\in C_{i}}\mid X_{-n-k},(B_{i})_{i=-n}^{n},(B_{i})_{i=-n-k}^{n-k}\right)\label{eq:mix2}
\end{equation}
because $X_{i-k}$ depends only on $\left(X_{-n-k},(B_{i})_{i=-n-k}^{n-k}\right)$.
We have that
\[
X_{i}=\left\{ W(i)\right\} =\left\{ W(i)-W(-n)\right\} +\left\{ W(-n)-W(n-k)\right\} +\left\{ W(n-k)\right\} 
\]
where $X_{n-k}=\left\{ W(n-k)\right\} $ depends only on $\left(X_{-n-k},(B_{i})_{i=-n-k}^{n-k}\right)$,
$Z_{i}:=\left\{ W(i)-W(-n)\right\} =\left\{ W_{0}(i)-W_{0}(-n)\right\} $
depends only on $(B_{i})_{i=-n}^{n}$, $\left\{ W(-n)-W(n-k)\right\} =\left\{ W_{0}(-n)-W_{0}(n-k)\right\} $
is independent of $\left(X_{-n-k},(B_{i})_{i=-n}^{n},(B_{i})_{i=-n-k}^{n-k}\right)$,
and $\left\{ W(-n)-W(n-k)\right\} $ and $\left\{ (k-2n)Y\right\} $
have the same distribution for a standard normally distributed vector
$Y$ that is independent of $\left(X_{-n-k},(B_{i})_{i=-n}^{n},(B_{i})_{i=-n-k}^{n-k}\right)$.
Thus
\[
E\left(\prod_{i=-n}^{n}I_{X_{i}\in C_{i}}\mid X_{-n-k},(B_{i})_{i=-n}^{n},(B_{i})_{i=-n-k}^{n-k}\right)=
\]
\[
E\left(\prod_{i=-n}^{n}I_{\left\{ (k-2n)Y\right\} \in C_{i}-Z_{i}-X_{n-k}}\mid X_{-n-k},(B_{i})_{i=-n}^{n},(B_{i})_{i=-n-k}^{n-k}\right)=
\]
\[
E\left(I_{\left\{ (k-2n)Y\right\} \in\bigcap_{i=-n}^{n}\left(C_{i}-Z_{i}-X_{n-k}\right)}\mid X_{-n-k},(B_{i})_{i=-n}^{n},(B_{i})_{i=-n-k}^{n-k}\right)
\]
and so it follows from (\ref{eq:mix2}) that
\[
E\left(\prod_{i=-n}^{n}I_{B_{i}\in H_{i}}I_{X_{i}\in C_{i}}I_{B_{i-k}\in G_{i}}I_{X_{i-k}\in D_{i}}\mid X_{-n-k},(B_{i})_{i=-n}^{n},(B_{i})_{i=-n-k}^{n-k}\right)=
\]
\begin{equation}
\left(\prod_{i=-n}^{n}I_{B_{i}\in H_{i}}I_{B_{i-k}\in G_{i}}I_{X_{i-k}\in D_{i}}\right)E\left(I_{\left\{ (k-2n)Y\right\} \in\bigcap_{i=-n}^{n}\left(C_{i}-Z_{i}-X_{n-k}\right)}\mid X_{-n-k},(B_{i})_{i=-n}^{n},(B_{i})_{i=-n-k}^{n-k}\right).\label{eq:mix3}
\end{equation}
Let $\left(\widetilde{B_{i}},\widetilde{X_{i}}\right)_{i=-n}^{n}$
be a random variable with the same distribution as $\left(B_{i-k},X_{i-k}\right)_{i=-n}^{n}$
such that $\left(\widetilde{B_{i}},\widetilde{X_{i}}\right)_{i=-n}^{n}$
is independent of $\left(Y,X_{-n-k},(B_{i})_{i=-n}^{n},(B_{i})_{i=-n-k}^{n-k}\right)$.
Then by the shift invariance (Lemma \ref{lem:shift invariance}) $\left(\widetilde{B_{i}},\widetilde{X_{i}}\right)_{i=-n}^{n}$
and $\left(B_{i},X_{i}\right)_{i=-n}^{n}$ have the same distribution.
Then
\[
\left(\prod_{i=-n}^{n}I_{B_{i}\in H_{i}}I_{B_{i-k}\in G_{i}}I_{X_{i-k}\in D_{i}}\right)E\left(I_{\left\{ (k-2n)Y\right\} \in\bigcap_{i=-n}^{n}\left(C_{i}-Z_{i}-X_{n-k}\right)}\mid X_{-n-k},(B_{i})_{i=-n}^{n},(B_{i})_{i=-n-k}^{n-k}\right)
\]
and
\[
\left(\prod_{i=-n}^{n}I_{B_{i}\in H_{i}}I_{\widetilde{B_{i}}\in G_{i}}I_{\widetilde{X_{i}}\in D_{i}}\right)E\left(I_{\left\{ (k-2n)Y\right\} \in\bigcap_{i=-n}^{n}\left(C_{i}-Z_{i}-\widetilde{X_{n}}\right)}\mid\widetilde{X_{-n}},(B_{i})_{i=-n}^{n},(\widetilde{B_{i}})_{i=-n}^{n}\right)
\]
have the same distribution. Hence by (\ref{eq:mix1}) and (\ref{eq:mix3})
it follows that
\[
\lim_{k\rightarrow\infty}R(A_{1}\cap S^{-k}(A_{2}))=
\]
\begin{equation}
\lim_{k\rightarrow\infty}E\left(\left(\prod_{i=-n}^{n}I_{B_{i}\in H_{i}}I_{\widetilde{B_{i}}\in G_{i}}I_{\widetilde{X_{i}}\in D_{i}}\right)E\left(I_{\left\{ (k-2n)Y\right\} \in\bigcap_{i=-n}^{n}\left(C_{i}-Z_{i}-\widetilde{X_{n}}\right)}\mid\widetilde{X_{-n}},(B_{i})_{i=-n}^{n},(\widetilde{B_{i}})_{i=-n}^{n}\right)\right).\label{eq:mix4}
\end{equation}
It follows from Lemma \ref{lem:kY} that
\[
\lim_{k\rightarrow\infty}E\left(I_{\left\{ (k-2n)Y\right\} \in\bigcap_{i=-n}^{n}\left(C_{i}-Z_{i}-\widetilde{X_{n}}\right)}\mid\widetilde{X_{-n}},(B_{i})_{i=-n}^{n},(\widetilde{B_{i}})_{i=-n}^{n}\right)=
\]
\[
E\left(I_{U\in\bigcap_{i=-n}^{n}\left(C_{i}-Z_{i}-\widetilde{X_{n}}\right)}\mid\widetilde{X_{-n}},(B_{i})_{i=-n}^{n},(\widetilde{B_{i}})_{i=-n}^{n}\right)=
\]
\[
E\left(I_{U\in\bigcap_{i=-n}^{n}\left(C_{i}-Z_{i}\right)}\mid\widetilde{X_{-n}},(B_{i})_{i=-n}^{n},(\widetilde{B_{i}})_{i=-n}^{n}\right)=
\]
almost surely where $U$ is a random variable that is independent
of $\left(\widetilde{X_{-n}},(B_{i})_{i=-n}^{n},(\widetilde{B_{i}})_{i=-n}^{n}\right)$
and is uniformly distributed in $\mathbb{R}^{d}/\mathbb{Z}^{d}$.
Thus by the dominated convergence theorem and (\ref{eq:mix4}) it
follows that
\[
\lim_{k\rightarrow\infty}R(A_{1}\cap S^{-k}(A_{2}))=
\]
\begin{equation}
E\left(\left(\prod_{i=-n}^{n}I_{B_{i}\in H_{i}}I_{\widetilde{B_{i}}\in G_{i}}I_{\widetilde{X_{i}}\in D_{i}}\right)E\left(I_{U\in\bigcap_{i=-n}^{n}\left(C_{i}-Z_{i}\right)}\mid\widetilde{X_{-n}},(B_{i})_{i=-n}^{n},(\widetilde{B_{i}})_{i=-n}^{n}\right)\right).\label{eq:mix5}
\end{equation}
Since $\left(U,(Z_{i})_{i=-n}^{n},(B_{i})_{i=-n}^{n}\right)$ is independent
of $\left(\widetilde{X_{-n}},(\widetilde{B_{i}})_{i=-n}^{n}\right)$
it follows that
\[
E\left(I_{U\in\bigcap_{i=-n}^{n}\left(C_{i}-Z_{i}\right)}\mid\widetilde{X_{-n}},(B_{i})_{i=-n}^{n},(\widetilde{B_{i}})_{i=-n}^{n}\right)=E\left(I_{U\in\bigcap_{i=-n}^{n}\left(C_{i}-Z_{i}\right)}\mid(B_{i})_{i=-n}^{n}\right)=
\]
\[
E\left(\prod_{i=-n}^{n}I_{U+Z_{i}\in C_{i}}\mid(B_{i})_{i=-n}^{n}\right).
\]
Thus by (\ref{eq:mix5})
\[
\lim_{k\rightarrow\infty}R(A_{1}\cap S^{-k}(A_{2}))=
\]
\[
E\left(\left(\prod_{i=-n}^{n}I_{B_{i}\in H_{i}}I_{\widetilde{B_{i}}\in G_{i}}I_{\widetilde{X_{i}}\in D_{i}}\right)E\left(\prod_{i=-n}^{n}I_{U+Z_{i}\in C_{i}}\mid(B_{i})_{i=-n}^{n}\right)\right)
\]
\[
E\left(\left(\prod_{i=-n}^{n}I_{\widetilde{B_{i}}\in G_{i}}I_{\widetilde{X_{i}}\in D_{i}}\right)E\left(\prod_{i=-n}^{n}I_{B_{i}\in H_{i}}I_{U+Z_{i}\in C_{i}}\mid(B_{i})_{i=-n}^{n}\right)\right)=
\]
\begin{equation}
R(A_{2})E\left(\prod_{i=-n}^{n}I_{B_{i}\in H_{i}}I_{U+Z_{i}\in C_{i}}\right)\label{eq:mix6}
\end{equation}
since $\prod_{i=-n}^{n}I_{\widetilde{B_{i}}\in G_{i}}I_{\widetilde{X_{i}}\in D_{i}}$
and $E\left(\prod_{i=-n}^{n}I_{B_{i}\in H_{i}}I_{U+Z_{i}\in C_{i}}\mid(B_{i})_{i=-n}^{n}\right)$
are independent and that $\left(\widetilde{B_{i}},\widetilde{X_{i}}\right)_{i=-n}^{n}$
and $\left(B_{i},X_{i}\right)_{i=-n}^{n}$ have the same distribution.

Similarly to the proof of Lemma \ref{lem:shift invariance} it can
be shown that $U-\left\{ W_{0}(-n)\right\} $ and $W_{0}$ are independent
and $W=X+W_{0}$ and $\left\{ U-\left\{ W_{0}(-n)\right\} \right\} _{0}+W_{0}$
have the same distribution. Since $Z_{i}=\left\{ W_{0}(i)\right\} -\left\{ W_{0}(-n)\right\} $
it follows that
\[
E\left(\prod_{i=-n}^{n}I_{B_{i}\in H_{i}}I_{U+Z_{i}\in C_{i}}\right)=E\left(\prod_{i=-n}^{n}I_{B_{i}\in H_{i}}I_{\{W(i)\}\in C_{i}}\right)=R(A_{1})
\]
and this finishes the proof combined with (\ref{eq:mix6}).
\end{proof}
\begin{lem}
\label{lem:ergodic  l1}Let
\[
f\left(\left(B_{i},X_{i}\right)_{i=-\infty}^{\infty}\right)=\#\left\{ Q\in\mathcal{Q}_{0}^{*}:W([0,1))\cap Q\neq\emptyset,W((-\infty,0))\cap Q=\emptyset\right\} ,
\]
i.e. the number of dyadic cubes of side length $1$ that is visited
by $W$ in the time interval $[0,1)$ but never visited before. Let
\[
g\left(\left(B_{i},X_{i}\right)_{i=-\infty}^{\infty}\right)=\#\left\{ Q\in\mathcal{Q}_{0}^{*}:W([0,1])\cap Q\neq\emptyset\right\} ,
\]
i.e. the number of dyadic cubes of side length $1$ that is visited
by $W$ in the time interval $[0,1]$. Then $f$ and $g$ are $\mathcal{L}^{1}$
functions with respect to the probability measure $R$.
\end{lem}

\begin{proof}
It is easy to see that $f$ and $g$ are measurable and that $0\leq f\leq g$.
Thus to prove the statement it is enough to show that
\[
E\left(g\left(\left(B_{i},X_{i}\right)_{i=-\infty}^{\infty}\right)\right)<\infty.
\]
Let $Y_{1}(t)$ be a Brownian motion in $\mathbb{R}$ started at $0$
and let $x_{1}\in[0,1)$. Then
\[
\#\left\{ Q\in\mathcal{Q}_{0}^{*}:x_{1}+Y_{1}([0,1])\cap Q\neq\emptyset\right\} \leq2+2\max_{t\in[0,1]}\left|x_{1}+Y_{1}(t)\right|\leq
\]
\begin{equation}
2+2x_{1}+2\max_{t\in[0,1]}\left|Y_{1}(t)\right|\leq4+2\max_{t\in[0,1]}\left|Y_{1}(t)\right|.\label{eq:1helper}
\end{equation}
Note that $\mathcal{Q}_{0}^{*}$ consists dyadic intervals of the
line in (\ref{eq:1helper}) because $Y_{1}$ is a Brownian motion
in $\mathbb{R}$, however, later on $Y$ is a Brownian path in $\mathbb{R}^{d}$
and so $\mathcal{Q}_{0}^{*}$ consists dyadic cubes of $\mathbb{R}^{d}$.
Let $Y(t)=(Y_{1}(t),\dots,Y_{d}(t))$ be a Brownian motion in $\mathbb{\mathbb{R}}^{d}$
started at $0$ and let $(x_{1},\dots,x_{d})\in[0,1)^{d}$. Then by
(\ref{eq:1helper})
\[
\#\left\{ Q\in\mathcal{Q}_{0}^{*}:x+Y([0,1])\cap Q\neq\emptyset\right\} \leq\prod_{i=1}^{d}\left(4+2\max_{t\in[0,1]}\left|Y_{i}(t)\right|\right).
\]
Since $Y_{1},\dots,Y_{d}$ are mutually independent it follows that
\[
E\left(g\left(\left(B_{i},X_{i}\right)_{i=-\infty}^{\infty}\right)\right)\leq E\left(\prod_{i=1}^{d}\left(4+2\max_{t\in[0,1]}\left|Y_{i}(t)\right|\right)\right)\leq\prod_{i=1}^{d}\left(4+2E\left(\max_{t\in[0,1]}\left|Y_{i}(t)\right|\right)\right)<\infty
\]
because $E\left(\max_{t\in[0,1]}\left|Y_{i}(t)\right|\right)<\infty$
by \cite[Theorem 2.21]{Peres-Morters-Broanian motion}.
\end{proof}
\begin{lem}
\label{prop:BlimNminus}There exist $0\leq\alpha\leq\beta<\infty$
such that
\begin{equation}
\lim_{N\rightarrow\infty}N^{-1}\#\left\{ Q\in\mathcal{Q}_{0}^{*}:W([0,N])\cap Q\neq\emptyset,W((-\infty,0))\cap Q=\emptyset\right\} =\alpha\label{eq:almost alpha}
\end{equation}
$R$ almost surely and
\begin{equation}
\lim_{N\rightarrow\infty}N^{-1}\sum_{i=0}^{N-1}\#\left\{ Q\in\mathcal{Q}_{0}^{*}:W([i,i+1])\cap Q\neq\emptyset\right\} =\beta\label{eq:almost beta}
\end{equation}
$R$ almost surely.
\end{lem}

\begin{proof}
Let $f$ and $g$ be as in Lemma \ref{lem:ergodic  l1}. Then
\[
\#\left\{ Q\in\mathcal{Q}_{0}^{*}:W([0,N])\cap Q\neq\emptyset,W((-\infty,0))\cap Q=\emptyset\right\} =\sum_{i=0}^{N-1}f\left(S^{i}\left(B_{i},X_{i}\right)_{i=-\infty}^{\infty}\right)
\]
(note that $W(N)\notin\partial Q$ almost surely for every $N\in\mathbb{Z}$
and $Q\in\mathcal{Q}_{0}^{*}$) and
\[
\sum_{i=0}^{N-1}\#\left\{ Q\in\mathcal{Q}_{0}^{*}:W([i,i+1])\cap Q\neq\emptyset\right\} =\sum_{i=0}^{N-1}g\left(S^{i}\left(B_{i},X_{i}\right)_{i=-\infty}^{\infty}\right).
\]
By Lemma \ref{lem:ergodic  l1} the functions $f$ and $g$ are $\mathcal{L}^{1}$
functions with respect to the probability measure $R$, by Lemma \ref{lem:shift invariance}
we have that $R$ is a shift invariant measure and by Lemma \ref{lem:mixing}
we have that $R$ is an ergodic measure since mixing implies ergodicity.
Hence by Birkhoff`s ergodic theorem (\cite[Theorem 1.14]{Walters ergodic}
and the Remark after the theorem) the statement follows for
\[
\alpha=E\left(f\left(\left(B_{i},X_{i}\right)_{i=-\infty}^{\infty}\right)\right)
\]
and
\[
\beta=E\left(g\left(\left(B_{i},X_{i}\right)_{i=-\infty}^{\infty}\right)\right).
\]
\end{proof}
\begin{prop}
\label{prop:BlimN}Let $0\leq\alpha\leq\beta<\infty$ be as in Lemma
\ref{prop:BlimNminus}. Let $\widehat{W_{n}}(t)$ be two-sided Brownian
motions in $\mathbb{R}^{d}$ started at $0$ for every $n\in\mathbb{N}$.
Let $Y_{n}$ be a random variable for every $n\in\mathbb{N}$, that
is independent of $\widehat{W_{n}}$ (not necessarily i.i.d.), takes
values in $\mathbb{R}^{d}$, the distributions of $Y_{n}$ are absolutely
continuous with respect to the Lebesgue measure and the density functions
of $\left\{ Y_{n}\right\} _{0}$ are uniformly bounded. Let $\widetilde{W_{n}}=\widehat{W_{n}}+Y_{n}$.
Then
\begin{equation}
f_{N}(\omega)=N^{-1}\#\left\{ Q\in\mathcal{Q}_{0}^{*}:\widetilde{W_{N}}([0,N])\cap Q\neq\emptyset,\widetilde{W_{N}}((-\infty,0))\cap Q=\emptyset\right\} \label{eq:alpha lim}
\end{equation}
converges to $\alpha$ in probability as $N$ goes to $\infty$ and
\begin{equation}
g_{N}(\omega)=N^{-1}\sum_{i=0}^{N-1}\#\left\{ Q\in\mathcal{Q}_{0}^{*}:\widetilde{W_{N}}([i,i+1])\cap Q\neq\emptyset\right\} \label{eq:beta lim}
\end{equation}
converges to $\beta$ in probability as $N$ goes to $\infty$.
\end{prop}

\begin{proof}
The quantity in (\ref{eq:alpha lim}) does not change if we replace
$Y_{n}$ by $\left\{ Y_{n}\right\} _{0}$, hence without the loss
of generality we assume that $Y_{n}=\left\{ Y_{n}\right\} _{0}$.
Let $\varepsilon>0$ be fixed. Let
\[
h_{N}(\omega)=N^{-1}\#\left\{ Q\in\mathcal{Q}_{0}^{*}:W([0,N])\cap Q\neq\emptyset,W((-\infty,0))\cap Q=\emptyset\right\} ,
\]
then
\begin{equation}
P(h_{N}<\alpha-\varepsilon)=\intop_{[0,1)^{d}}E\left(I_{f_{N}<\alpha-\varepsilon}\mid Y_{N}=x\right)\mathrm{d}x\label{eq:inteq}
\end{equation}
because $\widehat{W_{n}}$ and $W_{0}$ have the same distribution.
Let $D_{n}(x)$ be the density function of $\left\{ Y_{n}\right\} _{0}$
and let $M<\infty$ be the uniform bound that $D_{n}(x)\leq M$ for
every $n\in\mathbb{N}$ and Lebesgue almost every $x\in[0,1)^{d}$.
Then
\begin{equation}
P(f_{N}<\alpha-\varepsilon)=\intop_{[0,1)^{d}}E\left(I_{f_{N}<\alpha-\varepsilon}\mid Y_{N}=x\right)D_{N}(x)\mathrm{d}x\leq M\cdot P(h_{N}<\alpha-\varepsilon)\label{eq:pf}
\end{equation}
by (\ref{eq:inteq}). We have that $h_{N}$ converges to $\alpha$
almost surely by Lemma \ref{prop:BlimNminus} and so in probability.
Hence
\[
\lim_{N\rightarrow\infty}P(f_{N}<\alpha-\varepsilon)=0
\]
by (\ref{eq:pf}). Similarly we can show that
\[
\lim_{N\rightarrow\infty}P(f_{N}>\alpha+\varepsilon)=0,
\]
and hence
\begin{equation}
\lim_{N\rightarrow\infty}P(\left|f_{N}-\alpha\right|>\varepsilon)=0,\label{eq:probvege}
\end{equation}
i.e. $f_{n}$ converges to $\alpha$ in probability as $N$ goes to
$\infty$.

The proof of
\[
\lim_{N\rightarrow\infty}P(\left|g_{N}-\beta\right|>\varepsilon)=0
\]
for every $\varepsilon>0$ is similar to the proof of (\ref{eq:probvege}),
we omit the details.
\end{proof}

\subsection{Number of small cubes intersected by the Brownian path\label{subsec:small cubes}}

We say that a number $a\in\mathbb{R}$ is a \textit{dyadic number}
if $a=n\cdot2^{-k}$ for some $n\in\mathbb{Z}$, $k\in\mathbb{N}$.
We say that an interval $J\subseteq\mathbb{R}$ is an \textit{interval
with dyadic endpoints} if there exist dyadic numbers $a<b$ such that
$J$ is one of the following intervals $(a,b)$, $[a,b]$, $[a,b)$,
$(a,b]$.

Throughout this subsection let $W_{0}(t)$ be a two-sided Brownian
motion in $\mathbb{R}^{d}$ started at $0$. For a compact interval
$I=[a,b]\subseteq\mathbb{R}$ with dyadic endpoints we define the
following random variables:
\[
f_{k}^{I}=2^{-2k}\#\left\{ Q\in\mathcal{Q}_{k}^{*}:W_{0}(I)\cap Q\neq\emptyset,W_{0}((-\infty,a))\cap Q=\emptyset\right\} 
\]
and
\[
g_{k}^{I}=2^{-2k}\sum_{i=0}^{N_{k}-1}\#\left\{ Q\in\mathcal{Q}_{k}^{*}:W_{0}([a+i2^{-2k},a+(i+1)2^{-2k}])\cap Q\neq\emptyset\right\} 
\]
where $N_{k}=(b-a)2^{2k}$ (note, that $N_{k}$ is an integer for
large enough $k$ because $(b-a)$ is also a dyadic number).
\begin{lem}
\label{lem:scaling invariance}Let $r>0$ be fixed. Then $\left(W_{0}(t):t\in\mathbb{R}\right)$
and $\left(r^{-1}W_{0}(r^{2}t):t\in\mathbb{R}\right)$ have the same
distribution.
\end{lem}

Lemma \ref{lem:scaling invariance} is a folklore in the theory of
Brownian motions, see for example \cite[Lemma 1.7]{Peres-Morters-Broanian motion}.
\begin{lem}
\label{lem:cubes lim minus}Let $0\leq\alpha\leq\beta<\infty$ be
as in Lemma \ref{prop:BlimNminus}. Let $W_{0}(t)$ be a two-sided
Brownian motion in $\mathbb{R}^{d}$ ($d\geq3$) started at $0$.
Then for every $0<a<b<\infty$ dyadic numbers and for $I=[a,b]$ we
have that $f_{k}^{I}$ converges to $\alpha(b-a)$ in probability
as $k$ goes to $\infty$ and $g_{k}^{I}$ converges to $\beta(b-a)$
in probability as $k$ goes to $\infty$.
\end{lem}

\begin{proof}
Let $0<a<b<\infty$ be fixed dyadic numbers. Assume that $k$ is large
enough that $N_{k}$ is an integer. Let $\widetilde{W_{k}}(t)=2^{k}W_{0}(2^{-2k}t+a))$
for every $t\in\mathbb{R}$. Then
\[
\#\left\{ Q\in\mathcal{Q}_{k}^{*}:W_{0}(I)\cap Q\neq\emptyset,W_{0}((-\infty,a))\cap Q=\emptyset\right\} =
\]
\[
\#\left\{ Q\in\mathcal{Q}_{0}^{*}:\widetilde{W_{k}}([0,N_{k}])\cap Q\neq\emptyset,\widetilde{W_{k}}((-\infty,0))\cap Q=\emptyset\right\} 
\]
and
\[
\#\left\{ Q\in\mathcal{Q}_{k}^{*}:W_{0}([a+i2^{-2k},a+(i+1)2^{-2k}])\cap Q\neq\emptyset\right\} =
\]
\[
\#\left\{ Q\in\mathcal{Q}_{0}^{*}:\widetilde{W_{k}}([i,i+1])\cap Q\neq\emptyset\right\} .
\]
Hence, to complete the proof, we need to show that
\begin{equation}
N_{k}^{-1}\#\left\{ Q\in\mathcal{Q}_{0}^{*}:\widetilde{W_{k}}([0,N_{k})\cap Q\neq\emptyset,\widetilde{W_{n}}((-\infty,a))\cap Q=\emptyset\right\} \label{eq:eqegyyes}
\end{equation}
converges to $\alpha$ in probability and
\begin{equation}
N_{k}^{-1}\sum_{i=0}^{N_{k}-1}\#\left\{ Q\in\mathcal{Q}_{0}^{*}:\widetilde{W_{k}}([i,i+1])\cap Q\neq\emptyset\right\} \label{eq:eqkettes}
\end{equation}
converges to $\beta$ in probability (note that $2^{-2k}=(b-a)\cdot N_{k}^{-1}$).

We have that $\left(W_{0}(t):t\in\mathbb{R}\right)$ and $\left(W_{0}(t+a)-W_{0}(a):t\in\mathbb{R}\right)$
have the same distribution. Then for $\widehat{W_{k}}(t)=2^{k}W_{0}(2^{-2k}t+a))-2^{k}W_{0}(a)$
we get that $\left(\widehat{W_{k}}(t):t\in\mathbb{R}\right)$ and
$\left(W_{0}(t):t\in\mathbb{R}\right)$ have the same distribution
by Lemma \ref{lem:scaling invariance} for $r=2^{-k}$. Let $Y_{k}=2^{k}W_{0}(a)$.
Then the distributions of $Y_{k}$ are absolutely continuous with
respect to the Lebesgue measure and the density functions of $\left\{ Y_{k}\right\} _{0}$
are uniformly bounded by Lemma \ref{lem:kY}. Thus the quantities
in (\ref{eq:eqegyyes}) and (\ref{eq:eqkettes}) are converging in
probability to the desired limit by Proposition \ref{prop:BlimN}.
\end{proof}
\begin{lem}
\label{lem:0time}Let $W_{0}(t)$ be a two-sided Brownian motion in
$\mathbb{R}^{d}$ ($d\geq3$) started at $0$ and let $I\subseteq\mathbb{R}$
be a compact interval. Let
\[
D=\left\{ t\in\mathbb{R}\setminus\mathrm{int}I:\exists s\in I,W_{0}(t)=W_{0}(s)\right\} .
\]
Then $D$ is a random compact set and
\[
\dim_{H}D\leq1/2
\]
almost surely. In particular, $\lambda(D)=0$ almost surely.
\end{lem}

\begin{proof}
We have that
\[
W_{0}(D)=W_{0}(I)\cap W_{0}(\mathbb{R}\setminus\mathrm{int}I)
\]
hence almost surely the intersection of a compact set and a closed
set and so $W_{0}(D)$ is almost surely a compact set. It follows
from \cite[Lemma 9.4]{Peres-Morters-Broanian motion} and the reversibility
that $\dim_{H}W_{0}(D)\leq1$ almost surely for $d=3$. It follows
from \cite[Theorem 9.1]{Peres-Morters-Broanian motion} and the Markov
property that $D=\partial I$ almost surely for $d\geq4$. So $\dim_{H}W_{0}(D)\leq1$
almost surely. Since $W_{0}$ is almost surely a continuous function
it follows that 
\[
D=W_{0}^{-1}(W_{0}(D))\cap(\mathbb{R}\setminus\mathrm{int}I)
\]
is a closed set and bounded because of the transience of the Brownian
motion \cite[Theorem 3.20]{Peres-Morters-Broanian motion}, thus $D$
is compact. It follows from the fact that $\dim_{H}W_{0}(D)\leq1$
almost surely and from Kaufman`s dimension doubling theorem \cite[Theorem 9.28]{Peres-Morters-Broanian motion}
that
\[
\dim_{H}D\leq1/2
\]
almost surely.
\end{proof}
\begin{lem}
\label{lem:dyadic cover}Let $W_{0}(t)$ be a two-sided Brownian motion
in $\mathbb{R}^{d}$ ($d\geq3$) started at $0$, let $I\subseteq\mathbb{R}$
be a compact interval and let $\varepsilon>0$. Then almost surely
there exist finitely many random open intervals $J_{1},\dots,J_{m}$
with dyadic endpoints such that
\[
\mathrm{dist}\left(W_{0}(I),W_{0}(\mathbb{R}\setminus\left(\mathrm{int}I\bigcup\left(\cup_{i=1}^{m}J_{i}\right)\right)\right)>0
\]
and
\[
\sum_{i=1}^{m}\lambda(J_{i})<\varepsilon.
\]
\end{lem}

\begin{proof}
Let $D$ be as in Lemma \ref{lem:0time}. Since $\lambda(D)=0$ and
$D$ is compact almost surely we can cover $D$ with finitely many
random open intervals $J_{1},\dots,J_{m}$ with dyadic endpoints such
that $\sum_{i=1}^{m}\lambda(J_{i})<\varepsilon$. It is not hard to
see that we can choose the open intervals $J_{1},\dots,J_{m}$ on
a Borel measurable way, so we legitimately say random open intervals.

We have that $W_{0}(I)$ and $W_{0}(\mathbb{R}\setminus\left(\mathrm{int}I\bigcup\left(\cup_{i=1}^{m}J_{i}\right)\right)$
are disjoint almost surely due to the definition of $D$ (note that
the endpoints of $I$ are contained in $D$). Thus
\[
\mathrm{dist}\left(W_{0}(I),W_{0}(\mathbb{R}\setminus\left(\mathrm{int}I\bigcup\left(\cup_{i=1}^{m}J_{i}\right)\right)\right)>0
\]
almost surely because $W_{0}(I)$ is almost surely compact and $W_{0}(\mathbb{R}\setminus\left(\mathrm{int}I\bigcup\left(\cup_{i=1}^{m}J_{i}\right)\right)$
is almost surely closed.
\end{proof}
\begin{rem}
\label{nota:Q*}It follows from \cite[Corollary 3.19]{Peres-Morters-Broanian motion}
that $W(\mathbb{R})\cap\{0\}=\emptyset$ almost surely. Hence for
every deterministic interval $I\subseteq\mathbb{R}$ we have that
\[
\left\{ Q\in\mathcal{Q}_{k}:W(I)\cap Q\neq\emptyset\right\} =\left\{ Q\in\mathcal{Q}_{k}^{*}:W(I)\cap Q\neq\emptyset\right\} 
\]
almost surely. Similar holds for $W$ replaced by $W_{0}$ when $I\subseteq(0,\infty)$
because $W_{0}((0,\infty))\cap\{0\}=\emptyset$ almost surely.
\end{rem}

\begin{prop}
\label{lem:intunlim}Let $0\leq\alpha<\infty$ be as in Lemma \ref{lem:cubes lim minus}.
Let $W_{0}(t)$ be a two-sided Brownian motion in $\mathbb{R}^{d}$
($d\geq3$) started at $0$. Let $I=\cup_{n=1}^{N}I_{n}$ where $I_{1},\dots,I_{N}$
are disjoint compact intervals with positive dyadic endpoints. Then
\begin{equation}
h_{k}^{I}:=2^{-2k}\#\left\{ Q\in\mathcal{Q}_{k}:W_{0}(I)\cap Q\neq\emptyset\right\} \label{eq:h def}
\end{equation}
converges to $\alpha\lambda(I)$ in probability.
\end{prop}

\begin{proof}
To prove the statement of the proposition, by Lemma \ref{lem:subseq prob 0},
it is enough to show that for every $\varepsilon>0$ and for every
subsequence $\left\{ \alpha_{k}\right\} _{k=1}^{\infty}$ of $\mathbb{N}$
we can find a subsequence $\left\{ \beta_{k}\right\} _{k=1}^{\infty}$
of $\left\{ \alpha_{k}\right\} _{k=1}^{\infty}$ such that
\begin{equation}
\lim_{k\rightarrow\infty}P\left(\left|h_{\beta_{k}}^{I}-\alpha\lambda(I)\right|>\varepsilon\right)=0.\label{eq:celeq}
\end{equation}
It follows from Remark \ref{nota:Q*} that
\[
h_{k}^{I}=2^{-2k}\#\left\{ Q\in\mathcal{Q}_{k}^{*}:W_{0}(I)\cap Q\neq\emptyset\right\} 
\]
which we use throughout the proof instead of (\ref{eq:h def}).

Let $\varepsilon>0$ be fixed and $\left\{ \alpha_{k}\right\} _{k=1}^{\infty}$
be a subsequence of $\mathbb{N}$. Let $0\leq\beta<\infty$ be as
in Lemma \ref{lem:cubes lim minus}. For every compact interval $J$
with positive dyadic endpoints we have that $f_{k}^{J}$ converges
to $\alpha\lambda(J)$ in probability and $g_{k}^{J}$ converges to
$\beta\lambda(J)$ in probability by Lemma \ref{lem:cubes lim minus}.
Hence, by Lemma \ref{lem:mutual convergence subsequance}, we can
find an event $H$ with $P(H)=1$ and a subsequence $\left\{ \beta_{k}\right\} _{k=1}^{\infty}$
of $\left\{ \alpha_{k}\right\} _{k=1}^{\infty}$ such that $f_{\beta_{k}}^{J}(\omega)$
converges to $\alpha\lambda(J)$ and $g_{\beta_{k}}^{J}(\omega)$
converges to $\beta\lambda(J)$ for every interval $J$ with positive
dyadic endpoints for every outcome $\omega\in H$. We can further
assume, by Lemma \ref{lem:dyadic cover}, that there exist $m_{j}(\omega)\in\mathbb{N}$
and finitely many open intervals $J_{1}^{j}(\omega),\dots,J_{m_{j}}^{j}(\omega)$
with dyadic endpoints for $j=1,\dots,N$ such that
\begin{equation}
r_{j}:=\mathrm{dist}\left(W_{0}^{\omega}(I_{j}),W_{0}^{\omega}(\mathbb{R}\setminus\left(\mathrm{int}I_{j}\bigcup\left(\cup_{i=1}^{m_{j}}J_{i}^{j}(\omega)\right)\right)\right)>0\label{eq:distpos}
\end{equation}
and
\begin{equation}
\sum_{i=1}^{m_{j}(\omega)}\lambda(J_{i}^{j}(\omega))<\varepsilon(2N\beta)^{-1}\label{eq:lambdasumj}
\end{equation}
for every $\omega\in H$.

Let $\omega\in H$ be fixed. If $k$ is large enough that $\mathrm{diam}(Q)<2^{-1}r_{j}$
for some $j$ and $Q\in\mathcal{Q}_{k}^{*}$ then whenever $Q\cap W_{0}^{\omega}(I_{j})\neq\emptyset$
then, by (\ref{eq:distpos}), either $Q\cap W_{0}^{\omega}(I_{j})\neq\emptyset$
for $I_{j}=[a_{j},b_{j}]$ and $Q\cap W_{0}^{\omega}(-\infty,a_{j})=\emptyset$
or $Q\cap W_{0}^{\omega}(J_{i}^{j}(\omega))\neq\emptyset$ for some
$i$. Thus
\[
\sum_{j=1}^{N}f_{\beta_{k}}^{I_{j}}(\omega)\leq h_{\beta_{k}}^{I}(\omega)\leq\left(\sum_{j=1}^{N}f_{\beta_{k}}^{I_{j}}(\omega)\right)+\left(\sum_{j=1}^{N}\sum_{i=1}^{m_{j}(\omega)}g_{\beta_{k}}^{\overline{J_{i}^{j}}(\omega)}(\omega)\right)
\]
for large enough $k$, where $\overline{J_{i}}(\omega)$ is the closure
of $J_{i}(\omega)$. Hence
\[
\limsup_{k\rightarrow\infty}\left|h_{\beta_{k}}^{I}(\omega)-\sum_{j=1}^{N}f_{\beta_{k}}^{I_{j}}(\omega)\right|\leq\lim_{k\rightarrow\infty}\sum_{j=1}^{N}\sum_{i=1}^{m(\omega)}g_{\beta_{k}}^{\overline{J_{i}^{j}}(\omega)}(\omega)=\sum_{j=1}^{N}\sum_{i=1}^{m(\omega)}\beta\lambda(\overline{J_{i}}(\omega))<\varepsilon/2
\]
by (\ref{eq:lambdasumj}). Since this holds for every $\omega\in H$
it follows that
\begin{equation}
\lim_{k\rightarrow\infty}P\left(\left|h_{\beta_{k}}^{I}-\sum_{j=1}^{N}f_{\beta_{k}}^{I_{j}}(\omega)\right|>\varepsilon/2\right)=0.\label{eq:segeq}
\end{equation}

We have that
\[
P\left(\left|h_{\beta_{k}}^{I}-\alpha\lambda(I)\right|>\varepsilon\right)\leq P\left(\left|h_{\beta_{k}}^{I}-\sum_{j=1}^{N}f_{\beta_{k}}^{I_{j}}(\omega)\right|>\varepsilon/2\right)+P\left(\left|\sum_{j=1}^{N}f_{\beta_{k}}^{I_{j}}(\omega)-\alpha\lambda(I)\right|>\varepsilon/2\right),
\]
and hence (\ref{eq:celeq}) follows from (\ref{eq:segeq}) and the
fact that $\sum_{j=1}^{N}f_{\beta_{k}}^{I_{j}}(\omega)$ converges
to $\alpha\lambda(I)$ on $H$.
\end{proof}

\subsection{Occupation measure and limit measure\label{subsec:Occupation-measure-and}}

Throughout this subsection let $B_{0}(.)$ be a standard Brownian
motion in $\mathbb{R}^{d}$ ($d\geq3$) started at $0$, let $B=\{B_{0}(t):t\in[0,\infty)\}$
be the range of $B_{0}$. Let $\tau$ be the \textit{occupation measure
of $B_{0}(.)$}, that is
\begin{equation}
\tau(A)=\intop_{0}^{\infty}I_{B_{0}(t)\in A}\mathrm{d}t\label{eq:ocup def}
\end{equation}
for every Borel set $A\subseteq\mathbb{R}^{d}$, i.e. the amount of
time that the Brownian motion spends in $A$.
\begin{lem}
\label{lem:bound image}Let $Q\in\mathcal{Q}_{n}$ for some $n\in\mathbb{N}$.
Then $\lambda(\partial B_{0}^{-1}(Q))=0$ almost surely.
\end{lem}

\begin{proof}
Since $B_{0}$ is almost surely continuous it follows that $\partial B_{0}^{-1}(Q)\subseteq B_{0}^{-1}(\partial Q)$
almost surely. Hence it is enough to prove that $\lambda\left(B_{0}^{-1}(\partial Q)\right)=0$.
Let $B_{0}(t)=(B_{1}(t),\dots,B_{d}(t))$ for every $t\in[0,\infty)$.
Then
\[
\dim_{H}(B_{i}^{-1}(a))=1/2
\]
almost surely for every $i=1,\dots,d$ and $a\in\mathbb{R}$ by \cite[Theorem 9.34]{Peres-Morters-Broanian motion}.
For every side of $Q$ there exists $i\in\{1,\dots,d\}$ and $a\in\mathbb{R}$
such that the preimage of that side of $Q$ is contained in $B_{i}^{-1}(a)$.
Thus it follows that
\[
\lambda\left(B_{0}^{-1}(\partial Q)\right)=0
\]
because $Q$ has finitely many sides.
\end{proof}
\begin{lem}
\label{lem:ocup cover}Let $Q\in\mathcal{Q}_{n}$ for some $n\in\mathbb{N}$
such that $\mathrm{dist}(Q,0)>0$ and let $\varepsilon>0$ be fixed.
Then almost surely there exist random sets $I^{-}=\cup_{n=1}^{N^{-}}I_{n}^{-}$
where $I_{1}^{-},\dots,I_{N^{-}}^{-}$ are random disjoint compact
intervals with positive dyadic endpoints and $I^{+}=\cup_{n=1}^{N^{+}}I_{n}^{+}$
where $I_{1}^{+},\dots,I_{N^{+}}^{+}$ are random disjoint compact
intervals with positive dyadic endpoints such that
\[
I^{-}\subseteq B_{0}^{-1}(Q)\subseteq I^{+}
\]
and $\lambda(I^{+}\setminus I^{-})<\varepsilon$.
\end{lem}

\begin{proof}
We have that $\lambda(\partial B_{0}^{-1}(Q))=0$ almost surely by
Lemma \ref{lem:bound image}. Hence almost surely there exist random
sets $I^{-}=\cup_{n=1}^{N^{-}}I_{n}^{-}$ where $I_{1}^{-},\dots,I_{N^{-}}^{-}$
are random disjoint compact intervals with positive dyadic endpoints
and $I^{+}=\cup_{n=1}^{N^{+}}I_{n}^{+}$ where $I_{1}^{+},\dots,I_{N^{+}}^{+}$
are random disjoint compact intervals with positive dyadic endpoints
such that
\[
I^{-}\subseteq B_{0}^{-1}(Q)\subseteq I^{+}
\]
and $\lambda(I^{+}\setminus I^{-})<\varepsilon$. It is easy to see
that we can choose the intervals on a measurable way.
\end{proof}
\begin{lem}
\label{lem:fk}Let $Q\in\mathcal{Q}_{n}$ for some $n\in\mathbb{N}$
such that $\mathrm{dist}(Q,0)>0$ and let $G(x,y)$ be as in Remark
\ref{rem:green fn}. Then
\begin{equation}
\lim_{k\rightarrow\infty}\sup_{x\in Q}\frac{2^{k(2-d)}\left\Vert x\right\Vert ^{2-d}}{P(Q_{k}(x)\cap B\neq\emptyset)}=C_{G}([0,1]^{d})\label{eq:sup}
\end{equation}
and
\begin{equation}
\lim_{k\rightarrow\infty}\inf_{x\in Q}\frac{2^{k(2-d)}\left\Vert x\right\Vert ^{2-d}}{P(Q_{k}(x)\cap B\neq\emptyset)}=C_{G}([0,1]^{d}).\label{eq:inf}
\end{equation}
\end{lem}

\begin{proof}
For every $Q=[a_{1},b_{1})\times\dots\times[a_{d}.b_{d})\in Q_{k}$
let $x_{Q}=(a_{1},\dots,a_{d})\in Q$. By applying Lemma \ref{lem:hitting prob of compact sets}
and Remark \ref{rem:boundary dont count} to $A=[0,1)^{d}$ it follows
that
\begin{equation}
\lim_{k\rightarrow\infty}\sup_{x\in Q}\frac{2^{k(2-d)}\left\Vert x_{Q_{k}(x)}\right\Vert ^{2-d}}{P(Q_{k}(x)\cap B\neq\emptyset)}=C_{G}([0,1]^{d})\label{eq:xQ}
\end{equation}
and
\begin{equation}
\lim_{k\rightarrow\infty}\inf_{x\in Q}\frac{2^{k(2-d)}\left\Vert x_{Q_{k}(x)}\right\Vert ^{2-d}}{P(Q_{k}(x)\cap B\neq\emptyset)}=C_{G}([0,1]^{d}).\label{xqin}
\end{equation}
It is easy to show that
\[
\frac{\left\Vert x\right\Vert ^{2-d}}{\left\Vert x_{Q_{k}(x)}\right\Vert ^{2-d}}
\]
converges to $1$ uniformly on $Q$. Hence (\ref{eq:sup}) follows
from (\ref{eq:xQ}) and (\ref{eq:inf}) follows from (\ref{xqin}).
\end{proof}
\begin{lem}
\label{lem:ocub box}Let $0\leq\alpha<\infty$ be as in Proposition
\ref{lem:intunlim}, $G(x,y)$ be as in Remark \ref{rem:green fn}
and $Q\in\mathcal{Q}_{n}$ for some $n\in\mathbb{N}$ such that $\mathrm{dist}(Q,0)>0$.
Then
\[
\intop_{Q}\left\Vert x\right\Vert ^{2-d}\mathrm{d}\mathcal{C}_{k}(\lambda)
\]
converges to $\gamma\cdot\tau(Q)$ in probability as $k$ goes to
$\infty$, where
\begin{equation}
\gamma=\alpha\cdot C_{G}([0,1]^{d}).\label{eq:gamma}
\end{equation}
\end{lem}

\begin{proof}
Let
\[
M_{k}=\sup_{x\in Q}\frac{2^{k(2-d)}\left\Vert x\right\Vert ^{2-d}}{P(Q_{k}(x)\cap B\neq\emptyset)},
\]
then
\[
\frac{\intop_{S}\left\Vert x\right\Vert ^{2-d}\mathrm{d}x}{P(S\cap B\neq\emptyset)}\leq2^{-2k}M_{k}
\]
for every $S\in\mathcal{Q}_{k}$, $S\subseteq Q$ for every $k\geq n$.
Thus
\[
\intop_{Q}\left\Vert x\right\Vert ^{2-d}\mathrm{d}\mathcal{C}_{k}(\lambda)=\sum_{\begin{array}{c}
S\in\mathcal{Q}_{k}\\
S\subseteq Q
\end{array}}\frac{I_{S\cap B\neq\emptyset}}{P(S\cap B\neq\emptyset)}\intop_{S}\left\Vert x\right\Vert ^{2-d}\mathrm{d}x\leq
\]
\begin{equation}
\sum_{\begin{array}{c}
S\in\mathcal{Q}_{k}\\
S\subseteq Q
\end{array}}I_{S\cap B\neq\emptyset}2^{-2k}M_{k}=M_{k}2^{-2k}\#\left\{ S\in\mathcal{Q}_{k}:S\subseteq Q,\,B\cap S\neq\emptyset\right\} .\label{eq:nagymk}
\end{equation}
Similarly it can be shown that
\begin{equation}
\intop_{Q}\left\Vert x\right\Vert ^{2-d}\mathrm{d}\mathcal{C}_{k}(\lambda)\geq m_{k}2^{-2k}\#\left\{ S\in\mathcal{Q}_{k}:S\subseteq Q,\,B\cap S\neq\emptyset\right\} \label{eq:kismk}
\end{equation}
for
\[
m_{k}=\inf_{x\in Q}\frac{2^{k(2-d)}\left\Vert x\right\Vert ^{2-d}}{P(Q_{k}(x)\cap B\neq\emptyset)}.
\]

To prove the statement of the lemma, by Lemma \ref{lem:subseq prob 0},
it is enough to show that for every $\varepsilon>0$ and for every
subsequence $\left\{ \alpha_{k}\right\} _{k=1}^{\infty}$ of $\mathbb{N}$
we can find a subsequence $\left\{ \beta_{k}\right\} _{k=1}^{\infty}$
of $\left\{ \alpha_{k}\right\} _{k=1}^{\infty}$ such that
\begin{equation}
\lim_{k\rightarrow\infty}P\left(\left|\intop_{Q}\left\Vert x\right\Vert ^{2-d}\mathrm{d}\mathcal{C}_{\beta_{k}}(\lambda)-\gamma\cdot\tau(Q)\right|>\varepsilon\right)=0.\label{eq:celeq-1}
\end{equation}

Let $\varepsilon>0$ be fixed and $\left\{ \alpha_{k}\right\} _{k=1}^{\infty}$
be a subsequence of $\mathbb{N}$. For every $I=\cup_{n=1}^{N}I_{n}$
where $I_{1},\dots,I_{n}$ are disjoint compact intervals with positive
dyadic endpoints we have that
\[
h_{k}^{I}:=2^{-2k}\#\left\{ S\in\mathcal{Q}_{k}:B_{0}(I)\cap S\neq\emptyset\right\} 
\]
converges to $\alpha\lambda(I)$ in probability by Proposition \ref{lem:intunlim}.
Hence, by Lemma \ref{lem:mutual convergence subsequance}, we can
find an event $H$ with $P(H)=1$ and a subsequence $\left\{ \beta_{k}\right\} _{k=1}^{\infty}$
of $\left\{ \alpha_{k}\right\} _{k=1}^{\infty}$ such that $h_{\beta_{k}}^{I}(\omega)$
converges to $\alpha\lambda(I)$ for every $I=\cup_{n=1}^{N}I_{n}$
where $I_{1},\dots,I_{n}$ are disjoint compact intervals with positive
dyadic endpoints for every outcome $\omega\in H$. We can further
assume, by Lemma \ref{lem:ocup cover}, that for every $\omega\in H$
there exist $N^{-}(\omega),N^{+}(\omega)\in\mathbb{N}$ and sets $I^{-}(\omega)=\cup_{n=1}^{N^{-}(\omega)}I_{n}^{-}(\omega)$
where $I_{1}^{-}(\omega),\dots,I_{N^{-}}^{-}(\omega)$ are disjoint
compact intervals with positive dyadic endpoints and $I^{+}(\omega)=\cup_{n=1}^{N^{+}(\omega)}I_{n}^{+}(\omega)$
where $I_{1}^{+}(\omega),\dots,I_{N^{+}}^{+}(\omega)$ are disjoint
compact intervals with positive dyadic endpoints such that
\begin{equation}
I^{-}\subseteq B_{0}^{-1}(Q)\subseteq I^{+}\label{eq:minusplus}
\end{equation}
and
\begin{equation}
\lambda(I^{+}\setminus I^{-})<\varepsilon/2\gamma.\label{eq:epsilonka}
\end{equation}

Let $\omega\in H$ be fixed. Then
\[
\intop_{Q}\left\Vert x\right\Vert ^{2-d}\mathrm{d}\mathcal{C}_{\beta_{k},\omega}(\lambda)\leq M_{\beta_{k}}2^{-2\beta_{k}}\#\left\{ S\in\mathcal{Q}_{\beta_{k}}:S\subseteq Q,\,B_{\omega}\cap S\neq\emptyset\right\} \leq M_{\beta_{k}}\cdot h_{\beta_{k}}^{I^{+}}(\omega)
\]
by (\ref{eq:nagymk}) and (\ref{eq:minusplus}). Thus
\[
\limsup_{k\rightarrow\infty}\intop_{Q}\left\Vert x\right\Vert ^{2-d}\mathrm{d}\mathcal{C}_{\beta_{k},\omega}(\lambda)\leq\lim_{k\rightarrow\infty}M_{\beta_{k}}\cdot h_{\beta_{k}}^{I^{+}}(\omega)=C_{G}([0,1]^{d})\cdot\alpha\cdot\lambda(I^{+})
\]
\begin{equation}
\leq\gamma(\tau(Q)+\varepsilon/2\gamma)=\gamma\tau(Q)+\varepsilon/2\label{eq:limsupocup}
\end{equation}
by Lemma \ref{lem:fk}, the fact that $h_{\beta_{k}}^{I^{+}}(\omega)$
converges to $\lambda(I^{+})$, (\ref{eq:gamma}) and that $\lambda(I^{+})\leq\lambda(B_{0}^{-1}(Q))+\varepsilon/2\gamma=\tau(Q)+\varepsilon/2\gamma$
by (\ref{eq:epsilonka}) and the definition of $\tau$, (\ref{eq:ocup def}).
Similarly we can show that
\begin{equation}
\liminf_{k\rightarrow\infty}\intop_{Q}\left\Vert x\right\Vert ^{2-d}\mathrm{d}\mathcal{C}_{\beta_{k},\omega}(\lambda)\geq\gamma\tau(Q)-\varepsilon/2.\label{eq:liminfocup}
\end{equation}
Hence
\[
\limsup_{k\rightarrow\infty}P\left(\left|\intop_{Q}\left\Vert x\right\Vert ^{2-d}\mathrm{d}\mathcal{C}_{\beta_{k},\omega}(\lambda)-\gamma\tau(Q)\right|>\varepsilon\right)=0
\]
because (\ref{eq:limsupocup}) and (\ref{eq:liminfocup}) hold for
every $\omega\in H$ and $P(H)=1$. Thus (\ref{eq:celeq-1}) follows.
\end{proof}
\begin{thm}
\label{thm:ocup limit}Let $\nu(A)=\intop_{A}\left\Vert x\right\Vert ^{2-d}\mathrm{d}x$
for every Borel set $A\subseteq\mathbb{R}^{d}$. Then $\nu$ is a
locally finite Borel measure, $\nu=\nu_{R}$ and $\mathcal{C}(\nu)=\frac{1}{c(d)}\tau$
almost surely where $c(d)$ is as in (\ref{eq:cd def}).
\end{thm}

\begin{proof}
Clearly $\nu$ is a measure. For every $y\in\mathbb{R}^{d}$ such
that $\left\Vert y\right\Vert >r>0$ for some $r$ we have that $\nu(B(y,r))\leq\left(\left\Vert y\right\Vert -r\right)^{2-d}\lambda(B(y,r))$.
By the argument in the last paragraph of \cite[page 109]{Mattila book}
it follows that
\begin{equation}
\nu(B(0,1))=\intop_{B(0,1)}\left\Vert x\right\Vert ^{2-d}\mathrm{d}x<\infty\label{eq:finite int}
\end{equation}
and
\[
\intop_{B(0,1)}\left\Vert x-y\right\Vert ^{2-d}\mathrm{d}x\mathrm{d}y<\infty.
\]
Hence $\nu$ is a locally finite measure and $\nu_{R}=\nu$ by Proposition
\ref{prop:abs cont decomposition}.

The conditional measure $\mathcal{C}(\nu)$ of $\nu$ on $B$ exists
with respect to $\mathcal{Q}_{k}$ ($k\geq1$) with regularity kernel
$\varphi(x,y)=\left\Vert x-y\right\Vert ^{2-d}$ by Theorem \ref{thm:Brownian cond measure: Main}.
Let $\gamma$ be as in Lemma \ref{lem:ocub box}. It follows from
Lemma \ref{lem:ocub box} that $\mathcal{C}_{k}(\nu)(Q)$ converges
to $\gamma\cdot\tau(Q)$ in probability for every $n\in\mathbb{N}$
for every $Q\in\mathcal{Q}_{n}$ such that $\mathrm{dist}(Q,0)>0$.
By Property \textit{ii.)} and \textit{iv.)} of Definition \ref{def:def of cond meas}
we have that $\mathcal{C}(\nu)(Q)=\gamma\cdot\tau(Q)$ almost surely
for every $n\in\mathbb{N}$ for every $Q\in\mathcal{Q}_{n}$ such
that $\mathrm{dist}(Q,0)>0$. Hence it follows that
\[
\mathcal{C}(\nu)(Q)=\gamma\cdot\tau(Q)
\]
almost surely for every $n\in\mathbb{N}$ for every $Q\in\mathcal{Q}_{n}$
because $Q$ can be written as a countable union of boxes $Q_{i}\in\mathcal{Q}_{n_{i}}$
($i\in\mathbb{N}$) such that $\mathrm{dist}(Q_{i},0)>0$. It follows
from Property \textit{v.)} of Definition \ref{def:def of cond meas}
that $\mathcal{C}(\nu)(\{0\})=0$ almost surely. The set $\{\{0\}\}\cup\bigcup_{n=1}^{\infty}\mathcal{Q}_{n}$
forms a semiring that generates the Borel $\sigma$-algebra hence
it follows by Proposition \ref{lem:charateodory ineq} that
\[
\mathcal{C}(\nu)=\gamma\cdot\tau
\]
almost surely.

By Property \textit{ii.)} of Definition \ref{def:def of cond meas}
it follows that
\[
E\left(\gamma\cdot\tau([0,1]^{d})\right)=E\left(\mathcal{C}(\nu)([0,1]^{d})\right)=\nu([0,1]^{d})=\intop_{[0,1]^{d}}\left\Vert x\right\Vert ^{2-d}\mathrm{d}x.
\]
On the other hand
\[
E\left(\gamma\cdot\tau([0,1]^{d})\right)=\gamma\cdot E\left(\intop_{0}^{\infty}I_{B_{0}(t)\in[0,1]^{d}}\mathrm{d}t\right)=\gamma\cdot c(d)\cdot\intop_{[0,1]^{d}}\left\Vert x\right\Vert ^{2-d}\mathrm{d}x
\]
by \cite[Theorem 3.32]{Peres-Morters-Broanian motion} and \cite[Theorem 3.33]{Peres-Morters-Broanian motion}.
Thus $\gamma=c(d)^{-1}$ because the integral is positive and finite,
see (\ref{eq:finite int}).
\end{proof}
\begin{thm}
\label{thm:conditional measure  of lebesgue}We have that $\mathrm{d}\mathcal{C}(\lambda)=\frac{1}{c(d)}\left\Vert x\right\Vert ^{d-2}\mathrm{d}\tau$
almost surely where $c(d)$ is as in (\ref{eq:cd def}).
\end{thm}

\begin{proof}
The conditional measure $\mathcal{C}(\lambda)$ of $\lambda$ on $B$
exists with respect to $\mathcal{Q}_{k}$ ($k\geq1$) with regularity
kernel $\varphi(x,y)=\left\Vert x-y\right\Vert ^{2-d}$ by Theorem
\ref{thm:Brownian cond measure: Main}. The statement follows from
Theorem \ref{thm:ocup limit} and Property \textit{x{*}.)} of Definition
\ref{def:def of cond meas}.
\end{proof}
\begin{rem}
\label{rem:what is alpha}It turns out in the proof of Theorem \ref{thm:ocup limit}
that $\gamma=c(d)^{-1}$ and $\gamma=\alpha\cdot C_{G}([0,1]^{d})$
by (\ref{eq:gamma}) where $G(x,y)=c(d)\left\Vert x-y\right\Vert ^{2-d}$.
Hence it follows that
\[
\alpha=\frac{1}{C_{\varphi}([0,1]^{d})}
\]
for $\varphi(x,y)=\left\Vert x-y\right\Vert ^{2-d}$.
\end{rem}

\section{Conditional measure on the percolation limit sets of trees\label{sec:Conditional-measure-on}}

In this section we move to other metric spaces than $\mathbb{R}^{d}$
to study the conditional measure of measures on the boundary of a
tree when $B$ is a percolation limit set.
\begin{defn}
Let $T=(V,H,\zeta)$ be a countable graph with vertex set $V$, edge
set $H$ and a special vertex $\{\zeta\}$, that we call the\textit{
root of} $T$. We say that \textit{$T$ is a rooted tree with root
$\zeta$} if the following hold:

i.) for every $v,w\in V$ there exists a unique self-avoiding finite
path from vertex $v$ to vertex $w$,

ii.) the degree of every vertex $v\in V$ is finite,

iii.) let $T_{0}=\{\zeta\}$ and for every positive integer $n$ let
$T_{n}$ be the collection of vertices $v\in V$ such that there exists
a unique self-avoiding path of length $n$ from $\zeta$ to $v$.
Then $\left\{ (v,w)\in V:w\in T_{n+1}\right\} \neq\emptyset$ for
every $v\in T_{n}$ ($n=0,1,\dots$).
\end{defn}

We call an infinite self-avoiding path starting at $\zeta$ a \textit{ray}.
We denote the set of rays by $\partial T$. For every vertex $v\in V$
let $\left|v\right|$ be the unique $n\in\mathbb{N}$ such that $v\in T_{n}$.
For two rays $x,y\in\partial T$ let $x\wedge y\in V$ be the unique
vertex such that both $x$ and $y$ visit $x\wedge y$ and for every
$v\in T_{n}$ for every $n>\left|x\wedge y\right|$ at most one of
$x$ and $y$ visits $v$. We define a metric on $\partial T$ by
\[
d(x,y):=2^{-\left|x\wedge y\right|}.
\]
Then $X=\partial T$ is a compact separable metric space that is homeomorphic
to the Cantor set. For two vertices $v,w\in V$ let $v\wedge w\in V$
be the unique vertex such that the unique self-avoiding paths from
$\zeta$ to $v$ and $\zeta$ to $w$ both visit $v\wedge w$ and
for every $z\in T_{n}$ for every $n>\left|x\wedge y\right|$ at most
one of the unique self-avoiding paths from $\zeta$ to $v$ and $\zeta$
to $w$ visits $v$.

For the rest of this section let $\alpha>0$ be fixed. Let $\varphi_{\alpha}(r)=r^{-\alpha}$
for $r>0$ and so
\[
\varphi_{\alpha}(x,y)=d(x,y)^{-\alpha}=2^{-\alpha\left|x\wedge y\right|}
\]
for $x,y\in\partial T$. Then for every $\delta>0$ we have that (\ref{eq:Kernel restriction})
holds for $c_{2}=(1+2\delta)^{\alpha}$ and $c_{3}=0$. Also (\ref{eq:phi =00003Dinfty})
holds.

For a vertex $v\in V$ we denote by $[v]$ the set of rays of $\partial T$
that goes through vertex $v$. Let
\begin{equation}
\mathcal{Q}_{k}=\left\{ [v]:v\in T_{k},C_{\alpha}([v])>0\right\} .\label{eq:treeqk}
\end{equation}
Then $\mathcal{Q}_{k}$ is a sequence of finite families of Borel
subsets of $\partial T$ such that $Q\cap S=\emptyset$ for $Q,S\in\mathcal{Q}_{k}$,
for all $k\in\mathbb{N}$. It is easy to see that (\ref{eq:unique-subset})
holds.  For every $v\in T_{k}$ we have that $\mathrm{diam}([v])=2^{-k}$,
hence (\ref{eq:diameter_goes_to0}) and (\ref{eq:same size}) hold.
For $v,w\in T_{k}$, $v\neq w$ we have that $\mathrm{dist}([v],[w])\geq2^{-(k-1)}$,
and so
\[
\left\{ S\in\mathcal{Q}_{k}:\max\left\{ \mathrm{diam}(Q),\mathrm{diam}(S)\right\} \geq\delta\cdot\mathrm{dist}(Q,S)\right\} =\{Q\}
\]
for every $Q\in\mathcal{Q}_{k}$ for $1/2<\delta<1$. Thus (\ref{eq:bounded sundivision})
holds for such $\delta$ and $M_{\delta}=1$.

Let
\[
p=2^{-\alpha}
\]
be a probability parameter. For every vertex $v\in V$ let $Y(v)$
be a Bernoulli variable with parameter $p$, such that $Y(v)$ ($v\in T$)
are mutually independent. Let
\[
B_{N}=\bigcap_{k=1}^{N}\bigcup_{\begin{array}{c}
v\in T_{k}\\
Y(v)=1
\end{array}}[v]
\]
be a random compact set. Let
\[
B=\bigcap_{N=1}^{\infty}B_{N}
\]
be a random compact set which we call the \textit{percolation limit
set}.

The following result is due to Lyons, see for example \cite[Theorem 9.17]{Peres-Morters-Broanian motion}
\begin{prop}
\label{lem:tree capacity}For every compact set $K\subseteq\partial T$
\[
C_{\alpha}(K)\leq P(B\cap K\neq\emptyset)\leq2C_{\alpha}(K).
\]
In particular, $P(B\cap K\neq\emptyset)=0$ if $C_{\alpha}(K)=0$.
\end{prop}

By Proposition  \ref{lem:tree capacity} and (\ref{eq:treeqk}) it
follows that (\ref{eq:psoitivity of probability}) and (\ref{eq:lower hitting prob})
holds for $a=1$.

For $v,w\in V$ such that $[v]$ and $[w]$ are disjoint, we have
that
\[
\mathrm{dist}([v],[w])=2^{-k}
\]
for $k=\left|v\wedge w\right|$, hence if further $C_{\alpha}([v])>0$
and $C_{\alpha}([w])>0$ then
\[
\frac{P(B\cap[v]\neq\emptyset,B\cap[w]\neq\emptyset)}{P(B\cap[v]\neq\emptyset)\cdot P(B\cap[w]\neq\emptyset)}
\]
\[
=\frac{p^{k}\cdot P(B\cap[v]\neq\emptyset\mid B_{k}\cap[v\wedge w]\neq\emptyset)\cdot P(B\cap[w]\neq\emptyset\mid B_{k}\cap[v\wedge w]\neq\emptyset)}{p^{k}\cdot P(B\cap[v]\neq\emptyset\mid B_{k}\cap[v\wedge w]\neq\emptyset)\cdot p^{k}\cdot P(B\cap[w]\neq\emptyset\mid B_{k}\cap[v\wedge w]\neq\emptyset)}
\]
\begin{equation}
=p^{-k}=2^{-\alpha k}=2^{-\alpha\left|v\wedge w\right|}=\varphi_{\alpha}(\mathrm{dist}([v],[w])).\label{eq:upper alpha phi}
\end{equation}
Thus (\ref{eq:capacity-independence}) holds for $c=1$ for every
disjoint $Q$ and $S$ and so for every $\delta>0$. We have discussed
above that the assumptions of Section \ref{subsec:Special-assumptions}
are satisfied.

It follows from (\ref{eq:upper alpha phi}) that
\begin{equation}
F(x,y)=\underline{F}(x,y)=\overline{F}(x,y)=\varphi_{\alpha}(x,y),\label{eq:F egyenlo varphi}
\end{equation}
for every $x,y\in\partial T_{0}=\bigcap_{k=1}^{\infty}(\bigcup_{Q\in\mathcal{Q}_{k}}Q)$,
$x\neq y$ and $F(x,y)=0$ otherwise, see (\ref{def:upper and loweer F(x,y)}).
As we established above the conditions of Theorem \ref{cor:Double integral when W exists-summed}
and Theorem \ref{thm:intro non extinction} are satisfied and hence
we can conclude Theorem \ref{thm:cond meas on trees}.
\begin{rem}
\label{rem:X0 on vanish}If $I_{\alpha}(\nu)<\infty$ and $\nu(Q)>0$
then $C_{\alpha}(Q)>0$ because $I_{\alpha}(\nu\vert_{Q})<\infty$
and so $\mathrm{supp}\nu\subseteq\partial T_{0}$. Hence if $\nu(\partial T_{0})=0$
then $\nu_{R}=\nu_{\varphi_{\alpha}R}=0$. On the other hand if $\nu$
is a finite Borel measure such that $\nu(\partial T_{0})=0$ then
$\mathcal{C}_{k}(\nu)$ converges to $0$ in $\mathcal{L}^{1}$ by
Lemma \ref{lem:complem X0 is dead}. Hence for such $\nu$ the conditional
measure $\mathcal{C}(\nu)$ of $\nu$ on $B$ exists with respect
to $\mathcal{Q}_{k}$ ($k\geq1$) with regularity kernel $\varphi_{\alpha}$
and $\mathcal{C}(\nu)=0$ almost surely. Thus the assumption that
$\nu(\partial T\setminus\partial T_{0})=0$ is eliminated in the following
Theorem.
\end{rem}

\begin{thm}
\label{thm:cond meas on trees}Let $\nu$ be a finite, Borel measure
on $\partial T$. Then the conditional measure $\mathcal{C}(\nu)$
of $\nu$ on $B$ exists with respect to $\mathcal{Q}_{k}$ ($k\geq1$)
with regularity kernel $\varphi_{\alpha}$ and if $\tau$ is a finite,
Borel measure on $\partial T$ then
\[
E\left(\int\int f(x,y)\mathrm{d}\mathcal{C}(\nu)(x)\mathrm{d}\mathcal{C}(\tau)(y)\right)=\intop\intop\varphi_{\alpha}(x,y)f(x,y)\mathrm{d}\nu_{R}(x)\mathrm{d}\tau_{R}(y)
\]
for every $f:X\times X\longrightarrow\mathbb{R}$ Borel function with
$\intop\intop\varphi_{\alpha}(x,y)\left|f(x,y)\right|\mathrm{d}\nu_{R}(x)\mathrm{d}\tau_{R}(y)<\infty$.
Moreover,
\[
C_{\varphi_{\alpha}}(\nu)\leq P(\mathcal{C}(\nu)(X)>0)\leq2\overline{C_{\varphi_{\alpha}}}(\nu).
\]
\end{thm}

\begin{rem}
\label{rem:L2 lim}By Theorem \ref{thm:Special existence of conditional expectation finite+vague}
if $A\subseteq X$ is a Borel set such that $I_{\alpha}(\nu\vert_{A})<\infty$
then $\mathcal{C}_{k}(\nu)(A)$ converges to $\mathcal{C}(\nu)(A)$
in $\mathcal{L}^{2}$ (recall that $\nu\vert_{A}(T\setminus\partial T_{0})=0$
by Remark \ref{rem:X0 on vanish}).
\end{rem}

\subsection{Random multiplicative cascade measure as conditional measure }

Throughout this section  for a finite Borel measure $\nu$ on $\partial T$
let $\mathcal{C}(\nu)$ be the conditional measure of $\nu$ on $B$
with respect to $\mathcal{Q}_{k}$ ($k\geq1$) with regularity kernel
$\varphi_{\alpha}$ (which exists by Theorem \ref{thm:cond meas on trees}).
Let
\[
\mathcal{S}_{k}=\left\{ [v]:v\in T_{k}\right\} ,
\]
let $\mathcal{F}_{k}$ be the $\sigma$-algebra generated by the events
$\left\{ Q\cap B_{k}\neq\emptyset\right\} _{Q\in\mathcal{S}_{k}}$
and for a finite Borel measure $\nu$ on $\partial T$ let
\[
\mu_{k}^{\nu}=p^{-k}\nu\vert_{B_{k}}=\sum_{Q\in\mathcal{S}_{k}}P(Q\cap B_{k}\neq\emptyset)^{-1}\cdot I_{Q\cap B_{k}\neq\emptyset}\cdot\nu\vert_{Q}.
\]
Note that $\mu_{n}^{\nu}$ is the conditional measure of $\nu$ on
$B_{n}$ with respect to $\mathcal{S}_{k}$ ($k\geq1$) with regularity
kernel $\varphi(x,y)=1$.
\begin{prop}
\label{lem:cascade}Let $\nu$ be a finite Borel measure on $\partial T$.
There exists a random, finite Borel measure $\mu^{\nu}$ on $\partial T$
with the following properties:

1.) $\mu_{k}^{\nu}$ weakly converges to $\mu^{\nu}$ almost surely,

2.) for a countable collection of Borel sets $A_{n}\subseteq\partial T$
($n\in\mathbb{N}$) we have that $\mu_{k}^{\nu}(A_{n})$ converges
to $\mu^{\nu}(A_{n})$ for every $n\in\mathbb{N}$ as $k$ goes to
$\infty$ almost surely,

3.) $\mu^{\nu}=\mu^{\nu_{R}}$ almost surely for $\nu_{R}=\nu_{\varphi_{\alpha}R}$,

4.) $\mu^{\nu_{\perp}}=0$ almost surely for $\nu_{\perp}=\nu_{\varphi_{\alpha}\perp}$,

5.) $\mu^{\nu}=\sum_{i=1}^{\infty}\mu^{\nu_{i}}$ if $\nu=\sum\nu_{i}$,

6.) $E\left(\mu^{\nu}(A)\right)\leq\nu(A)$ for every Borel set $A\subseteq\partial T$.
\end{prop}

\begin{proof}
It is easy to check that the sequence of random measures $\mu_{k}^{\nu}$
is a $T$-martingale with respect to the filtration $\mathcal{F}_{k}$.
Hence Property \textit{1.)} and \textit{2.)} follows from \cite[Theorem 1]{Kahane-positive martingales}.
Property \textit{6.)} follows from Property \textit{2.)} and the nonnegative
martingale convergence theorem \cite[Theorem 5.2.9]{Durrett}. Property
\textit{5.)} follows from Proposition \ref{prop:SUM lem subseq}.

If $C_{\alpha}(K)=0$ for a compact set $K\subseteq\partial T$ then
$P(B\cap K\neq\emptyset)=0$ by Proposition \ref{lem:tree capacity}.
Hence, via an argument that is similar to the proof of Theorem \ref{prop:dieing singular part},
it can be shown that $\mu_{k}^{\nu_{\perp}}(\partial T)$ converges
to $0$ in probability. It implies that Property \textit{3.)} and
\textit{4.)} hold.
\end{proof}
\begin{rem}
\label{rem:mu is cascade}When $T$ is an $m$-ary tree for some $m\in\mathbb{N}$
and $\nu$ is the uniform measure on $\partial T$ then $\mu^{\nu}$
is the random multiplicative cascade measure with weight variables
$Y(v)/p$ for $v\in T$.
\end{rem}

\begin{lem}
\label{lem:percol bound}Let $\nu$ be a finite Borel measure on $\partial T$
such that $I_{\alpha}(\nu)<\infty$. Then
\[
\lim_{k\rightarrow\infty}\sum_{Q\in\mathcal{Q}_{k}}\frac{1}{P\left(Q\cap B\neq\emptyset\mid Q\cap B_{k}\neq\emptyset\right)}p^{-k}\nu(Q)^{2}=0.
\]
\end{lem}

\begin{proof}
We have that if $\nu(Q)>0$ then $\nu(Q)^{2}/I_{\alpha}(\nu\vert_{Q})\leq C_{\alpha}(Q)$.
It follows from Proposition \ref{lem:tree capacity} that $C_{\alpha}(Q)\cdot p^{-k}\leq P\left(Q\cap B\neq\emptyset\mid Q\cap B_{k}\neq\emptyset\right)$
for every $Q\in\mathcal{Q}_{k}$. Hence
\[
\sum_{Q\in\mathcal{Q}_{k}}\frac{1}{P\left(Q\cap B\neq\emptyset\mid Q\cap B_{k}\neq\emptyset\right)}p^{-k}\nu(Q)^{2}\leq\sum_{Q\in\mathcal{Q}_{k}}\frac{1}{C_{\alpha}(Q)}\nu(Q)^{2}
\]
\[
\leq\sum_{\begin{array}{c}
Q\in\mathcal{Q}_{k}\\
\nu(Q)\neq0
\end{array}}\frac{I_{\alpha}(\nu\vert_{Q})}{\nu(Q)^{2}}\nu(Q)^{2}\leq\sum_{Q\in\mathcal{Q}_{k}}I_{\alpha}(\nu\vert_{Q})\leq\iintop_{d(x,y)\leq2^{-k}}\varphi_{\alpha}(x,y)\mathrm{d}\nu(x)\mathrm{d}\nu(y).
\]
Since $I_{\alpha}(\nu)<\infty$ the statement follows because
\[
\iintop_{d(x,y)=0}\varphi_{\alpha}(x,y)\mathrm{d}\nu(x)\mathrm{d}\nu(y)=0
\]
by Fubini`s theorem.
\end{proof}
\begin{lem}
\label{lem:Efqs}Let
\[
f(Q,S)=\left(\frac{I_{Q\cap B_{k}\neq\emptyset}}{P(Q\cap B_{k}\neq\emptyset)}-\frac{I_{Q\cap B\neq\emptyset}}{P(Q\cap B\neq\emptyset)}\right)\left(\frac{I_{S\cap B_{k}\neq\emptyset}}{P(S\cap B_{k}\neq\emptyset)}-\frac{I_{S\cap B\neq\emptyset}}{P(S\cap B\neq\emptyset)}\right)
\]
for $Q,S\in\mathcal{Q}_{k}$. Then
\[
E(f(Q,S))=\begin{cases}
0 & \mathrm{if\,}Q\neq S\\
p^{-k}(1/P(Q\cap B\neq\emptyset\mid Q\cap B_{k}\neq\emptyset)-1) & \mathrm{if\,}Q=S
\end{cases}
\]
\end{lem}

\begin{proof}
We have that $P(Q\cap B_{k}\neq\emptyset)=P(S\cap B_{k}\neq\emptyset)=p^{k}$.
Since $Q\in\mathcal{Q}_{k}$ it follows that
\begin{equation}
P(Q\cap B\neq\emptyset\mid Q\cap B_{k}\neq\emptyset)=\frac{P(Q\cap B\neq\emptyset)}{p^{k}}>0\label{eq:Q fel Qk}
\end{equation}
and similarly for $S$. Then
\[
E\left(f(Q,S)\mid\mathcal{F}_{k}\right)=\frac{I_{Q\cap B_{k}\neq\emptyset}\cdot I_{S\cap B_{k}\neq\emptyset}}{p^{2k}}\cdot
\]
\[
\cdot E\left(\left(1-\frac{I_{Q\cap B\neq\emptyset}}{P(Q\cap B\neq\emptyset\mid Q\cap B_{k}\neq\emptyset)}\right)\left(1-\frac{I_{S\cap B\neq\emptyset}}{P(S\cap B\neq\emptyset\mid S\cap B_{k}\neq\emptyset)}\right)\mid\mathcal{F}_{k}\right).
\]
Given $\mathcal{F}_{k}$ if $Q\neq S$ then we have that $I_{Q\cap B\neq\emptyset}$
and $I_{S\cap B\neq\emptyset}$ are independent, hence $E\left(f(Q,S)\mid\mathcal{F}_{k}\right)=0$.

Let $q=P(Q\cap B\neq\emptyset\mid Q\cap B_{k}\neq\emptyset)$. If
$Q=S$ then
\[
E\left(f(Q,Q)\mid\mathcal{F}_{k}\right)=\frac{I_{Q\cap B_{k}\neq\emptyset}}{p^{2k}}E\left(\left(1-\frac{I_{Q\cap B\neq\emptyset}}{P(Q\cap B\neq\emptyset\mid Q\cap B_{k}\neq\emptyset)}\right)^{2}\mid\mathcal{F}_{k}\right)
\]
\[
=\frac{I_{Q\cap B_{k}\neq\emptyset}}{p^{2k}}\left((1-1/q)^{2}q+(1-q)\right)=\frac{I_{Q\cap B_{k}\neq\emptyset}}{p^{2k}}\cdot\frac{1-q}{q}.
\]
Hence
\[
E(f(Q,Q))=p^{-k}(1/q-1)=p^{-k}(1/P(Q\cap B\neq\emptyset\mid Q\cap B_{k}\neq\emptyset)-1).
\]
\end{proof}
\begin{prop}
\label{lem:L2 kul}Let $\nu$ be a finite Borel measure on $\partial T$
such that $I_{\alpha}(\nu)<\infty$. Then
\[
\lim_{k\rightarrow\infty}E\left((\mu_{k}^{\nu}(A)-\mathcal{C}_{k}(\nu)(A))^{2}\right)=0
\]
for every Borel set $A\subseteq\partial T$.
\end{prop}

\begin{proof}
If $\nu(A)=0$ then the proof is trivial so we assume that $\nu(A)>0$.
Without the loss of generality we can assume that $A=\partial T$
otherwise we replace $\nu$ by $\nu\vert_{A}$. By Remark \ref{rem:X0 on vanish}
we have that $\nu(Q)=0$ for every $Q\in\mathcal{S}_{k}\setminus\mathcal{Q}_{k}$.
Hence by Lemma \ref{lem:Efqs}
\[
E\left((\mu_{k}^{\nu}(A)-\mathcal{C}_{k}(\nu)(A))^{2}\right)=\sum_{Q\in\mathcal{Q}_{k}}\sum_{S\in\mathcal{Q}_{k}}E(f(Q,S))\nu(Q)\nu(S)
\]
\[
=\sum_{Q\in\mathcal{Q}_{k}}p^{-k}(1/P(Q\cap B\neq\emptyset\mid Q\cap B_{k}\neq\emptyset)-1)\nu(Q)^{2}\leq\sum_{Q\in\mathcal{Q}_{k}}\frac{1}{P\left(Q\cap B\neq\emptyset\mid Q\cap B_{k}\neq\emptyset\right)}p^{-k}\nu(Q)^{2}.
\]
Thus the statement follows from Lemma \ref{lem:percol bound}.
\end{proof}
\begin{prop}
\label{lem:l2 agree}Let $\nu$ be a finite Borel measure on $\partial T$
such that $I_{\alpha}(\nu)<\infty$. Then $\mu^{\nu}=\mathcal{C}(\nu)$
almost surely.
\end{prop}

\begin{proof}
We have that $\mathcal{C}_{k}(\nu)(Q)$ converges to $\mathcal{C}(\nu)(Q)$
in $\mathcal{L}^{2}$ for every $Q\in\mathcal{S}_{k}$ by Remark \ref{rem:L2 lim}.
Hence it follows by Proposition \ref{lem:L2 kul} that $\mu_{k}^{\nu}(Q)$
converges to $\mathcal{C}(\nu)(Q)$ in $\mathcal{L}^{2}$. So by Property
\textit{2.)} of Proposition \ref{lem:cascade} it follows that $\mu^{\nu}(Q)=\mathcal{C}(\nu)(Q)$
for every $Q\in\mathcal{S}_{k}$ almost surely. Since $\cup_{k=1}^{\infty}\mathcal{S}_{k}$
is a semiring that generates the Borel $\sigma$-algebra of $\partial T$
it follows by Proposition \ref{lem:charateodory ineq} that $\mu^{\nu}=\mathcal{C}(\nu)$
almost surely.
\end{proof}
As we noted in Remark \ref{rem:mu is cascade} the random measure
$\mu^{\nu}$ is a multiplicative cascade measure when $\nu$ is the
uniform measure. The next theorem states that $\mu^{\nu}$ is the
conditional measure $C(\nu)$.
\begin{thm}
\label{thm:cond agree cascade}Let $\nu$ be a finite Borel measure
on $\partial T$. Then $\mu^{\nu}=\mathcal{C}(\nu)$ almost surely.
\end{thm}

\begin{proof}
By Proposition \ref{decomposition} there exists a sequence of finite
Borel measures $\nu_{i}$ such that $\nu=\nu_{\varphi_{\alpha}\perp}+\sum_{i=1}^{\infty}\nu_{i}$
and $I_{\alpha}(\nu_{i})<\infty$ for every $i\in\mathbb{N}$. It
follows from Proposition \ref{lem:l2 agree} that $\mu^{\nu_{i}}=\mathcal{C}(\nu_{i})$
for every $i\in\mathbb{N}$ almost surely. It follows from Property
\textit{4.)} of Proposition \ref{lem:cascade} that $\mu^{\nu_{\varphi_{\alpha}\perp}}=0$
almost surely and it follows from Property \textit{vii.)} of Definition
\ref{def:def of cond meas} that $\mathcal{C}(\nu_{\varphi_{\alpha}\perp})=0$
almost surely. Thus it follows from Property \textit{5.)} of Proposition
\ref{lem:cascade} and Property \textit{ix.)} of Definition \ref{def:def of cond meas}
that
\[
\mu^{\nu}=\mu^{\nu_{\varphi_{\alpha}\perp}}+\sum_{i=1}^{\infty}\mu^{\nu_{i}}=\mathcal{C}(\nu_{\varphi_{\alpha}\perp})+\sum_{i=1}^{\infty}\mathcal{C}(\nu_{i})=\mathcal{C}(\nu)
\]
almost surely.
\end{proof}
\selectlanguage{english}%
\begin{center}
$\mathbf{Acknowledgements}$
\par\end{center}

The author was partially supported by the ERC grant (grant number
306494) and by the ERC grant (grant number 772466). The author would
like to thank Xiong Jin, Zemer Kosloff, Georg Berschneider, Yuval
Peres, David Simmons, Mike Hochman, Benjamin Weiss, Kenneth Falconer,
B\'alint T\'oth and G\'abor Pete for the many useful discussion.
\selectlanguage{british}%

\end{document}